\documentclass[leqno,11pt]{amsart}
\usepackage[top=1.25in, bottom=1.25in, left=1in, right=1in]{geometry}
\usepackage{amssymb,amsmath,latexsym,amsfonts,amsbsy,amsthm,dsfont,mathrsfs}
\usepackage{hyperref}
\usepackage{color}
\usepackage{graphicx}
\usepackage{comment}
\usepackage{enumerate}
\usepackage{xcolor,soul}
\usepackage{booktabs}
\usepackage{subcaption}
\usepackage{tikz}
\usetikzlibrary{arrows.meta, calc}

\definecolor{pku}{RGB}{139,0,18}

\let\pa=\partial
\let\f=\frac

\let\pa=\partial
\let\na=\nabla

\let\al=\alpha
\let\b=\beta
\let\e=\varepsilon
\let\d=\delta
\let\D=\Delta
\let\g=\gamma
\let\G=\Gamma
\let\ka=\kappa
\let\lam=\lambda
\let\Lam=\Lambda
\let\om=\omega
\let\Om=\Omega
\let\va=\varphi

\let\r=\rho
\let\s=\sigma
\let\th=\theta

\let\z=\zeta

\DeclareMathOperator*{\esssup}{esssup}

\newcommand{\beq}{\begin{equation}}
\newcommand{\eeq}{\end{equation}}
\newcommand{\beqo}{\begin{equation*}}
\newcommand{\eeqo}{\end{equation*}}
\newcommand{\ben}{\begin{eqnarray}}
\newcommand{\een}{\end{eqnarray}}
\newcommand{\beno}{\begin{eqnarray*}}
\newcommand{\eeno}{\end{eqnarray*}}

\numberwithin{equation}{section}

\newtheorem{thm}{Theorem}[section]
\newtheorem{lem}{Lemma}[section]
\newtheorem{cor}{Corollary}[section]
\newtheorem{prop}{Proposition}[section]

\theoremstyle{definition}

\theoremstyle{remark}

\newtheorem{case}{Case}
\newtheorem{rmk}{Remark}[section]

\newcommand{\pv}{\mathrm{p.v.}}

\newcommand{\CC}{\mathcal{C}}

\newcommand{\CI}{\mathcal{I}}

\newcommand{\CS}{\mathcal{S}}

\newcommand{\CT}{\mathcal{T}}

\newcommand{\CM}{\mathcal{M}}

\newcommand{\R}{\mathbb{R}}

\newcommand{\BC}{\mathbb{C}}
\newcommand{\BW}{\mathbb{W}}
\newcommand{\BR}{\mathbb{R}}
\newcommand{\BZ}{\mathbb{Z}}
\newcommand{\BT}{\mathbb{T}}
\newcommand{\BH}{\mathbb{H}}

\newcommand{\BM}{\mathbb{M}}

\newcommand{\bfR}{\mathbf{R}}

\allowdisplaybreaks

\begin{document}
\title[Uniformly Rotating Vortex Patches with 90-Degree Corners]{Uniformly Rotating Vortex Patches with 90-Degree Corners}
\author{De Huang}
\address{School of Mathematical Sciences, Peking University, Beijing 100871, China}
\email{dhuang@math.pku.edu.cn}

\author{Jiajun Tong}
\address{Beijing International Center for Mathematical Research, Peking University, Beijing 100871, China}
\email{tongj@bicmr.pku.edu.cn}

\author{Xiaopeng Zheng}
\address{School of Mathematical Sciences, Peking University, Beijing 100871, China}
\email{zhengxiaopeng@stu.pku.edu.cn}

\begin{abstract}
For the 2-D incompressible Euler equation, we prove the existence of $m$-fold symmetric uniformly rotating vortex patches with 90-degree corners for all $m\geq 12$.
Properties of these patches and the induced stationary flows in the co-rotating frame are characterized.
Our construction is based on a novel fixed-point method.
\end{abstract}

\date{\today}
\maketitle

\tableofcontents

\section{Introduction}

\subsection{Problem formulation and main results}
\label{sec: problem formulation}

Consider the vorticity-stream formulation of the 2-D incompressible Euler equation in $\BR^2$:
\beq
\pa_t \om + u \cdot \na \om = 0,\quad u= -\na^\perp\varphi,
\label{eqn: 2D incompressible Euler}
\eeq
where $u = u(x,t)$ and $\om = \om(x,t)$ denote the velocity and the vorticity, respectively, and where $\varphi$ solves
\beq
-\D \varphi(x,t) = \om(x,t)\mbox{ in }\BR^2,\quad |\na \varphi|\to 0\mbox{ as }|x|\to +\infty.
\label{eqn: stream function of 2D incompressible Euler}
\eeq
In this paper, we shall study its uniformly rotating vortex patch solutions.
A bounded simply connected domain $D_0\subset \BR^2$ is called a \emph{uniformly rotating vortex patch} or a \emph{(rotating) V-state} to the 2-D incompressible Euler equation in $\BR^2$, if for some angular velocity $a \in \BR$,
\[
\om(x,t) := \mathds{1}_{D_t}(x),\quad D_t := \begin{pmatrix}
\cos (a t) & -\sin (a t)\\
\sin (a t) & \cos (a t)
\end{pmatrix} D_0,
\]
gives a weak solution to \eqref{eqn: 2D incompressible Euler} and \eqref{eqn: stream function of 2D incompressible Euler} in the Yudovich sense \cite{Yudovich1963} (see Section \ref{sec: related works}).
Here $\mathds{1}_{D_t}(x)$ denotes the characteristic function of the set $D_t$.
In other words, this solution $\om = \om(x,t)$ to the 2-D Euler equation corresponds to a vortex patch rotating with constant angular velocity with its shape being time-invariant.
The goal of this paper is to rigorously construct V-states that have 90-degree corners along their boundaries, and also characterize their  properties.
We will discuss the historical background of this problem and related topics in Section \ref{sec: related works} below.

In order to formulate this problem mathematically, we assume that $D_0\subset \BR^2$ is a V-state with angular velocity $a\in \BR$, and that $\pa D_0$ has only one connected component.
Let $\phi = \phi(x)$ solve
\beq
-\D \phi = \om(x,0) = \mathds{1}_{D_0}(x)\mbox{ in }\BR^2,\quad |\na \phi|\to 0\mbox{ as }|x|\to +\infty.
\label{eqn: equation for phi}
\eeq
We shall call $\phi$ the \textit{stream function} in $\BR^2$; it can be determined up to an additive constant.
Define the \textit{modified stream function} by
\begin{equation}\label{eqt: def of phi_dagger}
\phi_\dag := \phi + \f{a}2 |x|^2.
\end{equation}
In the \emph{co-rotating frame}, which is the coordinate rotating with the patch, it induces a stationary flow field
\beq
v = -\na^\perp \phi_\dag  = u - a x^\perp.
\label{eqn: flow field in the corotating frame}
\eeq
Here $\na^\perp \phi_\dag = (\na \phi_\dag)^\perp := (-\pa_{x_2}\phi_\dag, \, \pa_{x_1}\phi_\dag)$.
Since $D_0$ is a V-state, it is time-invariant under the flow $v$ in the co-rotating frame.
By the classic theory \cite{MajdaBertozzi2001}, $\phi_\dag$ must be constant along $\pa D_0$.
As we are pursuing V-states with 90-degree corners, without loss of generality, we may assume that
\[
\mbox{$(1,0)\in \pa D_0$, and $\pa D_0$ admits a right angle at $(1,0)$.}
\]
This gives $\phi_\dag \equiv \phi_\dag(1,0)$ along $\pa D_0$,
i.e.,
\beq
\phi(x)-\phi(1,0) = \f{a}{2}\big(1-|x|^2\big) \quad \mbox{along }\pa D_0.
\label{eqn: constraint satisfied by phi along the patch boundary}
\eeq
Moreover, by the steadiness of $D_0$ in the co-rotating frame and the fact that $\pa D_0$ has a right angle at $(1,0)$, $v(1,0) = (0,0)$.
This together with \eqref{eqn: flow field in the corotating frame} implies
\beq
a = -\pa_{x_1}\phi(1,0),\quad 0 = \pa_{x_2}\phi(1,0).
\label{eqn: formula for angular velocity}
\eeq
Note that, using the polar coordinate notation, the first formula in \eqref{eqn: formula for angular velocity} can be written as $a= -(\f{1}{r}\pa_r \phi)|_{(1,0)}$, which agrees with the heuristics of the angular velocity.
Therefore, \eqref{eqn: constraint satisfied by phi along the patch boundary} becomes
\beq
\phi(x) - \phi(1,0)
= -\f12 \pa_{x_1}\phi(1,0)\big(1-|x|^2\big)\quad \mbox{along }\pa D_0.
\label{eqn: constraint satisfied by phi along the patch boundary final}
\eeq

In this paper, we seek such V-states $D_0$ with the following additional properties:
\begin{enumerate}
\item[(A)] $D_0$ can be represented in polar coordinates, i.e., for some bounded $f:\BT\to (0,+\infty)$,
\beq
D_0 = D_0(f) := \big\{(r\cos\th,r \sin \th)\in \BR^2:\,\th\in \BT,\,r\in [0,f(\th))\big\}.
\label{eqn: form of D_0}
\eeq
Here $\BT:= \BR/(2\pi \BZ)$, i.e., $\BT$ is $[-\pi,\pi]$ with the two end-points identified.

\item[(B)] $D_0$ has $m$-fold symmetry with respect to the origin for some $m\geq 2$ $(m\in \BZ_+)$, i.e.,
\[
\mbox{with }R_m:=
\begin{pmatrix}
\cos \f{2\pi}{m} & -\sin \f{2\pi}{m}\\
\sin \f{2\pi}{m} & \cos \f{2\pi}{m}
\end{pmatrix},
\mbox{ it holds that $R_m x\in D_0$ whenever $x\in D_0$.}
\]
\item[(C)] $D_0$ is symmetric with respect to the horizontal axis, i.e., $(x_1,-x_2)\in D_0$ whenever $(x_1,x_2)\in D_0$.
Note that under this assumption, the second equation of \eqref{eqn: formula for angular velocity} holds automatically.
\end{enumerate}
We are thus motivated to define with $m\geq 2$ $(m\in \BZ_+)$ that
\beq
\begin{split}
\BM_0:= \big\{f:\BT\to (0,1]:&\; f\mbox{ is $\f{2\pi}{m}$-periodic and even,}\\
&\;\mbox{$f$ is non-increasing on $[0,\f{\pi}{m}]$,} \\
&\;\mbox{$f<1$ in $(0,\f{2\pi}{m})$, and } f(0) = \lim_{\th\to 0}f(\th) = 1
\big\}.
\end{split}
\label{eqn: function set M_0}
\eeq
Here the evenness of $f = f(\th)$ is defined with respect to $\th = 0$.
Our goal is to find suitable $f = f(\th)\in \BM_0\cap C(\BT)$, such that if we introduce $D_0 = D_0(f)$ by \eqref{eqn: form of D_0} and define $\phi$ by \eqref{eqn: equation for phi}, then \eqref{eqn: constraint satisfied by phi along the patch boundary final} holds, which means that $D_0$ is a V-state with the angular velocity $a= -\pa_{x_1}\phi(1,0)$, and moreover, $\pa D_0$ admits a 90-degree corner at $(1,0)$.
Let us remark that the monotonicity of $f$ assumed in \eqref{eqn: function set M_0} will play a crucial role in our construction, which will be clear later.

Our first result is that, for all suitably large $m\in \BZ_+$, there exist such $m$-fold symmetric V-states $D_0$ with 90-degree corners.
We also characterize their geometric properties as well as their angular velocities.

\begin{thm}[Existence of V-states with 90-degree corners and their  characterizations]
\label{thm: main existence theorem}
Let $\BM_0$ be defined as in \eqref{eqn: function set M_0}.
There exists a universal integer $2\leq m_0\leq 12$, such that for each $m \geq m_0$ $(m\in \BZ_+)$, there exists $f\in \BM_0\cap C(\BT)$ that depends on $m$, such that the following holds.

\begin{enumerate}[(i)]
\item \label{statement: rotating patch}
If we define $D_0 = D_0(f)$ by \eqref{eqn: form of D_0} and define $\phi$ by \eqref{eqn: equation for phi}, then \eqref{eqn: constraint satisfied by phi along the patch boundary final} holds.

\item \label{statement: lower and upper bound for f}
For any $\th\in [0,\f{2\pi}{m}]$,
\[
\max\left\{e^{-\f{M}{m}},\, e^{-\f{\Lam}{m}(\pi^2-(m\th-\pi)^2)}\right\}
\leq
f(\th)\leq e^{-\f{\lam}{m}(\pi^2-(m\th-\pi)^2)},
\]
where $M>0$ and $\Lam\geq 1\geq \lam>0$ are universal constants independent of $m$.

\item \label{statement: local analyticity of f and piecewise C^1}
$f\in C^1([0, \f{2\pi}{m}])$, and $f$ is analytic in $(0, \f{2\pi}{m})$.

\item \label{statement: upper bound for f'}
$f'(\th)<0$ for $\th\in(0,\f{\pi}{m})$, $f'(\f{\pi}{m})=0$, and $\|f'\|_{L^\infty(\BT)}\leq C$, where $C>0$ is a universal constant independent of $m$.

\item \label{statement: derivative at the end points}
It holds that
\begin{align*}
f'_+(0) := \lim_{\th\to 0^+} \f{f(\th)-f(0)}{\th-0} = -1,\quad &
f'_-(0) := \lim_{\th\to 0^-} \f{f(\th)-f(0)}{\th-0} = 1,
\\
f'(0^+):= \lim_{\th\to 0^+}f'(\th) = -1,\quad &
f'(0^-):= \lim_{\th\to 0^-}f'(\th) = 1.
\end{align*}
Moreover,
\[
\limsup_{\th\to 0}\left|\left(1+\f{\ln f(\th)}{|\th|}\right) \ln |m\th| \right| \leq C
\]
for some universal constant $C>0$ independent of $m$.
\end{enumerate}
In other words, for any such $f$, the resulting $D_0(f)$ gives a uniformly rotating vortex patch with $m$-fold symmetry of the 2-D incompressible Euler equation,
which has 90-degree corners along its boundary at exactly $m$ points $\{(\cos\f{2k\pi}{m},\sin \f{2k\pi}{m}):\, k = 0,1,\cdots,m-1\}$.

Moreover, the angular velocity of the constructed rotating vortex patch is $a=-\pa_{x_1}\phi(1,0)>0$, which enjoys the estimate
\beq
-\f{M}{m}\leq a-\f12 \leq -\f{C\lam}{m}.
\label{eqn: estimate for angular velocity}
\eeq
Here $M$ and $\lam$ are the universal constants introduced above, and $C>0$ is another universal constant independent of $m$.

In addition, for all $m\geq 12$, the above statements hold with $M=4$.

\begin{figure} \centering
    \begin{subfigure}[b]{0.45\textwidth}

    \begin{tikzpicture}
        \node[inner sep=0pt] (img)
        {
        \includegraphics[width=1\textwidth]{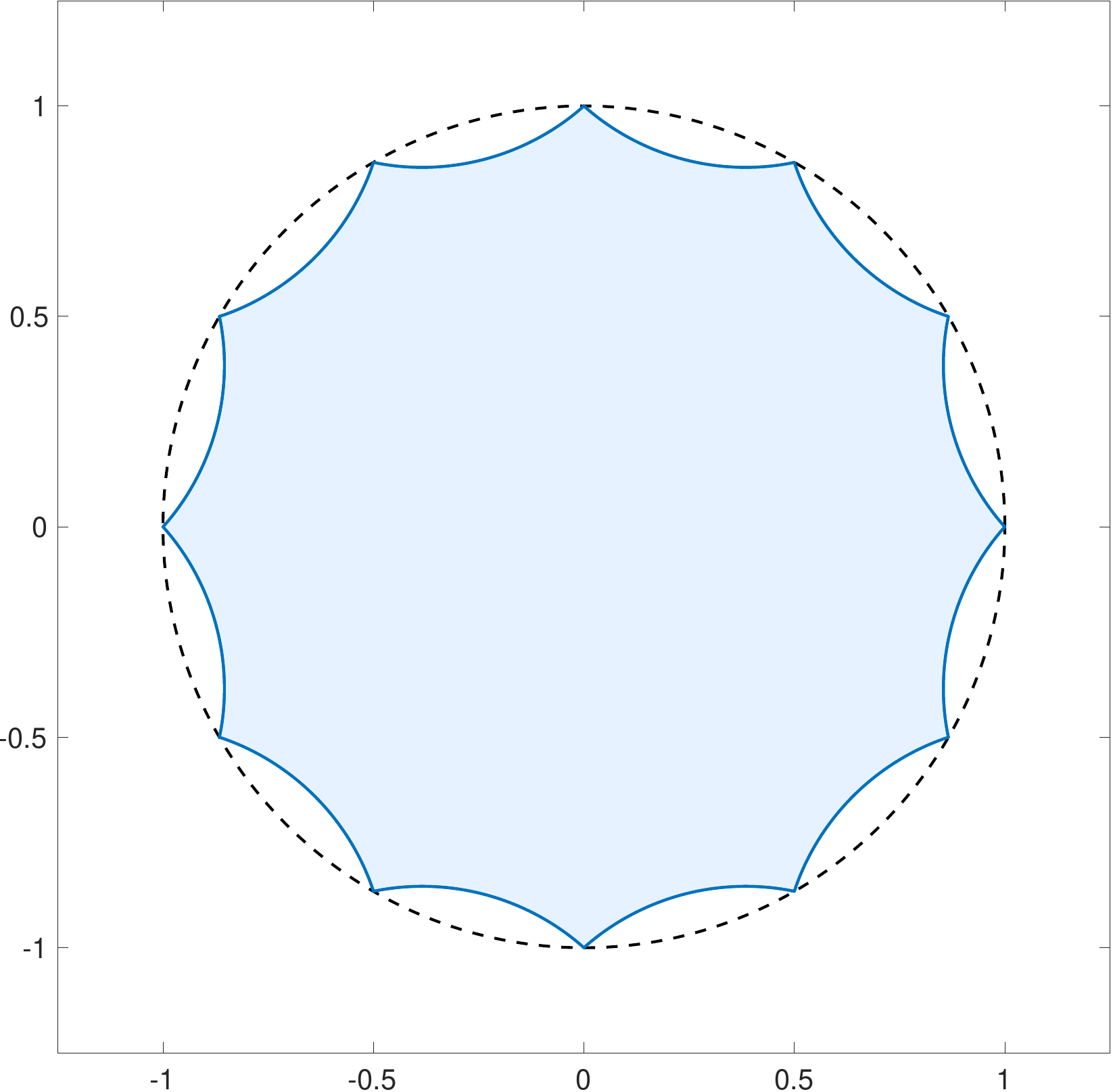}
        };
        \node[align=center, font=\bfseries]
        at($(img.center) + (0.0275\textwidth,0.02\textwidth)$){$\om \equiv 1$};
    \end{tikzpicture}
    \caption{\small A 12-fold V-state with 90-degree corners}
    \end{subfigure}\quad
    \begin{subfigure}[b]{0.45\textwidth}         \includegraphics[width=1\textwidth]{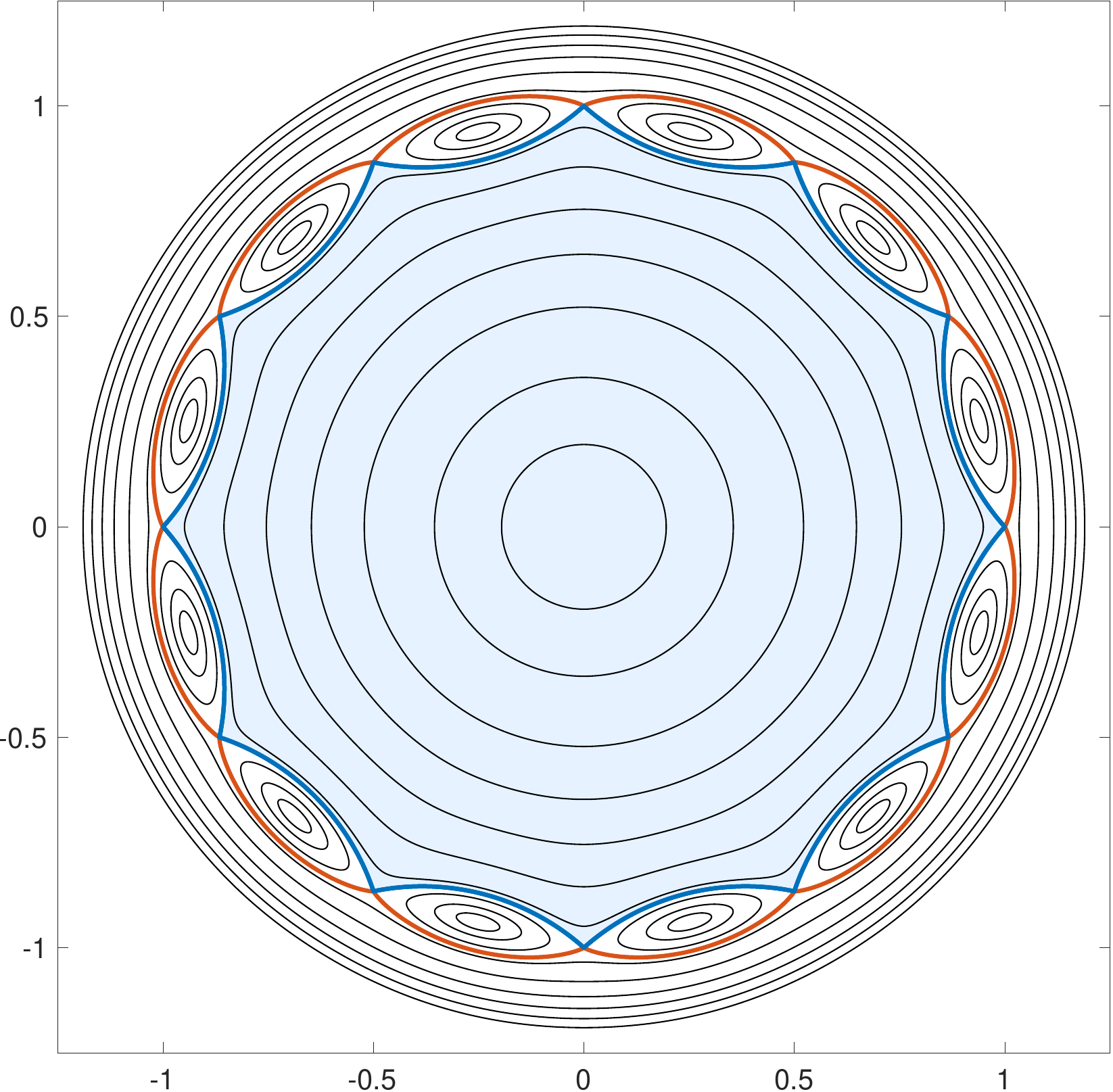}
        \caption{\small Level sets of $\phi_\dagger$}
    \end{subfigure}
    \caption[12-fold_intro]{\small
    (A) A numerically computed 12-fold uniformly rotating vortex patch (V-state) with 90-degree corners. The dashed black curve represents the unit circle; and the light blue region indicates the interior of the patch $D_0$ where $\om\equiv 1$.
    (B) Level sets of the corresponding modified stream function $\phi_\dagger$ (defined in \eqref{eqn: modified stream function repeat}). The solid blue curve is the boundary of the V-state and also the part of $\s_0$ (a special level set of $\phi_\dag$ defined in \eqref{eqn: modified stream function repeat} and \eqref{eqn: general level set in main thm}) inside the unit circle; the solid red curve is the part of $\s_0$ outside the unit circle; and the solid black curves represent selected level sets of $\phi_\dagger$ other than $\s_0$.
    }
    \label{fig:12-fold_intro}
\end{figure}

\begin{rmk}
For simplicity, we only stated the properties of $f=f(\th)$ for $\theta$ lying in the intervals such as $(0,\f{2\pi}{m})$ and $(0,\f{\pi}{m})$, etc., but one can easily extend them to their corresponding full ranges by the $\f{2\pi}{m}$-periodicity and evenness of $f$ which is assumed in the definition of $\BM_0$.
\end{rmk}

\begin{rmk}
\label{rmk: claim optimality m = 12}
The claim $m_0\leq 12$ means that all the statements of Theorem \ref{thm: main existence theorem} hold at least for all $m\geq 12$, i.e., we have proved the existence of $m$-fold symmetric rotating V-states with 90-degree corners for all $m\geq 12$.
This explicit threshold is achieved as a result of carefully choosing the parameters to balance several constraints arising in the argument. This will be explained in Remark \ref{rmk: optimality of m=12 explained} after we finish proving Theorem \ref{thm: main existence theorem}.
\end{rmk}

\begin{rmk}
Our results on the regularity of the patch boundary agree with those established in \cite{WangZhangZhou2026} for general rotating vortex patches with Lipschitz boundary, and the characterizations of the patch boundary near the corners also agree with the formal analysis in \cite{Overman1986}.
Moreover, our estimate for the angular velocity $a$ as well as the upper and lower bounds for $f$ agree with those in \cite{Park2022} derived for general simply-connected $m$-fold symmetric V-states.
\end{rmk}

\begin{rmk}
The proof of Theorem \ref{thm: main existence theorem} (and other results of this paper) is purely analytic and involves no computer assistance.
Yet, as an illustration, a $12$-fold V-state with 90-degree corners, which is computed numerically, is shown in Figure \ref{fig:12-fold_intro}(A).
We will explain our numerical method and present more numerical results in Appendix \ref{sec: numerical}.
It is noteworthy that, although the uniqueness of the $m$-fold symmetric V-states with 90-degree corners has not been rigorously established, we can robustly obtain the presented numerical solutions in the simulations with no other solutions found, which suggests the uniqueness in a suitable class of patches.
See the discussions in Appendix \ref{sec: numerical}.
\end{rmk}
\end{thm}

For a uniformly rotating vortex patch $D_0$ obtained in Theorem \ref{thm: main existence theorem}, its boundary curve $\pa D_0$ lies inside the unit disk and forms a streamline in the co-rotating frame connecting the neighboring corners.
In the next theorem, we show that there is another streamline lying outside the unit disk that also connects the neighboring corners, and moreover, these two streamlines together form the level set of the modified stream function $\phi_\dag$ with the value $\phi_\dag(1,0)$.

\begin{thm}[A streamline outside the unit disk that connects the neighboring corners]
\label{thm: heteroclinic trajectory}
Let $m_0\in \BZ_+$ be given by Theorem \ref{thm: main existence theorem}.
Given any integer $m\geq m_0$, fix an $f\in \BM_0\cap C(\BT)$ such that it satisfies all the conditions of Theorem \ref{thm: main existence theorem}.
Let $D_0 = D_0(f)$ and $\phi$ be defined as in \eqref{eqn: form of D_0} and \eqref{eqn: equation for phi}, respectively.
Define the modified stream function $\phi_\dag$ as (see \eqref{eqt: def of phi_dagger} and \eqref{eqn: formula for angular velocity})
\beq
\phi_\dag(x) :=\phi(x) - \frac{1}{2}\pa_{x_1}\phi(1,0)|x|^2.
\label{eqn: modified stream function repeat}
\eeq
Then the following holds.
\begin{itemize}
\item
There exists a unique $\hat f:\BT\to [1,+\infty)$ with $1/\hat f \in \BM_0 \cap C(\BT)$ corresponding to $f$, such that
\begin{enumerate}[(i)]
\item For all $\th \in \BT$,
\[
\phi_\dag\left(\hat{f}(\th)\cos\th,\, \hat{f}(\th)\sin\th \right) \equiv \phi_\dag(1,0) = \phi(1,0) - \frac{1}{2}\pa_{x_1}\phi(1,0).
\]

\item $\hat{f}\in C^1([0, \f{2\pi}{m}])$, and $\hat{f}$ is analytic in $(0, \f{2\pi}{m})$.

\item $\hat{f}'(\th)>0$ for $\th\in(0,\f{\pi}{m})$, and $\hat{f}'(\f{\pi}{m})=0$.

\item
It holds that
$\hat{f}'_+(0) = \hat{f}'(0^+) = 1$, $\hat{f}'_-(0) = \hat{f}'(0^-) = -1$,
and
\[
\limsup_{\th\to 0}\left|\left(1-\f{\ln \hat{f}(\th)}{|\th|}\right) \ln |m\th| \right| \leq C
\]
for some universal constant $C>0$ independent of $m$.
\end{enumerate}

\item
Denote $\hat{D}_0:=\{(r\cos\th,r\sin\th):\,\th\in \BT,\, r\in [0,\hat{f}(\th))\}$.
Then
\[
\phi_\dag(x) - \phi_\dag(1,0)
\begin{cases}
< 0, & \mbox{if }x\in \hat{D}_0\setminus\overline{D_0},\\
= 0, & \mbox{if }x\in \pa D_0\cup \pa \hat{D}_0,\\
> 0, & \mbox{otherwise.}
\end{cases}
\]
In particular, if we denote
\beq
\s_0:=\big\{x\in \BR^2:\;\phi_\dag(x) - \phi_\dag(1,0) = 0\big\},
\label{eqn: general level set in main thm}
\eeq
and let $B_1$ denote the unit open disk centered at the origin, then
\begin{align*}
\s_0\cap \overline{B_1}&= \pa D_0 =
\big\{(f(\th)\cos\th,f(\th)\sin \th):\, \th\in\BT\big\} ,\\
\s_0\cap B_1^c &= \pa \hat{D}_0 = \big\{(\hat{f}(\th)\cos\th,\hat{f}(\th)\sin \th):\, \th\in\BT\big\} .
\end{align*}
\end{itemize}

\begin{rmk}\label{rmk: 4 90-degree angles at corner}
The theorem states that, there are two curves $\pa D_0$ and $\pa \hat{D}_0$, which lie inside and outside the unit circle respectively, representing the only two streamlines that connect the neighboring corners of the patch.
In fact, they together form the level set of $\phi_\dag$ with the value $\phi_\dag(1,0)$.
They meet only at the corner points $\{(\cos \f{2k\pi}{m},\,\sin \f{2k\pi}{m}):\,k = 0,1,\cdots, m-1\}$ and form four $90$-degree angles at each corner point.
It is not difficult to further justify that, in a neighborhood of each corner point, $\pa D_0\cup \pa \hat{D}_0$ can be re-parameterized as a union of two $C^1$-curves which intersect perpendicularly.

As an illustration, for the numerically computed 12-fold V-state with 90-degree corners shown in Figure \ref{fig:12-fold_intro}(A), we plot some level sets of its modified stream function $\phi_\dag$ in Figure \ref{fig:12-fold_intro}(B).
In particular, $\pa D_0$ and $\pa \hat D_0$ are plotted there as the solid blue curve and the solid red curve, respectively.
\end{rmk}
\end{thm}

In the final theorem, we describe the stationary flows in the co-rotating frame induced by the rotating vortex patches.

\begin{thm}[Characterizations of the flow field in the co-rotating frame]
\label{thm: flow field}
Let $m_0$, $m$, $f$, $D_0$, and $\phi_\dag$ be given as in the condition of Theorem \ref{thm: heteroclinic trajectory}.
Denote the flow field in the co-rotating frame by $v := -\na^\perp \phi_\dag$  (see \eqref{eqn: flow field in the corotating frame}).
For $x\neq 0$, define $v_r(x) := v(x)\cdot \f{x}{|x|}$ and $v_\th(x) := v(x)\cdot \f{x^\perp}{|x|}$.

\begin{itemize}
\item
Denote $\mathcal{S}:= \{(\cos \f{2k\pi}{m},\,\sin \f{2k\pi}{m}):\,k = 0,1,\cdots, m-1\}$, which is the set of singular points (corners) along $\pa D_0$.
Then \begin{enumerate}[(i)]
\item $v(x) = (0,0)$ for $x\in \{(0,0)\}\cup \CS$.

\item For $x\neq (0,0)$,
\beq
v_r(x)
\begin{cases}
<0,&\mbox{if }x\in \bigcup_{k = 0}^{m-1} \left\{\left(r\cos \th,\,r\sin\th\right):\, r>0,\, \th \in \big(\f{2k\pi}{m},\f{(2k+1)\pi}{m}\big)\right\},\\
>0,&\mbox{if }x\in \bigcup_{k = 0}^{m-1} \left\{\left(r\cos \th,\,r\sin\th\right):\, r>0,\, \th \in \big(\f{(2k+1)\pi}{m},\f{(2k+2)\pi}{m}\big)\right\},\\
=0,&\mbox{otherwise.}
\end{cases}
\label{eqn: sign of v_r}
\eeq

\item \label{statement: v_theta is negative outside B_1}
$v_\theta(x) < 0$ for any $x\in \BR^2$ with $|x|\geq 1$ and $x\not \in \mathcal{S}$.
\item \label{statement: v_theta is negative near corners inside D_0}
$v_\theta(x) > 0$ for any
\beq
\begin{split}
x\in &\; \bigcup_{k = 0}^{m-1} \left\{\left(r\cos \f{(2k+1)\pi}{m},\,r\sin\f{(2k+1)\pi}{m}\right):\, r\in \left(0,f\left(\f{\pi}{m}\right)\right]\right\}\\
&\; \cup \bigcup_{k = 0}^{m-1} \left\{(r\cos \th,\,r\sin\th)\in D_0:\, r\cos\left(\th-\f{2k\pi}{m}\right)\in [r_*,1)
\right\},
\end{split}
\label{eqn: v_theta is positive in suitable regions}
\eeq
where $r_*\in (0,1)$ is a universal constant independent of $m$.
\end{enumerate}

\item
Denote the set of critical points of $\phi_\dag$ by $\CC := \{x\in\BR^2:\,\na \phi_\dag(x) = (0,0)\}$, which is also the set of stagnation points of $v$, i.e., $\CC = \{x\in \BR^2:\, v(x) = (0,0)\}$.
Then $\CC$ has $m$-fold symmetry with respect to the origin, and it is a finite subset of \beq
\begin{split}
\mathcal{S} &\; \cup \bigcup_{k = 0}^{m-1} \left\{\left(r\cos\f{2k\pi}{m},\, r\sin \f{2k\pi}{m}\right):\, r\in [0,r_*)\right\}\\
&\; \cup \bigcup_{k = 0}^{m-1}\left\{\left(r\cos\f{(2k+1)\pi}{m},\, r\sin \f{(2k+1)\pi}{m}\right):\,r\in \left(f\left(\f{\pi}{m}\right),1\right)\right\},
\end{split}
\label{eqn: location of stagnation points}
\eeq
where $r_*\in (0,1)$ is the universal constant introduced above.

\item
If additionally, \beq
2\cos \f{2\pi}{m} + 1> e^{M/m},
\label{eqn: condition for folding one petal}
\eeq
where $M>0$ is the universal constant introduced in Theorem \ref{thm: main existence theorem}, then $v_\th(x) > 0$ for all $x\in \overline{B_{r_m}}\setminus \{(0,0)\}$, where $r_m:=f(\f{\pi}{m})\cos\f{\pi}{m}$.

In particular, \eqref{eqn: condition for folding one petal} holds for all $m\geq 12$ with $M = 4$, for which the existence and characterization of the $m$-fold symmetric V-states with 90-degree corners has been established in Theorem \ref{thm: main existence theorem}.

\item
Lastly, there exists a universal $m_1\in \BZ_+$, such that whenever $m\geq m_1$ ($m\in \BZ_+$), $v_\th(x)>0$ for all $x\in D_0\setminus \{(0,0)\}$, which is an improvement over \eqref{eqn: v_theta is positive in suitable regions}.
Moreover, $\CC$ is a finite subset of (cf.\;\eqref{eqn: location of stagnation points})
\[
\CS\cup \{(0,0)\}\cup \bigcup_{k = 0}^{m-1} \left\{\left(r\cos\f{(2k+1)\pi}{m},\, r\sin \f{(2k+1)\pi}{m}\right):\,r\in \left(f\left(\f{\pi}{m}\right),1\right)\right\}. \]
\end{itemize}

\begin{rmk}
The results in this theorem are motivated by what is illustrated in Figure \ref{fig:12-fold_intro}(B).
In \eqref{eqn: sign of v_r}, we fully characterize the sign of the radial velocity $v_r$ in $\BR^2\setminus \{(0,0)\}$, which agrees with the numerical observation.
However, the sign of $v_\th$ (i.e., the direction of the rotational flow) is more complicated to study.
We have proved that $v_\th<0$ outside the unit disk, $v_\th>0$ in some regions inside $D_0$ and yet close to the corners, and moreover, when $m\geq m_1$, we can show that $v_\th > 0$ all over $D_0\setminus \{(0,0)\}$.
When $m$ is not so large but the assumption \eqref{eqn: condition for folding one petal} is satisfied, we can prove $v_\th>0$ at least in a deleted neighborhood of the origin.
This can be treated as a characterization of $v_\th$ away from the corners, which is expected in view of the numerical result.
For small $m$, we have not ruled out the possible presence of stagnation points of $v$ along the line segments $\{(r\cos\f{2k\pi}{m},\, r\sin \f{2k\pi}{m}):\,r\in (0,1)\}$ $(k=0,\cdots,m-1)$, as this seems to depend on detailed geometric properties of $D_0$.
Regarding the flow in $B_1\setminus D_0$, we can show that there are at most finitely many stagnation points of $v$ and they are only located along the line segments $\{(r\cos\f{(2k+1)\pi}{m},\, r\sin \f{(2k+1)\pi}{m}):\,r\in (f(\f{\pi}{m}),1)\}$ $(k=0,\cdots,m-1)$.
In fact, there is at least one along each of these line segments, which corresponds to the global minima of $\phi_\dag$; see Theorem \ref{thm: heteroclinic trajectory} and Figure \ref{fig:12-fold_intro}(B).
Nevertheless, it is non-trivial to confirm that there is only one stagnation point in $\{(r\cos\th,r\sin\th)\in \hat{D}_0\setminus \overline{D_0}:\;\th\in (0,\f{2\pi}{m})\}$, which is one of the cat's-eye-shaped regions enclosed by $\pa D_0$ and $\pa\hat{D}_0$ shown in Figure \ref{fig:12-fold_intro}(B).
This issue is related to the conjectured cat's-eye-type structure in the flow field of an $m$-fold symmetric V-state \cite[Conjecture 1.4]{HassainiaMasmoudiWheeler2020}, but unfortunately, such desired feature crucially relies on more detailed geometric properties of $D_0$ as well as $\hat{D}_0$.
\end{rmk}

\begin{rmk}
Theorem \ref{thm: heteroclinic trajectory} and Theorem \ref{thm: flow field} imply that, in the original physical coordinate, passive particles initially lying in $\hat{D}_0\setminus \overline{D_0}$ will rotate with the patch despite that they are always outside the support of vorticity, while those passive particles starting outside $\overline{\hat{D}_0}$ lag behind as they rotate with a slower speed.
This phenomenon is in the same spirit as the notion ``vortex atmosphere" defined for traveling vortices \cite{ChoiJeongSim2025}.
We note that a rigorous definition of a corresponding notion for the rotating vortices requires more considerations.
\end{rmk}
\end{thm}

\subsection{Historical background}
\label{sec: related works}
Vortex patch solutions are an important subject in the study of the 2-D incompressible Euler equation.
It refers to the solutions to \eqref{eqn: 2D incompressible Euler} and \eqref{eqn: stream function of 2D incompressible Euler} with the initial data of the form $\om(x,0) = \mathds{1}_{\Om_0}(x)$, where $\Om_0\subset \BR^2$ is a bounded region.
Since $\om(x,0)\in L^1\cap L^\infty(\BR^2)$, by the classic Yudovich theory \cite{Yudovich1963}, the global weak solution is uniquely well-defined, and in addition, $\om(x,t) = \mathds{1}_{\Omega_t}(x)$ for all $t\in \BR$ with $\Om_t$ being a time-varying bounded region.
Vortex patch solutions (and more generally, patch-type solutions) can also arise in the Euler equation posed on bounded domains \cite{CaoWan2022,CaoWanWangZhan2021,HassainiaRoulley2025BoundaryEffects,KiselevLi2019,LongWangZeng2019,Turkington1983a} and manifolds (notably, spheres) \cite{GarciaHassainiaRoulley2025,KimuraOkamoto1987,PolvaniDritschel1993}, or in the context of other related active scalar equations, such as the surface quasi-geostrophic equation and its generalizations \cite{CastroCordobaGomezSerrano2016SQG,ChaeConstantinCordobaGancedoWu2012,CordobaCordobaGancedo2018,CordobaFontelosManchoRodrigo2005,Gancedo2008,GancedoPatel2021,HassainiaHmidi2015,HassainiaHmidiMasmoudi2025KAM,KiselevLuo2025,KiselevRyzhikYaoZlatos2016,KiselevYaoZlatos2017}, but in what follows, we shall only focus on the vortex patches in the 2-D incompressible Euler equation in $\BR^2$.

Dynamics of general vortex patch solutions has been studied extensively.
For patches with initially regular boundary, whether they may develop finite-time singularity was once a subject of investigation and debate both numerically and analytically \cite{Alinhac1991,Bertozzi1991, Buttke1989, ConstantinTiti1988, Dritschel1985, Dritschel1988, DritschelMcIntyre1990, Majda1986, ZabuskyHughesRoberts1979,ZouOvermanWuZabusky1988}.
In the early 1990s, Chemin \cite{Chemin1990_1991,Chemin1993} and Bertozzi--Constantin \cite{BertozziConstantin1993} resolved this issue by proving that, if the initial patch $\Om_0$ has $C^{1,\al}$-boundary with $\al\in (0,1)$, then $\pa \Omega_t$ will remain $C^{1,\al}$ for all time; also see \cite{Serfati1994}.
Recently, \cite{KiselevLuo2023} showed ill-posedness of $C^2$-vortex patches, while \cite{Lee2025} established local well-posedness of $C^{1,\va}$-vortex patches where $C^{1,\va}$ is defined using certain modulus of continuity.
If the initial patch boundary has singularities, however, Danchin \cite{Danchin1997} proved that the singularities in a closed subset simply get transported by the flow, and away from that, the boundary regularity persists for all time.
He also obtained stability of cusp-like structures \cite{Danchin2000}.
In \cite{CarrilloSoler2000,CohenDanchin2000}, dynamics of corners was studied numerically, and it was observed that acute angles along the patch boundary shrink instantaneously as $t$ becomes positive, while obtuse angles widen.
Singular vortex patches with multiple corners emanating from the origin were investigated systematically in the monograph by Elgindi--Jeong \cite{ElgindiJeong2023} (also see the companion paper \cite{ElgindiJeong2020} discussing long-time dynamics of such patches), where
several well-posedness results were established under certain $m$-fold symmetry condition (with $m \geq 3$). In addition, they also proved ill-posedness results which imply that isolated acute and obtuse corners cannot evolve continuously in time, and that isolated 90-degree corners are generically ill-posed as well.
Very recently, \cite{ElgindiJo2025} proved the cusp formation of vortex patches with acute angles; also see similar results in \cite{ElgindiJeong2023,HoffPerepelitsa2009} under different symmetry assumptions.
From these results, it is clear that isolated corners on the patch boundary which can persist for some time can only be of 90 degrees, but there has been no rigorous construction for such an example to the best of our knowledge.

Uniformly rotating vortex patches (or rotating V-states) is a special class of vortex patch solutions which has been studied for a long time.
Back in the 19th century, disks (also called the Rankine vortices in the literature) and Kirchhoff ellipses \cite{Kirchhoff1876} were already known as explicit examples.
In 1880, Thomson (Lord Kelvin) \cite{Thomson1880ColumnarVortex} analyzed the linearized equation for infinitesimal perturbations around a 3-D cylindrical vortex.
The result suggests that, for the 2-D Euler equation, $m$-fold symmetric rotating vortex patches ($m\geq 3$) may be obtained by bifurcating from the unit disk at the angular velocity $\f{m-1}{2m}$ (see Lamb \cite[Art.~158]{Lamb1932Hydrodynamics}).
Systematic search of such rotating patches started from Deem--Zabusky \cite{DeemZabusky1978}, where the term ``V-state" was introduced and several $m$-fold symmetric rotating patches were numerically computed by bifurcating from the disk.
In 1982, Burbea \cite{Burbea1982Motions} provided the first existence proof using the Crandall--Rabinowitz theorem, where for each $m\geq 3$, a branch of $m$-fold symmetric V-states was found bifurcating from the disk at the angular velocity predicted by Lord Kelvin.
In \cite{BurbeaLandau1982}, $m$-fold V-states further along the bifurcation branches were numerically computed for $m \in \{3,4,5,6\}$, with their stability examined.
Wu--Overman~II--Zabusky  \cite{WuOvermanZabusky1984} developed more accurate numerical methods, and for the first time, they obtained limiting V-states with corners.
It was further confirmed in \cite{Overman1986} that these corners must be of $90$ degrees by an asymptotic analysis near the corners of a hypothetical vortex patch.
Hmidi--Mateu--Verdera \cite{HmidiMateuVerdera2013} fixed the defect in the argument of \cite{Burbea1982Motions}, and further showed that the V-states along the local bifurcation curve have $C^\infty$-boundary, which was later improved to be analytic \cite{CastroCordobaGomezSerrano2016}.
In a milestone work \cite{HassainiaMasmoudiWheeler2020}, Hassainia--Masmoudi--Wheeler extended the earlier local bifurcation results into a global one for all $m\geq 3$, and provided characterization of the solutions along the whole branches.
They performed numerical simulations to confirm the results by Wu--Overman~II--Zabusky \cite{WuOvermanZabusky1984}, and proposed a conjecture that, for every $m\geq 3$, $m$-fold symmetric V-states with 90-degree corners should exist as the weak limits of patches along each bifurcation branch.
As far as we know, this conjecture remains open.
Recently, using computer-assisted proof techniques, Castro-L\'{o}pez--G\'{o}mez-Serrano \cite{CastroLopezGomezSerrano2025} constructed a non-convex V-state with 6-fold symmetry and analytic boundary.
Let us also mention some works searching for new V-states based on the Kirchhoff ellipses \cite{CastroCordobaGomezSerrano2016,CerretelliWilliamson2003,
HmidiMateu2016,Kamm1987,LuzzattoFegizWilliamson2010}, studying doubly-connected rotating vortex patches \cite{deLaHozEtAl2016DoublyConnected, HmidiMateu2016DegenerateBifurcation, WangXuZhou2024DegenerateDoublyConnected}, and constructing quasi-periodic vortex patch solutions using the KAM theory \cite{BertiHassainiaMasmoudi2023TimeQuasiperiodic, HassainiaHmidiRoulley2024InvariantKAM}.

Besides these studies which focus on the existence of various rotating V-states, there are many analytic studies devoted to establishing their mathematical and physical properties.
For example, rigidity results have been rigorously proved in \cite{FanWangZhan2025RadialSymmetryEuler, Fraenkel2000, GomezSerranoParkShiYao2021, Hmidi2015, Huang2025Rigidity}, and stability and instability issues were investigated in \cite{ChoiJeong2022, GuoHallstromSpirn2004, Love1893, Tang1987, Wan1986, WanPulvirenti1985CircularPatchStability}.
Park \cite{Park2022} derived some quantitative estimates for the V-states.
Very recently, with the observation that uniformly rotating vortex patches give rise to sign-changing unstable free boundary problems \cite{AnderssonShahgholianWeiss2012, MonneauWeiss2007}, Wang--Zhang--Zhou \cite{WangZhangZhou2026} proved that for a V-state with Lipschitz boundary, its singular set can contain at most finitely many points; moreover, the patch boundary near each singular point should consist of two $C^1$-arcs meeting at the right angle, while it is smooth elsewhere.
However, it was not known whether such V-states with Lipschitz boundary and non-empty singular set do exist, although for the (one-phase) unstable free boundary problem, examples of cross-shaped singularities have been constructed in a bounded domain by choosing suitable boundary data \cite{AnderssonWeiss2006}.

\subsection{Key ideas in the proof}
\label{sec: scheme of the proof}
Recall that in Theorems \ref{thm: main existence theorem}-\ref{thm: flow field}, we have presented the existence of $m$-fold symmetric V-states with 90-degree corners for all suitably large $m$, and also characterized these patches as well as the flows they induce in the co-rotating frames.
This provides not only the first rigorous construction of rotating V-states in $\BR^2$ with singular boundaries which fulfills the assumptions of \cite{WangZhangZhou2026}, but also the first example of a single vortex patch with persistent 90-degree corners, which are very fragile according to \cite{ElgindiJeong2023}.
Besides, these V-states may also help resolve the above-mentioned conjecture in \cite{HassainiaMasmoudiWheeler2020}.
In this part, let us explain the key ideas involved in the proofs of the main results, especially the existence.

The proof of the existence uses a novel fixed-point argument that is inspired by an earlier work of the first two authors \cite{HuangTong2025} on finding the so-called Sadovskii vortex patch \cite{ChoiJeongSim2025Sadovskii,Sadovskii1971}.
In fact, the problems of rigorously constructing the Sadovskii vortex patch and the $m$-fold symmetric V-states with 90-degree angles share some similar difficulties: we seek some non-trivial solutions that are far away from any known explicit solutions, and meanwhile, we want them to have certain geometric features.
Compared with the more classic approaches of constructing solutions, such as the bifurcation method and the variational framework, the direct fixed-point method we are going to introduce below has the strength that it has better control of detailed geometric properties of the solution (e.g., monotonicity of functions, and size of certain angles at some special points), while it relies on little knowledge of the closed-form solutions.
In this paper, the proof of the existence proceeds as follows.
\begin{enumerate}
\item
By introducing a change of coordinates (see \eqref{eqn: def of zeta}, \eqref{eqn: def of Phi}, and Lemma \ref{lem: integral formula for Psi}), we can transform the task of constructing the $m$-fold symmetric V-states with 90-degree corners into the following one: find a nontrivial continuous $g:\BT\to [0,+\infty)$, such that, if we define $\mu:=\f{1}{m}$, $\Omega:=\{(x_1,x_2)\in \BT\times \BR:\, x_1\in \BT,\, x_2\in[0,g(x_1)]\}$, and (see \eqref{eqn: integral representation of Psi})
\[
\Psi(x) = \Psi(x;g)
:= -\f{1}{4\pi}\int_\Om e^{-2\mu y_2}
\left[\ln \left(\f{\cosh(x_2-y_2)- \cos(x_1-y_1)}{\cosh(0-y_2)-\cos(\pi-y_1)}\right)
-x_2\right] dy,
\]
then it should hold that (see \eqref{eqn: constraint for the boundary curve in Psi})
\[
\Psi(x_1,g(x_1);g) = \partial_{x_2}\Psi(\pi,0;g)\cdot \frac{1-e^{-2\mu g(x_1)}}{2\mu}
\quad \text{for $x_1\in\BT$}.
\]

\item
Let $F(x_1,x_2;g):=\f{1}{x_2}\Psi(x_1,x_2;g)$.
Then the above condition can be rewritten as
\beq
F(x_1,g(x_1);g) = \partial_{x_2}\Psi(\pi,0;g)\cdot \frac{1-e^{-2\mu g(x_1)}}{2\mu g(x_1)}
\quad \text{for $x_1\in\BT$}.
\label{eqn: equation for fixed point g first glance}
\eeq
We can show that (see Proposition \ref{prop: monotonicity of F}), if $g$ is even on $[-\pi,\pi]$ and non-increasing on $[0,\pi]$ (plus some other additional assumptions), $F(x_1,x_2;g)$ is strictly decreasing in both $x_1$ and $x_2$ in $[0,\pi]\times (0,+\infty)$.
This monotonicity property allows us to introduce a mapping $\bfR$ as follows: given $g:\BT \to [0,+\infty)$, we can define the function $\bfR(g):\BT \to [0,+\infty)$ by (see Proposition \ref{prop: implicit mapping} and \eqref{eqn: def of R(g)})
\beq
F(x_1,\bfR(g)(x_1);g) = \partial_{x_2}\Psi(\pi,0;g)\cdot \frac{1-e^{-2\mu g(x_1)}}{2\mu g(x_1)},
\quad \forall\, x_1\in\BT.
\label{eqn: implicit mapping R first glance}
\eeq
In fact, $\bfR(g)$ is well-defined as an even function on $[-\pi,\pi]$, and it is strictly decreasing on $[0,\pi]$.
Then in view of \eqref{eqn: equation for fixed point g first glance}, the task of finding a desired $g$ becomes seeking for a continuous non-trivial fixed point of the mapping $\bfR$.

\item
To obtain the fixed point, we can prove that, for suitably large $m$, $\bfR$ preserves an a priori upper bound (see Proposition \ref{prop: a priori upper bound general m}), i.e., there is a universal $M>0$, such that $g \leq M$ on $\BT$ implies $\bfR(g) < M$ on $\BT$.
Moreover, we will show that the functions $x\mapsto \Lam(\pi^2-x^2)$ and $x\mapsto \lam(\pi^2-x^2)$ on $\BT$, with $\Lam>0$ being suitably large and $\lam>0$ being suitably small, can serve as upper and lower ``barrier functions" (see Proposition \ref{prop: upper barrier new} and Proposition \ref{prop: lower barrier}).
The role of the upper barrier can be interpreted as follows: if $g$ stays below $\Lam(\pi^2-x^2)$ and ``touches it from below", i.e., $g(x)\leq \Lam (\pi^2-x^2)$ for all $x \in \BT$ while the equality is achieved at some $x_*\in \BT$, then $\bfR(g)$ must be strictly smaller than $g$ at the touching point, i.e., $\bfR(g)(x_*)<g(x_*) = \Lam(\pi^2-x_*^2)$.
Note that this is a single-point comparison between $\bfR(g)$ and the upper barrier instead of a global one throughout $\BT$.
The lower barrier can be understood in a similar manner.
These barrier functions provide non-trivial bounds in the design of the function set on which the fixed-point argument is applied (see below).

\item With the upper bound $M$ and the barrier functions in hand, we define a class of functions that lie between the upper bound $M$, the upper barrier, and the lower barrier (see \eqref{eqt: def of function set D})
\beq
\begin{split}
\mathbb{W}:=\big\{g:\BT\to [0,\infty):&\; \text{$g$ is even and lower semi-continuous on }\BT,\\
&\;\text{$g$ is non-increasing on $[0,\pi]$},\\
&\;\lambda (\pi^2-x^2)\leq g(x)\leq \min\{\Lambda  (\pi^2-x^2), M\}\text{ on $\BT$}\big\},
\end{split}
\label{eqn: function space W first glance}
\eeq
We can apply the Schauder fixed-point theorem to show that a truncated version of the mapping $\bfR$ (see \eqref{eqn: truncated mapping hat R}) has a fixed point in $\mathbb{W}$; in fact, to be very precise, we have to work with an $L^2$-version of $\mathbb{W}$ (see Section \ref{sec: proof of existence of the fixed point} for all the technicality).
Thanks to the properties of the barrier functions, we can further argue that this fixed point is actually a fixed point of $\bfR$ in $\mathbb{W}$.

\item
The fixed point $g$ found above satisfies \eqref{eqn: equation for fixed point g first glance}, but it is not necessarily continuous on $\BT$.
We define (see \eqref{eqn: def of F_dag})
\[
F_\dag(x_1,x_2;g):= F(x_1,x_2;g) -\pa_{x_2}\Psi(\pi,0;g)\cdot \f{1-e^{-2\mu x_2}}{2\mu x_2},
\]
and observe that it is possible to prove a non-degeneracy result for $F_\dag$ (in the $x_2$-direction) at some special points of interest: if $(x_1,x_2)\in [0,\pi)\times (0,g(0))$ satisfies $F_\dag(x_1,x_2;g) = 0$ and $x_1 = g^{-1}(x_2)$, we have $\pa_{x_2} F_\dag (x_1,x_2;g) < 0$; see the quantitative statements in Proposition \ref{prop: non-degeneracy along level set} and Proposition \ref{prop: non-degeneracy along level set x_1 greater than pi over 2}.
This enables us to show the continuity of $g$.
Higher regularity of $g$ in $(-\pi,\pi)$ then follows from the implicit function theorem and the standard arguments in the elliptic free boundary problem.
By transforming this $g$ back to the quantities in the original coordinate, we  obtain an $m$-fold symmetric V-state.

\item Finally, by proving a local expansion of $\Psi$ near the point $(\pi,0)$ (see Lemma \ref{lem: local expansion}), we can further show that at the end-points $\pm \pi$, the one-sided derivative of $g$ satisfies $g'_-(\pi) = -1$ and $g_+'(-\pi) = 1$
(see Proposition \ref{prop: g slope at pi}).
This implies that the resulting V-state in the original coordinate has 90-degree corners at exactly $m$ points.
\end{enumerate}

Once the existence of the desired V-state has been obtained together with its regularity and quantitative estimates, it is not difficult to prove the other main results.
We omit the discussion.

Although the parameter $\mu = \f1m$ cannot take the value zero, it is illuminating to first study the formal limiting case $\mu=0$, which turns out to be simpler.
When $\mu = 0$, the condition \eqref{eqn: equation for fixed point g first glance} for the fixed point $g$ reduces to $F(x_1,g(x_1);g) = \partial_{x_2}\Psi(\pi,0;g)$ for all $x_1\in\BT$; note that the right-hand side is a constant, which is also the case in the construction of the Sadovskii vortex patch \cite{HuangTong2025}.
The implicit mapping in \eqref{eqn: implicit mapping R first glance} can still be well defined by $F(x_1,\bfR(g)(x_1);g) = \partial_{x_2}\Psi(\pi,0;g)$ for all $x_1\in \BT$.
By deriving a negative upper bound for $\pa_{x_2}F$, which can be viewed as a quantitative version of the monotonicity of $F$ in $x_2$ mentioned above, one can directly obtain the continuity of $\bfR(g)$ and even $C_{loc}^{1,\al}$-regularity of $\bfR(g)$ in $(-\pi,\pi)$ (under suitable conditions).
Indeed, for any $\mu \geq 0$, $\Psi$ satisfies $-\D \Psi = e^{-2\mu x_2}\mathds{1}_\Omega$ in $\BT\times \BR$, so it is automatically in $C^{1,\al}_{loc}(\BT\times \BR)$, which implies that $F = \f{1}{x_2}\Psi$ has $C^{1,\al}_{loc}$-regularity in $\BT\times (0,+\infty)$.
Then by the implicit function theorem, $\bfR(g)$ should be locally $C^{1,\al}$ in $(-\pi,\pi)$.
This argument requires relatively lower regularity of $g$.
With some extra efforts, one can derive more quantitative characterizations along this line.
This allows us to directly apply the fixed-point argument in some function class with high regularity (cf.\;\eqref{eqn: function space W first glance}, and also see the argument in \cite{HuangTong2025}), which thus has the desired compactness, and there is no need to justify the continuity of the fixed point in a separate proof.

In principle, the case of small $\mu$ can be treated as a perturbation of the limiting case $\mu = 0$.
Indeed, when $\mu \ll 1$, the function $x_2 \mapsto (1-e^{-2\mu x_2})/(2\mu x_2)$ on the right-hand sides of \eqref{eqn: equation for fixed point g first glance} and \eqref{eqn: implicit mapping R first glance} is approximately $1$ for $x_2$ being not so large, and it only decreases slowly as $x_2$ increases in $(0,+\infty)$.
As a result, the existence proof in the case $0<\mu\ll 1$ can be made very similar to that for the case $\mu=0$.
We only remark that, unlike the case $\mu = 0$, the regularity of $\bfR(g)$ in the case $\mu>0$ will be no better than that of $g$ (see \eqref{eqn: implicit mapping R first glance}), but it is still possible to use the smallness of $\mu$ to show certain higher regularity bounds for $\bfR(g)$ given that $g$ satisfies the same bounds.
Therefore, we can still apply the fixed-point argument just as before in a function class with high regularity assumptions.

A natural question is how large $\mu$ can be so that our fixed-point method is still applicable.
This directly relates to the issue that for $m$ in what range we can construct the $m$-fold symmetric rotating V-states with 90-degree corners.
We do encounter smallness constraints on $\mu$ at several places in the argument.
In order to figure out the admissible range of $\mu$, it is necessary to make those smallness constraints as explicit as possible.
For that purpose, we will track the universal constants explicitly when deriving several estimates regarding $F$ (see Section \ref{sec: estimates with explicit constants}), which leads to the explicit smallness conditions  for $\mu$ and $M$ in Corollary \ref{cor: M=4 and mu leq 1/12} and Proposition \ref{prop: upper barrier new}.
However, on the other hand, the strategy (discussed in the previous two paragraphs) of first proving a priori higher-order estimates for $\bfR(g)$ and then applying the fixed-point argument in a function class with some high-regularity bounds does not readily lead to an explicit smallness condition.
The reason is that, to derive the estimates for $\bfR(g)$, one has to work with general inputs $g$ in a suitable function class and study their corresponding $\Psi$ and $F$, which can be rather complicated.
Instead, we choose to first obtain the fixed point in a function class with weaker regularity (see \eqref{eqn: function space W first glance} and Section \ref{sec: proof of existence of the fixed point}), and justify its continuity as well as higher regularity later by using the non-degeneracy result for $F_\dag$.
Note that in the latter step, we will only need the non-degeneracy of $F_\dag$ at those points $(x_1,x_2)$ satisfying that $F_\dag(x_1,x_2;g) = 0$ and $x_1 = g^{-1}(x_2)$, which are not arbitrary points, and moreover, we will have the extra information that $F_\dag$ is generated by a fixed point $g$, so it is much more tractable to derive the desired explicit smallness condition in this way.
Eventually, these smallness conditions combined lead to an explicit admissible range of $\mu$, which will be summarized and explained in detail in Remark \ref{rmk: optimality of m=12 explained}.
That gives rise to the final conclusion that $m_0\leq 12$ in Theorem \ref{thm: main existence theorem}, i.e., $m$-fold symmetric rotating V-states with 90-degree corners exist at least for all $m\geq 12$ (see Remark \ref{rmk: claim optimality m = 12}).

\subsection{Organization of the paper}
The rest of this paper is organized as follows.
In Section \ref{sec: reformulation}, as was explained before, we will transform the problem into a new coordinate by a conformal mapping, and introduce the key quantities $\Psi$ and $F$.
After showing the monotonicity property of $F$, we will reformulate the problem into a fixed-point problem regarding the mapping $\bfR$.
In Section \ref{sec: estimates}, we will derive a priori estimates for $\Psi$, $F$, and related quantities; some of the estimates come with explicit constants.
We also establish an upper bound preserved by $\bfR$, and construct the upper and lower barrier functions.
Section \ref{sec: existence of the fixed point} is devoted to proving the existence of the fixed point of $\bfR$.
Regularity of the fixed point $g$ will be addressed in Section \ref{sec: regularity}.
We will start from establishing the non-degeneracy result for $F_\dag$, and then show the continuity and higher regularity for $g$ in $(-\pi,\pi)$.
After that, by proving a local expansion of $\Psi$ near the point $(\pi,0)$, we justify the properties of the fixed point $g$ near the end-points $\pm \pi$.
At the end of Section \ref{sec: regularity}, we will present the proof of Theorem \ref{thm: main existence theorem}.
In Section \ref{sec: flow field}, we investigate the stationary flows in the co-rotating frame induced by the V-states, and prove Theorem \ref{thm: heteroclinic trajectory} and Theorem \ref{thm: flow field}.
In Appendix \ref{sec: proof of M = 4 mu = 1/12}, we present the proof of Corollary \ref{cor: M=4 and mu leq 1/12}, which directly leads to the explicit threshold $12$ for $m$.
We prove some auxiliary calculus lemmas in Appendix \ref{sec: calculus lemmas}.
Finally, in Appendix \ref{sec: numerical}, we explain our numerical method to compute the $m$-fold symmetric V-states with 90-degree angles and present more numerical results.

\subsection*{Acknowledgement}
The authors are all supported by the National Key R\&D Program of China under the grant No.~2021YFA1001500.
DH is also supported by the National Natural Science Foundation of China under the grant NSFC No.~12288101.

\section{Reformulation of the Problem}
\label{sec: reformulation}

The main goal of this section is to reformulate the problem of finding a V-state $D_0$ that satisfies \eqref{eqn: constraint satisfied by phi along the patch boundary final} into a fixed-point problem.
We start from transforming the problem into a more convenient coordinate, and decomposing the stream function.

\subsection{Conformal mapping and decomposition of the stream function}
\label{sec: conformal mapping decomposing stream function}
Let $B_1\subset \BR^2$ denote the unit open disk in $\BR^2$ centered at the origin.
Note that the general patch $D_0$ given in \eqref{eqn: form of D_0} (not necessarily with the right angles) is a subset of $B_1$, and \eqref{eqn: constraint satisfied by phi along the patch boundary final} is a constraint along $\pa D_0\subset \overline{B_1}$.

By virtue of the $m$-fold symmetry of $D_0$, $\phi$ also enjoys the $m$-fold symmetry, i.e., for any $x\in B_1$, in the complex notation, $\phi(e^{2\pi i/m} x) = \phi(x)$.
We introduce an infinite sector
\[
S_m^\infty := \big\{(r\cos\th,r \sin \th)\in \BR^2:\,\th\in [0,2\pi/m],\,r\in (0,+\infty)\big\}
\]
and a sector of the unit circle
\[
S_m := \big\{(r\cos\th,r \sin \th)\in \BR^2:\,\th\in [0,2\pi/m],\,r\in (0,1]\big\}.
\]
Note that the restriction of $\phi$ on $S_m$ already encodes all the information of $\phi$ on $B_1$.
Also denote
\[
\BH:=\BT\times \R,\mbox{ and }\BH_+:=\BT\times [0,+\infty).
\]
We introduce a conformal mapping $\zeta:S_m^\infty\to \BH$ as follows:
\beq
\zeta:\,(r\cos\th,r\sin\th)\mapsto (m\th-\pi,-m\ln r).
\label{eqn: def of zeta}
\eeq
Equivalently, using the complex notation, we may write it as
\[
\zeta:\,z\mapsto -m\cdot i\ln z-\pi.
\]
Its inverse is given by $\zeta^{-1}:\BH\to S_m^\infty$,
\[
\zeta^{-1}:\,(x_1,x_2)\mapsto  \left(e^{-\f{x_2}{m}}\cos \f{x_1+\pi}{m},\, e^{-\f{x_2}{m}}\sin \f{x_1+\pi}{m}\right).
\]
Note that $\zeta$ also maps $S_m$ to $\BH_+$ bijectively and conformally.

Using the mapping $\zeta$, we can transform the $\phi$-problem defined on $S_m$ to an equivalent formulation on $\BH_+$.
For that purpose, let us introduce some new notations.
\begin{itemize}
\item
Define $\Phi:\BH\to \BR$ by
\beq
\Phi(x_1,x_2)
:= -m^2\big[\phi\circ\zeta^{-1}(x_1,x_2) - \phi\circ\zeta^{-1}(\pi,0) \big],\quad  \forall\, (x_1,x_2)\in \BH,
\label{eqn: def of Phi}
\eeq
or equivalently
\[
\Phi\circ\zeta(x_1,x_2)
:= -m^2\big(\phi(x_1,x_2) - \phi(1,0) \big),\quad  \forall\, (x_1,x_2)\in \BR^2.
\]
It is not difficult to derive that
\beq
\pa_{x_2}\Phi(\pm \pi,0) = m\partial_{x_1}\phi(1,0).
\label{eqn: relation between derivatives of phi and Phi}
\eeq
We note that in this formula, $\pa_{x_2}$ on the left-hand side represents the partial derivative with respect to the $x_2$-variable in the new coordinate after the transformation, while $\pa_{x_1}$ on the right-hand side denotes the partial derivative with respect to the $x_1$-variable in the physical coordinate (i.e., the co-rotating frame) before the transformation.

\item
Define $g:\BT\to \BR$ by
\beq
g(x_1):= -m\ln f\left(\f{x_1+\pi}{m}\right), \quad \forall\, x_1\in \BT = (-\pi,\pi].
\label{eqn: def of g}
\eeq
In fact, this is determined by the relation (see \eqref{eqn: def of zeta})
\[
\zeta\big(f(\th)\cos \th, f(\th)\sin\th \big) = \big(m\th-\pi,g(m\th-\pi)\big),\quad \forall\, \th \in \BT,
\]
i.e., the graph of $g$ on $\BT$ represents the image of the graph of $f = f(\th)$ in the polar coordinate under the mapping $\zeta$.
Correspondingly, let (cf.\;\eqref{eqn: function set M_0})
\beq
\begin{split}
\BM:= \Big\{g:\BT\to [0,\infty):\;& \text{$g$ is even}, \;
\text{$g$ is non-increasing on $[0,\pi]$},\\
&\mbox{$g>0$ in $(-\pi,\pi)$},\; g(\pm \pi)= \lim_{x\to\pi^-}g(x)=\lim_{x\to -\pi^+}g(x) =0 \Big\}.
\end{split}
\label{eqn: function set tilde M_0}
\eeq
It is clear that \eqref{eqn: def of g} gives a bijection from $\BM_0$ to $\BM$, whose inverse is given by
\beq
f(\th):= e^{-\f{1}{m}g(m\th-\pi)},\quad \th\in \BT. \label{eqn: transformation from g to f}
\eeq
We illustrate the conformal mapping $\zeta$ and the induced correspondence between $f$ and $g$ for $m=12$ in Figure \ref{fig:mapping}.

\item
Let
\beq
\Om := \zeta\big(S_m\setminus D_0\big) = \big\{(x_1,x_2)\in \BH_+:\, x_1\in \BT,\, x_2\in[0,g(x_1)] \big\}.
\label{eqn: def of Omega domain}
\eeq

\item
Let $h:\BT\to \BR$ be given by
\beq
\begin{split}
h(x_1):= &\; -m^2\big[\phi\circ \zeta^{-1}(x_1,0)
- \phi\circ \zeta^{-1}(\pi,0)\big]\\
= &\; -m^2\left[\phi\left(\cos\f{x_1+\pi}{m},\, \sin\f{x_1+\pi}{m}\right) -\phi(1,0)\right].
\end{split}
\label{eqn: def of h}
\eeq
By \eqref{eqn: def of Phi}, $h(x_1) = \Phi(x_1,0)$.
\end{itemize}

Using these notations as well as \eqref{eqn: relation between derivatives of phi and Phi},
we can rewrite the condition \eqref{eqn: constraint satisfied by phi along the patch boundary final} for a general rotating vortex patch as
\beq
\Phi(x_1,x_2) = \partial_{x_2}\Phi(\pi,0)\cdot \frac{1-e^{-2x_2/m}}{2/m}
\quad \mbox{on }\zeta\big(S_m\cap \pa D_0\big).
\label{eqn: constraint for Phi along the image of the patch boundary}
\eeq
For general $f\in \BM_0$ (or equivalently, $g\in \BM$),
\[
\g:= \{(x_1,g(x_1)):\,x_1\in \BT\} \subset \zeta(S_m\cap \pa D_0),
\]
so a necessary condition of \eqref{eqn: constraint for Phi along the image of the patch boundary} is that
\beq
\Phi(x_1,x_2) = \partial_{x_2}\Phi(\pi,0)\cdot \frac{1-e^{-2x_2/m}}{2/m}
\quad \mbox{along }\g.
\label{eqn: constraint for the boundary curve in Phi}
\eeq

\begin{rmk}
\label{rmk: equivalence of the conditions along boundary}
If additionally $f\in C(\BT)$ (or equivalently, $g\in C(\BT)$), we have $\g = \zeta(S_m\cap \pa D_0)$ and thus \eqref{eqn: constraint for the boundary curve in Phi} is equivalent to \eqref{eqn: constraint for Phi along the image of the patch boundary}.
Therefore, in order to have \eqref{eqn: constraint satisfied by phi along the patch boundary final} (or equivalently, \eqref{eqn: constraint for Phi along the image of the patch boundary}) hold, it is convenient to first find $g\in \BM$ such that \eqref{eqn: constraint for the boundary curve in Phi} holds, and then further justify \eqref{eqn: constraint for Phi along the image of the patch boundary} by showing the continuity of $g$.
\end{rmk}

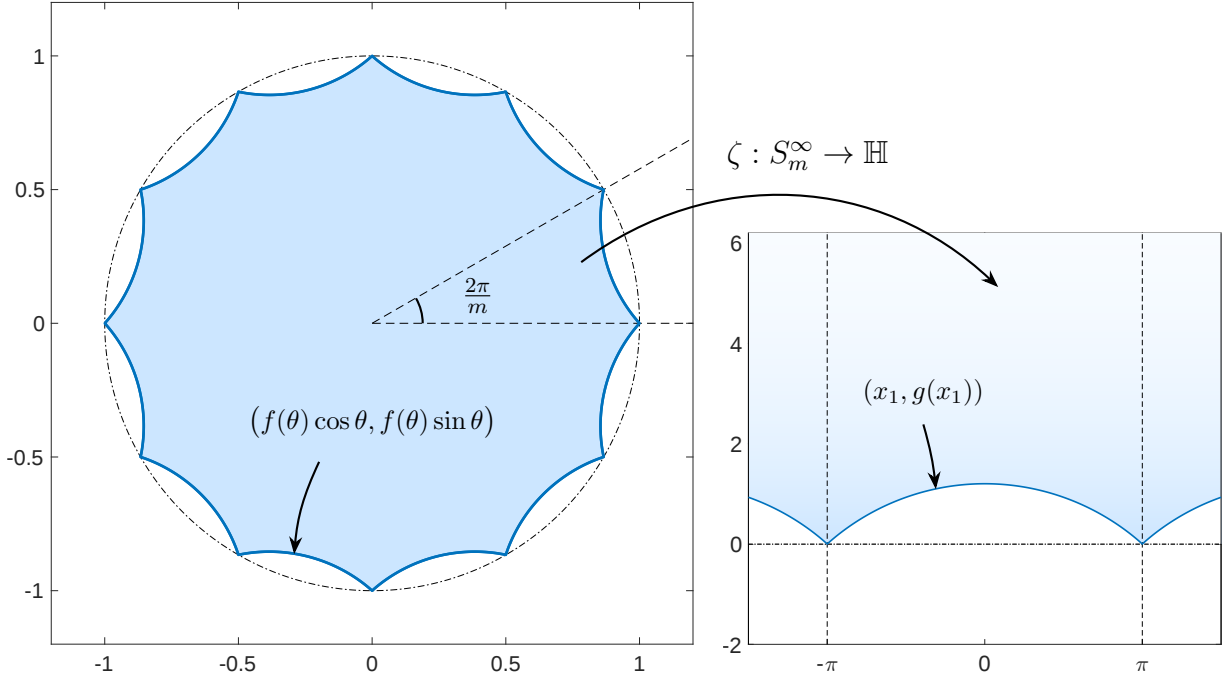
\begin{figure} \centering

    \begin{tikzpicture}
    \node[inner sep=0pt] (img)
    {
    \includegraphics[width=0.55\textwidth]{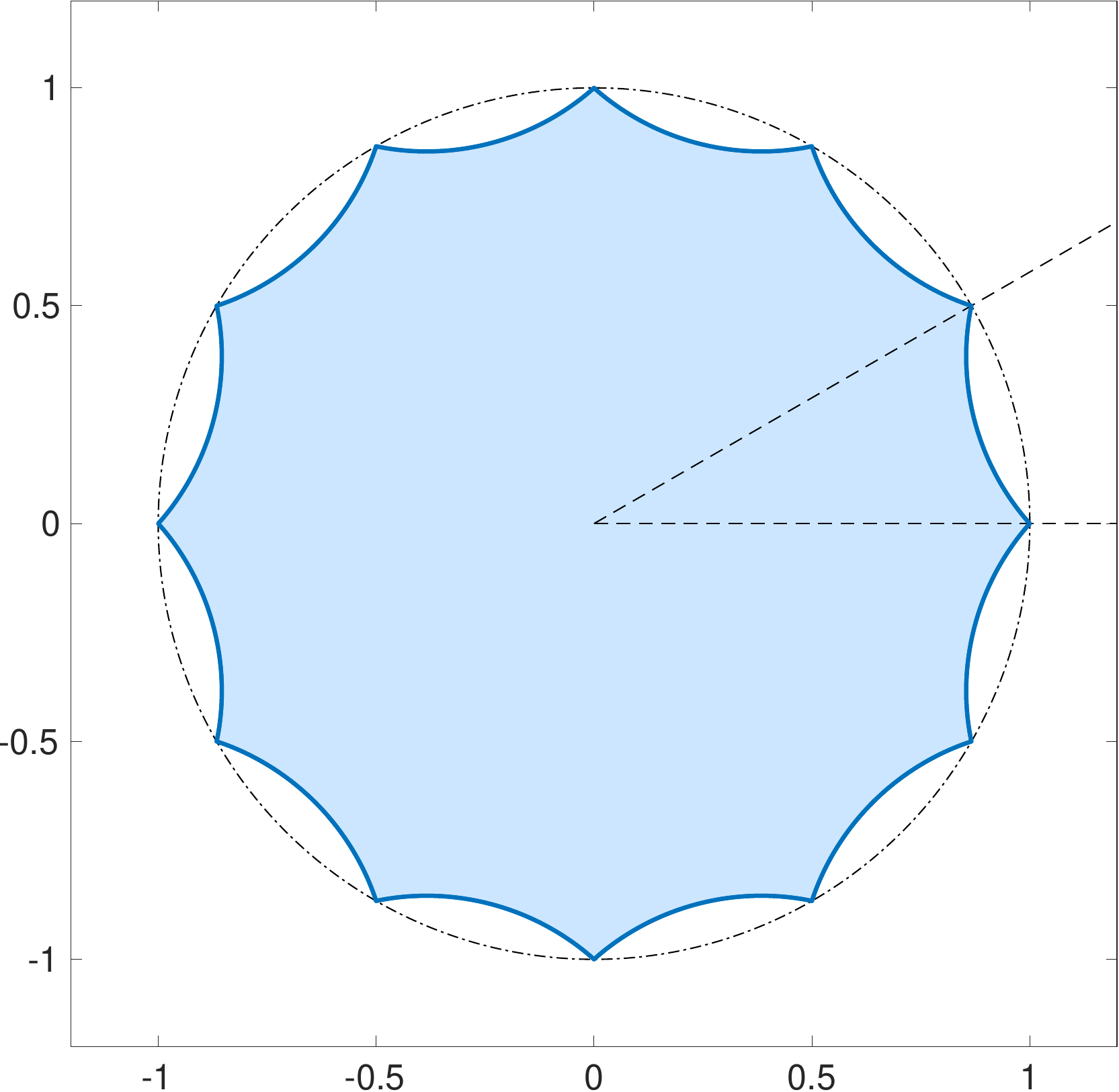}\quad
    \includegraphics[width=0.4\textwidth]{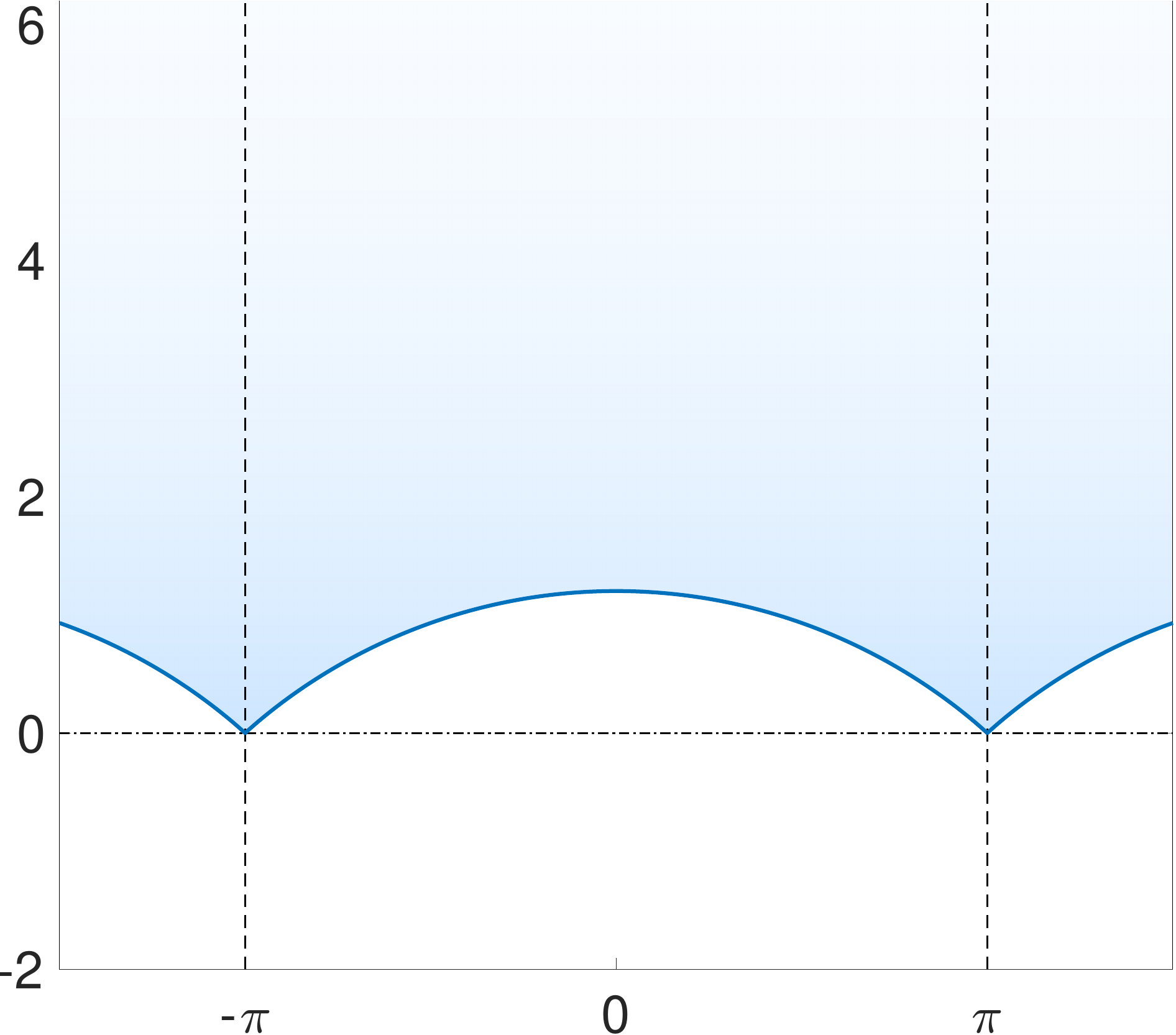}
    };

      \draw[-Stealth, thick, black]
      ($(img.west) + (0.46\textwidth, 0.06\textwidth)$)
      to[bend left=40]
      node[above = 6pt,pos=0.53,font=\large] {$\zeta:S_m^\infty\to \BH$}
      ($(img.east) + (-0.18\textwidth, 0.04\textwidth)$);

      \draw[-Stealth, thick, black]
      ($(img.west) + (0.25\textwidth, -0.1\textwidth)$)
      to[bend left=-10]
      node[above = 5pt,pos=0,font=\small] {\qquad \qquad $\big(f(\th)\cos\th,f(\th)\sin \th\big)$}
      ($(img.west) + (0.23\textwidth, -0.174\textwidth)$);

      \draw[-Stealth, thick, black]
      ($(img.east) + (-0.24\textwidth, -0.07\textwidth)$)
      to[bend left=5]
      node[above = 3pt,pos=0,font=\small] {$(x_1,g(x_1))$}
      ($(img.east) + (-0.23\textwidth, -0.122\textwidth)$);

      \draw[thick, black]
      ($(img.west) + (0.333\textwidth, 0.011\textwidth)$)
      arc[start angle=0, end angle=30, radius=0.04\textwidth];

      \node[align=center, font=\bfseries]
      at($(img.center) + (-0.11\textwidth,0.0333\textwidth)$){$\f{2\pi}{m}$};

    \end{tikzpicture}

    \caption[Mapping]{\small An illustration of the conformal mapping $\zeta$ and the induced correspondence between $f$ and $g$ in the case $m=12$.}
    \label{fig:mapping}
\end{figure}

\subsection{Integral representations}

To study the equation \eqref{eqn: constraint for the boundary curve in Phi}, it is convenient to have integral representations of $\Phi$ and related quantities in hand. To that end, we shall first derive an integral representation of $\phi$ before the transform.

\begin{lem}
\label{lem: integral presentation of Phi}
Recall that $\phi$ is defined in \eqref{eqn: equation for phi}, where $D_0$ is given by \eqref{eqn: form of D_0}.
For any $x\in \BR^2$ that can be written as $x = (r\cos \th, r\sin \th)$, where $r \geq 0$ and $\th \in \BT$, we have
\beqo
\phi(x) = \phi(x;f)
= -\f1{4\pi}\int_{0}^{2\pi/m}\int_0^{f(\al)} R\ln \big[r^{2m}+R^{2m} - 2 r^m R^m\cos(m(\th-\al))\big]\,dR\,d\al.
\eeqo

\begin{proof}
Let $\G(x) := -\f{1}{2\pi}\ln |x|$ denote the fundamental solution to the Laplace's equation $-\D \G(x) = \d(x)$ in $\BR^2$.
Then by \eqref{eqn: equation for phi}, $\phi(x) = \G*\mathds{1}_{D_0}(x)$.
In what follows, we would like to incorporate the $m$-fold symmetry of $D_0$ with the fundamental solution.
Using the complex notation, we calculate that, for any $x,y\in \BC$, \begin{align*}
\G_m(x,y) := &\; \sum_{k = 0}^{m-1}\G\big(x-ye^{2i\pi k/m}\big)\\
= &\; -\f{1}{2\pi} \sum_{k = 0}^{m-1}\ln \big|x-ye^{2i\pi k/m}\big|
= -\f{1}{2\pi}\ln \left|\prod_{k = 0}^{m-1} \big(x-ye^{2i\pi k/m}\big)\right|\\
= &\; -\f{1}{2\pi}\ln \left|y^m\prod_{k = 0}^{m-1} \left(\f{x}{y}-e^{2i\pi k/m}\right)\right|
= -\f{1}{2\pi}\ln \left|y^m \left(\f{x^m}{y^m}-1\right)\right|\\
= &\;-\f{1}{2\pi}\ln \big|x^m-y^m\big|.
\end{align*}
Let
\[
D_m := \big\{(r\cos\th,r \sin \th)\in \BR^2:\,\th\in [0,2\pi/m],\,r\in [0,f(\th))\big\}.
\]
Then for $x\in \BC$, with abuse of the complex notation, \[
\phi(x) = \int_{D_0}\G(x-y)\,dy = \int_{D_m}\G_m(x,y)\,dy
= -\f1{2\pi}\int_{D_m} \ln \big|x^m-y^m\big|\,dy.
\]
Writing $x\in \BC$ and $y\in D_m$ as $x = r e^{i\th}$ and $y = Re^{i\al}$ respectively, we can deduce that
\[
\phi(x)
= -\f1{4\pi}\int_{0}^{2\pi/m}\int_0^{f(\al)} R\ln \big[r^{2m}+R^{2m} - 2 r^m R^m\cos(m(\th-\al))\big]\,dR\,d\al,
\]
which is the desired claim.
\end{proof}
\end{lem}

\begin{lem}
\label{lem: integral formula for Psi}
Let $\Phi$ and $\Om$ be defined in \eqref{eqn: def of Phi} and \eqref{eqn: def of Omega domain}, respectively, and let $\mu:=\f{1}{m}$.
Then, for $x=(x_1,x_2)\in \BH$,
\beq
\Phi(x) = \Phi(x;g)
= \f{1}{4\pi}\int_{\mathbb{H}_+\setminus\Om} e^{-2\mu y_2}
\left[\ln \left(\f{\cosh(x_2-y_2)- \cos(x_1-y_1)}{\cosh(0-y_2)-\cos(\pi-y_1)}\right)
-x_2\right] dy.
\label{eqn: integral representation of Phi}
\eeq
Moreover, for $x_2\geq 0$, $\Phi$ has the decomposition
\beq
\Phi(x) = \frac{e^{-2\mu x_2}-1}{(2\mu)^2} + \Psi(x),
\label{eqn: decomposition of Phi}
\eeq
where
\beq
\Psi(x) = \Psi(x;g)
:= -\f{1}{4\pi}\int_\Om e^{-2\mu y_2}
\left[\ln \left(\f{\cosh(x_2-y_2)- \cos(x_1-y_1)}{\cosh(0-y_2)-\cos(\pi-y_1)}\right)
-x_2\right] dy,
\label{eqn: integral representation of Psi}
\eeq
and
\begin{equation}
\partial_{x_2}\Psi(\pi,0)
= \f{1}{4\pi}\int_\Om e^{-2\mu y_2}\left[\f{\sinh y_2}{\cosh y_2 + \cos y_1} + 1\right] dy.
\label{eqn: end point constant}
\end{equation}

\begin{proof}
Thanks to Lemma \ref{lem: integral presentation of Phi} as well as \eqref{eqn: def of Phi}, we can directly compute that\begin{align*}
&\; \Phi(x_1,x_2)\\
= &\; -m^2\left( \phi\left(e^{-\f{x_2}{m}}\cos \f{x_1+\pi}{m},\, e^{-\f{x_2}{m}}\sin \f{x_1+\pi}{m}\right)
- \phi(1,0)\right)\\
= &\; \f{m^2}{4\pi}\int_{0}^{2\pi/m}\int_0^{f(\al)} R
\ln \left(\f{e^{-2x_2}+R^{2m} - 2 e^{-x_2} R^m\cos(x_1+\pi-m\al)}{1+R^{2m} - 2R^m\cos(0-m\al)}\right) dR\,d\al
\\
= &\; \f{m^2}{4\pi}\int_{0}^{2\pi/m}\int_0^{f(\al)} R
\left[\ln \left(\f{\f{e^{-2x_2}+R^{2m}}{2 e^{-x_2} R^m} - \cos(x_1+\pi-m\al)}{\f{1+R^{2m}}{2R^m}-\cos(0-m\al)}\right)
+ \ln \left(\f{2 e^{-x_2} R^m}{2R^m}\right)\right] dR\,d\al
\\
= &\; \f{m^2}{4\pi}\int_{0}^{2\pi/m}\int_0^{f(\al)} R
\left[\ln \left(\f{\cosh(x_2+\ln R^m)- \cos(x_1+\pi-m\al)}{\cosh(\ln R^m)-\cos(0-m\al)}\right)
-x_2\right] dR\,d\al
\\
= &\; -\f{1}{4\pi}\int_{0}^{2\pi/m}\int_0^{f(\al)} R^2
\left[\ln \left(\f{\cosh(x_2+\ln R^m)- \cos(x_1+\pi-m\al)}{\cosh(\ln R^m)-\cos(0-m\al)}\right)\right.\\
&\;\qquad \qquad \qquad \qquad \qquad \quad
-x_2\bigg] \, d (-m\ln R)\,d(m\al-\pi)
\\
= &\; \f{1}{4\pi}\int_{-\pi}^{\pi} \int_{-m\ln f((y_1+\pi)/m)}^{\infty} e^{-2y_2/m}
\left[\ln \left(\f{\cosh(x_2-y_2)- \cos(x_1-y_1)}{\cosh(0-y_2)-\cos(\pi-y_1)}\right)
-x_2\right] dy_2\,dy_1
\\
= &\; \f{1}{4\pi}\int_{-\pi}^{\pi} \int_{g(y_1)}^{\infty} e^{-2y_2/m}
\left[\ln \left(\f{\cosh(x_2-y_2)- \cos(x_1-y_1)}{\cosh(0-y_2)-\cos(\pi-y_1)}\right)
-x_2\right] dy_2\,dy_1\\
= &\; \f{1}{4\pi}\int_{\mathbb{H}_+\setminus\Om} e^{-2\mu y_2}
\left[\ln \left(\f{\cosh(x_2-y_2)- \cos(x_1-y_1)}{\cosh(0-y_2)-\cos(\pi-y_1)}\right)
-x_2\right] dy,
\end{align*}
as claimed.

Next, using the fact that
\beq
\int_{-\pi}^{\pi}\ln (2\cosh x_2+ 2\cos x_1 )\, dx_1 = 2\pi|x_2|,
\label{eqn: integral of a ln function on torus}
\eeq
we can derive for $x_2\geq 0$ that
\begin{align*}
\Phi(x_1,x_2)= &\; \f{1}{4\pi}\int_{-\pi}^{\pi} \int_{g(y_1)}^{\infty} e^{-2\mu y_2}
\left[\ln \left(\f{\cosh(x_2-y_2)- \cos(x_1-y_1)}{\cosh(0-y_2)-\cos(\pi-y_1)}\right)
-x_2\right] dy_2\,dy_1\\
= &\; \frac{1}{4\pi}\int_{-\pi}^{\pi}\int_0^{\infty}e^{-2\mu y_2} \left[\ln\left(\frac{\cosh(x_2-y_2)-\cos(x_1-y_1)}{\cosh(0-y_2) - \cos(\pi-y_1)}\right) - x_2\right] dy_2 dy_1 \\
&\; - \frac{1}{4\pi}\int_{-\pi}^{\pi}\int_0^{g(y_1)}e^{-2\mu y_2} \left[\ln\left(\frac{\cosh(x_2-y_2)-\cos(x_1-y_1)}{\cosh(0-y_2) - \cos(\pi-y_1)}\right) - x_2\right] dy_2 dy_1\\
= &\; \frac{1}{4\pi}\int_0^{\infty}e^{-2\mu y_2} \int_{-\pi}^{\pi}\left[\ln\left(\frac{\cosh(x_2-y_2)-\cos(x_1-y_1)}{\cosh(0-y_2) - \cos(\pi-y_1)}\right) - x_2\right] dy_1 dy_2 \\
&\; - \frac{1}{4\pi}\int_{-\pi}^{\pi}\int_0^{g(y_1)}e^{-2\mu y_2} \left[\ln\left(\frac{\cosh(x_2-y_2)-\cos(x_1-y_1)}{\cosh(0-y_2) - \cos(\pi-y_1)}\right) - x_2\right] dy_2 dy_1\\
= &\; \frac{1}{2}\int_0^{\infty}e^{-2\mu y_2}\left(|x_2-y_2| - x_2- y_2\right)  dy_2 \\
&\; - \frac{1}{4\pi}\int_\Om e^{-2\mu y_2} \left[\ln\left(\frac{\cosh(x_2-y_2)-\cos(x_1-y_1)}{\cosh(0-y_2) - \cos(\pi-y_1)}\right) - x_2\right] dy\\
= &\; \frac{e^{-2\mu x_2}-1}{(2\mu)^2} + \Psi(x_1,x_2).
\end{align*}
The formula for $\pa_{x_2}\Psi(\pi,0)$ follows from direct calculation.
\end{proof}
\end{lem}

In the rest of the paper, we will always denote $\mu := \f{1}{m}$. Substituting \eqref{eqn: decomposition of Phi} into the solution condition \eqref{eqn: constraint for the boundary curve in Phi} yields
\beq
\Psi(x_1,g(x_1);g) = \partial_{x_2}\Psi(\pi,0;g)\cdot \frac{1-e^{-2\mu g(x_1)}}{2\mu}
\quad \text{for $x_1\in\BT$}.
\label{eqn: constraint for the boundary curve in Psi}
\eeq
In view of Remark \ref{rmk: equivalence of the conditions along boundary}, our first goal is to find some $g\in \BM$ such that this condition holds.
To better study $\Psi$ in $\mathbb{H}_+$, let us introduce the fundamental solution to $-\D$ on $\BT\times \BR$:
\beq
G(x_1,x_2) := -\f{1}{4\pi}\ln (2\cosh x_2 - 2\cos x_1).
\label{eqn: def of G}
\eeq
Then \eqref{eqn: integral representation of Psi} can be rewritten as
\beqo
\Psi(x_1,x_2)
= \int_\Om e^{-2\mu y_2}
\left[G(x_1-y_1,x_2-y_2)- G(\pi-y_1,0-y_2)+\f{1}{4\pi}x_2\right] dy.
\eeqo
It is straightforward to check that $\Psi$ solves
\beq
-\D \Psi(x) = e^{-2\mu x_2}\mathds{1}_\Omega(x)\mbox{ in $\BH_+$},\quad \Psi(x_1,0) = h(x_1) \mbox{ for }x_1\in \BT,
\label{eqn: equation for Psi}
\eeq
where
\beq
h(x_1) = h(x_1;g) =
\int_\Om e^{-2\mu y_2}
\big[G(x_1-y_1,0-y_2)-G(\pi-y_1,0-y_2)\big]\, dy.
\label{eqn: formula for h}
\eeq
We shall further split $\Psi$ as
\[\Psi = \Psi_1 + \Psi_2.\]
Here $\Psi_1$ is the solution to
\beqo
-\D \Psi_1(x) = e^{-2\mu x_2}\mathds{1}_\Omega(x) \mbox{ in $\BH_+$},\quad \Psi_1(x_1,0) = 0 \mbox{ for }x_1\in \BT,
\eeqo
while $\Psi_2$ is the solution to
\beq
-\D\Psi_2(x) = 0 \mbox{ in $\BH_+$},\quad \Psi_2(x_1,0) = h(x_1)\mbox{ for }x_1\in \BT.
\label{eqn: equation for Psi_2}
\eeq
We can use the reflection principle to obtain an integral representation of $\Psi_1$ as
\beq
\begin{split}
\Psi_1(x_1,x_2) &=
\int_\Omega e^{-2\mu y_2} \left[G(x_1-y_1,x_2-y_2)-G(x_1-y_1,x_2+y_2)\right] dy\\
&= -\f{1}{4\pi}\int_\Om e^{-2\mu y_2} \ln \left(\f{\cosh(x_2-y_2)- \cos(x_1-y_1)}{\cosh(x_2+y_2)-\cos(x_1-y_1)}\right) dy.
\end{split}
\label{eqn: representation of Psi_1}
\eeq
Since $\Psi_2 = \Psi-\Psi_1$, \beq
\begin{split}
\Psi_2(x_1,x_2) &=
\int_\Omega e^{-2\mu y_2} \left[G(x_1-y_1,x_2+y_2) - G(\pi-y_1,0-y_2)+\f{1}{4\pi}x_2\right] dy\\
&= -\f{1}{4\pi}\int_\Om e^{-2\mu y_2}
\left[\ln \left(\f{\cosh(x_2+y_2)- \cos(x_1-y_1)}{\cosh y_2-\cos(\pi-y_1)}\right)
-x_2\right] dy.
\end{split}
\label{eqn: representation of Psi_2 old}
\eeq
Alternatively, since $\Psi_2$ is harmonic in $\mathbb{H}_+$ with a Dirichlet boundary condition  $\Psi_2(x_1,0) = h(x_1)$ along $\BT\times \{0\}$ (see \eqref{eqn: equation for Psi_2}),
we immediately obtain a second representation of $\Psi_2$:
\beq
\Psi_2(x_1,x_2) = \int_\BT P(x_1-y_1,x_2)h(y_1)\,dy_1,
\label{eqn: representation of Psi_2}
\eeq
where
\beq
P(x_1,x_2) := \f{1}{2\pi}\cdot \f{\sinh x_2}{\cosh x_2- \cos x_1}
\label{eqn: Poisson kernel}
\eeq
is the Poisson kernel on $\BH=\BT\times [0,+\infty)$.

\subsection{The quantity $F$ and its monotonicity}
\label{sec: monotonicity}
Define
\beq
F(x) = F(x;g):= \f{1}{x_2} \Psi(x;g),
\label{eqn: def of F}
\eeq
and correspondingly,
\beq
F_1(x) := \f{1}{x_2} \Psi_1(x),\quad
F_2(x) := \f{1}{x_2} \Psi_2(x).
\label{eqn: def of F_1 F_2}
\eeq
The desired equality \eqref{eqn: constraint for the boundary curve in Psi} can then be rewritten as
\beq
F(x_1,g(x_1);g) = \partial_{x_2}\Psi(\pi,0;g)\cdot \frac{1-e^{-2\mu g(x_1)}}{2\mu g(x_1)}
\quad \text{for $x_1\in\BT$}.
\label{eqn: constraint for the boundary curve in F}
\eeq
From now on, we shall treat $\mu$ as a parameter that continuously ranges in $(0,\f12]$, instead of only taking the discrete values in $\{\f1m:\,m\in \BZ_+,\, m\geq 2\}$ --- although these values will still be of particular interest.
In addition, we will also allow $\mu$ to take the value $0$, which can be treated as a formal limiting case.
Since
\[
\lim_{\mu\to 0^+} \f{1-e^{-2\mu x_2}}{2\mu x_2} = \lim_{x_2\to 0^+} \f{1-e^{-2\mu x_2}}{2\mu x_2} = 1,
\]
in the case $\mu=0$, we will always interpret the function $x_2\mapsto \f{1-e^{-2\mu x_2}}{2\mu x_2}$ (e.g.\;on the right-hand side of \eqref{eqn: constraint for the boundary curve in F}) as the constant $1$ for all $x_2\geq 0$.

The motivation for introducing $F$, instead of directly working with $\Psi$, is that we can prove the following monotonicity property of $F$ in $\BH_+$.
\begin{prop}
\label{prop: monotonicity of F}
Assume that $g \in \BM$ (see \eqref{eqn: function set tilde M_0}). Then $\pa_{x_1}F(0,x_2;g)=\pa_{x_1}F(\pm\pi,x_2;g)=0$ for $x_2>0$, $\pa_{x_1}F(x_1,x_2;g)< 0$ on $(0,\pi)\times (0,+\infty)$, and $\pa_{x_2}F(x_1,x_2;g)<0$ on $[-\pi,\pi]\times (0,+\infty)$.
\end{prop}

We first study $F_1$.
\begin{lem}
\label{lem: monotonicity of F_1}
For any $g \in \BM$, $\pa_{x_1}F_1(x_1,x_2)\leq 0$ on $[0,\pi]\times (0,+\infty)$ with the equality attained only when $x_1=0,\pi$, and $\pa_{x_2}F_1(x_1,x_2)<0$ on $[-\pi,\pi]\times (0,+\infty)$.

\begin{proof}
Since $g\in \BM$ is decreasing on $[0,\pi]$, we define for $x_2 \geq 0$ that
\beq
g^{-1}(x_2)
:=
\begin{cases}
\sup\{x_1\in [0,\pi]:\, g(x_1)\geq x_2\}, & \mbox{if }x_2 \leq g(0),
\\
0, & \mbox{if } x_2 > g(0).
\end{cases}
\label{eqn: def of g inverse}
\eeq
To show the monotonicity of $F_1$ in the $x_1$-direction, we derive from \eqref{eqn: representation of Psi_1} that
\begin{align*}
&\; \pa_{x_1}\Psi_1(x_1,x_2) \\
= &\;
\int_\Omega e^{-2\mu y_2} \cdot \pa_{x_1}\big[G(x_1-y_1,x_2-y_2)-G(x_1-y_1,x_2+y_2)\big]\, dy
\\
= &\;
- \int_0^{g(0)} e^{-2\mu y_2}  \int_{-g^{-1}(y_2)}^{g^{-1}(y_2)} \pa_{y_1}\big[G(x_1-y_1,x_2-y_2)-G(x_1-y_1,x_2+y_2)\big]\, dy_1\,dy_2
\\
= &\;
\int_0^{g(0)} e^{-2\mu y_2}\big[G(x_1+g^{-1}(y_2),x_2-y_2)-G(x_1+g^{-1}(y_2),x_2+y_2)\\
&\;\qquad \qquad \quad  - G(x_1-g^{-1}(y_2),x_2-y_2)
+ G(x_1-g^{-1}(y_2),x_2+y_2)\big]\, dy_2
\\
= &\;
\f1{4\pi} \int_0^{g(0)} e^{-2\mu y_2}\\
&\;\qquad \quad\cdot
\ln\left(\f{\cosh(x_2+y_2)-\cos(x_1+g^{-1}(y_2))} {\cosh(x_2-y_2)-\cos(x_1+g^{-1}(y_2))}\cdot \f{\cosh(x_2-y_2)-\cos(x_1-g^{-1}(y_2))}{\cosh(x_2+y_2)-\cos(x_1-g^{-1}(y_2))} \right) dy_2.
\end{align*}
Observe that for $x_1\in [0,\pi]$ and $x_2 \geq 0$,
\begin{equation}\label{eqn: strict estimate}
\begin{split}
&\big[\cosh(x_2+y_2)-\cos(x_1+g^{-1}(y_2))\big]
\big[\cosh(x_2-y_2)-\cos(x_1-g^{-1}(y_2))\big]
\\
&-\big[\cosh(x_2-y_2)-\cos(x_1+g^{-1}(y_2))\big] \big[\cosh(x_2+y_2)-\cos(x_1-g^{-1}(y_2))\big]
\\
&= \big[\cosh(x_2+y_2)-\cosh(x_2-y_2)\big]
\big[ \cos(x_1+g^{-1}(y_2))-\cos(x_1-g^{-1}(y_2))\big]
\\
&= -2\big[\cosh(x_2+y_2)-\cosh(x_2-y_2)\big]
\sin x_1 \sin g^{-1}(y_2) \\
&\leq 0.
\end{split}
\end{equation}
In the last inequality, we used the fact that $g^{-1}(y_2)\in [0,\pi]$.
Since $g\in \BM$ (see \eqref{eqn: function set tilde M_0}), $g(x)$ is continuous at $x=\pi$ and strictly positive on $(0,\pi)$. This implies that $\sin g^{-1}(y_2)>0$ at least for $y_2\in (0,\delta)$ for some $\delta>0$. Hence, for any $(x_1,x_2)\in (0,\pi)\times (0,+\infty)$, the inequality in the last line of \eqref{eqn: strict estimate} is strict for $y_2\in (0,\delta)$, which implies
\[
\pa_{x_1}F_1(x_1,x_2) = \f{1}{x_2}\cdot \pa_{x_1}\Psi_1(x_1,x_2)< 0.
\]

To show the monotonicity of $F_1$ in $x_2$, we derive from \eqref{eqn: representation of Psi_1} that
\begin{align*}
&x_2 \pa_{x_2}\Psi_1(x) - \Psi_1(x)
\\
&= \f{x_2}{4\pi}\int_{-\pi}^\pi \int_0^{g(y_1)}
e^{-2\mu y_2}\\
&\qquad\qquad\qquad \cdot
\left[\f{\sinh(x_2+y_2)}{\cosh(x_2+y_2)-\cos(x_1-y_1)}
-\f{\sinh(x_2-y_2)}{\cosh(x_2-y_2)-\cos(x_1-y_1)}\right] dy_2\,dy_1
\\
&\quad\, -\f{1}{4\pi}\int_{-\pi}^\pi \int_0^{g(y_1)} e^{-2\mu y_2}
\ln\left(\f{\cosh(x_2+y_2)-\cos(x_1-y_1)}{\cosh(x_2-y_2)-\cos(x_1-y_1)}\right)dy_2\,dy_1\\
&=: \f{1}{4\pi}\int_{-\pi}^\pi \r\big(x_2; g(y_1),\cos(x_1-y_1),\mu\big)\,dy_1,
\end{align*}
where
\beq
\begin{split}
&\r(x;a,b,\mu)\\
&:= \int_0^a e^{-2\mu y}\left[x\left(\f{\sinh(x+y)}{\cosh(x+y)-b}
-\f{\sinh(x-y)}{\cosh(x-y)-b}\right)- \ln\left(\f{\cosh(x+y)-b}{\cosh(x-y)-b}\right)\right]dy.
\label{eqn: def of rho a b}
\end{split}
\eeq
Note that $\r\big(x_2; g(y_1),\cos(x_1-y_1),\mu\big)$ is absolutely integrable as a function of $y_1\in[-\pi,\pi]$. In what follows, we shall prove that $\r(x;a,b,\mu)<0$ for any $x>0$, $a>0$, $b\in [-1,1)$, and $\mu \geq 0$.
Once this is confirmed, we will immediately have that, for any $\mu \geq 0$ and $x_2 > 0$,
\[
x_2 \pa_{x_2}\Psi_1(x) - \Psi_1(x) <0,
\]
and thus
\[
\pa_{x_2}F_1(x) = \f{1}{x_2^2}\left(x_2 \pa_{x_2}\Psi_1(x) - \Psi_1(x)\right)<0.
\]

For arbitrary $\mu \geq 0$, we deduce from \eqref{eqn: def of rho a b} that
\begin{align*}
\r(x;a,b,\mu)
=&\; \int_0^a e^{-2\mu y}\cdot \pa_y \r(x;y,b,0)\, dy
\\
=&\; \int_0^a e^{-2\mu a}\cdot \pa_y \r(x;y,b,0)\, dy
+ \int_0^a \big[ e^{-2\mu y}-e^{-2\mu a}\big]\cdot \pa_y \r(x;y,b,0)\, dy
\\
=&\; e^{-2\mu a}\r(x;a,b,0)
+2\mu \int_0^a e^{-2\mu y} \r(x;y,b,0)\, dy.
\end{align*}
We thus denote $\eta(x;a,b):= \r(x;a,b,0)$.
It suffices to show that $\eta(x;a,b)<0$ for any $x>0$, $a>0$, and $b\in [-1,1)$.

It is not difficult to verify that, for any fixed $a>0$ and $b\in [-1,1)$, $\eta(0;a,b)=0$ and
\[
\lim_{x\to +\infty}\eta(x;a,b) = -2\int_0^a y\,dy = -a^2 \leq 0.
\]
Moreover,
\begin{align*}
&\; \eta_x (x;a,b)\\
= &\; x \int_0^a \pa_x\left(\f{\sinh(x+y)}{\cosh(x+y)-b}
-\f{\sinh(x-y)}{\cosh(x-y)-b}\right) dy
\\
= &\; x \int_0^a \pa_y\left(\f{\sinh(x+y)}{\cosh(x+y)-b}
+\f{\sinh(x-y)}{\cosh(x-y)-b}\right) dy
\\
= &\; x\left(\f{\sinh(x+a)}{\cosh(x+a)-b}
+ \f{\sinh(x-a)}{\cosh(x-a)-b} - \f{2\sinh x}{\cosh x-b}\right)
\\
= &\;
\f{4x(\sinh \f{a}{2})^2 \sinh x}{(\cosh(x+a)-b)(\cosh(x-a)-b)(\cosh x-b)}
\cdot \big(b\cosh x + b^2 -\cosh a -1\big).
\end{align*}
Hence, the sign of $\eta_x (x;a,b)$ is determined by that of $\va_{a,b}(x):=b\cosh x + b^2 -\cosh a -1$.
We proceed in two cases.
\begin{enumerate}
\item If $b\in [-1,0]$, then $\eta_x(x;a,b)<0$ for $x\geq 0$, which implies $\eta(x;a,b)<0$ for all $x>0$;

\item If $b\in (0,1)$, $\va_{a,b}(x)$ is strictly increasing on $[0,+\infty)$.
Moreover, $\va_{a,b}(0) = b+b^2 -\cosh a -1 < 0$, and $\lim_{x\to +\infty}\va_{a,b}(x) = +\infty$.
Let $x_*$ be its unique positive root.
Then $\eta_x(x;a,b)<0$ for $x<x_*$, while $\eta_x(x;a,b)>0$ for $x>x_*$.
This together with the facts $\eta(0;a,b)=0$ and $\lim_{x\to +\infty}\eta(x;a,b)\leq 0$ implies that $\eta(x;a,b)<0$ for all $x>0$.
\end{enumerate}
This completes the proof.
\end{proof}
\end{lem}

To study $F_2$, we will need the following lemmas.
\begin{lem}
\label{lem: monotonicity of h}
Given any $g \in \BM$, let $h$ be given by \eqref{eqn: formula for h}.
Then
\begin{enumerate}
\item $h = h(x_1)$ is even on $[-\pi,\pi]$, and $h\in C^{1,\alpha}(\BT)$ for any $\al\in (0,1)$;
\item $h'(x_1)\leq 0$ for $x_1\in [0,\pi]$, with the equality only attained at $x_1 = 0,\pi$;
\item $h(x_1)\geq 0$ on $\BT$, with the equality only achieved at $x_1 = \pm \pi$.
\end{enumerate}

\begin{proof}
That $h$ is even follows from \eqref{eqn: formula for h}.
The $C^{1,\al}(\BT)$-regularity of $h$ follows from \eqref{eqn: def of h} and the fact that $\phi\in C^{1,\alpha}_{loc}(\BR^2)$ due to the standard regularity theory for elliptic equations.

Let $g^{-1}$ be defined as in \eqref{eqn: def of g inverse}.
Then by \eqref{eqn: def of G} and \eqref{eqn: formula for h},
\begin{align*}
h'(x_1)
= &\;  \int_\Om e^{-2 \mu y_2} \cdot \pa_{x_1} G(x_1-y_1,0-y_2)\, dy_1\,dy_2
\\
= &\; \f{1}{4\pi}\int_0^{g(0)} e^{-2 \mu y_2} \int_{-g^{-1}(y_2)}^{g^{-1}(y_2)} \pa_{y_1}\left( \ln \big[2\cosh y_2 -2\cos(x_1-y_1)\big]\right) dy_1\,dy_2
\\
= &\; \f{1}{4\pi}\int_0^{g(0)} e^{-2 \mu y_2} \cdot \ln \left(\f{2\cosh y_2 -2\cos(x_1-g^{-1}(y_2))}{2\cosh y_2 -2\cos(x_1+g^{-1}(y_2))}\right) dy_2.
\end{align*}
Since for $x_1\in [0,\pi]$,
\beqo
\cos(x_1-g^{-1}(y_2)) - \cos(x_1+g^{-1}(y_2))
= 2\sin x_1 \sin g^{-1}(y_2) \geq 0,
\eeqo
we conclude that $h'(x_1)\leq 0$ for $x_1\in [0,\pi]$. Since $g>0$ on $[0,\pi)$, the function $g^{-1}(y_2)$ cannot take the values $0$ and $\pi$ only; otherwise, by the continuity of $g(x_1)$ at $x_1=\pi$, $g$ must be constant zero on $(0,\pi]$, which is a contradiction. Hence, $h'(x_1)<0$ for all $x_1\in (0,\pi)$.

Finally, that strict positivity of $h$ in $(-\pi,\pi)$ follows from the strict monotonicity of $h$ in $(0,\pi)$ and the fact $h(\pm \pi) = 0$.
\end{proof}
\end{lem}

\begin{lem}
\label{lem: monotonicity of P/x_2}
Recall that the Poisson kernel $P = P(x_1,x_2)$ is given by \eqref{eqn: Poisson kernel}.

\begin{enumerate}
\item For any $x_1\in \BT$ and $x_2>0$,
\[
\f{1}{x_2}P(x_1,x_2) > 0,
\]
and
\[
\pa_{x_2}\left[\f{1}{x_2}P(x_1,x_2)\right]
\leq \f{(x_2-\sinh x_2)(1+\cosh x_2)}{2\pi x_2^2(\cosh x_2- \cos x_1)^2}
< 0;
\]
\item For any $x_1\in \BT\setminus \{0\}$,
\[
\lim_{x_2\to 0^+}\f{1}{x_2}P(x_1,x_2) = \f{1}{2\pi}\cdot \f{1}{1-\cos x_1}.
\]
\end{enumerate}

\begin{proof}
We only show the second inequality, as the others are trivial.
We calculate that, for any $x_1\in \BT$ and $x_2 > 0$,
\begin{align*}
&\; \pa_{x_2}\left[\f{1}{x_2}P(x_1,x_2)\right]
\\
= &\; \f{1}{2\pi}\cdot \f{\cosh x_2\cdot x_2(\cosh x_2- \cos x_1)-\sinh x_2 (x_2\sinh x_2 + \cosh x_2- \cos x_1)}{x_2^2(\cosh x_2- \cos x_1)^2}
\\
= &\; \f{x_2 - \sinh x_2 \cosh x_2  - (x_2-\tanh x_2)\cosh x_2\cos x_1}{2\pi x_2^2(\cosh x_2- \cos x_1)^2}
\\
\leq &\; \f{x_2 - \sinh x_2 \cosh x_2  + (x_2-\tanh x_2)\cosh x_2}{2\pi x_2^2(\cosh x_2- \cos x_1)^2}
\\
= &\; \f{(x_2-\sinh x_2)(1+\cosh x_2)}{2\pi x_2^2(\cosh x_2- \cos x_1)^2}
< 0.
\end{align*}
Here we used the fact that $\tanh x_2 < x_2 < \sinh x_2$ for any $x_2> 0$.
\end{proof}
\end{lem}

As long as $g\in \BM$, we can prove monotonicity of $F_2$ in $[0,\pi]\times [0,+\infty)$.
\begin{lem}
\label{lem: monotonicity of F_2}
For any $g \in \BM$, $\pa_{x_1}F_2(x_1,x_2)\leq 0$ on $[0,\pi]\times (0,+\infty)$ with the equality attained only when $x_1=0,\pi$, and $\pa_{x_2}F_2(x_1,x_2)< 0$ on $[-\pi,\pi]\times (0,+\infty)$.

\begin{proof}
To show the monotonicity in $x_1$, we fix $x_2 > 0$.
By \eqref{eqn: representation of Psi_2} and Lemma \ref{lem: monotonicity of h},
\begin{align*}
\pa_{x_1} F_2(x_1,x_2)
&= \f{1}{x_2}\int_\BT P(y_1,x_2)h'(x_1-y_1)\,dy_1\\
&= \f{1}{x_2}\int_\BT P(x_1-y_1,x_2)h'(y_1)\,dy_1\\
&= \f{1}{x_2}\int_0^{\pi} \left(P(x_1-y_1,x_2) - P(x_1+y_1,x_2)\right)h'(y_1)\,dy_1\\
&= \f{1}{\pi x_2}\int_0^{\pi} \frac{\sinh(x_2)\sin(x_1)\sin(y_1)}{(\cosh(x_2)-\cos(x_1-y_1))(\cosh(x_2)-\cos(x_1+y_1))}\cdot h'(y_1)\,dy_1\\
&\leq 0.
\end{align*}
Note that by Lemma \ref{lem: monotonicity of h}, $h'(y_1)<0$ for $y_1\in(0,\pi)$. Hence, $\pa_{x_1}F_2(x_1,x_2)< 0$ for $x_1 \in (0, \pi)$ and $x_2 > 0$.

By \eqref{eqn: representation of Psi_2}, we have
\[
\pa_{x_2}F_2(x_1,x_2)= \int_\BT \pa_{x_2}\left[\f{1}{x_2} P(x_1-y_1,x_2)\right] h(y_1)\,dy_1.
\]
Then the strict monotonicity in $x_2$ follows from Lemma \ref{lem: monotonicity of h} and Lemma \ref{lem: monotonicity of P/x_2}.
\end{proof}
\end{lem}

Proposition \ref{prop: monotonicity of F} is an immediate consequence of Lemma \ref{lem: monotonicity of F_1} and Lemma \ref{lem: monotonicity of F_2}.

\subsection{Reformulation into a fixed-point problem}
By virtue of the monotonicity of $F$, we can study the problem \eqref{eqn: constraint for the boundary curve in F} from an iterative perspective as in the following proposition.

\begin{prop}\label{prop: implicit mapping}
Given $g\in \BM$, there exists a unique $\tilde g \in \BM$ such that
\[F(x_1,\tilde g(x_1);g) = \partial_{x_2}\Psi(\pi,0;g)\cdot \frac{1-e^{-2\mu g(x_1)}}{2\mu g(x_1)},
\quad x_1\in\BT.\]
Moreover, $\tilde g$ is strictly decreasing on $[0,\pi]$.
\end{prop}

\begin{proof}
Fix $g\in \BM$.
We define
\beq
H(x_1,x_2) = H(x_1,x_2;g) := F(x_1,x_2;g) - \partial_{x_2}\Psi(\pi,0;g)\cdot \frac{1-e^{-2\mu g(x_1)}}{2\mu g(x_1)}.
\label{eqn: def of H as F minus RHS in the implicit mapping}
\eeq
In particular, $H(\pm\pi,x_2)= F(\pm\pi,x_2;g) - \partial_{x_2}\Psi(\pi,0;g)$ since $g(\pm\pi)=0$. Note that the dependence of $H(x_1,x_2)$ in $x_2$ only comes from that of $F(x_1,x_2)$.  For $x_1\in [-\pi,\pi]$, it is straightforward to check by \eqref{eqn: representation of Psi_1} and \eqref{eqn: representation of Psi_2} that
\begin{align*}
\lim_{x_2\to +\infty} \Psi(x_1,x_2) &= \lim_{x_2\to +\infty} \Psi_1(x_1,x_2) + \lim_{x_2\to +\infty} \Psi_2(x_1,x_2)\\
&= \frac{1}{2\pi}\int_{\Omega}e^{-2\mu y_2} y_2\, dy + \frac{1}{2\pi}\int_{-\pi}^{\pi}h(y_1)\, dy_1 \in (0,+\infty).
\end{align*}
Hence, for all $x_1\in [-\pi,\pi]$, $\lim_{x_2\to +\infty} F(x_1,x_2) = 0$. Also note that $\lim_{x_2\to 0+}F(x_1,x_2) = +\infty$ for $x_1\in(-\pi,\pi)$, and $\lim_{x_2\to 0+}F(\pm\pi,x_2) = \partial_{x_2}\Psi(\pi,0)$. From \eqref{eqn: end point constant}, we know $\pa_{x_2}\Psi(\pi,0)>0$. It follows that
\[\lim_{x_2\to 0^+} H(x_1,x_2) = +\infty,\quad x\in(-\pi,\pi),\quad \text{and}\quad \lim_{x_2\to 0^+} H(\pm \pi,x_2) = 0,\]
while
\[\lim_{x_2\to +\infty} H(x_1,x_2) = - \partial_{x_2}\Psi(\pi,0;g)\cdot \frac{1-e^{-2\mu g(x_1)}}{2\mu g(x_1)} <0.\]
By Proposition \ref{prop: monotonicity of F}, $\pa_{x_2}H(x_1,x_2) = \pa_{x_2}F(x_1,x_2)<0$ on $[-\pi,\pi]\times (0,+\infty)$. As a consequence, for each $x_1\in[-\pi,\pi]$, $x_2\mapsto H(x_1,x_2)$ has a unique root in $[0,+\infty)$, which we denote as $\tilde g(x_1)$; in addition, $\tilde g(x_1)>0$ for $x_1\in(-\pi,\pi)$ and $\tilde g(\pm\pi)=0$.
Obviously, $\tilde g(x_1)$ is even on $[-\pi,\pi]$.

Next, we show that $\tilde g(x_1)$ is strictly decreasing in $x_1$ on $[0,\pi]$. By Proposition \ref{prop: monotonicity of F}, $\pa_{x_1}F(x_1,x_2)< 0$ on $(0,\pi)\times (0,+\infty)$, so $F(x_1,x_2)$ is strictly decreasing in $x_1$ on $[0,\pi]\times (0,+\infty)$. Also, $g(x_1)$ is non-increasing in $x_1$ on $[0,\pi]$, which implies the function
\[x_1\mapsto -\pa_{x_2}\Psi(\pi,0)\cdot \frac{1-e^{-2\mu g(x_1)}}{2\mu g(x_1)}\]
is non-increasing in $x_1$ on $[0,\pi]$. As a result, $H(x_1,x_2)$ is strictly decreasing on $[0,\pi]\times (0,+\infty)$. Therefore, for any $s,t\in[0,\pi]$ with $s<t$,
\[
H\big(s,\tilde g(t)\big) > H\big(t,\tilde g(t)\big) =0 = H\big(s,\tilde g(s)\big).
\]
This implies $\tilde g(t)< \tilde g(s)$ since $H(x_1,x_2)$ is strictly decreasing in $x_2$.

Finally, to prove $\tilde g\in \BM$, it remains to show that $\lim_{x_1\to \pi^-}\tilde g(x_1)=0$. We prove this by contradiction. Suppose $\lim_{x_1\to \pi^-}\tilde g(x_1)=\tilde x_2>0$. Since $F$ is continuous on $\BT\times (0,+\infty)$ and $\lim_{x_1\to \pi^-} g(x_1)=0=g(\pi)$, we have
\begin{align*}
H(\pi,\tilde x_2) &= F(\pi,\tilde x_2) - \partial_{x_2}\Psi(\pi,0)\\
&= \lim_{x_1\to \pi^-} F(x_1,\tilde g(x_1)) - \partial_{x_2}\Psi(\pi,0)\cdot \frac{1-e^{-2\mu g(x_1)}}{2\mu g(x_1)}\\
&= \lim_{x_1\to \pi^-} H(x_1,\tilde g(x_1)) =0.
\end{align*}
However, since $H(\pi,0) = 0$ and $\pa_{x_2}H(\pi,x_2)<0$ for all $x_2>0$, we should have $H(\pi,\tilde x_2)<0$, which is a contradiction. Therefore, $\lim_{x_1\to \pi^-}\tilde g(x_1)=0$.
By symmetry, $\lim_{x_1\to -\pi^+}\tilde g(x_1)=0$.
This completes the proof.
\end{proof}

Thanks to Proposition \ref{prop: implicit mapping}, we can define an implicit mapping $\bfR:\BM \to \BM$ as follows.
Fix $\mu \geq 0$.
For any $g\in \BM$, we define $\Psi$ and $F$ as in \eqref{eqn: integral representation of Psi} and \eqref{eqn: def of F}, respectively, and we define the image of $g$ under the mapping $\bfR$ to be the function $\bfR(g):\BT \to [0,+\infty)$ which satisfies
\beq
F(x_1,\bfR(g)(x_1);g) = \partial_{x_2}\Psi(\pi,0;g)\cdot \frac{1-e^{-2\mu g(x_1)}}{2\mu g(x_1)},
\quad \forall\, x_1\in\BT.
\label{eqn: def of R(g)}
\eeq
Proposition \ref{prop: implicit mapping} guarantees the existence and uniqueness of such $\bfR(g)\in \BM$, and moreover, $\bfR(g)$ is strictly decreasing in $[0,\pi]$.
In view of \eqref{eqn: def of R(g)}, the problem of finding $g\in \BM$ such that \eqref{eqn: constraint for the boundary curve in F} (which is equivalent to \eqref{eqn: constraint for the boundary curve in Psi} and also \eqref{eqn: constraint for the boundary curve in Phi}) holds, is equivalent to finding a fixed point of $\bfR$ in $\BM$.
We will show the existence of the fixed point in the next two sections.

\section{Estimates for $F$ and $\bfR(g)$}
\label{sec: estimates}
In this section, we shall prove some a priori estimates for $F$ and $\bfR(g)$.
They are essential for the fixed-point argument in Section \ref{sec: existence of the fixed point} as they will suggest the choice of the set on which we should apply the fixed-point theorem (e.g.\;see \eqref{eqt: def of function set D_0} and \eqref{eqt: def of function set D} below).
The main results of this section are Proposition \ref{prop: a priori upper bound general m}, Proposition \ref{prop: upper barrier new}, and Proposition \ref{prop: lower barrier}.

\subsection{An a priori upper bound}
We first show that, when $\mu$ is close to $0$, the mapping $\bfR$ can preserve an a priori upper bound for $g$.

\begin{prop}\label{prop: a priori upper bound general m}
There exist universal constants $M\geq 1$ and $\mu_0\in (0,\f12]$, such that as long as $\mu \in [0,\mu_0]$, for any $g\in \BM$, that $g(0)\leq M$ implies $\mathbf{R}(g)(0) < M$.

\begin{proof}
We will prove a stronger result: there is a universal strictly increasing positive function $M_*=M_*(c)$ defined on $[0,+\infty)$, such that, as long as $\mu\in [0,\f12]$ and $M\geq M_*(2\mu M)$, it holds that for any $g\in \BM$, that $g(0) \leq M$ would imply $\mathbf{R}(g)(0) < M$.
Once this is confirmed, we can simply take, for instance, $M= M_*(1)$ and $\mu_0 \leq (2M)^{-1}$ to achieve the desired statement.

With $M>0$ to be chosen, we take $g\in \BM$ such that $g(0) \leq M$.
For $x_2 \geq 0$, let $g^{-1}(x_2)$ be defined as in \eqref{eqn: def of g inverse}, which is a non-negative decreasing function on $[0,+\infty)$.
With $\mu\geq 0$, let $\Psi$ be defined in terms of $g$ as in Lemma \ref{lem: integral formula for Psi}.
In order to have $\mathbf{R}(g)(0) < M$, thanks to the monotonicity of $F$ (see Proposition \ref{prop: monotonicity of F}), it suffices to achieve that \beq
\f{\Psi(0, M)}{M} <
\pa_{x_2} \Psi(\pi, 0)\cdot \f{1-e^{-2\mu M}}{2\mu M}
\label{eqn: F at the peak general m}
\eeq
provided that $g(0) \leq M$.

By Lemma \ref{lem: integral formula for Psi},
\begin{align*}
\f{\Psi(0, M)}{M}
= &\; -\f{1}{4\pi M} \int_\Om e^{-2\mu y_2}
\left[\ln \left( \f{\cosh(M-y_{2})-\cos y_1}{\cosh y_2 + \cos y_{1}} \right) - M \right] dy
\\
= &\; \f{1}{4\pi M} \int_\Om e^{-2\mu y_2} \left[ \ln \left(\frac{\cosh y_2 + \cos y_{1}}{\cosh(M-y_2)-\cos y_1} \right) +  M \right] dy
\\
\leq &\; \f{1}{4\pi M}  \int_0^{g(0)} e^{-2\mu y_2} \int_{-g^{-1}(y_2)}^{g^{-1}(y_2)} \left[\ln \left(\f{\cosh y_2 + 1}{\cosh(M -y_2)-1} \right) +  M\right] dy_1\,dy_2
\\
= &\; \f{1}{2\pi M} \int_0^{g(0)} e^{-2\mu y_2} g^{-1}(y_2) \left[2\ln \left(\f{\cosh \f{y_2}{2}}{\sinh \f{M-y_2}{2}} \right) +  M \right] dy_2
\\
= &\; \f{1}{2\pi} \int_0^1 e^{-2\mu M s} g^{-1}(M s) \left[ 2\ln \left( \f{\cosh \f{M s}{2}}{\sinh \f{ M (1-s)}{2}}\right) +  M \right] ds.
\end{align*}
In the last line, we used the fact that $g(0) \leq M$ and $g^{-1}(y_2) = 0$ for all $y_2>g(0)$.
On the other hand, \begin{align*}
\pa_{x_2} \Psi (\pi, 0)
= &\; \f{1}{4\pi} \int_\Om e^{-2\mu y_2} \left[\f{\sinh y_{2}}{\cosh y_2 + \cos y_1}+1\right] dy
\\
\geq &\; \f{1}{4\pi} \int_0^{g(0)} e^{-2\mu y_2} \int_{-g^{-1}(y_2)}^{g^{-1}(y_2)} \left[\f{\sinh y_2}{\cosh y_2 + 1} + 1\right] dy_1\,dy_2
\\
= &\; \f{1}{2\pi} \int_{0}^{g(0)} e^{-2\mu y_2} g^{-1}(y_{2})\left[\tanh\left(\frac{y_{2}}{2}\right) +1\right] dy_2
\\
= &\; \f{1}{2\pi}\int_0^1 e^{-2\mu M s} g^{-1}(Ms)\cdot M\left[\tanh\left(\f{Ms}{2}\right)+1\right] ds.
\end{align*}

Let
\beq
A(s)=
A_{\mu,M}(s):= 2\ln \left( \f{\cosh \f{M s}{2}}{\sinh \f{M(1-s)}{2}} \right)
+ M - M\left[\tanh\left(\f{Ms}{2}\right) + 1\right]\f{1-e^{-2\mu M}}{2\mu M}.
\label{eqn: def of A(s)}
\eeq
In view of the estimates above, in order to achieve \eqref{eqn: F at the peak general m}, it suffices to show that
\beq
\int_0^1 g^{-1}(Ms) \cdot e^{-2\mu M s} A_{\mu,M}(s)\, ds < 0. \label{eqn: sufficient condition for the sign of F at the peak general m}
\eeq
Since the function $s\mapsto g^{-1}(Ms)$ is non-negative and decreasing on $[0,1]$, we only have to guarantee that, for any $s\in (0,1)$,
\beq
\int_0^{s} e^{-2\mu M s'} A_{\mu,M}(s')\, ds' < 0.
\label{eqn: further sufficient condition for the sign of F at the peak general m}
\eeq
Indeed, if the function $s\mapsto g^{-1}(Ms)$ is a piecewise constant decreasing function, the integral in \eqref{eqn: sufficient condition for the sign of F at the peak general m} can be written as a linear combination of the integrals $\int_0^{s} e^{-2\mu M s'}A(s')\, ds'$ with non-negative coefficients.
The general case can be approximated by the case of piecewise constant decreasing functions.

We first observe that $A_{\mu,M}(0)\leq 0$ is a necessary condition for \eqref{eqn: further sufficient condition for the sign of F at the peak general m} to hold for all $s\in(0,1)$.
We derive that
\beq
\begin{split}
A_{\mu,M}(0) &= -2\ln \left(\sinh \f{M}{2}\right)
+ M\left[1 - \f{1-e^{-2\mu M}}{2\mu M}\right]
\\
&= - 2 \ln \left(\f{1-e^{-M}}{2}\right)
- M \cdot \f{1-e^{-2\mu M}}{2\mu M} =: \z_0 (M;2\mu M),
\end{split}
\label{eqn: A(0)}
\eeq
where for $M > 0$ and $c\geq 0$, we define
\beq
\z_0(M;c):= - 2 \ln \left(\f{1-e^{-M}}{2}\right)
- M \cdot \f{1-e^{-c}}{c}.
\label{eqn: def of zeta_0}
\eeq
Clearly, $\z_0(M;c)$ is strictly decreasing in $M > 0$ and strictly increasing in $c \geq 0$.
Moreover, for any fixed $c\geq 0$,
\[
\lim_{M\to 0^+}\z_0(M;c) = +\infty,\quad
\lim_{M\to +\infty} \z_0(M;c) = -\infty,
\]
so there exists a unique $M_0 = M_0(c) > 0$, such that $\z_0(M_0(c);c)= 0$.
Besides, $c\mapsto M_0(c)$ is strictly increasing.
This implies that, for any fixed $c\geq 0$, as long as $2\mu M \in [0, c]$ and $M\geq M_0(c)$, it holds that
\[
A_{\mu,M}(0) = \z_0(M;2\mu M)\leq \z_0\big(M_0(c); c\big) \leq 0.
\]

Next, we differentiate $A$ in $s$ and find that
\[
A_{\mu,M}'(s)= M \tanh\left(\f{Ms}{2}\right) + M\coth \left(\f{M(1-s)}{2}\right)
- \f{M^2}{2} \cdot \f{1}{\cosh^2\f{M s}{2}}\cdot \f{1-e^{-2\mu M}}{2\mu M},
\]
which is strictly increasing in $[0,1)$.
Also observe that $\lim_{s\to 1^-}A(s) = +\infty$.
As a result, if $A_{\mu,M}(0)\leq 0$, there exists a unique $s_*\in [0,1)$, such that $A(s)\leq 0$ for all $s\in [0,s_*]$ while $A(s)>0$ for all $s\in (s_*,1)$.
Hence, for any $s\in (0,1)$, it holds that
\beq
\int_0^{s} e^{-2\mu M s'} A_{\mu,M}(s')\, ds' < \max\left\{0,\; \int_0^1 e^{-2\mu M s}A_{\mu,M}(s)\, ds\right\}.
\label{eqn: bounding the integral from 0 to s by the one on 0 to 1}
\eeq
Therefore, to guarantee \eqref{eqn: further sufficient condition for the sign of F at the peak general m}, it suffices to achieve that
\beq
\int_0^1 e^{-2\mu M s} A_{\mu,M}(s) \, ds \leq 0.
\label{eqn: the integral from 0 to 1 should be nonpositive}
\eeq

To better study the integral in \eqref{eqn: the integral from 0 to 1 should be nonpositive}, we perform the following derivation.
First, by integration by parts,
\begin{align*}
&\;\int_0^1 e^{-2\mu M s} \tanh\left(\f{Ms}{2}\right) ds\\
= &\;\f{2}{M}\left.\left[e^{-2\mu M s} \ln\cosh\left(\f{Ms}{2}\right)\right]\right|_{s=0}^1
- \f{2}{M} \int_0^1 \big(-2\mu M e^{-2\mu M s}  \big)\cdot \ln\cosh\left(\f{Ms}{2}\right) ds\\
= &\; \f1M e^{-2\mu M} \cdot 2\ln\cosh\left(\f{M}{2}\right)
+ 4\mu \int_0^1 e^{-2\mu M s}  \ln\cosh\left(\f{Ms}{2}\right)ds.
\end{align*}
Hence,
\beq
\begin{split}
&\;\int_0^1 e^{-2\mu M s} A_{\mu,M}(s)\, ds\\
= &\; 2\int_0^1 e^{-2\mu M s} \cdot \ln \cosh \left(\f{Ms}{2}\right) ds
-2\int_0^1 e^{-2\mu M s} \cdot \ln \sinh \left(\f{M(1-s)}{2}\right) ds \\
&\; + \int_0^1 e^{-2\mu M s} \left[M - M \cdot \f{1-e^{-2\mu M}}{2\mu M}\right] ds\\
&\; -\f{1-e^{-2\mu M}}{2\mu M}\cdot M\int_0^1 e^{-2\mu M s} \tanh\left(\f{Ms}{2}\right) ds
\\
= &\; 2\int_0^1 e^{-2\mu M s} \cdot \ln \cosh \left(\f{Ms}{2}\right) ds
-2\int_0^1 e^{-2\mu M (1-s)} \cdot \ln \sinh \left(\f{M s}{2}\right) ds \\
&\; + \left[M - M \cdot \f{1-e^{-2\mu M}}{2\mu M}\right]\f{1-e^{-2\mu M}}{2\mu M} \\
&\; -\f{1-e^{-2\mu M}}{2\mu M}\cdot e^{-2\mu M} \cdot 2\ln\cosh\left(\f{M}{2}\right)
+ 2\big(1-e^{-2\mu M}\big) \int_0^1 e^{-2\mu M s}  \ln\cosh\left(\f{Ms}{2}\right)ds
\\
= &\; 2 e^{-2\mu M} \int_0^1 \left[e^{-2\mu M s} \cdot \ln \cosh \left(\f{Ms}{2}\right)
- e^{2\mu M s} \cdot \ln \sinh \left(\f{Ms}{2}\right)\right] ds \\
&\; + \left[M - M \cdot \f{1-e^{-2\mu M}}{2\mu M}- e^{-2\mu M} \cdot 2\ln\cosh\left(\f{M}{2}\right) \right] \f{1-e^{-2\mu M}}{2\mu M}.
\end{split}
\label{eqn: calculate the integral exp -2 mu Ms A}
\eeq
We further derive by Taylor expansion that, for $\mu \in [0,\f12)$,
\begin{align*}
&\; \int_0^1 \left[e^{-2\mu M s} \cdot \ln \cosh \left(\f{Ms}{2}\right)
- e^{2\mu M s} \cdot \ln \sinh \left(\f{Ms}{2}\right)\right] ds
\\
= &\; \f1{M}\int_0^M \left[e^{-2\mu s} \cdot \ln \left(\f12\cdot  e^{\f{s}{2}} (1+e^{-s})\right)
- e^{2\mu s} \cdot \ln\left(\f12 \cdot e^{\f{s}{2}}(1-e^{-s})\right)\right] ds
\\
= &\; \f1{M}\int_0^M \left(e^{-2\mu s}
- e^{2\mu s}\right)\left(\f{s}{2}-\ln 2\right) ds
\\
&\; + \f1{M}\int_0^M \left[e^{-2\mu s}\sum_{k = 1}^\infty\f{(-1)^{k-1}}{k} (e^{-s})^k
- e^{2\mu s} \sum_{k = 1}^\infty\f{(-1)^{k-1}}{k}(-e^{-s})^k \right] ds
\\
= &\; -\f1{M}\int_0^M s\sinh (2\mu s)\, ds
+ \f{2\ln 2}{M}\int_0^M \sinh (2\mu s) \,ds
\\
&\; + \f1{M}\int_0^M \sum_{k = 1}^\infty \left[\f{(-1)^{k-1}}{k} e^{-s(k+2\mu)} + \f{1}{k} e^{-s(k-2\mu)}\right] ds
\\
= &\; -\f{1}{4\mu^2 M}\big[2\mu M\cosh (2\mu M)-\sinh (2\mu M)\big]
+ \f{2\ln 2}{2\mu M} \big[\cosh (2\mu M)-1\big]
\\
&\; + \f1{M} \left[\f{1-e^{-(1+2\mu)M}}{1+2\mu}
+  \f{1-e^{-(1-2\mu)M}}{1-2\mu} \right]
\\
&\; + \f1{M}\sum_{k = 2}^\infty \left[\f{(-1)^{k-1}}{k} \cdot \f{1-e^{-(k+2\mu)M}}{k+2\mu}
+ \f{1}{k} \cdot\f{1-e^{-(k-2\mu)M}}{k-2\mu} \right].
\end{align*}
Since $z\mapsto \f{e^{-z}}{z}$ is decreasing for $z>0$,
\begin{align*}
&\; \int_0^1 \left[e^{-2\mu M s} \cdot \ln \cosh \left(\f{Ms}{2}\right)
- e^{2\mu M s} \cdot \ln \sinh \left(\f{Ms}{2}\right)\right] ds
\\
\leq
&\; -\f{1}{4\mu^2 M}\big[2\mu M\cosh (2\mu M)-\sinh (2\mu M)\big]
+ \f{2\ln 2}{2\mu M} \big[\cosh (2\mu M)-1\big]
\\
&\;+\f1{M}\left[\f{1-e^{-(1+2\mu)M}}{1+2\mu}
+  \f{1-e^{-(1-2\mu)M}}{1-2\mu} \right]
+ \f1{M}\sum_{k = 2}^\infty  \f{1}{k} \left[\f{(-1)^{k-1}}{k+2\mu}
+\f{1}{k-2\mu}\right]
\\
=
&\; -\f{1}{8\mu^2M}\big[2\mu M (e^{2\mu M}+e^{-2\mu M}) - (e^{2\mu M}-e^{-2\mu M})\big]
+ \f{\ln 2}{2\mu M} \big( e^{2\mu M}+e^{-2\mu M}-2\big)
\\
&\; +\f1{M}\left[\f{1-e^{-(1+2\mu)M}}{1+2\mu}
+  \f{1-e^{-(1-2\mu)M}}{1-2\mu} \right]
\\
&\; + \f{1}{M}\sum_{k = 2}^\infty \f1{k} \left[\f{1+(-1)^{k-1}}{k}
+ 2\mu \left(-\f{(-1)^{k-1}}{(k+2\mu)k} +\f{1}{k(k-2\mu)}\right)\right].
\end{align*}
Since
\[
\sum_{k = 2}^\infty \f1{k}\cdot \f{1+(-1)^{k-1}}{k}
= \sum_{j = 1}^\infty \f{2}{(2j-1)^2} - 2 = \f{\pi^2}{4}-2,
\]
we find that
\begin{align*}
&\; \int_0^1 \left[e^{-2\mu M s} \cdot \ln \cosh \left(\f{Ms}{2}\right)
- e^{2\mu M s} \cdot \ln \sinh \left(\f{Ms}{2}\right)\right] ds
\\
\leq
&\; -\f{1}{2} e^{2\mu M} \cdot M \left[\f{1+e^{-4\mu M}}{2\mu M} -\f{1-e^{-4\mu M}}{(2\mu M)^2} \right]
+ \ln 2 \cdot e^{2\mu M} \cdot 2\mu M  \left(\f{1-e^{-2\mu M}}{2\mu M}\right)^2
\\
&\;
+ \f1{M}\left[\f{1-e^{-(1+2\mu)M}}{1+2\mu}
+ \f{1-e^{-(1-2\mu)M}}{1-2\mu} -2 \right]
+ \f{\pi^2}{4M}
+ \f{2\mu}{M} \sum_{k = 2}^\infty  \f{1}{k^2} \left(\f{(-1)^{k}}{k+2\mu}
+ \f{1}{k-2\mu}\right).
\end{align*}
Plugging this into \eqref{eqn: calculate the integral exp -2 mu Ms A}, we obtain that
\begin{align*}
&\;\int_0^1 e^{-2\mu M s} A_{\mu,M}(s)\, ds\\
\leq &\;
- M \left[\f{1+e^{-4\mu M}}{2\mu M} -\f{1-e^{-4\mu M}}{(2\mu M)^2} \right]
+ 2\ln 2 \cdot 2\mu M  \left(\f{1-e^{-2\mu M}}{2\mu M}\right)^2
\\
&\; + e^{-2\mu M}\cdot \f{\pi^2}{2M}
+ 2e^{-2\mu M}\cdot \f1{M} \left[\f{1-e^{-(1+2\mu)M}}{1+2\mu}
+ \f{1-e^{-(1-2\mu)M}}{1-2\mu} -2 \right]
\\
&\;
+ e^{-2\mu M}\cdot \f{4\mu}{M} \sum_{k = 2}^\infty  \f{1}{k^2} \left(\f{(-1)^{k}}{k+2\mu} +\f{1}{k-2\mu}\right)
 \\
&\; + M\left[\f{1-e^{-2\mu M}}{2\mu M} -\f{1-2e^{-2\mu M}+e^{-4\mu M}}{(2\mu M)^2} \right]
- e^{-2\mu M} \cdot 2\ln\cosh\left(\f{M}{2}\right)\f{1-e^{-2\mu M}}{2\mu M}
\\
= &\;
e^{-2\mu M}\left[\f{\pi^2}{2M} + \f{2}{M} \left(\f{1-e^{-(1+2\mu)M}}{1+2\mu}
+ \f{1-e^{-(1-2\mu)M}}{1-2\mu} -2 \right) \right.\\
&\; \qquad \quad \left. - 2M \cdot \f{1}{2\mu M}  \left(1-\f{1-e^{-2\mu M}}{2\mu M}\right)
+\f{4\mu M}{M^2} \sum_{k = 2}^\infty  \f{1}{k^2} \left(\f{(-1)^{k}}{k+2\mu} +\f{1}{k-2\mu}\right)\right]
\\
&\; + 2\ln 2\cdot \big(1-e^{-2\mu M}\big)
\cdot \f{1-e^{-2\mu M}}{2\mu M}
- e^{-2\mu M}\left[2\ln\cosh\left(\f{M}{2}\right)-M\right] \f{1-e^{-2\mu M}}{2\mu M}.
\end{align*}
Since $2\ln\cosh\f{z}{2} - z > -2\ln 2$ for all $z\geq 0$,
\beq
\begin{split}
&\;\int_0^1 e^{-2\mu M s} A_{\mu,M}(s)\, ds\\
< &\;
e^{-2\mu M}\left[\f{\pi^2}{2M}
+ \f{2}{M}
\left(\f{1-e^{-(1+2\mu)M}}{1+2\mu} + \f{1-e^{-(1-2\mu)M}}{1-2\mu} -2 \right)
\right.\\
&\; \qquad \quad +\f{4\mu M}{M^2} \sum_{k = 2}^\infty  \f{1}{k^2} \left(\f{(-1)^{k}}{k+2\mu} +\f{1}{k-2\mu}\right)
\\
&\; \qquad \quad \left. - 2M \cdot \f{1}{2\mu M}  \left(1-\f{1-e^{-2\mu M}}{2\mu M}\right)
+2\ln 2 \cdot \f{1-e^{-2\mu M}}{2\mu M}\cdot e^{2\mu M}
\right]
\\
=:&\; \r_1(M;\mu).
\end{split}
\label{eqn: def of rho_1(M,mu)}
\eeq

For $\mu\in [0,\f12]$, there exists some universal constant $B_0>0$, such that
\[
0<\sum_{k = 2}^\infty  \f{1}{k^2} \left(\f{(-1)^{k}}{k+2\mu} +\f{1}{k-2\mu}\right) \leq B_0.
\]
Since $\f{1-e^{-z}}{z}\in [0,1]$ for all $z\geq 0$,
\begin{align*}
\r_1(M;\mu)
\leq &\;
e^{-2\mu M}\left[\f{\pi^2}{2M} + \f{2}{M}(1+M-2)+  \f{4\mu M}{M^2} B_0\right. \\
&\;\qquad  \left.  - 2M \cdot \f{1}{2\mu M}  \left(1-\f{1-e^{-2\mu M}}{2\mu M}\right) + 2\ln 2 \cdot e^{2\mu M}\right]\\
=: &\; e^{-2\mu M} \zeta_1(M;2\mu M),
\end{align*}
where
\[
\zeta_1(M;c):=
\f1M \left(\f{\pi^2}{2}-2\right)+ 2 + \f{2c}{M^2} B_0 - 2M \cdot \f{1}{c}  \left(1-\f{1-e^{-c}}{c}\right)
+ 2\ln 2\cdot e^c.
\]
Clearly, $\zeta_1$ is strictly decreasing in $M>0$ and strictly increasing in $c\geq 0$.
Besides, for any $c\geq 0$,
\[
\lim_{M\to 0^+} \zeta_1(M;c) = +\infty,\quad
\lim_{M\to +\infty} \zeta_1(M;c) = -\infty,
\]
so there exists a unique $M_1 = M_1(c) > 0$, such that $\zeta_1(M_1(c);c)= 0$.
Moreover, $c\mapsto M_1(c)$ is strictly increasing.
This implies that, for any fixed $c\geq 0$, as long as $2\mu M \in [0, c]$ and $M\geq M_1(c)$, it holds that
\[
\zeta_1(M;2\mu M)\leq \zeta_1\big(M_1(c); c\big)= 0,
\]
and as a result,
\[
\int_0^1 e^{-2\mu M s} A_{\mu,M}(s)\, ds
< \r_1(M;\mu) \leq e^{-2\mu M} \zeta_1(M;2\mu M) \leq 0.
\]

Denote $M_*(c) := \max\{M_0(c),\,M_1(c)\}$, which is a strictly increasing positive function.
We conclude that as long as
\[
(M,2\mu M) \in \Big\{(\tilde{M},\tilde{c}):\;\tilde{c}\geq 0,\; \tilde{M}\geq M_*(\tilde{c})\Big\},
\]
we will have
\beq
A_{\mu,M}(0) = \z_0 (M;2\mu M) \leq 0,\quad
\int_0^1 e^{-2\mu M s} A_{\mu,M}(s)\, ds < \r_1(M;\mu)\leq 0.
\label{eqn: sign conditions related to A(s)}
\eeq
By virtue of \eqref{eqn: bounding the integral from 0 to s by the one on 0 to 1}, this implies \eqref{eqn: further sufficient condition for the sign of F at the peak general m}.
Then \eqref{eqn: sufficient condition for the sign of F at the peak general m} and thus \eqref{eqn: F at the peak general m} follow immediately.

This completes the proof.
\end{proof}
\end{prop}

With some extra efforts, we can even find an explicit pair of constants $M$ and $\mu_0$ that makes Proposition \ref{prop: a priori upper bound general m} hold.
\begin{cor}
\label{cor: M=4 and mu leq 1/12}
Proposition \ref{prop: a priori upper bound general m} holds with $M=4$ and $\mu_0 = \f1{12}$.
\end{cor}

Thanks to the arguments in the proof of Proposition \ref{prop: a priori upper bound general m}, in order to show Corollary \ref{cor: M=4 and mu leq 1/12}, it suffices to verify \eqref{eqn: sign conditions related to A(s)} for any $\mu\in[0,\frac{1}{12}]$ with $M = 4$, that is, for any $\mu\in [0,\f1{12}]$, $\z_0(4; 8\mu)\leq 0$ and $\r_1(4;\mu)\leq 0$, where $\zeta_0$ and $\r_1$ were defined in \eqref{eqn: def of zeta_0} and \eqref{eqn: def of rho_1(M,mu)} respectively.
However, the proof, which is fully analytic and rigorous, involves technical handling of various formulas and lengthy calculation, which can be very distracting.
Therefore, we will leave it to Appendix \ref{sec: proof of M = 4 mu = 1/12}.

\begin{rmk}
\label{rmk: (M,mu) pair with larger mu}
It is noteworthy that $(M,\mu_0) = (4,\f{1}{12})$ here is only a convenient pair of constants that makes Proposition \ref{prop: a priori upper bound general m} hold.
If one wants to prove the existence of V-states with 90-degree corners for a larger range of $m$ in the current framework, it is natural to ask whether there are other pairs of $(M,\mu)$ with larger values of $\mu$, e.g., $\mu = \f{1}{11}, \f{1}{10}, \cdots$, etc., such that (cf.\;\eqref{eqn: sign conditions related to A(s)})
\beq
\z_0 (M;2\mu M) \leq 0,\quad \int_0^1 e^{-2\mu M s} A_{\mu,M}(s)\, ds \leq 0,
\label{eqn: conditions for a valid pair of M and mu}
\eeq
where $\zeta_0$ and $A_{\mu,M}$ are defined in \eqref{eqn: def of zeta_0} and \eqref{eqn: def of A(s)}, respectively.
Indeed, if $M>0$ satisfies \eqref{eqn: conditions for a valid pair of M and mu}, then it is an a priori upper bound preserved by $\bfR$ with this $\mu$ (but not necessarily for any smaller $\mu'$), i.e., as long as $g\in \BM$ and $g(0)\leq M$, we would have $\mathbf{R}(g)(0) < M$.
We investigate this issue by numerically plotting the functions
\[
M\mapsto \z_0 (M;2\mu M),\quad
M\mapsto \int_0^1 e^{-2\mu M s} A_{\mu,M}(s)\, ds
\]
with different values of $\mu$, and computing the set \beq
\CM(\mu):= \left\{M>0:\, \z_0 (M;2\mu M) \leq 0,\,\int_0^1 e^{-2\mu M s} A_{\mu,M}(s)\, ds \leq 0\right\},
\label{eqn: def of CM mu admissible range of M}
\eeq
i.e., the admissible range of $M$ so that it can serve as an a priori upper bound preserved by $\bfR$ with the given $\mu$.
The numerical results are listed in Table \ref{tab: smallest M for various mu}.
We will come back to them in Remark \ref{rmk: optimality of m=12 explained} after proving Theorem \ref{thm: main existence theorem} at the end of Section \ref{sec: regularity}.
At this moment, we only note that $\CM(\mu) = \varnothing$ when $\mu = \f{1}{2},\f13,\cdots,\f18$, which suggests that the current analysis is not enough for handling the cases of $m = 2,3,\cdots, 8$.

\begin{table}
\centering
\begin{tabular}{cc}
\toprule[1pt]
$\mu$ & $\CM(\mu)$
\\[3.5pt]
\toprule[0.5pt]
$\f{1}{12}$ & $\approx [3.9526, 16.8312]$ \\[2.5pt]
$\f{1}{11}$ & $\approx [4.1461, 14.4242]$ \\[2.5pt]
$\f{1}{10}$ & $\approx [4.4467, 11.9733]$ \\[2.5pt]
$\f{1}{9}$ & $\approx [5.0364, 9.3048]$ \\[2.5pt]
$\mu \in \left\{\f{1}{8},\f{1}{7},\cdots, \f{1}{2}\right\}$ & $\varnothing$
\\[2.5pt]
\bottomrule[1pt]
\end{tabular}
\caption{$\CM(\mu)$ (defined in \eqref{eqn: def of CM mu admissible range of M}) that is numerically found for various $\mu$'s of interest.
}
\label{tab: smallest M for various mu}
\end{table}
\end{rmk}

\subsection{Estimates with explicit constants}
\label{sec: estimates with explicit constants}
In this part, we would like to prove estimates for $F$ (as well as $F_1$ and $F_2$) which contain explicit constants.
This will play a crucial role later in making the smallness condition on $\mu$ explicit (e.g.\;see Proposition \ref{prop: upper barrier new} below).

\begin{lem}
\label{lem: upper bound for -F_1 x_1}
Suppose that $g\in \BM$.
Let $F_1$ be defined as in \eqref{eqn: representation of Psi_1} and \eqref{eqn: def of F}.
For any $x_1\in [0,\pi]$ and $x_2>0$,
\[
-\pa_{x_1}F_1(x_1,x_2)
\leq
\left(\sqrt{2} + \f{2+2\ln 3}{\pi}\right) \f{\sin x_1}{x_2}
+ \f{2\sin x_1}{\pi x_2} \ln \left(1+\f{x_2}{4\sin x_1}\right).
\]

\begin{proof}
Recall in the proof of Lemma \ref{lem: monotonicity of F_1}, we have shown that \begin{align*}
&\;-\pa_{x_1}\Psi_1(x_1,x_2) \\
= &\;
\f1{4\pi} \int_0^{g(0)} e^{-2\mu y_2}\\
&\;\qquad \cdot
\ln\left(\f {\cosh(x_2-y_2)-\cos(x_1+g^{-1}(y_2))}{\cosh(x_2+y_2)-\cos(x_1+g^{-1}(y_2))} \cdot \f{\cosh(x_2+y_2)-\cos(x_1-g^{-1}(y_2))}{\cosh(x_2-y_2)-\cos(x_1-g^{-1}(y_2))} \right) dy_2.
\end{align*}
Note that
\begin{align*}
&\; \big[\cosh(x_2-y_2)-\cos(x_1+g^{-1}(y_2))\big] \big[\cosh(x_2+y_2)-\cos(x_1-g^{-1}(y_2))\big]
\\
&\; -\big[\cosh(x_2+y_2)-\cos(x_1+g^{-1}(y_2))\big]
\big[\cosh(x_2-y_2)-\cos(x_1-g^{-1}(y_2))\big]
\\
= &\; -\big[\cosh(x_2+y_2)-\cosh(x_2-y_2)\big]
\big[ \cos(x_1+g^{-1}(y_2))-\cos(x_1-g^{-1}(y_2))\big]
\\
= &\; 4\sinh x_2 \sinh y_2 \sin x_1 \sin g^{-1}(y_2)
\geq 0,
\end{align*}
and
\[
\cosh x_2 - \cos x_1 = 2\sinh^2 \f{x_2}{2} + 2\sin^2 \f{x_1}{2}.
\]
So by the definition of $F_1$ in \eqref{eqn: def of F},
\begin{align*}
&\; -\pa_{x_1}F_1(x_1,x_2)\\
= &\; \f{1}{4\pi x_2}\int_0^{g(0)} e^{-2\mu y_2} \cdot \ln\left(1
+\f {\sinh x_2 \sinh y_2 \sin x_1 \sin g^{-1}(y_2)}{(\sinh^2 \f{x_2+y_2}{2}
+ \sin^2\f{x_1+g^{-1}(y_2)}{2})
(\sinh^2 \f{x_2-y_2}{2} + \sin^2\f{x_1-g^{-1}(y_2)}{2})} \right) dy_2
\\
= &\; \f{1}{4\pi x_2}\int_0^{g(0)} e^{-2\mu y_2}\cdot \mathds{1}_{\{\sin g^{-1}(y_2)\leq 2\sin x_1\}} \\
&\;\qquad\qquad \cdot \ln\left(1
+\f {\sinh x_2 \sinh y_2 \sin x_1 \sin g^{-1}(y_2)}{(\sinh^2 \f{x_2+y_2}{2}
+ \sin^2\f{x_1+g^{-1}(y_2)}{2})
(\sinh^2 \f{x_2-y_2}{2} + \sin^2\f{x_1-g^{-1}(y_2)}{2})} \right) dy_2
\\
&\; + \f{1}{4\pi x_2}\int_0^{g(0)} e^{-2\mu y_2}\cdot \mathds{1}_{\{\sin g^{-1}(y_2)> 2\sin x_1\}} \\
&\;\qquad\qquad \cdot \ln\left(1
+\f {\sinh x_2 \sinh y_2 \sin x_1 \sin g^{-1}(y_2)}{(\sinh^2 \f{x_2+y_2}{2}
+ \sin^2\f{x_1+g^{-1}(y_2)}{2})
(\sinh^2 \f{x_2-y_2}{2} + \sin^2\f{x_1-g^{-1}(y_2)}{2})} \right) dy_2
\\
=:&\; I_1+I_2.
\end{align*}

Since for any $x_2,y_2>0$,
\[
\sinh^2\f{x_2-y_2}{2} \geq \f14 |x_2-y_2|^2,
\]
and
\[
\sinh x_2 \sinh y_2 = \sinh^2 \f{x_2+y_2}{2} - \sinh^2 \f{x_2-y_2}{2},
\]
we can bound $I_1$ as follows
\begin{align*}
I_1
\leq &\;\f{1}{4\pi x_2}\int_0^{g(0)} \ln\left(1
+\f {2\sinh x_2 \sinh y_2 \sin^2 x_1}{\sinh^2 \f{x_2+y_2}{2}
\cdot \sinh^2 \f{x_2-y_2}{2}} \right) dy_2
\\
\leq &\;\f{1}{4\pi x_2}\int_0^{g(0)} \ln\left(1
+\f {2\sin^2 x_1}{\f14 |x_2-y_2|^2} \right) dy_2
\leq \f{1}{4\pi x_2}\int_{\BR} \ln\left(1
+\f {8\sin^2 x_1}{|z|^2} \right) dz
\\
= &\; \f{\sqrt{8}\sin x_1}{4\pi x_2}\int_{\BR}\ln \left(1 + \f{1}{s^2}\right)ds
= \f{\sqrt{2} \sin x_1}{x_2}.
\end{align*}

To bound $I_2$, we first observe that, under the assumption $\sin g^{-1}(y_2)> 2\sin x_1\geq 0$,
\[
2\sin \f{g^{-1}(y_2)-x_1}{2}\cos \f{g^{-1}(y_2)+x_1}{2} = \sin g^{-1}(y_2) - \sin x_1 \geq \f12 \sin g^{-1}(y_2),
\]
which implies
\[
\sin^2 \f{g^{-1}(y_2)-x_1}{2}\geq \f1{16} \sin^2 g^{-1}(y_2),
\]
Hence, \begin{align*}
I_2 \leq
&\;
\f{1}{4\pi x_2}\int_0^{g(0)}
\ln\left(1
+\f {\sinh x_2 \sinh y_2 \sin x_1 \sin g^{-1}(y_2)}
{\sinh^2 \f{x_2+y_2}{2} \cdot (\sinh^2\f{x_2-y_2}{2} +\f{1}{16}\sin^2 g^{-1}(y_2))} \right) dy_2
\\
\leq
&\;
\f{1}{4\pi x_2}\int_0^{g(0)}
\ln\left(1
+\f {\sinh x_2 \sinh y_2 \sin x_1 \sin g^{-1}(y_2)}{ \sinh^2 \f{x_2+y_2}{2} \cdot \f12
\sinh\f{|x_2-y_2|}{2}\sin g^{-1}(y_2)} \right) dy_2
\\
\leq
&\;
\f{1}{4\pi x_2}
\int_0^{2x_2}
\ln\left(1
+\f {2\sin x_1}{\sinh\f{|x_2-y_2|}{2}} \right) dy_2
\\
&\; +
\f{1}{4\pi x_2} \int_{2x_2}^\infty
\ln\left(1
+\f {2\sinh x_2 \sinh y_2 \sin x_1}{\sinh^2 \f{x_2+y_2}{2}
\cdot \sinh\f{y_2-x_2}{2}} \right) dy_2\\
=: &\; I_{2,1} + I_{2,2}.
\end{align*}
We derive that
\begin{align*}
I_{2,1} \leq &\;
\f{1}{4\pi x_2}\int_0^{2x_2}
\ln\left(1
+\f {4\sin x_1}{|x_2-y_2|} \right) dy_2
=
\f{1}{2\pi x_2}\int_0^{x_2}
\ln\left(1
+\f {4\sin x_1}{z} \right) dz
\\
= &\;
\f{2\sin x_1}{\pi x_2} \int_0^{x_2/(4\sin x_1)}
\ln\left(1+\f {1}{t} \right) dt
\\
= &\; \f{2\sin x_1}{\pi x_2}\left[\left(1+\f{x_2}{4\sin x_1}\right)\ln \left(1+\f{x_2}{4\sin x_1}\right) - \f{x_2}{4\sin x_1}\ln \left(\f{x_2}{4\sin x_1}\right) \right]
\\
\leq &\; \f{2\sin x_1}{\pi x_2}\left[1
+ \ln \left(1+\f{x_2}{4\sin x_1}\right) \right].
\end{align*}
On the other hand, since
\[
\sinh \f{x_2+y_2}{2} \sinh\f{y_2-x_2}{2} = \sinh^2\f{y_2}{2}-\sinh^2\f{x_2}{2},
\]
and \[
\sinh \f{x_2+y_2}{2} = \cosh\f{x_2}{2}\cosh\f{y_2}{2}\left(\tanh\f{x_2}{2}+\tanh\f{y_2}{2}\right)
\geq \cosh\f{x_2}{2}\sinh\f{y_2}{2},
\]
we find that
\begin{align*}
I_{2,2}
\leq &\;
\f{1}{2\pi x_2}\int_{2x_2}^\infty
\f {\sinh x_2 \sinh y_2 \sin x_1}{\sinh \f{x_2+y_2}{2}\cdot (\sinh^2\f{y_2}{2}-\sinh^2\f{x_2}{2})}\, dy_2
\\
\leq &\;
\f{\sin x_1}{\pi x_2}\int_{2x_2}^\infty
\f{4\sinh \f{x_2}{2}\cosh\f{x_2}{2} \cdot \sinh \f{y_2}{2}}
{\cosh \f{x_2}{2}\sinh\f{y_2}{2}\cdot
(\sinh^2\f{y_2}{2}-\sinh^2\f{x_2}{2})}\cdot \f12 \cosh\f{y_2}{2}\, d y_2
\\
= &\;
\f{2\sin x_1}{\pi x_2}\int_{\sinh x_2}^\infty
\f{2\sinh \f{x_2}{2}}{t^2-\sinh^2\f{x_2}{2}} \, dt
\\
= &\;
\f{2\sin x_1}{\pi x_2}\ln \left(\f{\sinh x_2 +\sinh\f{x_2}{2}}{\sinh x_2-\sinh\f{x_2}{2}}\right)
=
\f{2\sin x_1}{\pi x_2}\ln \left(\f{2\cosh \f{x_2}{2} +1}{2\cosh \f{x_2}{2}-1}\right)
\\
\leq &\; \f{2\ln 3}{\pi}\cdot \f{\sin x_1}{x_2}.
\end{align*}
Summarizing all the estimates above proves the desired bound.
\end{proof}
\end{lem}

In the next lemma, we shall bound $\pa_{x_1}F_2$ in terms of $\pa_{x_2}F_2$.
\begin{lem}
\label{lem: bounding F_2 x_1 in terms of F_2 x_2}
For any $x_1\in [\f{\pi}{2},\pi]$ and $x_2 > 0$,
\[
-\pa_{x_1} F_2(x_1,x_2)
\leq \f{x_2\sinh x_2 \sin x_1}
{x_2\cosh x_2-\sinh x_2} \cdot\left(-
\pa_{x_2} F_2(x_1,x_2)\right).
\]

\begin{proof}
Throughout the proof, we will always assume that $x_1\in [\f{\pi}{2},\pi]$ and $x_2 > 0$.
In particular, $\cos x_1 \leq 0$.

By \eqref{eqn: representation of Psi_2} and \eqref{eqn: def of F}, for $i = 1,2$,
\beq
\begin{split}
\pa_{x_i} F_2(x_1,x_2)
= &\; \int_\BT \pa_{x_i} \left[\f{1}{x_2}P(x_1-y_1,x_2)\right] h(y_1)\,dy_1
\\
= &\; \int_0^\pi \left[\pa_{x_i}\left(\f{P}{x_2}\right)(x_1-y_1,x_2)
+ \pa_{x_i}\left(\f{P}{x_2}\right)(x_1+y_1,x_2)\right] h(y_1)\,dy_1.
\end{split}
\label{eqn: formula for F_2_(x_2) new}
\eeq
We are thus motivated to study
\beq
Q_i(x_1,x_2,y_1)
:=
\pa_{x_i}\left(\f{P}{x_2}\right)(x_1-y_1,x_2)
+ \pa_{x_i}\left(\f{P}{x_2}\right)(x_1+y_1,x_2).
\label{eqn: def of Q_i}
\eeq
We will show in Lemma \ref{lem: simplification of Q_i} in Appendix \ref{sec: calculus lemmas} that
\beq
\begin{split}
Q_1(x_1,x_2,y_1) = &\; -\f{x_2\sinh x_2 \sin x_1 \cdot q_1(\cos y_1)}
{\pi x_2^2(\cosh x_2- \cos (x_1-y_1))^2(\cosh x_2- \cos (x_1+y_1))^2},
\\
Q_2(x_1,x_2,y_1)
= &\; -\f{(x_2\cosh x_2-\sinh x_2)\cdot q_2(\cos y_1)}{\pi x_2^2(\cosh x_2- \cos (x_1-y_1))^2(\cosh x_2- \cos (x_1+y_1))^2},
\end{split}
\label{eqn: formulas for Q_i}
\eeq
where we defined for $s\in [-1,1]$ that
\begin{align*}
q_1(s):= &\; s (\cosh^2 x_2 + 1 + \cos^2 x_1 - s^2)-2\cosh x_2 \cos x_1,
\\
q_2(s) := &\; (k(x_2)+s\cos x_1) (\cosh x_2 - s\cos x_1)^2
\\
&\;
+ (k(x_2) + 2\cosh x_2 - s\cos x_1) \sin^2 x_1 (1-s^2),
\end{align*}
and for $x_2 > 0$,
\[
k(x_2):= \f{\sinh x_2 \cosh x_2-x_2}{x_2\cosh x_2-\sinh x_2}.
\]
Note that both $q_1(s)$ and $q_2(s)$ are cubic polynomials of $s$.
In the proof below, we will need the fact that
\beq
k(x_2)\in (2,2\cosh x_2)\mbox{ for any }x_2 > 0.
\label{eqn: estimate for k}
\eeq
We leave its proof to Lemma \ref{lem: bound for k(x_2)}.

Define
\[
w(s) := q_1(s)(-1+\cos x_1+\cos^2 x_1) + q_2(s).
\]
By direct calculation and applying \eqref{eqn: estimate for k}, we find that for $x_2 > 0$ and $s\in [-1,1]$,
\begin{align*}
w''(s)
= &\;q_1''(s)\cos x_1 + q_2''(s) - (1-\cos^2 x_1) q_1''(s)\\
= &\; - 2\big[-k(x_2)(\cos^2 x_1-\sin^2 x_1)+2\cosh x_2\big] + 6s \sin^2 x_1
\\
\leq &\; 2\big[k(x_2)-2\cosh x_2\big] - 4k(x_2)\sin^2 x_1 + 6\sin^2 x_1
\\
< &\; -8\sin^2 x_1 + 6\sin^2 x_1 \leq 0.
\end{align*}
Hence, the function $s\mapsto w(s)$ is strictly concave.
This implies that
\[
\inf_{s\in [-1,1]}w(s) =\min\big\{w(1),\, w(-1)\big\}.
\]
By \eqref{eqn: estimate for k},
\begin{align*}
w(1)
= &\; (-1+\cos x_1+\cos^2 x_1)\left[ (\cosh^2 x_2 + \cos^2 x_1)-2\cosh x_2 \cos x_1\right]\\
&\;
+ (k(x_2)+\cos x_1) (\cosh x_2 - \cos x_1)^2
\\
= &\; (k(x_2)-1+2\cos x_1+\cos^2 x_1) (\cosh x_2 - \cos x_1)^2
\\
\geq &\; (1+2\cos x_1+ \cos^2 x_1) (\cosh x_2-\cos x_1)^2 \geq 0,
\end{align*}
and
\begin{align*}
w(-1)
= &\; (-1+\cos x_1+\cos^2 x_1)\left[-(\cosh^2 x_2 + \cos^2 x_1)-2\cosh x_2 \cos x_1\right]\\
&\;
-(-k(x_2)+\cos x_1) (\cosh x_2 + \cos x_1)^2
\\
= &\; (k(x_2)+1-2\cos x_1-\cos^2 x_1) (\cosh x_2 + \cos x_1)^2
\\
\geq &\; (3-2\cos x_1-\cos^2 x_1) (\cosh x_2 + \cos x_1)^2 \geq 0.
\end{align*}
Therefore, for any $s\in [-1,1]$,
\[
w(s)\geq 0.
\]
Using \eqref{eqn: estimate for k} once again, we see that for any $x_2 >0$ and $s\in [-1,1]$,
\[
q_2(s) \geq (2-|s|) (\cosh x_2 - |s|)^2 > 0,
\]
so the preceding inequality is equivalent to
\[
\f{q_1(s)}{q_2(s)}\leq \f{1}{1-\cos x_1-\cos^2 x_1}\leq 1.
\]
Note that $\cos x_1 \leq 0$ by assumption.
That $q_2>0$ also implies $Q_2(x_1,x_2,y_1)<0$ for any $x_1 \in [\f{\pi}{2},\pi]$ and $x_2>0$.
Combining these facts with the formulas for $Q_i$ (see \eqref{eqn: formulas for Q_i}) yields that
\begin{align*}
-Q_1(x_1,x_2,y_1)
= &\; \f{x_2\sinh x_2 \sin x_1}
{x_2\cosh x_2-\sinh x_2}
\cdot \f{q_1(\cos y_1)}{q_2(\cos y_1)}\cdot \big(-Q_2(x_1,x_2,y_1)\big)
\\
\leq &\; -\f{x_2\sinh x_2 \sin x_1}
{x_2\cosh x_2-\sinh x_2} \cdot Q_2(x_1,x_2,y_1).
\end{align*}

Now by \eqref{eqn: formula for F_2_(x_2) new}, \eqref{eqn: def of Q_i}, and the fact that $h\geq 0$ (see Lemma \ref{lem: monotonicity of h}),
\begin{align*}
-\pa_{x_1} F_2(x_1,x_2)
= &\; -\int_0^\pi Q_1(x_1,x_2,y_1) h(y_1)\,dy_1
\\
\leq &\; - \f{x_2\sinh x_2 \sin x_1}
{x_2\cosh x_2-\sinh x_2} \int_0^\pi
Q_2(x_1,x_2,y_1) h(y_1)\,dy_1\\
= &\; -\f{x_2\sinh x_2 \sin x_1}
{x_2\cosh x_2-\sinh x_2} \cdot \pa_{x_2} F_2(x_1,x_2).
\end{align*}
This completes the proof.
\end{proof}
\end{lem}

Next, we derive upper and lower bounds for $F(\pi,x_2)-\pa_{x_2}\Psi(\pi,0)$.
\begin{lem}
\label{lem: control vertical displacement general mu upper and lower bounds}
Let $\mu \geq 0$.
Then for any $x_2 \in (0,g(0)]$,
\begin{align*}
F(\pi,x_2)-\pa_{x_2}\Psi(\pi,0)
\leq &\;
-\f{1}{\pi}\int_0^{x_2} e^{-2\mu y_2} \cdot \arctan\left(\tanh\f{y_2}{2}\tan\f{g^{-1}(x_2)}{2}\right) dy_2,
\end{align*}
and \begin{align*}
F(\pi,x_2)-\pa_{x_2}\Psi(\pi,0)
\geq &\; -\f{7}{8}x_2
-\f{x_2}{4\pi} \int_{\min\{2x_2,g(0)\}}^{g(0)} \min\left\{ \tan\f{g^{-1}(y_2)}{2},\, \f{1}{\sinh \f{y_2}{2}}\right\} dy_2.
\end{align*}

\begin{proof}
By Lemma \ref{lem: integral formula for Psi} and the definition of $F$ (see \eqref{eqn: def of F}),
\begin{align*}
&\; F(\pi,x_2)\\
= &\; -\f{1}{4\pi x_2}\int_0^{g(0)} e^{-2\mu y_2} \int_{-g^{-1}(y_2)}^{g^{-1}(y_2)} \ln \left(\f{\cosh(x_2-y_2)+ \cos y_1}{\cosh y_2 +\cos y_1}\right)dy_1\,dy_2
+ \f{1}{4\pi}\int_\Om e^{-2\mu y_2} \, dy
\\
= &\; -\f{1}{4\pi x_2}\int_0^{g(0)} e^{-2\mu y_2} \int_{-g^{-1}(y_2)}^{g^{-1}(y_2)} \int_{|y_2|}^{|x_2-y_2|} \f{\sinh z}{\cosh z+ \cos y_1}\,dz\, dy_1\,dy_2
+ \f{1}{4\pi}\int_\Om e^{-2\mu y_2} \, dy
\\
= &\; -\f{1}{4\pi x_2}\int_0^{g(0)}e^{-2\mu y_2}  \int_{|y_2|}^{|x_2-y_2|} \int_{-g^{-1}(y_2)}^{g^{-1}(y_2)} \f{\sinh z}{\cosh z+ \cos y_1}\, dy_1\,dz\,dy_2
+ \f{1}{4\pi}\int_\Om e^{-2\mu y_2} \, dy.
\end{align*}
Since
\[
\f{\sinh z}{\cosh z+ \cos y_1} = \f{d}{dy_1}\left[2\arctan\left(\tanh\f{z}{2}\tan\f{y_1}{2}\right)\right],
\]
we obtain that
\begin{align*}
&\; F(\pi,x_2)\\
= &\; -\f{1}{4\pi x_2}\int_0^{g(0)} e^{-2\mu y_2} \int_{|y_2|}^{|x_2-y_2|} 4\arctan\left(\tanh\f{z}{2}\tan\f{g^{-1}(y_2)}{2}\right) dz\,dy_2
+ \f{1}{4\pi}\int_\Om e^{-2\mu y_2} \, dy.
\end{align*}
Analogously, thanks to Lemma \ref{lem: integral formula for Psi},
\begin{align*}
\pa_{x_2}\Psi(\pi,0)
= &\; \f{1}{4\pi}\int_0^{g(0)}e^{-2\mu y_2} \int_{-g^{-1}(y_2)}^{g^{-1}(y_2)} \f{\sinh y_2}{\cosh y_2+ \cos y_1} \, dy_1\,dy_2
+ \f{1}{4\pi}\int_\Om e^{-2\mu y_2} \, dy
\\
= &\; \f{1}{4\pi}\int_0^{g(0)} e^{-2\mu y_2} \cdot 4\arctan\left(\tanh\f{y_2}{2}\tan\f{g^{-1}(y_2)}{2}\right) dy_2
+ \f{1}{4\pi}\int_\Om e^{-2\mu y_2} \, dy.
\end{align*}

Using the above formulas, we split $F(\pi,x_2)-\pa_{x_2}\Psi(\pi,0)$ into two parts:
\begin{align*}
&\; F(\pi,x_2)- \pa_{x_2}\Psi(\pi,0)\\
= &\; -\f{1}{4\pi x_2}\int_0^{g(0)} e^{-2\mu y_2} \int_{|y_2|}^{|x_2-y_2|} 4\arctan\left(\tanh\f{z}{2}\tan\f{g^{-1}(y_2)}{2}\right) dz\,dy_2
\\
&\; -\f{1}{4\pi}\int_0^{g(0)} e^{-2\mu y_2} \cdot 4\arctan\left(\tanh\f{y_2}{2}\tan\f{g^{-1}(y_2)}{2}\right) dy_2
\\
=: &\; I_1 + I_2,
\end{align*}
where
\begin{align*}
I_1:= &\; -\f{1}{\pi x_2}\int_0^{x_2} e^{-2\mu y_2} \int_{y_2}^{x_2-y_2} \arctan\left(\tanh\f{z}{2}\tan\f{g^{-1}(y_2)}{2}\right) dz\,dy_2
\\
&\; -\f{1}{\pi}\int_0^{x_2} e^{-2\mu y_2} \cdot \arctan\left(\tanh\f{y_2}{2}\tan\f{g^{-1}(y_2)}{2}\right) dy_2,
\end{align*}
and
\begin{align*}
I_2:=
&\; \f{1}{\pi x_2}\int_{x_2}^{g(0)} e^{-2\mu y_2}   \int_{y_2-x_2}^{y_2} \arctan\left(\tanh\f{z}{2}\tan\f{g^{-1}(y_2)}{2}\right) dz\,dy_2
\\
&\; -\f{1}{\pi}\int_{x_2}^{g(0)} e^{-2\mu y_2} \cdot \arctan\left(\tanh\f{y_2}{2}\tan\f{g^{-1}(y_2)}{2}\right) dy_2.
\end{align*}

We first show the upper bound.
Obviously,
\begin{align*}
I_2 \leq
&\; \f{1}{\pi x_2}\int_{x_2}^{g(0)} e^{-2\mu y_2}   \int_{y_2-x_2}^{y_2} \arctan\left(\tanh\f{y_2}{2}\tan\f{g^{-1}(y_2)}{2}\right) dz\,dy_2
\\
&\; -\f{1}{\pi}\int_{x_2}^{g(0)} e^{-2\mu y_2} \cdot \arctan\left(\tanh\f{y_2}{2}\tan\f{g^{-1}(y_2)}{2}\right) dy_2
\\
= &\; 0.
\end{align*}
To bound $I_1$, we further rewrite it as follows:
\begin{align*}
I_1= &\; -\f{1}{\pi x_2}\int_0^{x_2/2} e^{-2\mu y_2} \int_{y_2}^{x_2-y_2} \arctan\left(\tanh\f{z}{2}\tan\f{g^{-1}(y_2)}{2}\right) dz\,dy_2
\\
&\; +\f{1}{\pi x_2}\int_0^{x_2/2} e^{-2\mu (x_2-y_2)} \int_{y_2}^{x_2-y_2} \arctan\left(\tanh\f{z}{2}\tan\f{g^{-1}(x_2-y_2)}{2}\right) dz\,dy_2
\\
&\; -\f{1}{\pi}\int_0^{x_2} e^{-2\mu y_2} \cdot \arctan\left(\tanh\f{y_2}{2}\tan\f{g^{-1}(y_2)}{2}\right) dy_2
\\
=:&\; I_{1,1} + I_{1,2} + I_{1,3}.
\end{align*}
Since $g^{-1}$ is decreasing, for any $y_2\in [0,\f{x_2}{2}]$,
\[
e^{-2\mu(x_2-y_2)}\leq e^{-2\mu y_2},\quad g^{-1}(x_2-y_2)\leq g^{-1}(y_2).
\]
This implies that $I_{1,1}+I_{1,2}\leq 0$.
Finally, using the facts that $g^{-1}$ is decreasing, and that $\tanh\f{y_2}{2}\leq 1$,
\begin{align*}
I_{1,3}
\leq &\;
-\f{1}{\pi}\int_0^{x_2} e^{-2\mu y_2} \cdot \arctan\left(\tanh\f{y_2}{2}\tan\f{g^{-1}(x_2)}{2}\right) dy_2.
\end{align*}
Combining all the above estimates yields the desired upper bound.

Next we show the lower bound.
A trivial bound for $I_1$ is that
\[
I_1
\geq -\f{1}{4\pi x_2}\int_0^{x_2/2} \int_{y_2}^{x_2-y_2} 2\pi \,dz\,dy_2
-\f{1}{4\pi}\int_0^{x_2} 2\pi\, dy_2
= -\f{5}{8} x_2.
\]

On the other hand, observe that, for $a\geq 0$ and $z\geq 0$,
\[
\f{d^2}{dz^2}\left[\arctan\left(a\tanh\f{z}{2}\right)\right]
= -\f{a (a^2 + 1) \sinh z}{((a^2 + 1) \cosh z - a^2 + 1)^2} \leq 0,
\]
so
\[
z\mapsto \arctan\left(a\tanh\f{z}{2}\right)\mbox{ is concave on $[0,+\infty)$}.
\]
As a result, \begin{align*}
I_2
=&\; \f{1}{\pi x_2}\int_{x_2}^{g(0)} e^{-2\mu y_2} \int_{y_2-x_2}^{y_2} \left[\arctan\left(\tanh\f{z}{2}\tan\f{g^{-1}(y_2)}{2}\right)\right.\\
&\; \qquad \qquad \qquad \qquad \quad \left.- \arctan\left(\tanh\f{y_2}{2}\tan\f{g^{-1}(y_2)}{2}\right)\right]
dz\,dy_2
\\
\geq &\; \f{1}{\pi}\int_{x_2}^{g(0)} e^{-2\mu y_2} \left[\f12 \arctan\left(\tanh\f{y_2}{2}\tan\f{g^{-1}(y_2)}{2}\right)\right.\\
&\;\qquad \qquad \qquad \quad
+ \f12\arctan\left(\tanh\f{y_2-x_2}{2}\tan\f{g^{-1}(y_2)}{2}\right)\\
&\; \qquad \qquad \qquad \quad \left.- \arctan\left(\tanh\f{y_2}{2}\tan\f{g^{-1}(y_2)}{2}\right)\right] dy_2
\\
=&\; \f{1}{2\pi}\int_{x_2}^{g(0)} e^{-2\mu y_2} \arctan\left(\f{(\tanh\f{y_2-x_2}{2} -\tanh\f{y_2}{2})\tan\f{g^{-1}(y_2)}{2}}{1+\tanh\f{y_2-x_2}{2} \tanh\f{y_2}{2}\tan^2\f{g^{-1}(y_2)}{2}}\right) dy_2.
\end{align*}
Since $y_2\geq x_2$, and
\[
\tanh\f{y_2-x_2}{2} -\tanh\f{y_2}{2}
= \f{-\sinh\f{x_2}{2}}{\cosh \f{y_2-x_2}{2}\cosh\f{y_2}{2}}
\geq
-\tanh\f{x_2}{2},
\]
we have that
\begin{align*}
\f{(\tanh\f{y_2-x_2}{2} -\tanh\f{y_2}{2})\tan\f{g^{-1}(y_2)}{2}}{1+\tanh\f{y_2-x_2}{2} \tanh\f{y_2}{2}\tan^2\f{g^{-1}(y_2)}{2}}
\geq &\;
\left(\tanh\f{y_2-x_2}{2} -\tanh\f{y_2}{2}\right)\tan\f{g^{-1}(y_2)}{2}
\\
\geq &\;
-\tanh\f{x_2}{2} \tan\f{g^{-1}(y_2)}{2}.
\end{align*}
If additionally $y_2 \geq 2x_2$, we have $y_2-x_2\geq \f{y_2}{2}\geq x_2$, so by Young's inequality,
\begin{align*}
&\;\f{(\tanh\f{y_2-x_2}{2} -\tanh\f{y_2}{2})\tan\f{g^{-1}(y_2)}{2}}{1+\tanh\f{y_2-x_2}{2} \tanh\f{y_2}{2}\tan^2\f{g^{-1}(y_2)}{2}}
\\
\geq
&\;\left(\tanh\f{y_2-x_2}{2} -\tanh\f{y_2}{2}\right)\cdot \f{\tan\f{g^{-1}(y_2)}{2}}{2\sqrt{\tanh\f{y_2-x_2}{2} \tanh\f{y_2}{2}} \cdot \tan\f{g^{-1}(y_2)}{2}}
\\
= &\;-\f{\sinh\f{x_2}{2}}{\cosh \f{y_2-x_2}{2}\cosh\f{y_2}{2}}\cdot
\f{1}{2\sqrt{\tanh\f{y_2-x_2}{2} \tanh\f{y_2}{2}}}
\\
= &\;-\f{\sinh\f{x_2}{2}}{\sqrt{2\cosh \f{y_2}{2}\cdot \sinh (y_2-x_2)   \sinh \f{y_2}{2} }}
\\
\geq &\;-\f{\sinh\f{x_2}{2}}{\sqrt{(\cosh^2 \f{y_2}{4} +1)\cdot \sinh^2 \f{y_2}{2} }}
\\
\geq &\; -\f{\sinh\f{x_2}{2}}{\cosh \f{y_2}{4} \cdot \sinh \f{y_2}{2}}
\geq -\f{\tanh\f{x_2}{2}}{\sinh \f{y_2}{2}}.
\end{align*}
With this, we can further derive that
\begin{align*}
I_2
\geq &\; -\f{1}{2\pi}\int_{x_2}^{\min\{2x_2,g(0)\}}
\arctan\left(\tanh\f{x_2}{2} \tan\f{g^{-1}(y_2)}{2}\right)
dy_2
\\
&\; -\f{1}{2\pi}\int_{\min\{2x_2,g(0)\}}^{g(0)} \arctan\left(\min\left\{\tanh\f{x_2}{2} \tan\f{g^{-1}(y_2)}{2},\,
\f{\tanh\f{x_2}{2}}{\sinh \f{y_2}{2}}
\right\}\right) dy_2
\\
\geq &\; -\f{1}{2\pi}\int_{x_2}^{2x_2}
\f{\pi}{2}\, dy_2
-\f{1}{2\pi} \tanh\f{x_2}{2}  \int_{\min\{2x_2,g(0)\}}^{g(0)} \min\left\{\tan\f{g^{-1}(y_2)}{2},\,
\f{1}{\sinh \f{y_2}{2}}
\right\} dy_2
\\
\geq &\; -\f{1}{4}x_2
-\f{x_2}{4\pi} \int_{\min\{2x_2,g(0)\}}^{g(0)} \min\left\{ \tan\f{g^{-1}(y_2)}{2},\, \f{1}{\sinh \f{y_2}{2}}\right\} dy_2.
\end{align*}
Combining the estimates above concludes the lower bound.
\end{proof}
\end{lem}

\subsection{More estimates for $\Psi$}
The following two estimates for $\Psi$ will also be useful later.

\begin{lem}
\label{lem: control horizontal displacement of Psi}
For any $x_1\in [0,\pi]$ and $x_2\geq 0$,
\begin{align*}
&\;\Psi(x_1,x_2)-\Psi(\pi,x_2)\\
\geq &\; \f{1}{2\pi} \int_0^{g(0)} e^{-2\mu y_2}\sin g^{-1}(y_2) \cdot \ln\left(1+ \f{1+\cos x_1}
{\cosh(x_2-y_2) -1 + 2 \cos^2\f{ g^{-1}(y_2)}{2}}\right) dy_2.
\end{align*}
In particular, for any $x_1\in[0,\pi]$,
\[
h(x_1) \geq \f{1}{2\pi} \int_0^{g(0)} e^{-2\mu y_2}\sin g^{-1}(y_2) \cdot \ln\left(1+ \f{1+\cos x_1}{\cosh y_2 -1 + 2 \cos^2\f{ g^{-1}(y_2)}{2}}\right) dy_2.
\]

\begin{proof}
By Lemma \ref{lem: integral formula for Psi},
\beq
\begin{split}
-\pa_{x_1}\Psi(x_1,x_2)
= &\; \f{1}{4\pi}\int_0^{g(0)} e^{-2\mu y_2}\int_{-g^{-1}(y_2)}^{g^{-1}(y_2)}
\pa_{x_1}\big[\ln (\cosh(x_2-y_2)- \cos(x_1-y_1) )\big] \, dy_1\, dy_2
\\
= &\; \f{1}{4\pi}\int_0^{g(0)} e^{-2\mu y_2}\int_{-g^{-1}(y_2)}^{g^{-1}(y_2)}
-\pa_{y_1}\big[\ln (\cosh(x_2-y_2)- \cos(x_1-y_1))\big] \, dy_1\, dy_2
\\
= &\; \f{1}{4\pi}\int_0^{g(0)} e^{-2\mu y_2} \cdot \ln \left(\f{\cosh(x_2-y_2)-\cos(x_1 + g^{-1}(y_2))}{\cosh (x_2-y_2)-\cos(x_1-g^{-1}(y_2))}\right)dy_2.
\end{split}
\label{eqn: formula for Psi_x_1}
\eeq
Observe that for any $0<a<b$,
\[
\ln \left(\f{b}{a}\right) = \int_a^b \f{1}{s} \,ds
\geq (b-a)\left(\f{a+b}{2}\right)^{-1} = \f{2(b-a)}{a+b},
\]
and thus
\begin{align*}
& \;-\pa_{x_1}\Psi(x_1,x_2)\\
\geq &\; \f{1}{4\pi}\int_0^{g(0)} e^{-2\mu y_2} \cdot \f{2(\cos(x_1-g^{-1}(y_2))-\cos(x_1+g^{-1}(y_2)))}{2\cosh (x_2-y_2) - 2\cos x_1 \cos g^{-1}(y_2)}\, dy_2
\\
= &\;\f{1}{2\pi}\int_0^{g(0)} e^{-2\mu y_2} \cdot  \f{\sin x_1 \sin g^{-1}(y_2)}{\cosh(x_2-y_2) + \cos x_1  -\cos x_1 (1+\cos g^{-1}(y_2))}\, dy_2
\\
\geq&\; \f{1}{2\pi}\int_0^{g(0)}  e^{-2\mu y_2} \cdot \f{\sin x_1 \sin g^{-1}(y_2)}{\cosh(x_2-y_2) + \cos x_1 + 2 \cos^2\f{ g^{-1}(y_2)}{2}}\, dy_2
\\
= &\;\f{1}{2\pi}\int_0^{g(0)}  e^{-2\mu y_2} \sin g^{-1}(y_2) \cdot \f{d}{dx_1}\left[-\ln \left(\cosh(x_2-y_2) + \cos x_1 + 2 \cos^2\f{ g^{-1}(y_2)}{2}\right)\right] dy_2.
\end{align*}
This implies that
\begin{align*}
& \;\Psi(x_1,x_2)-\Psi(\pi,x_2)\\
= &\;\int_{x_1}^\pi -\pa_{x_1}\Psi(x,x_2)\,dx
\\
\geq &\; \f{1}{2\pi} \int_0^{g(0)} e^{-2\mu y_2}\sin g^{-1}(y_2) \cdot \ln\left(1+ \f{1+\cos x_1}{\cosh(x_2-y_2) -1 + 2 \cos^2\f{ g^{-1}(y_2)}{2}}\right) dy_2,
\end{align*}
which proves the first claim.
Since $h(x_1) = \Psi(x_1,0)-\Psi(\pi,0)$, taking $x_2 = 0$ immediately leads to the second claim.
\end{proof}
\end{lem}

\begin{lem}
\label{lem: upper bound for F x_1 derivative}

Suppose that $g\in \BM$. For $x_1\in [0,\f{\pi}{2}]$ and $x_2>0$,
\begin{align*}
&\; -\pa_{x_1}\Psi(x_1,x_2)\\
\leq
&\; C\sin x_1 \ln \left(2 + \f{x_2}{\sin x_1}\right)
+
C\int_0^{g(0)} e^{-2\mu y_2}\cdot\ln \left(1 + \f{ \sin x_1 \sin g^{-1}(y_2)}{y_2^2 + x_2^2 + \sin^2 \f{x_1}{2} +  \sin^2 \f{g^{-1}(y_2)}{2}}\right)dy_2, \end{align*}
while for $x_1\in [\f{\pi}{2},\pi]$ and $x_2>0$,
\begin{align*}
&\; -\pa_{x_1}\Psi(x_1,x_2)\\
\leq
&\; C\sin x_1 \ln \left(2 + \f{x_2}{\sin x_1}\right)
+ C \int_0^{g(0)} e^{-2\mu y_2}\cdot \ln \left(1 + \f{ \sin x_1 \sin g^{-1}(y_2)}{y_2^2 + x_2^2 +  \cos^2 \f{x_1}{2} +  \cos^2 \f{g^{-1}(y_2)}{2}}\right) dy_2.
\end{align*}
Here $C>0$ are universal constants that are independent of $\mu$.

\begin{proof}
By \eqref{eqn: formula for Psi_x_1},
\begin{align*}
&\; -\pa_{x_1}\Psi(x_1,x_2)\\
= &\; \f{1}{4\pi}\int_0^{g(0)} e^{-2\mu y_2} \cdot \ln \left(1 + \f{\cos(x_1 - g^{-1}(y_2))-\cos(x_1 +g^{-1}(y_2))}{\cosh (x_2-y_2)-\cos(x_1 - g^{-1}(y_2))}\right)dy_2
\\
= &\; \f{1}{4\pi}\int_0^{g(0)} e^{-2\mu y_2}\cdot \mathds{1}_{\{\sin g^{-1}(y_2)\leq 2\sin x_1\}} \ln \left(1 + \f{ 2\sin x_1 \sin g^{-1}(y_2)}{2\sinh^2\f{x_2-y_2}{2}+ 2\sin^2\f{x_1 - g^{-1}(y_2)}{2}}\right)dy_2
\\
&\; + \f{1}{4\pi}\int_0^{g(0)} e^{-2\mu y_2}\cdot \mathds{1}_{\{\sin g^{-1}(y_2)> 2\sin x_1\}} \ln \left(1 + \f{ 2\sin x_1 \sin g^{-1}(y_2)}{2\sinh^2\f{x_2-y_2}{2}+ 2\sin^2\f{x_1 - g^{-1}(y_2)}{2}}\right)dy_2
\\
=:&\; I_1+I_2.
\end{align*}

Since $\sinh^2 z\geq z^2$, $I_1$ can be bounded by
\[
I_1
\leq C\int_0^{g(0)}\ln \left(1 + \f{C\sin^2 x_1}{|x_2-y_2|^2}\right)dy_2
\leq C\sin x_1\int_{\BR}\ln \left(1 + \f{1}{s^2}\right)ds
\leq C\sin x_1,
\]
where $C>0$ is a universal constant.

We further split $I_2$ into two terms
\begin{align*}
I_2 \leq
&\; C\int_0^{2x_2} e^{-2\mu y_2}\cdot \mathds{1}_{\{\sin g^{-1}(y_2)> 2\sin x_1\}}  \ln \left(1 + \f{ C\sin x_1 \sin g^{-1}(y_2)}{|x_2-y_2|^2 +\sin^2\f{x_1 - g^{-1}(y_2)}{2}}\right)dy_2
\\
&\; + C\int_{[0,g(0)]\cap [2x_2,\infty)} e^{-2\mu y_2}\cdot \mathds{1}_{\{\sin g^{-1}(y_2)> 2\sin x_1\}}
\ln \left(1 + \f{ C\sin x_1 \sin g^{-1}(y_2)}{|x_2-y_2|^2 + \sin^2\f{x_1 - g^{-1}(y_2)}{2}}\right)dy_2\\
=: &\; I_{21} + I_{22}.
\end{align*}
We will need some lower bounds for $\sin^2\f{x_1 - g^{-1}(y_2)}{2}$.
Note that, under the condition $\sin g^{-1}(y_2)> 2\sin x_1$, we must have
$x_1\in [0,\f{\pi}{6})\cup (\f{5\pi}{6},\pi]$, and
\[
g^{-1}(y_2) \in \big(\arcsin (2\sin x_1),\, \pi-\arcsin (2\sin x_1)\big).
\]
If $x_1\in [0,\f{\pi}{6})$, we find that $g^{-1}(y_2)\in (2x_1,\pi-2x_1)$,
so
\[
\left|\f{g^{-1}(y_2)-x_1}{2}\right| = \f{g^{-1}(y_2)-x_1}{2}
\in\left[\f14 g^{-1}(y_2),\,\f12 g^{-1}(y_2)\right].
\]
This gives
\[
\sin^2 \f{g^{-1}(y_2)-x_1}{2} \geq C\sin^2 \f{g^{-1}(y_2)}{2}
\geq C\sin^2 \f{x_1}{2} + C\sin^2 \f{g^{-1}(y_2)}{2},
\]
where $C>0$ is universal.
In the last inequality, we used the fact that $\sin^2 \f{x_1}{2} \leq \sin^2 \f{g^{-1}(y_2)}{2}$.
Otherwise, if $x_1\in (\f{5\pi}{6},\pi]$, we have that $g^{-1}(y_2) \in (2(\pi-x_1),\pi-2(\pi-x_1))$, so
\[
\left|\f{g^{-1}(y_2)-x_1}{2}\right|= \f{(\pi-g^{-1}(y_2))-(\pi-x_1)}{2}\in\left[\f14 (\pi-g^{-1}(y_2)),\,\f12 (\pi-g^{-1}(y_2))\right].
\]
This analogously gives
\[
\sin^2 \f{g^{-1}(y_2)-x_1}{2} \geq C\sin^2 \f{\pi-g^{-1}(y_2)}{2} \geq C\cos^2\f{x_1}{2} +  C\cos^2 \f{g^{-1}(y_2)}{2},
\]
where $C>0$ is universal.
In both cases,
\[
\sin^2 \f{g^{-1}(y_2)-x_1}{2}
\geq C\sin^2 g^{-1}(y_2).
\]

With these bounds, we further derive that
\begin{align*}
I_{21} \leq
&\; C\int_0^{2x_2} \ln \left(1 + \f{ C\sin x_1 \sin g^{-1}(y_2)}{|x_2-y_2|^2 + \sin^2 g^{-1}(y_2)}\right)dy_2\\
\leq
&\; C\int_0^{2x_2} \ln \left(1 + \f{ C\sin x_1 }{|x_2-y_2|}\right)dy_2
\leq C\sin x_1 \ln \left(2 + \f{x_2}{\sin x_1}\right).
\end{align*}
For $I_{22}$, we note that, when $y_2 \geq 2x_2$, $|x_2-y_2|^2 \geq Cy_2^2 + Cx_2^2$.
If $x_1 \in [0,\f{\pi}{2}]$,
\begin{align*}
I_{22}
\leq
&\; C\int_{0}^{g(0)} e^{-2\mu y_2}\cdot \ln \left(1 + \f{ C\sin x_1 \sin g^{-1}(y_2)}{y_2^2 + x_2^2 + \sin^2\f{x_1 - g^{-1}(y_2)}{2}}\right)dy_2
\\
\leq
&\;
C \int_{0}^{g(0)} e^{-2\mu y_2}\cdot  \ln \left(1 + \f{ \sin x_1 \sin g^{-1}(y_2)}{y_2^2 + x_2^2 + \sin^2 \f{x_1}{2} +  \sin^2 \f{g^{-1}(y_2)}{2}}\right)dy_2.
\end{align*}
If $x_1 \in [\f{\pi}{2},\pi]$, similarly, \[
I_{22}
\leq
C \int_0^{g(0)} e^{-2\mu y_2}\cdot \ln \left(1 + \f{ \sin x_1 \sin g^{-1}(y_2)}{y_2^2 + x_2^2 + \cos^2 \f{x_1}{2} +  \cos^2 \f{g^{-1}(y_2)}{2}}\right)dy_2.
\]

Summarizing all the estimates, we proved the desired claim.
\end{proof}
\end{lem}

\subsection{Barrier functions}
In this part, we shall prove two very important propositions regarding the existence of what we call the upper and lower barrier functions for the mapping $\bfR$.
As is explained in Section \ref{sec: scheme of the proof}, these barrier functions will provide non-trivial upper and lower bounds in the design of the function set on which the fixed-point argument is applied (see \eqref{eqt: def of function set D_0} and \eqref{eqt: def of function set D} below).

\begin{prop}
\label{prop: upper barrier new}
Define $u(x) := \pi^2 -x^2$.
Let $M$ be given in Proposition \ref{prop: a priori upper bound general m}.
Let $c_*>0$ be the unique positive root to $\f{1-e^{-c}}{c}=\f12$. Then for any $\mu< \f{c_*}{2M}$, there exists $\Lam_0 = \Lam_0(\mu,M)>0$, such that the following holds: for any $\Lam\geq \Lam_0(\mu,M)$, if $g\in \BM$ satisfies $g(x)\leq M$, $g(x)\leq \Lam u(x)$ for all $x\in [0,\pi]$, and $g(x_*)= \Lam u(x_*)$ at some $x_*\in [0,\pi)$, then $\bfR(g)(x_*)< g(x_*)$.

\begin{proof}
We only consider the case $g(x_*)<M$, as otherwise this can be readily proved by Proposition \ref{prop: a priori upper bound general m} and the monotonicity of $\bfR (g)$.
We also assume that $\Lam$ is suitably large so that $x_*\in [\f{\pi}2,\pi]$ and $g(x_*)>\pi-x_*$.
This can be achieved, because from $g(x_*)=\Lam u(x_*)\in [0,M]$  we can derive that
\beqo
\pi-x_*  = \f{g(x_*)}{\Lam(\pi + x_*)}\leq \f{M}{\pi \Lam}.
\eeqo

We prove by contradiction.
Suppose $\bfR(g)(x_*)\geq g(x_*)$.
Then by the definition of $\bfR(g)$ (see \eqref{eqn: def of R(g)}) and the monotonicity of $F$, it must hold that
\[
F\big(x_*,g(x_*)\big)\geq F\big(x_*,\bfR(g)(x_*)\big) = \pa_{x_2}\Psi(\pi,0)\cdot \f{1-e^{-2\mu g(x_*)}}{2\mu g(x_*)}.
\]
This can be equivalently written as
\beq
F\big(x_*,g(x_*)\big) - \pa_{x_2}\Psi(\pi,0)
\geq -\pa_{x_2}\Psi(\pi,0)\left[1-\f{1-e^{-2\mu g(x_*)}}{2\mu g(x_*)}\right].
\label{eqn: difference at x_* and the end point}
\eeq

We first bound the right-hand side of \eqref{eqn: difference at x_* and the end point}.
By Lemma \ref{lem: integral formula for Psi} and the assumption $g(x)\leq M$, \beq
0\leq \pa_{x_2}\Psi(\pi,0)
\leq \f{1}{4\pi}\int_0^M e^{-2\mu y_2}
\int_{-\pi}^{\pi}\left[ \f{\sinh y_2}{\cosh y_2+ \cos y_1}+1\right] dy_1\,dy_2
= \f{1}{2\mu}\left(1-e^{-2\mu M}\right).
\label{eqn: upper bound for Psi_x2 at pi 0}
\eeq
So
\beqo
-\pa_{x_2}\Psi(\pi,0)\left[1-\f{1-e^{-2\mu g(x_*)}}{2\mu g(x_*)}\right]
\geq -\f{1}{2\mu}\left(1-e^{-2\mu M}\right)\left[1-\f{1-e^{-2\mu g(x_*)}}{2\mu g(x_*)}\right].
\eeqo

Next we bound the left-hand side of \eqref{eqn: difference at x_* and the end point} from above.
To that end, we rewrite it as
\beq
\begin{split}
&\;F\big(x_*,g(x_*)\big) - \pa_{x_2}\Psi(\pi,0)
\\
= &\; \left[F\big(x_*,g(x_*)\big) -F\big(\pi,g(x_*)-(\pi-x_*)\big)\right]
+\left[F\big(\pi,g(x_*)-(\pi-x_*)\big)- \pa_{x_2}\Psi(\pi,0)\right].
\end{split}
\label{eqn: compare the touching point and the endpoint}
\eeq
Using the fact that $\pa_{x_2}F_1\leq 0$,
\beq
\begin{split}
&\;F\big(x_*,g(x_*)\big) -F\big(\pi,g(x_*)-(\pi-x_*)\big)\\
= &\; -\int_{x_*}^{\pi} \big[(1,-1)\cdot
\na (F_1+F_2)\big] \big(y_1,g(x_*)-(y_1-x_*)\big)\, dy_1
\\
\leq &\; \int_{x_*}^{\pi}
\big(-\pa_{x_1} F_1-\pa_{x_1} F_2+\pa_{x_2} F_2\big)\big(y_1,g(x_*)-(y_1-x_*)\big)\, dy_1.
\end{split}
\label{eqn: tilted displacement}
\eeq
In order to bound the last line, we need some preparations.
Observe that, for any $y_1\in [x_*,\pi]$, it holds that $y_1\geq x_* \geq \f{\pi}{2}$ by assumption, and with $y_2 := g(x_*)-(y_1-x_*)$,
\begin{align*}
\f{\sin y_1}{y_2}
\leq &\;\f{\pi-y_1}{g(x_*)+(\pi-y_1)-(\pi-x_*)}
\\
\leq &\; \f{\pi-x_*}{g(x_*)}
= \f{1}{\Lam (\pi+x_*)}
\leq \f{2}{3\pi\Lam}.
\end{align*}
We also claim that
\[
y_2\mapsto \f{y_2^2 \sinh y_2}{y_2 \cosh y_2 - \sinh y_2}
\mbox{ is increasing on $(0,+\infty)$.}
\]
Indeed,
\[
\f{d}{dy_2}\left[\f{y_2^2 \sinh y_2}{y_2 \cosh y_2 - \sinh y_2}\right]
=\f{y_2(y_2^2 + y_2 \sinh y_2 \cosh y_2 -2\sinh^2 y_2)}{(y_2\cosh y_2 - \sinh y_2)^2},
\]
while
\[
\f{d}{dy_2}\big[y_2^2 + y_2 \sinh y_2 \cosh y_2 -2\sinh^2 y_2\big]
= y_2 + 2y_2\cosh^2 y_2 - 3\cosh y_2 \sinh y_2,
\]
and
\[
\f{d^2}{dy_2^2}\big[y_2^2 + y_2 \sinh y_2 \cosh y_2 -2\sinh^2 y_2\big]
= 4\sinh y_2\big(y_2\cosh y_2 - \sinh y_2 \big) \geq 0.
\]
This proves the desired claim.
As a result, for $y_1\in [x_*,\pi]$ and $y_2:= g(x_*)-(y_1-x_*)$,
\begin{align*}
\f{y_2\sinh y_2 \sin y_1}{y_2\cosh y_2-\sinh y_2}
= &\; \f{y_2^2\sinh y_2}{y_2\cosh y_2-\sinh y_2}\cdot \f{\sin y_1}{y_2}\\
\leq &\; \f{M^2\tanh M}{M-\tanh M}\cdot \f{2}{3\pi\Lam}=:\f{A_M}{\Lam},
\end{align*}
where we denoted
\[
A_M:=\f{2}{3\pi}\cdot \f{M^2\tanh M}{M-\tanh M}.
\]

By Lemma \ref{lem: upper bound for -F_1 x_1} and the fact that $z\mapsto z\ln(1+\f{1}{z})$ is increasing on $(0,+\infty)$,
\begin{align*}
&\;\int_{x_*}^{\pi}
-\pa_{x_1} F_1\big(y_1,g(x_*)-(y_1-x_*)\big)\, dy_1
\\
\leq &\; \int_{x_*}^{\pi}\left.\left[
\left(\sqrt{2} + \f{2+2\ln 3}{\pi}\right) \f{\sin y_1}{y_2}
+ \f{2\sin y_1}{\pi y_2} \ln \left(1+\f{y_2}{4\sin y_1}\right)\right]\right|_{y_2 = g(x_*)-(y_1-x_*)}\, dy_1
\\
\leq &\; (\pi-x_*)\left[
\left(\sqrt{2} + \f{2+2\ln 3}{\pi}\right) \f{1}{\pi\Lam}
+ \f{2}{\pi\Lam}\ln \left(1+\f{\pi\Lam}{4}\right)\right].
\end{align*}
By Lemma \ref{lem: bounding F_2 x_1 in terms of F_2 x_2},
\begin{align*}
&\;\int_{x_*}^{\pi}
\big[-\pa_{x_1} F_2+\pa_{x_2} F_2\big]\big(y_1,g(x_*)-(y_1-x_*)\big)\, dy_1
\\
\leq &\;\int_{x_*}^{\pi}
\left.\left[\left(\f{y_2\sinh y_2 \sin y_1}{y_2\cosh y_2 - \sinh y_2} - 1\right)\cdot \big(-\pa_{x_2} F_2\big)(y_1,y_2)\right]\right|_{y_2=g(x_*)-(y_1-x_*)}\, dy_1
\\
\leq &\;\int_{x_*}^{\pi}
\left(\f{A_M}{\Lam}- 1\right)\cdot \big(-\pa_{x_2} F_2\big)\big(y_1,g(x_*)-(y_1-x_*)\big) \, dy_1.
\end{align*}
Since $\pa_{x_2}F_2\leq 0$, we may assume $\Lam \geq A_M$, so that the above estimate implies
\[
\int_{x_*}^{\pi}
\big[-\pa_{x_1} F_2+\pa_{x_2} F_2\big]\big(y_1,g(x_*)-(y_1-x_*)\big)\, dy_1 \leq 0.
\]
Combining the estimates above with \eqref{eqn: tilted displacement}, we obtain that
\beq
\begin{split}
&\;F\big(x_*,g(x_*)\big) -F\big(\pi,g(x_*)-(\pi-x_*)\big)\\
\leq &\; g(x_*)\cdot \f{\pi-x_*}{g(x_*)}\cdot \f{1}{\pi\Lam} \left[
\left(\sqrt{2} + \f{2+2\ln 3}{\pi}\right)
+ 2\ln \left(1+\f{\pi\Lam}{4}\right)\right]
\\
\leq &\; g(x_*)\cdot \f{2}{3\pi^2 \Lam^2}\left[
\sqrt{2} + \f{2+2\ln 3}{\pi}
+ 2\ln \left(1+\f{\pi\Lam}{4}\right)\right].
\end{split}
\label{eqn: bound for the tilted difference}
\eeq

To bound the last term in \eqref{eqn: compare the touching point and the endpoint}, we apply Lemma \ref{lem: control vertical displacement general mu upper and lower bounds}: \begin{align*}
&\;F\big(\pi,g(x_*)-(\pi-x_*)\big)- \pa_{x_2}\Psi(\pi,0)\\
\leq &\; -\f{1}{\pi}\int_0^{g(x_*)-(\pi-x_*)} e^{-2\mu y_2} \cdot \arctan\left(\tanh\f{y_2}{2}\tan\f{g^{-1}(g(x_*)-(\pi-x_*))}{2}\right) dy_2
\\
\leq &\; -\f{1}{\pi}\int_0^{g(x_*)-(\pi-x_*)} e^{-2\mu y_2} \cdot \arctan\left(\tanh\f{y_2}{2}\tan\f{x_*}{2}\right) dy_2
\\
= &\; -\f{1}{\pi}\int_0^{g(x_*)-(\pi-x_*)} e^{-2\mu y_2} \left[\f{\pi}{2}- \arctan\left(\f{1}{\tanh\f{y_2}{2}}\tan\f{\pi-x_*}{2}\right)\right] dy_2
\\
\leq &\; -\f{1}{2} \cdot \f{1-e^{-2\mu (g(x_*)-(\pi-x_*))}}{2\mu}
+ \f{1}{\pi}\int_0^{g(x_*)} e^{-2\mu y_2}  \min\left\{\f{\pi}{2},\, \f{1}{\tanh\f{y_2}{2}} \cdot \f{\pi-x_*}{2}\right\} dy_2.
\end{align*}
In the last line, we used the fact that, for $\b\geq 1$ and $z\geq 0$, it holds that $\arctan (\b z)\leq \min\{\f{\pi}{2},\b \arctan z\}$.
By the mean value theorem,
\[
-\f12 \cdot \f{1-e^{-2\mu (g(x_*)-(\pi-x_*))}}{2\mu }
\leq - \f{1-e^{-2\mu g(x_*)}}{4\mu} + \f{\pi-x_*}{2}
\leq g(x_*)\left[- \f{1-e^{-2\mu g(x_*)}}{4\mu g(x_*)}
+ \f{1}{3\pi\Lam}\right].
\]
Since $z\mapsto \f{\tanh z}{z}$ is decreasing on $(0,+\infty)$, we use the fact $x_*\in [\f{\pi}{2},\pi]$ to deduce that, for $z\in (0,1]$,
\begin{align*}
\f{1}{\tanh\f{g(x_*)z}{2}} \cdot \f{\pi-x_*}{2}
= &\; \f{\f{g(x_*)z}{2}}{\tanh\f{g(x_*)z}{2}} \cdot  \f{\f{\pi-x_*}{2}}{\f{g(x_*)z}{2}}
\\
\leq &\;
\f{\f{Mz}{2}}{\tanh\f{Mz}{2}} \cdot \f{1}{\Lam(\pi+x_*)z}
\leq \f{M}{3\pi \Lam\tanh\f{Mz}{2}}.
\end{align*}
Hence,
\begin{align*}
&\; \f{1}{\pi}\int_0^{g(x_*)} \min\left\{\f{\pi}{2},\, \f{1}{\tanh\f{y_2}{2}}\cdot\f{\pi-x_*}{2}\right\} dy_2
\\
= &\; \f{g(x_*)}{\pi}\int_0^1 \min\left\{\f{\pi}{2},\, \f{1}{\tanh\f{g(x_*)z}{2}}\cdot\f{\pi-x_*}{2}\right\} dz
\\
\leq &\; \f{g(x_*)}{\pi}\int_0^1 \min\left\{\f{\pi}{2},\, \f{M}{3\pi \Lam\tanh\f{M z}{2}} \right\} dz.
\end{align*}
In summary,
\begin{align*}
&\;F\big(\pi,g(x_*)-(\pi-x_*)\big)- \pa_{x_2}\Psi(\pi,0)\\
\leq &\;  g(x_*)\left[- \f{1-e^{-2\mu g(x_*)}}{4\mu g(x_*)}
+ \f{1}{3\pi\Lam}
+ \f{1}{\pi}\int_0^1 \min\left\{\f{\pi}{2},\, \f{M}{3\pi \Lam\tanh\f{Mz}{2}} \right\} dz\right].
\end{align*}
Combining this with \eqref{eqn: difference at x_* and the end point}-\eqref{eqn: compare the touching point and the endpoint} as well as \eqref{eqn: bound for the tilted difference}, we find that
\beq
\begin{split}
&\;
\f{1-e^{-2\mu g(x_*)}}{2\mu g(x_*)} -\left(1-e^{-2\mu M}\right)\cdot \f{1}{\mu g(x_*)}\left(1-\f{1-e^{-2\mu g(x_*)}}{2\mu g(x_*)}\right)
\\
\leq &\;
\f{4}{3\pi^2 \Lam^2}\left[
\sqrt{2} + \f{2+2\ln 3}{\pi}
+ 2\ln \left(1+\f{\pi\Lam}{4}\right)\right]\\
&\;
+ \f{2}{3\pi\Lam}
+ \f{2}{\pi}\int_0^1 \min\left\{\f{\pi}{2},\, \f{M}{3\pi \Lam\tanh\f{Mz}{2}} \right\} dz.
\end{split}
\label{eqn: final inequality in upper barrier prelim}
\eeq

Let
\[
\r(z) := \f{2}{1-e^{-2z} } \left(1-\f{1-e^{-2z}}{2z}\right) = 1 + \f{1}{\tanh z} - \f{1}{z}.
\]
For $z>0$,
\[
\r'(z) = \f1{z^2}-\f{1}{\sinh^2 z} >0,
\]
so $\r =\r(z)$ is strictly increasing on $(0,+\infty)$.
Also note that $\eta(z):=\f{1-e^{-z}}{z}$ is strictly decreasing on $(0,+\infty)$.
This allows us to handle the left-hand side of \eqref{eqn: final inequality in upper barrier prelim} as follows:
\begin{align*}
&\;
\f{1-e^{-2\mu g(x_*)}}{2\mu g(x_*)} -\left(1-e^{-2\mu M}\right)\cdot \f{1}{\mu g(x_*)}\left(1-\f{1-e^{-2\mu g(x_*)}}{2\mu g(x_*)}\right)
\\
= &\;\f{1-e^{-2\mu g(x_*)}}{2\mu g(x_*)}
\Big[1-\left(1-e^{-2\mu M}\right) \r\big(\mu g(x_*)\big)\Big]
\\
\geq
&\; \f{1-e^{-2\mu M}}{2\mu M}
\left[1-\left(1-e^{-2\mu M}\right) \r\big(\mu M\big)\right]
\\
= &\;
\f{1-e^{-2\mu M}}{\mu M}
\left[\f{1-e^{-2\mu M}}{2\mu M}- \f12\right].
\end{align*}
Recall that $c_*>0$ is defined to be the unique positive root to $\eta(c) = \f12$, which means $\eta(z) > \f12$ for any $z\in (0,c_*)$.
Since we assumed $2\mu M < c_*$, this immediately gives that
\begin{align*}
\f{1-e^{-2\mu M}}{\mu M}
\left[\f{1-e^{-2\mu M}}{2\mu M}- \f12\right] > 0.
\end{align*}
On the other hand, the right-hand side of \eqref{eqn: final inequality in upper barrier prelim} converges to zero if $\Lam$ is sent to $+\infty$.
Therefore, whenever $2\mu M < c_*$, we may choose $\Lam$ to be sufficiently large, such that
\begin{align*}
&\; \f{4}{3\pi^2 \Lam^2}\left[
\sqrt{2} + \f{2+2\ln 3}{\pi}
+ 2\ln \left(1+\f{\pi\Lam}{4}\right)\right]\\
&\;
+ \f{2}{3\pi\Lam}
+ \f{2}{\pi}\int_0^1 \min\left\{\f{\pi}{2},\, \f{M}{3\pi \Lam\tanh\f{Mz}{2}} \right\} dz
\\
\leq &\; \f{1-e^{-2\mu M}}{\mu M}
\left[\f{1-e^{-2\mu M}}{2\mu M}- \f12\right],
\end{align*}
which would lead to a contradiction to \eqref{eqn: final inequality in upper barrier prelim}.

This completes the proof.
\end{proof}

\begin{rmk}
\label{rmk: universality of large Lambda}
Recall that $c_*>0$ is the unique positive root to $\f{1-e^{-c}}{c}=\f12$.
Since
\[
\ln 4= 2\int_1^2 \f{1}{z}\, dz < 2\cdot\f12 \left(1+\f{1}{2}\right) = \f{3}{2},
\]
we have $e^{-3/2}< \f14$.
This together with the fact that $c\mapsto \f{1-e^{-c}}{c}$ is strictly decreasing on $(0,+\infty)$ implies that $c_*>\f32$.
It is then not difficult to see from the last part of the proof that, with $M$ and $\mu_0$ being the universal constants fulfilling the statement of Proposition \ref{prop: a priori upper bound general m}, for all $\mu \in [0, \f{3}{4M}]\cap [0,\mu_0]$, $\Lam_0(\mu,M)$ in the statement of Proposition \ref{prop: upper barrier new} may be chosen uniformly as a universal constant, which is denoted by $\Lam_*$.
Without loss of generality, we may further assume that $\Lam_*\geq 1$ and $ \Lam_* \pi^2\geq M$.

\end{rmk}
\end{prop}

\begin{prop}
\label{prop: lower barrier}
Let $\mu\in [0,\f12]$ and fix $\Lam\geq 1$.
Denote $u(x):= \pi^2-x^2$.
Then there exists $\lam_0\in (0,1]$, which depends only on $\Lam$ but not on $\mu$, such that the following holds for any $\lam\in (0,\lam_0]$: if $g\in \BM$ satisfies $\lam u(x)\leq g(x)\leq \Lam u(x)$ for all $x\in [0,\pi]$, and $g(x_*)= \lam u(x_*)$ at some $x_*\in [0,\pi)$, then $\bfR(g)(x_*)> g(x_*)$.

\begin{rmk}
\label{rmk: universality of small lambda}
If we set $\Lam \geq 1$ in the condition of this proposition to be $\Lam_*$ defined in Remark \ref{rmk: universality of large Lambda}, we can obtain $\lam_0$ here as a universal constant as well.
We denote this universal constant to be $\lam_*$.
Without loss of generality, we assume that $\lam_* \pi^2  \leq  M$.
\end{rmk}

\begin{proof} Assume that $g\in \BM$ satisfies $\lam u(x)\leq g(x)\leq \Lam u(x)$ for all $x\in [0,\pi]$, and $g(x_*)= \lam u(x_*)$ at some $x_*\in [0,\pi)$.
To show $\bfR(g)(x_*)> g(x_*)$, thanks to the monotonicity of $F$ (see Proposition \ref{prop: monotonicity of F}) and the definition of $\bfR(g)$ in \eqref{eqn: def of R(g)}, it suffices to achieve that
\beqo
F(x_*,g(x_*))-\pa_{x_2}\Psi(\pi,0)>0,
\eeqo
i.e.,
\beq
F(x_*,g(x_*))-F(\pi,g(x_*))
> - \big(F(\pi,g(x_*))-\pa_{x_2}\Psi(\pi,0)\big).
\label{eqn: sign of F at the lower touch point}
\eeq

We first study the right-hand side of \eqref{eqn: sign of F at the lower touch point}.
By Lemma \ref{lem: control vertical displacement general mu upper and lower bounds},
\begin{align*}
&\; - \big(F(\pi,g(x_*))-\pa_{x_2}\Psi(\pi,0)\big)\\
\leq &\;
Cg(x_*)\left(1
+\int_{\min\{2g(x_*),g(0)\}}^{g(0)} \min\left\{ \tan\f{g^{-1}(y_2)}{2},\, \f{1}{\sinh \f{y_2}{2}}\right\} dy_2
\right),
\end{align*}
where $C>0$ is a universal constant.
We calculate that
\begin{align*}
&\;\int_{\min\{2g(x_*),g(0)\}}^{g(0)} \min\left\{\tan\f{g^{-1}(y_2)}{2},\, \f{1}{\sinh\f{y_2}{2}} \right\} dy_2
\\
\leq &\; \int_{0}^\infty \min\left\{\tan\f{x_*}{2},\, \f{1}{\sinh\f{y_2}{2}} \right\} dy_2
\\
\leq &\; \int_0^1 \min\left\{\tan\f{x_*}{2},\, \f{2}{y_2} \right\} dy_2
+ \int_1^\infty \f{1}{\sinh\f{y_2}{2}} \, dy_2
\\
\leq &\;
C\ln \left(2+\tan\f{x_*}{2}\right).
\end{align*}
Therefore,
\beq
- \big(F(\pi,g(x_*))-\pa_{x_2}\Psi(\pi,0)\big)
\leq Cg(x_*)\ln \left(2+\tan\f{x_*}{2}\right),
\label{eqn: upper bound for RHS}
\eeq
where $C>0$ is universal.

\setcounter{case}{0}

Next let us bound the left-hand side of \eqref{eqn: sign of F at the lower touch point} from below.
By Lemma \ref{lem: control horizontal displacement of Psi},
\beq
\begin{split}
&\; F(x_*,g(x_*))-F(\pi,g(x_*))\\
\geq &\; \f{1}{2\pi g(x_*)} \int_0^{g(0)}e^{-2\mu y_2} \sin g^{-1}(y_2) \cdot \ln\left(1+ \f{1+\cos x_*}
{\cosh(g(x_*)-y_2) -1 + 2 \cos^2\f{g^{-1}(y_2)}{2}}\right) dy_2.
\end{split}
\label{eqn: F x_* g(x_*) - F pi g(x_*)}
\eeq
We proceed in two cases.
\begin{case}
We first assume $x_*\in [0,\f{2\pi}{3}]$.
Under this condition, for $y_2\in [0,\lam u(\f{2\pi}{3})]\subset [0,g(\f{2\pi}{3})]$, it holds that $y_2 \leq g(x_*)$; moreover, $g^{-1}(y_2)\in [\f{2\pi}{3},\pi]$, which implies $\sin g^{-1}(y_2)\geq \sqrt{3}\cos \f{g^{-1}(y_2)}{2}$.
Hence, \eqref{eqn: F x_* g(x_*) - F pi g(x_*)} gives
\begin{align*}
&\; F (x_*,g(x_*))-F(\pi,g(x_*))
\\
\geq &\; \f{1}{2\pi \lam u(x_*)} \int_0^{\lam u(2\pi/3)} e^{-2\mu\lam u(\f{2\pi}{3})} \cdot \sqrt{3}\cos \f{g^{-1}(y_2)}{2} \cdot \ln\left(1+ \f{\f12}
{\cosh g(x_*)-1 + 2 \cos^2\f{ g^{-1}(y_2)}{2}}\right) dy_2.
\end{align*}
Since $g(x_*) = \lam u(x_*) \leq \pi^2\lam$,
\[
\cosh g(x_*)-1\leq \cosh (\pi^2 \lam)-1 \leq C\lam^2,
\]
where $C$ is universal given that $\lam \leq 1$.
Also, $e^{-2\mu \lam u(\f{2\pi}{3})} \geq C$ for some universal $C>0$, since $\lam \leq 1$ and $\mu\in [0,\f12]$.
Therefore, for some universal $C>0$, \beq
\begin{split}
&\; F(x_*,g(x_*))-F(\pi,g(x_*))
\\
\geq &\; \f{C}{\lam} \int_0^{\lam u(2\pi/3)} \cos\f{ g^{-1}(y_2)}{2} \cdot \ln\left(1+ \f{1}
{\lam^2 + \cos^2\f{g^{-1}(y_2)}{2}}\right) dy_2
\\
=: &\; \f{C}{\lam} \int_0^{\lam u(2\pi/3)} \r_\lam\left(\cos\f{ g^{-1}(y_2)}{2}\right) dy_2,
\end{split}
\label{eqn: lower bound for LHS prelim general mu}
\eeq
where we defined
\[
\r_\lam(s):= s\ln\left(1+ \f{1}{\lam^2 + s^2}\right).
\]

We claim that there exists a unique $s_* = s_*(\lam)>0$, such that $\r_\lam=\r_\lam(s)$ is increasing in $[0,s_*(\lam)]$ and decreasing in $[s_*(\lam),+\infty)$.
Indeed,
\[
\r'_\lam(s) = \ln\left(1+ \f{1}{\lam^2 + s^2}\right)
- \f{2s^2}{(1+\lam^2+s^2)(\lam^2+s^2)},
\]
so
\[
\big(1+\lam^2+s^2\big)\r_\lam'(s)
= \f{\ln(1-\f{1}{1+\lam^2 + s^2})}{-\f{1}{1+\lam^2+s^2}}
- \f{2s^2}{\lam^2+s^2}
\]
is strictly decreasing for $s\in [0,+\infty)$.
Moreover, $\r_\lam'(0)>0$, and
\[
\r_\lam'\big(\sqrt{1+\lam^2}\big) = \ln\left(1+ \f{1}{1+2\lam^2}\right)
- \f{2(1+\lam^2)}{2(1+\lam^2)(1+2\lam^2)} <0.
\]
This implies the claim.

Now to bound the last line of \eqref{eqn: lower bound for LHS prelim general mu}, we deduce under the assumption $g(x)\leq \Lam u(x)$ that, for any $y_2\in [0,\lam u(\f{2\pi}{3})]\subset [0,g(\f{2\pi}{3})]$,
\[
\pi-g^{-1}(y_2)\leq \f{\pi}{3},
\quad
\pi-g^{-1}(y_2)\geq \f{y_2}{\Lam (\pi + g^{-1}(y_2))} \geq
\f{y_2}{2\pi \Lam},
\]
so $\cos\f{g^{-1}(y_2)}{2} \in [\f{Cy_2}{\Lam}, \f12]$ for some universal $C\in (0,1)$.
This together with the above claim gives that
\[
\r_\lam\left(\cos\f{ g^{-1}(y_2)}{2}\right)
\geq \min\left\{ \r_\lam\left(\f{Cy_2}{\Lam}\right),\, \r_\lam\left(\f{1}{2}\right)\right\}
\geq C\min\left\{\r_\lam\left(\f{y_2}{\Lam}\right), \, 1 \right\}
\geq C\r_\lam\left(\f{y_2}{\Lam}\right),
\]
where $C$ is universal.
In the last inequality, we used the fact that $\r_\lam(s) \leq C$ on $[0,+\infty)$ for some universal $C>0$ that is independent of $\lam\in [0,1]$.
Therefore,
\begin{align*}
\int_0^{\lam u(2\pi/3)} \r_\lam\left(\cos\f{ g^{-1}(y_2)}{2}\right) dy_2
\geq &\; C\int_0^{\lam u(2\pi/3)} \r_\lam\left(\f{y_2}{\Lam}\right) dy_2
\\
= &\; C\Lam \int_0^{\lam u(2\pi/3)/\Lam} s \ln\left(1+ \f{1}
{\lam^2 + s^2}\right) ds
\\
\geq &\; C\Lam \cdot \left(\f{\lam}{\Lam}\right)^2 \ln\left(1+ \f{1}
{\lam^2 + (\f{\lam}{\Lam})^2}\right)\\
\geq &\; \f{C\lam^2}{\Lam} \ln\left(1+ \f{1}{2\lam^2}\right),
\end{align*}
where $C$ is universal.
Here we used the facts that $u(2\pi/3) \geq 1$ and $\lam\leq 1\leq \Lam$.
Putting this back to \eqref{eqn: lower bound for LHS prelim general mu} yields that
\[
F(x_*,g(x_*))-F(\pi,g(x_*))
\geq \f{C\lam}{\Lam} \ln\left(1+ \f{1}
{2\lam^2}\right),
\]
where $C$ is universal.

In view of this estimate and \eqref{eqn: upper bound for RHS}, in order to achieve \eqref{eqn: sign of F at the lower touch point}, we only need to guarantee that
\beqo
\f{\lam}{\Lam} \ln\left(1+ \f{1}
{2\lam^2}\right)
> C\lam u(x_*)\ln \left(2+\tan\f{x_*}{2}\right),
\eeqo
where $C$ is universal.
Since $x_*\in [0,\f{2\pi}{3}]$, this holds when $\lam$ is taken to be suitably small, which only depends on $\Lam$.
\end{case}

\begin{case}
Next we consider the case $x_*\in [\f{2\pi}{3},\pi]$.
We still have \eqref{eqn: F x_* g(x_*) - F pi g(x_*)}.
Observe that, for $y_2\in[\lam u(x_*),\lam u(\f{\pi}{2})]\subset [ g(x_*),g(\f{\pi}{2})]$,
\[
1+\cos x_* = 2\cos^2 \f{x_*}{2}\leq 2\cos^2 \f{g^{-1}(y_2)}{2}
\leq C\big(\pi-g^{-1}(y_2)\big)^2,
\]
and
\[
\cosh (g(x_*)-y_2)-1\leq C(y_2-g(x_*))^2,
\]
where $C$ is universal.
As a result,
\[
\f{1+\cos x_*}
{\cosh(g(x_*)-y_2) -1 + 2 \cos^2\f{g^{-1}(y_2)}{2}}
\leq 1,
\]
while
\[
\f{1+\cos x_*}
{\cosh(g(x_*)-y_2) -1 + 2 \cos^2\f{g^{-1}(y_2)}{2}}\geq
\f{C(1+\cos x_*)}
{(y_2-g(x_*))^2 + (\pi-g^{-1}(y_2))^2},
\]
where $C>0$ is universal.
Therefore, we derive from \eqref{eqn: F x_* g(x_*) - F pi g(x_*)} that
\begin{align*}
&\; F\big(x_*,g(x_*)\big)-F\big(\pi,g(x_*)\big)
\\
\geq &\; \f{C}{g(x_*)} \int_{\lam u(x_*)}^{\lam u(\pi/2)} e^{-2\mu \lam u(\pi/2)} \big(\pi-g^{-1}(y_2)\big)\cdot \f{1+\cos x_*}
{\cosh(g(x_*)-y_2) -1 + 2 \cos^2\f{g^{-1}(y_2)}{2}}\, dy_2
\\
\geq &\; \f{C(1+\cos x_*)}{g(x_*)} \int_{\lam u(x_*)}^{\lam u(\pi/2)}
\f{\pi-g^{-1}(y_2)}{(y_2-g(x_*))^2 + (\pi-g^{-1}(y_2))^2}\, dy_2,
\end{align*}
where $C$ is universal.
In the last line, we used the facts that $\mu\in [0,\f12]$ and $\lam\in [0,1]$, so $e^{-2\mu \lam u(\pi/2)}> C$ for some universal $C>0$.

By the assumption $\lam u(x)\leq g(x)\leq \Lam u(x)$, for $y_2\in [\lam u(x_*),\lam u(\f{\pi}{2})]\subset [g(x_*),g(\f{\pi}{2})]$,
\beq
\pi-g^{-1}(y_2)\in\left[\pi-x_*,\,\f{\pi}{2}\right]\cap \left[\f{C_1}{\Lam}y_2,\,\f{C_2}{\lam}y_2\right],
\label{eqn: bounds for pi-g^-1}
\eeq
where $C_1$ and $C_2$ are universal constants.
Hence,
\[
(y_2-g(x_*))^2 \leq y_2^2 \leq C\Lam^2\big(\pi-g^{-1}(y_2)\big)^2,
\]
This gives that
\begin{align*}
&\; F\big(x_*,g(x_*)\big)-F\big(\pi,g(x_*)\big)
\\
\geq &\; \f{C(1+\cos x_*)}{g(x_*)} \int_{\lam u(x_*)}^{\lam u(\pi/2)} \f{\pi-g^{-1}(y_2)}{(1+\Lam^2)(\pi-g^{-1}(y_2))^2}\, dy_2
\\
\geq &\; \f{C(1+\cos x_*)}{(1+\Lam^2) \cdot \lam u(x_*)} \int_{\lam u(x_*)}^{\lam u(\pi/2)} \f{1}{\f{C_2}{\lam}y_2 }\, dy_2,
\end{align*}
where $C$ is universal.
Here we used \eqref{eqn: bounds for pi-g^-1} in the second inequality.
Calculating the integral yields that \[
F\big(x_*,g(x_*)\big)-F\big(\pi,g(x_*)\big)
\geq \f{Cu(x_*)}{\Lam^2} \cdot \ln \left(\f{u(\pi/2)}{u(x_*)}\right),
\]
where $C$ is universal.

In view of this estimate and \eqref{eqn: upper bound for RHS}, we find that in order to achieve \eqref{eqn: sign of F at the lower touch point}, we only need
\beqo
\f{u(x_*)}{\Lam^2} \ln\left(\f{u(\pi/2)}{u(x_*)}\right)
> Cg(x_*)\ln \left(2+\tan\f{x_*}{2}\right),
\eeqo
where $C$ is universal. This is equivalent to
\[
\ln\left(\f{\pi^2 - (\f{\pi}{2})^2}{\pi^2 -x_*^2}\right)
> C\lam \Lam^2\ln \left(2+\tan\f{x_*}{2}\right).
\]
Since $x_* \in [\f{2\pi}{3},\pi]$, it is clear that this holds when $\lam$ is taken to be suitably small, which only depends on $\Lam$.
\end{case}

In summary, by taking $\lam$ to be smaller than some $\lam_0\leq 1$, which only depends on $\Lam$ but not on $\mu\in [0,\f12]$, we can make \eqref{eqn: sign of F at the lower touch point} hold.
That further gives $\mathbf{R}(g)(x_*)>g(x_*)$ due to the monotonicity of $F$.
\end{proof}
\end{prop}

\section{Existence of the Fixed Point}
\label{sec: existence of the fixed point}

Recall that \eqref{eqn: def of R(g)} not only defines the mapping $\bfR$, but also reformulates the task of finding $g\in \BM$ that satisfies \eqref{eqn: constraint for the boundary curve in Phi} into a fixed-point problem for $\bfR$ in $\BM$ (see its definition in \eqref{eqn: function set tilde M_0}).
In this section, we will prove the desired existence of the fixed point of $\bfR$ on a suitably designed subset of $\BM$.
The main result of this section is Proposition \ref{prop: existence of fixed point}.

From now on, if not otherwise stated, we will always use $M\geq 1$ and $\mu_0\in (0,\f12]$ to denote the universal constants that fulfill Proposition \ref{prop: a priori upper bound general m} (also see Corollary \ref{cor: M=4 and mu leq 1/12}), and we will always assume $\mu \in[0,\mu_0]$.

\subsection{Continuity of the mapping $\bfR$}
We will first show the continuity of $\bfR(g)$ with respect to $g$ in the standard $L^2$-norm on $\BT=[-\pi,\pi]$ under suitable conditions.
To that end, we will start from proving continuous dependence of $\pa_{x_2}\Psi(\pi,0;g)$ and $F(x_1,x_2;g)$ on $g$.

\begin{lem}\label{lem: bound for Psi_x_2}
For any $g \in \BM$ with $g(0)\leq M$,  $\pa_{x_2}\Psi(\pi, 0;g)\in [0,M]$.
Moreover, for any $g_1,g_2 \in \BM$ with $g_1,g_2\leq M$ on $[-\pi,\pi]$,
\[
\big|\pa_{x_2}\Psi(\pi, 0;g_1) - \pa_{x_2}\Psi(\pi, 0;g_2)\big| \leq C_M\|g_1-g_2\|_{L^2(\BT)}^{1/2},
\]
where $C_M$ is a universal constant depending only on $M$.

\begin{proof}
That $\pa_{x_2}\Psi(\pi,0;g)\in [0,M]$ has been proved in \eqref{eqn: upper bound for Psi_x2 at pi 0}.

To justify the claimed upper bound for $|\pa_{x_2}\Psi(\pi, 0;g_1) - \pa_{x_2}\Psi(\pi, 0;g_2)|$, we denote
\[
g_l(x) := \min \{g_1(x), g_2(x)\}, \quad
g_u(x) := \max \{g_1(x), g_2(x)\}.
\]
Then \begin{align*}
&\;\big|\pa_{x_2}\Psi(\pi, 0;g_1) - \pa_{x_2}\Psi(\pi, 0;g_2)\big|
\\
= &\;\frac{1}{4\pi} \left| \int_{-\pi}^{\pi}\int_{g_2(y_1)}^{g_1(y_1)}e^{-2\mu y_2} \left[\frac{\sinh y_2}{\cosh y_2 + \cos y_1} + 1\right] dy_2\, dy_1 \right| \\
\leq &\;\frac{1}{4\pi} \int_{-\pi}^{\pi}\int_{g_l(y_1)}^{g_u(y_1)} \left[\frac{\sinh y_2}{\cosh y_2 + \cos y_1} + 1\right] dy_2\, dy_1  \\
= &\;\frac{1}{2\pi} \int_{0}^{\pi}
\left[\int_{g_l(y_1)}^{g_u(y_1)} \f{2\sinh \f{y_2}{2}\cosh \f{y_2}{2}}{2\sinh^2\f{y_2}{2} + 2\cos^2 \f{y_1}{2}} \, dy_2\right] + g_u(y_1)-g_l(y_1)\, dy_1
\\
\leq &\;\frac{1}{2\pi} \int_{0}^{\pi}
\left[\int_{g_l(y_1)}^{g_u(y_1)} \f{\sinh \f{y_2}{2}\cosh \f{y_2}{2}}{\sinh^2\f{y_2}{2} + \cos^2 \f{y_1}{2}} \, dy_2\right]dy_1 +  C\|g_1 - g_2\|_{L^2},
\end{align*}
where $C$ is a universal constant.
Since $\f{a}{a^2+b^2}\leq\f{2}{a+b}$ for all $a,b>0$,
\begin{align*}
&\;\int_{g_l(y_1)}^{g_u(y_1)} \f{\sinh \f{y_2}{2}\cosh \f{y_2}{2}}{\sinh^2\f{y_2}{2} + \cos^2 \f{y_1}{2}} \, dy_2
\\
\leq &\;
2\int_{g_l(y_1)}^{g_u(y_1)} \f{\cosh \f{y_2}{2}}{\sinh\f{y_2}{2} + \cos \f{y_1}{2}} \, dy_2
= 4\ln\left(\frac{\sinh\frac{g_u(y_1)}{2} +\cos\frac{y_1}{2}}{\sinh\frac{g_l(y_1)}{2}+\cos\frac{y_1}{2}}\right)\\
\leq &\; 4\ln\left(1 + \frac{\sinh\frac{g_u(y_1)}{2} - \sinh\frac{g_l(y_1)}{2}}{\cos\frac{y_1}{2}}\right)\\
\leq &\; C_M \ln\left(1 + \frac{g_u(y_1)-g_l(y_1)}{\pi-y_1}\right),
\end{align*}
where $C_M$ depends on $M$.
Note that $\ln(1+z^2)\leq 2z\ln (1+z)$ for all $z\geq 0$.
Indeed, if $z>1$, $\ln(1+z^2)\leq 2\ln(1+z)\leq 2z\ln(1+z)$, while if $z\in [0,1]$, $\ln(1+z^2)\leq z^2 \leq z\int_0^z\f{2}{1+s}\,ds = 2z\ln (1+z)$.
Using this inequality, we can further deduce that
\begin{equation}\label{eqt: difference inequality}
\begin{split}
&\;\int_0^{\pi} \ln\left(1 + \frac{g_u(y_1)-g_l(y_1)}{\pi-y_1}\right) dy_1\\
\leq &\; 2\int_0^{\pi} \left(\frac{g_u(y_1)-g_l(y_1)}{\pi-y_1}\right)^{1/2}\ln\left(1 + \left(\frac{g_u(y_1)-g_l(y_1)}{\pi-y_1}\right)^{1/2}\right) dy_1\\
\leq &\; C\|g_1 - g_2\|_{L^2}^{1/2} \left(\int_0^{\pi} \left|\frac{1}{(\pi-y_1)^{1/2}}\ln\left(1 + \frac{(2M)^{1/2}}{(\pi-y_1)^{1/2}}\right)\right|^{4/3} dy_1\right)^{3/4}\\
\leq &\; C_M\|g_1 - g_2\|_{L^2}^{1/2}.
\end{split}
\end{equation}
Sequencing all estimates above and using the fact $g_1,g_2\leq M$ leads to the desired bound.
\end{proof}
\end{lem}

\begin{lem}\label{lem: continuity of F}
	Given any $g_1, g_2 \in \BM$ with $g_1(0), g_2(0) \leq M$, for all $x_1 \in [-\pi, \pi]$ and $x_2\in (0,M]$,
    \[
    |F(x_1, x_2; g_1) - F(x_1, x_2; g_2)|\leq  \frac{C_M}{x_2^{1/4}}\left(1 + \frac{\pi-x_1}{x_2}\right)^{3/4} \|g_1-g_2\|_{L^2(\BT)}^{1/2},
    \]
    where $C_M$ is a universal constant depending only on $M$.

\begin{proof}
Again let us write $g_l(x) := \min \{g_1(x), g_2(x)\}$, and $g_u(x) := \max \{g_1(x), g_2(x)\}$.
We may assume $x_1\in[0,\pi]$ by symmetry.
Then by Lemma \ref{lem: integral formula for Psi} and \eqref{eqn: def of F},
    \begin{align*}
         &\; |F(x_1, x_2; g_1) - F(x_1, x_2; g_2)| \\
         =&\; \frac{1}{4\pi x_2}\left|\int_{-\pi}^{\pi}\int_{g_2(y_1)}^{g_1(y_1)}e^{-2\mu y_2}\left[\ln\left(\frac{\cosh(x_2 - y_2) - \cos(x_1 - y_1)}{\cosh y_2 + \cos y_1}\right) - x_2\right]dy_2\, dy_1\right|\\
         \leq &\;\frac{C}{x_2}\int_{-\pi}^{\pi} \int_{g_l(y_1)}^{g_u(y_1)}\left|\ln\left(\frac{\cosh(x_2 - y_2) - \cos(x_1 - y_1)}{\cosh y_2 + \cos y_1}\right)\right|dy_2\, dy_1 + C\|g_1 - g_2\|_{L^2},
    \end{align*}
    where $C>0$ is a universal constant.
    Regarding the first term above, we can derive that
    \begin{align*}
        &\;\int_{-\pi}^{\pi}\int_{g_l(y_1)}^{g_u(y_1)}\left|\ln\left(\frac{\cosh(x_2 - y_2) - \cos(x_1 - y_1)}{\cosh y_2 + \cos y_1}\right)\right|dy_2\, dy_1 \\
        \leq &\;\int_{-\pi}^{\pi}\int_{g_l(y_1)}^{g_u(y_1)}\ln\left(1 + \frac{|\cosh(x_2 - y_2) - \cos(x_1 - y_1)-(\cosh y_2 + \cos y_1)|}{\min\{\cosh y_2 + \cos y_1,\cosh(x_2 - y_2) - \cos(x_1 - y_1)\}}\right)dy_2\, dy_1 \\
        \leq &\;\int_{-\pi}^{\pi}\int_{g_l(y_1)}^{g_u(y_1)}\ln\left(1 + \frac{|\cosh(x_2 - y_2) - \cos(x_1 - y_1)-(\cosh y_2 + \cos y_1)|}{\cosh y_2 + \cos y_1}\right)dy_2\, dy_1\\
        &\; + \int_{-\pi}^{\pi}\int_{g_l(y_1)}^{g_u(y_1)}\ln\left(1 + \frac{|\cosh(x_2 - y_2) - \cos(x_1 - y_1)-(\cosh y_2 + \cos y_1)|}{\cosh(x_2 - y_2) - \cos(x_1 - y_1)}\right)dy_2\, dy_1\\
        =: &\; I_1 + I_2.
    \end{align*}

    We first handle $I_1$.
    Note that
    \begin{align*}
        &\; \cosh(x_2 - y_2) - \cos(x_1 - y_1)-(\cosh y_2 + \cos y_1) \\
        =&\; (\cosh x_2-1)\cosh y_2-(\cos x_1+1)\cos y_1 - \sinh x_2\sinh y_2 - \sin x_1\sin y_1.
    \end{align*}
    We hence have
    \begin{align*}
        I_1 \leq &\; \int_{-\pi}^{\pi}\int_{g_l(y_1)}^{g_u(y_1)}\ln\left(1 + \frac{(\cosh x_2-1)\cosh y_2 + (\cos x_1+1)|\cos y_1|}{\cosh y_2 + \cos y_1}\right)  dy_2\, dy_1
        \\
        &\;+ \int_{-\pi}^{\pi}\int_{g_l(y_1)}^{g_u(y_1)} \ln\left(1 + \frac{\sinh x_2\sinh y_2 + \sin x_1 |\sin y_1|}{\cosh y_2 + \cos y_1}\right) dy_2\, dy_1
        \\
        =:  &\; I_{1,1} + I_{1,2}.
\end{align*}
$I_{1,1}$ can be bounded as follows:
\begin{align*}
I_{1,1}
= &\;
2\int_0^{\pi}\int_{g_l(y_1)}^{g_u(y_1)}\ln\left(1 + \frac{2\sinh^2\frac{x_2}{2}\cosh y_2
+ 2\sin^2\frac{\pi-x_1}{2}|\cos y_1|}{2\sinh^2\frac{y_2}{2}+2\sin^2\frac{\pi-y_1}{2}}\right)dy_2\,dy_1
\\
\leq &\;
2\int_0^{\pi}\int_{g_l(y_1)}^{g_u(y_1)}\ln\left(1 + \frac{C_M\sinh^2\frac{x_2}{2}
+ \sin^2\frac{\pi-x_1}{2}}{\sin^2\frac{\pi-y_1}{2}}\right)dy_2\,dy_1
\\
= &\;
2\int_0^{\pi} \big(g_u(y_1)-g_l(y_1)\big)\ln\left(1 + \frac{C_M\sinh^2\frac{x_2}{2} + \sin^2\frac{\pi-x_1}{2}}
{\sin^2\frac{\pi-y_1}{2}}\right) dy_1.
\end{align*}
By the H\"{o}lder's inequality and the interpolation inequality,
\begin{align*}
I_{1,1}
\leq
&\;
C\|g_u-g_l\|_{L^4}
\left[\int_0^{\pi} \left|\ln\left(1 + \frac{C_M(x_2^2 + (\pi-x_1)^2)}{(\pi-y_1)^2}\right)\right|^{4/3} dy_1 \right]^{3/4}
\\
\leq
&\;
CM^{1/2}\|g_u-g_l\|_{L^2}^{1/2}
\left[\int_0^\infty \left|\ln\left(1 + \frac{1}{s^2}\right)\right|^{4/3} ds \cdot C_M\big(x_2^2 + (\pi-x_1)^2\big)^{1/2} \right]^{3/4}
\\
\leq &\;
C_M \|g_u-g_l\|_{L^2}^{1/2}
(x_2 + \pi -x_1)^{3/4},
\end{align*}
where $C_M$ depends on $M$.
On the other hand, using the facts that $\ln (1+z)\leq z$ for all $z\geq 0$ and that $\f{a}{a^2+b^2}\leq\f{2}{a+b}$ for all $a,b>0$, we can derive that
    \begin{align*}
        &\; I_{1,2} \\
        =
        &\; 2\int_0^{\pi}\int_{g_l(y_1)}^{g_u(y_1)} \ln\left(1+ \frac{\sinh x_2\cdot 2\sinh \f{y_2}{2}\cosh\f{y_2}{2}}{2\sinh^2\frac{y_2}{2}+2\sin^2\frac{\pi-y_1}{2}} + \frac{\sin(\pi-x_1)\cdot 2\sin\f{\pi-y_1}{2} \cos\f{\pi-y_1}{2}}{2\sinh^2\frac{y_2}{2}+2\sin^2\frac{\pi-y_1}{2}}\right) dy_2\,dy_1
        \\
        \leq
        &\; C\int_0^{\pi}\int_{g_l(y_1)}^{g_u(y_1)}
        \frac{\sinh x_2\cdot \cosh\f{y_2}{2}}{\sinh\frac{y_2}{2}+\sin\frac{\pi-y_1}{2}} + \frac{\sin(\pi-x_1)\cdot \cos\f{\pi-y_1}{2}}{\sinh\frac{y_2}{2}+\sin\frac{\pi-y_1}{2}}\, dy_2\,dy_1 \\
                                \leq
        &\; C\big(\sinh x_2 + \sin(\pi-x_1)\big) \int_0^{\pi}\int_{g_l(y_1)}^{g_u(y_1)} \frac{\cosh\f{y_2}{2}}{\sinh\frac{y_2}{2}+\sin\frac{\pi-y_1}{2}} \,dy_2\,dy_1
        \\
        =
        &\; C\big(\sinh x_2 + \sin(\pi-x_1)\big) \int_0^{\pi} 2\ln\left(\frac{\sinh\frac{g_u(y_1)}{2}+\sin\frac{\pi-y_1}{2}}{\sinh\frac{g_l(y_1)}{2}+\sin\frac{\pi-y_1}{2}} \right) dy_1
        \\
        \leq
        &\; C_M(x_2 + \pi-x_1) \int_0^{\pi} \ln\left(1+\frac{\sinh\frac{g_u(y_1)}{2}-\sinh\frac{g_l(y_1)}{2}}{\sin\frac{\pi-y_1}{2}} \right) dy_1\\
        \leq
        &\; C_M(x_2 + \pi-x_1) \int_0^{\pi} \ln\left(1+\frac{C_M(g_u(y_1)-g_l(y_1))}{\pi-y_1} \right) dy_1,
    \end{align*}
    where $C_M$ depends on $M$.
By virtue of \eqref{eqt: difference inequality}, \[
I_{1,2}
\leq  C_M(x_2 + \pi-x_1) \|g_1 - g_2\|_{L^2}^{1/2}.
\]
Putting all the estimates above together yields that
    \[
    I_1 \leq  I_{1,1}+ I_{1,2} \leq C_M(x_2 + \pi-x_1)^{3/4} \|g_1 - g_2\|_{L^2}^{1/2}.
    \]

To bound $I_2$, we introduce a change of variables $\tilde y_2 := x_2-y_2$ and $\tilde y_1 := \pi + x_1-y_1$, and rewrite $I_2$ as
    \begin{align*}
        &\; I_2 \\
        = &\; \int_{-\pi}^{\pi}\int_{g_l(y_1)}^{g_u(y_1)}\ln\left(1 + \frac{|\cosh(x_2 - y_2) - \cos(x_1 - y_1)-(\cosh y_2 + \cos y_1)|}{\cosh(x_2 - y_2) - \cos(x_1 - y_1)}\right)dy_2\, dy_1\\
        = &\; \int_{x_1}^{2\pi+x_1}\int_{x_2-g_u(\pi + x_1-\tilde y_1)}^{x_2-g_l(\pi + x_1-\tilde y_1)}\ln\left(1 + \frac{|\cosh \tilde{y}_2 + \cos \tilde{y}_1-(\cosh(x_2-\tilde{y_2}) - \cos(x_1-\tilde{y}_1))|}{\cosh \tilde{y}_2 + \cos \tilde {y}_1}\right) d\tilde {y}_2\, d\tilde y_1\\
        = &\; \int_{-\pi}^{\pi}\int_{x_2-g_u(\pi + x_1-\tilde y_1)}^{x_2-g_l(\pi + x_1-\tilde y_1)}\ln\left(1 + \frac{|\cosh \tilde{y}_2 + \cos \tilde{y}_1 -(\cosh(x_2-\tilde y_2) - \cos(x_1-\tilde y_1))|}{\cosh \tilde{y}_2 + \cos \tilde{y}_1}\right)d\tilde y_2\, d\tilde y_1.
    \end{align*}
    Therefore, we can bound $I_2$ in the same way as we did for $I_1$, which leads to
    \[
    I_2 \leq C_M(x_2 + \pi-x_1)^{3/4} \|g_1 - g_2\|_{L^2}^{1/2}.
    \]

    In summary, we obtain
    \begin{align*}
    |F(x_1, x_2; g_1) - F(x_1, x_2; g_2)|
         &\leq \frac{C}{x_2}(I_1 + I_2) + C\|g_1 - g_2\|_{L^2}\\
         &\leq \frac{C_M}{x_2^{1/4}}\left(1 + \frac{\pi-x_1}{x_2}\right)^{3/4} \|g_1 - g_2\|_{L^2}^{1/2}
    \end{align*}
    as desired.
    Here $C_M>0$ is a constant that only depends on $M$.
\end{proof}
\end{lem}

The next lemma will also be useful in showing the continuity of the mapping $\bfR$.

\begin{lem}
\label{lem: lower bound for F_2 x_2}
Let $\mu\in [0,\mu_0]$.
Assume $g\in \BM$ satisfies $g(0)\leq M$.
Let $h$ be given by \eqref{eqn: formula for h}.
Let $F_2$ be defined by \eqref{eqn: def of F_1 F_2} (also see \eqref{eqn: representation of Psi_2}).
Then for any $x_1\in [0,\pi]$ and $x_2 \in (0,M]$,
\[
-\pa_{x_2}F_2(x_1,x_2;g)
\geq \frac{C_M}{x_2^2}\cdot \min\left\{h\left(x_1-\f{x_2}{2M}\right),\,h(1)\right\},
\]
where $C_M>0$ only depends on $M$.

If $g\in \BM$ additionally satisfies $\lambda (\pi^2-x^2)\leq g(x) \leq \min \{\Lambda(\pi^2-x^2), M \}$ on $[-\pi,\pi]$ for some constants $\Lambda>\lambda>0$, then for any $x_1\in [0,\pi]$ and $x_2 \in (0,M]$,
\[
-\pa_{x_2}F_2(x_1,x_2;g)\geq C_{M,\lambda,\Lambda}\left(1+\frac{(\pi-x_1)^2}{x_2^2}\right),
\]
where $C_{M,\lambda,\Lambda}>0$ only depends on $M$, $\lambda$, and $\Lambda$.

\begin{proof}
We argue as in the proof of Lemma \ref{lem: monotonicity of P/x_2} and find that, for any $x_1\in [-\pi,\pi]$ and $x_2\in (0,M]$,
\begin{align*}
\pa_{x_2}\left[\f{1}{x_2}P(x_1,x_2)\right]
&\leq \f{x_2 - \sinh x_2 \cosh x_2  + (x_2-\tanh x_2)\cosh x_2}{2\pi x_2^2(\cosh x_2- \cos x_1)^2}
\\
&= \f{(x_2-\sinh x_2)(1+\cosh x_2)}{2\pi x_2^2(2\sinh^2\f{x_2}{2} + 2\sin^2\f{x_1}{2})^2}\\
&\leq -\f{C_Mx_2^3\cdot 1}{x_2^2(x_1^2+x_2^2)^2}
= -\f{C_Mx_2}{(x_1^2+x_2^2)^2},
\end{align*}
where $C_M$ depends on $M$. In the last line, we used the fact that for any $x_2\in [0,M]$,
\[
\sinh x_2 -x_2 \geq C_1x_2^3,
\quad
\sinh x_2 \leq C_2x_2,
\]
for some constants $C_1,C_2 > 0$ that depend on $M$.
Hence, by \eqref{eqn: representation of Psi_2},
\[
-\pa_{x_2}F_2(x_1,x_2)
= -\int_{-\pi}^{\pi} \f{\pa}{\pa x_2}\left[\f{1}{x_2}P(y_1,x_2)\right] h(x_1-y_1)\,dy_1
\geq \int_{-\pi}^{\pi} \f{C_Mx_2}{(y_1^2+x_2^2)^2} \cdot h(x_1-y_1)\,dy_1.
\]
For $x_1\in[0,\pi]$, $x_2\leq M$, and $y_1\in [x_2/(2M), x_2/M]$, we have $x_1-y_1\in[x_1-x_2/M,x_1-x_2/(2M)]\subset [-1,x_1-x_2/(2M)]$.
Recall that $h$ is even and is non-increasing on $[0,\pi]$ due to Lemma \ref{lem: monotonicity of h}, so $h(x_1-y_1)\geq \min\{h(x_1-x_2/(2M)),h(1)\}$.
Hence,
\begin{align*}
-\pa_{x_2}F_2(x_1,x_2)
&\geq C_M\int_{x_2/(2M)}^{x_2/M} \f{x_2}{(y_1^2+x_2^2)^2} \cdot h(x_1-y_1)\,dy_1\\
&\geq C_M \min\left\{h\left(x_1-\f{x_2}{2M}\right),\,h(1)\right\} \int_{x_2/(2M)}^{x_2/M} \f{x_2}{(y_1^2+x_2^2)^2}\,dy_1\\
&\geq \frac{C_M}{x_2^2}\cdot \min\left\{h\left(x_1-\f{x_2}{2M}\right),\,h(1)\right\},
\end{align*}
which proves the first claim.

Thanks to Lemma \ref{lem: control horizontal displacement of Psi}, for any $x\in[0,\pi]$,
\begin{align*}
    h(x) &\geq \f{1}{2\pi} \int_0^{g(0)} e^{-2\mu y_2}\sin g^{-1}(y_2) \cdot \ln\left(1+ \f{1+\cos x}{\cosh y_2 -1 + 2 \cos^2\f{ g^{-1}(y_2)}{2}}\right) dy_2\\
    &\geq \f{1}{2\pi} \int_0^{g(0)} e^{-2\mu M}\sin g^{-1}(y_2) \cdot \ln\left(1+ \f{2\cos^2\frac{x}{2}}{\cosh M + 1}\right) dy_2\\
    &\geq C_M (\pi-x)^2\int_0^{g(0)} \sin g^{-1}(y_2) \, dy_2.
         \end{align*}
Here we used the assumption $\mu \leq \mu_0 \leq \f12$.
If we additionally assume that $\lambda (\pi^2-x^2)\leq g(x) \leq \min\{\Lambda(\pi^2-x^2), M\}$ on $[-\pi,\pi]$ for some $\Lambda>\lambda>0$, then $\pi - y_2/(\pi\lambda) \leq g^{-1}(y_2) \leq \pi - y_2/(2\pi\Lambda)$ for $y_2\in[0,g(0)]$.
Hence, \[
\int_0^{g(0)} \sin g^{-1}(y_2)\, dy_2
\geq \int_0^{\lam \pi^2/2} \sin \left(\f{y_2}{2\pi \Lam}\right) dy_2
\geq C\int_0^{\lambda\pi^2/2} \f{y_2}{\Lambda}\, dy_2
\geq \frac{C\lambda^2}{\Lambda},
\]
which gives
\[
h(x)\geq C_{M,\lambda,\Lambda} (\pi-x)^2\mbox{ for all }x\in [0,\pi].
\]
Since $h$ is even on $[-\pi,\pi]$, this estimate also applies to all $x\in [-\f{\pi}{2},\pi]$.
Hence, for any $x_1\in[0,\pi]$ and $x_2\leq M$,
\[
  h(1) \geq C_{M,\lambda,\Lambda}\geq C_{M,\lambda,\Lambda}\left(x_2^2 + (\pi-x_1)^2\right),
\]
and since $x_1-x_2/(2M)\geq -\f12$ in this case, it also holds that
\[
h\left(x_1-\f{x_2}{2M}\right) \geq C_{M,\lambda,\Lambda}\left(\frac{x_2}{2M} + \pi-x_1\right)^2\geq C_{M,\lambda,\Lambda}\left(x_2^2 + (\pi-x_1)^2\right).
\]
As a result, \begin{align*}
-\pa_{x_2}F_2(x_1,x_2)
&\geq \frac{C_M}{x_2^2}\cdot \min\left\{h\left(x_1-\f{x_2}{2M}\right),\,h(1)\right\}\geq C_{M,\lambda,\Lambda}\left(1+\frac{(\pi-x_1)^2}{x_2^2}\right),
\end{align*}
as desired.
\end{proof}
\end{lem}

With the preparation above, we can now show that, under suitable assumptions, $\bfR(g)$ is $C^{2/5}$-continuous with respect to $g$ in the $L^2$-topology.

\begin{prop}\label{prop: continuity of R}
Let $M$ and $\mu_0$ be given in Proposition \ref{prop: a priori upper bound general m}, and fix $\mu\in[0,\mu_0]$.
Let $\lambda$ and $\Lambda$ be two constants such that $\Lambda > \lambda >0$ and $\pi^2\lambda \leq M$. Then, for any $g_1, g_2 \in \BM$ that satisfy $\lambda (\pi^2-x^2)\leq g_i(x) \leq \min \{\Lambda(\pi^2-x^2), M\}$ on $[-\pi,\pi]$ $(i=1,2)$,
\[
\|\bfR(g_1)-\bfR(g_2)\|_{L^2(\BT)}\le C_{M,\lambda,\Lambda}\|g_1-g_2\|_{L^2(\BT)}^{2/5},
\]
where the constant $C_{M,\lambda,\Lambda}>0$ only depends on $M$, $\lambda$, and $\Lambda$.

\begin{proof}
We only prove for $\mu\in (0,\mu_0]$.
The case $\mu = 0$ can be justified similarly up to minor modification.

    For $g\in \BM$, we introduce an auxiliary function (also see \eqref{eqn: def of H as F minus RHS in the implicit mapping})
    \[
    H(x_1,x_2;g) := F(x_1,x_2;g) - \pa_{x_2}\Psi(\pi,0;g)\cdot \frac{1-e^{-2\mu g(x_1)}}{2\mu g(x_1)}.
    \]
    By the definition of $\bfR$ (see \eqref{eqn: def of R(g)}), $H(x,\bfR(g_i)(x);g_i) \equiv 0$ on $[-\pi,\pi]$ $(i=1,2)$.

    Fix $x\in [0,\pi]$.
    Without loss of generality, we assume $\bfR(g_1)(x)\geq \bfR(g_2)(x)$.
    Let us first derive upper and lower bounds for $|H(x,\bfR(g_1)(x);g_2)|$.
    Thanks to Lemma \ref{lem: bound for Psi_x_2},
    \begin{align*}
         &\;\left|\pa_{x_2}\Psi(\pi,0;g_1)\cdot \f{1-e^{-2\mu g_1(x)}}{2\mu g_1(x)} - \pa_{x_2}\Psi(\pi,0;g_2)\cdot \f{1-e^{-2\mu g_2(x)}}{2\mu g_2(x)}\right|\\
         \leq &\;\big|\pa_{x_2}\Psi(\pi,0;g_1)- \pa_{x_2}\Psi(\pi,0;g_2)\big|\cdot  \f{1-e^{-2\mu g_1(x)}}{2\mu g_1(x)}\\
         &\;+\pa_{x_2}\Psi(\pi,0;g_2)\left|\frac{1-e^{-2\mu g_1(x)}}{2\mu g_1(x)} - \frac{1-e^{-2\mu g_2(x)}}{2\mu g_2(x)}\right|\\
         \leq &\;\big|\pa_{x_2}\Psi(\pi,0;g_1)- \pa_{x_2}\Psi(\pi,0;g_2)\big|  + C_M\mu|g_1(x)-g_2(x)|\\
         \leq &\;C_M\left(\|g_1-g_2\|_{L^2}^{1/2} + |g_1(x)-g_2(x)|\right).
    \end{align*}
    We also apply Lemma $\ref{lem: continuity of F}$ with $(x_1,x_2) = (x,\bfR(g_1)(x))$ to obtain that
    \begin{align*}
     &\; \big|F(x,\bfR(g_1)(x);g_1) - F(x,\bfR(g_1)(x);g_2)\big|
     \\
     \leq &\; \frac{C_M}{\bfR(g_1)(x)^{1/4}}\left(1 + \frac{\pi-x}{\bfR(g_1)(x)}\right)^{3/4} \|g_1-g_2\|_{L^2}^{1/2}\\
     \leq &\;\frac{C_M}{\bfR(g_1)(x)^{1/4}}\left(1 + \frac{(\pi-x)^2}{\bfR(g_1)(x)^2}\right) \|g_1-g_2\|_{L^2}^{1/2}.
    \end{align*}
    In the last line, we used the Young's inequality.
    Hence,     $|H(x,\bfR(g_1)(x);g_2)|$ can be bounded from above as follows:
    \begin{equation}\label{eqt: continuity step 2}
    \begin{split}
     &\; |H(x,\bfR(g_1)(x);g_2)|\\
     = &\; \big|H(x,\bfR(g_1)(x);g_1) - H(x,\bfR(g_1)(x);g_2)\big| \\
     \leq &\; \big|F(x,\bfR(g_1)(x);g_1) - F(x,\bfR(g_1)(x);g_2)\big| \\
     &\; + \left|\pa_{x_2}\Psi(\pi,0;g_1)\cdot \frac{1-e^{-2\mu g_1(x)}}{2\mu g_1(x)} - \pa_{x_2}\Psi(\pi,0;g_2)\cdot \frac{1-e^{-2\mu g_2(x)}}{2\mu g_2(x)}\right| \\
     \leq &\; \frac{C_M}{\bfR(g_1)(x)^{1/4}}\left(1 + \frac{(\pi-x)^2}{\bfR(g_1)(x)^2}\right) \left(\|g_1-g_2\|_{L^2}^{1/2} + |g_1(x)-g_2(x)|\right).
     \end{split}
    \end{equation}
    Here we used the fact that $\bfR(g_1)(x)\leq M$ due to Proposition \ref{prop: a priori upper bound general m} and the monotonicity of $\bfR(g_1)$.
    On the other hand, using the mean value theorem,         \begin{align*}
     |H(x,\bfR(g_1)(x);g_2)| &= \big|H(x,\bfR(g_1)(x);g_2) - H(x,\bfR(g_2)(x);g_2)\big| \\
     &= \big|F(x,\bfR(g_1)(x);g_2) - F(x,\bfR(g_2)(x);g_2)\big|\\
     &= |\pa_{x_2}F(x,\xi;g_2)|\cdot |\bfR(g_1)(x)-\bfR(g_2)(x)|,
     \end{align*}
    where $\xi\in[\bfR(g_2)(x),\bfR(g_1)(x)]$.
    By Lemma \ref{lem: monotonicity of F_1}, Lemma \ref{lem: monotonicity of F_2}, and Lemma \ref{lem: lower bound for F_2 x_2},
            \begin{align*}
     |H(x,\bfR(g_1)(x);g_2)|      &\geq |\pa_{x_2}F_2(x,\xi;g_2)|\cdot |\bfR(g_1)(x)-\bfR(g_2)(x)|\\
     &\geq C_{M,\lambda,\Lambda}\left(1+\frac{(\pi-x)^2}{\xi^2}\right)|\bfR(g_1)(x)-\bfR(g_2)(x)|\\
     &\geq C_{M,\lambda,\Lambda}\left(1+\frac{(\pi-x)^2}{\bfR(g_1)(x)^2}\right)|\bfR(g_1)(x)-\bfR(g_2)(x)|,
     \end{align*}
            where $C_{M,\lambda,\Lambda}>0$ depends on $M$, $\Lambda$, and $\lam$.
    Combining this with     \eqref{eqt: continuity step 2},     we obtain that
    \[
        |\bfR(g_1)(x)-\bfR(g_2)(x)| \leq \frac{C_{M,\lambda,\Lambda}}{\bfR(g_1)(x)^{1/4}}\left(\|g_1-g_2\|_{L^2}^{1/2} + |g_1(x)-g_2(x)|\right).
    \]

    We have proved that, for any $x\in[-\pi,\pi]$,
    \begin{align*}
     |\bfR(g_1)(x)-\bfR(g_2)(x)|^{5/4} &\leq |\bfR(g_1)(x)-\bfR(g_2)(x)|\cdot \max\big\{\bfR(g_1)(x),\bfR(g_2)(x)\big\}^{1/4}\\
     &\leq C_{M,\lambda,\Lambda}\left(\|g_1-g_2\|_{L^2}^{1/2} + |g_1(x)-g_2(x)|\right).
    \end{align*}
    This readily implies
    \[
        \|\bfR(g_1)-\bfR(g_2)\|_{L^2} \leq C_{M,\lambda,\Lambda}\left(\|g_1-g_2\|_{L^2}^{2/5} + \|g_1-g_2\|_{L^{8/5}}^{4/5}\right)\leq C_{M,\lambda,\Lambda}\|g_1-g_2\|_{L^2}^{2/5},
    \]
    as desired.
 \end{proof}
\end{prop}

\subsection{Proof of the existence of the fixed point}
\label{sec: proof of existence of the fixed point}

Let $M$ and $\mu_0$ be given in Proposition \ref{prop: a priori upper bound general m}.
Consider $\mu\in[0,\mu_0]\cap [0,\f{3}{4M}]$. Define (cf.\;\eqref{eqn: function set tilde M_0})
\begin{equation}\label{eqt: def of function set D_0}
\begin{split}
\mathbb{W}_0:= \big\{ g\in L^2(\BT):
&\;
\text{$g$ is even on }[-\pi,\pi], \;\text{$g$ is non-increasing on $[0,\pi]$},\\
&\;\lambda (\pi^2-x^2)\leq g(x)\leq \min\{\Lambda  (\pi^2-x^2), M\} \text{ on }\BT\big\}.
\end{split}
\end{equation}
Here we take
\[
\Lambda:=\Lambda_*,\quad \lambda:=\lambda_*,
\]
which are determined in Remark \ref{rmk: universality of large Lambda} and Remark \ref{rmk: universality of small lambda}, respectively (also see Proposition \ref{prop: upper barrier new} and Proposition \ref{prop: lower barrier}); also recall that $\lam_* \pi^2 \leq M \leq \Lam_* \pi^2$.
Note that $\BW_0$ is equipped with the standard $L^2(\BT)$-topology, so $g\in \BW_0$ being non-increasing on $[0,\pi]$ should be understood in the following sense: there exists a non-increasing function $\tilde{g}$ defined pointwise on $[0,\pi]$ such that $g = \tilde{g}$ almost everywhere on $[0,\pi]$.
Apparently, $\mathbb{W}_0$ is convex and closed in  $L^2(\BT)$.
Moreover, since any set of non-increasing $L^2$-functions on $[0,\pi]$ with a uniform $L^\infty$-bound is pre-compact in $L^2([0,\pi])$,  $\mathbb{W}_0$ is also compact in the $L^2(\BT)$-topology.

Since the elements in $\BW_0$ are $L^2$-functions on $\BT$, their pointwise values are not well-defined.
For convenience, let us define for $g\in \BW_0$ that
\[
\CT(g)(x):=
\begin{cases}
\esssup_{y\in [x,\pi]} g(y)& \mbox{if }x\in [0,\pi),\\
\esssup_{y\in [-\pi,x]} g(y)& \mbox{if }x\in (-\pi,0],\\
0& \mbox{if }x=\pm \pi.
\end{cases}
\]
Here $\CT(g)$ is a function well-defined pointwise on $\BT$.
Using the assumptions on $g\in \BW_0$, it is not difficult to verify the following claims: \begin{itemize}
\item $\CT(g) = g$ almost everywhere on $\BT$.

Indeed, by assumption, there exists $\tilde{g}:\BT\to \BR$ such that $\tilde{g}$ is even on $[-\pi,\pi]$, non-increasing on $[0,\pi]$, and $g = \tilde{g}$ almost everywhere on $\BT$.
Thanks to the monotonicity, $\tilde{g}$ is continuous almost everywhere on $\BT$.
Since $\CT(\tilde{g})(x) = \tilde{g}(x)$ holds at the continuous point of $\tilde{g}$, it holds almost everywhere on $\BT$.
On the other hand, such $\tilde{g}$ can also be understood as an element of $\BW_0$; in fact, $\tilde{g} = g$ in the $L^2$-sense, so $\CT(\tilde{g})(x)\equiv \CT(g)(x)$ on $\BT$.
Therefore, up to a measure-zero set, it holds that $\CT(g)(x) = \CT(\tilde{g})(x) = \tilde{g}(x) = g(x)$.

\item $\CT(g) \in \BM$ and it is lower semi-continuous on $\BT$.

That $\CT(g)\in \BM$ follows from the definition.
To show the lower semi-continuity, we first verify that, for $x\in (0,\pi)$, $\CT(g)(x) \leq \liminf_{x'\to x}\CT(g)(x')$.
In fact, thanks to the monotonicity, $\CT(g)(x)\leq \lim_{x'\to x^-}\CT(g)(x')$, and
\[
\lim_{x'\to x^+}\CT(g)(x') = \esssup_{y\in (x,\pi]} g(y) = \esssup_{y\in [x,\pi]} g(y) = \CT(g)(x).
\]
The lower semi-continuity at $x \in (-\pi,0)$, $x=0$ and $x=\pm \pi$ are then straightforward.

\item $\CT$ is injective from $\BW_0$ to $\BM$.

If not, suppose $\CT(g_1) = \CT(g_2)$ for some $g_1\neq g_2$ in $\BW_0$.
Then we take $\tilde{g}_1$ and $\tilde{g}_2$ that coincide with $g_1$ and $g_2$ almost everywhere on $\BT$ respectively, which are non-increasing on $[0,\pi]$.
By the argument above, $\CT(\tilde{g}_1) = \CT(g_1) = \CT(g_2) = \CT(\tilde{g}_2)$ pointwise, and thus $g_1=\tilde{g}_1 = \CT(\tilde{g}_1) = \CT(\tilde{g}_2) = \tilde{g}_2=g_2$ almost everywhere.
This contradicts the fact that $g_1 \neq g_2$ in $\BW_0$.
\end{itemize}
We thus call $\CT(g)$ the lower semi-continuous representative of $g$.
We also define
\begin{equation}\label{eqt: def of function set D}
\begin{split}
\mathbb{W}:= &\;\{ \CT(g):\; g\in \BW_0\}
\\
= &\; \big\{ g\in\BM:\; \mbox{$g$ is lower semi-continuous on }\BT,\\
&\;\qquad \qquad \;\;
\lambda (\pi^2-x^2)\leq g(x)\leq \min\{\Lambda  (\pi^2-x^2), M\}\text{ on $\BT$}\big\},
\end{split}
\end{equation}
where $M$, $\Lam$, and $\lam$ are the same as in the definition of $\BW_0$ (see the note that follows \eqref{eqt: def of function set D_0}).
$\CT$ is thus a bijection from $\BW_0$ to $\BW\subset \BM$.
We denote its inverse on $\BW$ by $\CI: \BW\to \BW_0$, which identifies each function in $\BW$ as an $L^2$-function in a natural way.
Therefore, it holds that $\CI(\CT(g))=g$ for all $g\in \BW_0$, and for any $g_\dag \in \BW$, $\CT(\CI (g_\dag))(x) \equiv g_\dag(x)$ on $\BT$.

Now we are ready to prove the existence of the fixed points of $\bfR$ in $\BW$.
\begin{prop}\label{prop: existence of fixed point}
Let $M$ and $\mu_0$ be given by Proposition \ref{prop: a priori upper bound general m}.
Let $\BW$ be defined as in \eqref{eqt: def of function set D}.
For any $\mu\in[0,\mu_0]\cap [0,\f{3}{4M}]$, the mapping $\bfR$ admits a fixed point $g_*\in\mathbb{W}$.
Here and in what follows, we call a function $g\in \BW$ a fixed point of $\bfR$ if $\bfR(g)(x)= g(x)$ for all $x\in [-\pi,\pi]$.

Moreover, for any fixed point $g$ of $\bfR$ in $\BW$, $g$ is strictly decreasing on $[0,\pi]$, and it satisfies $\lambda (\pi^2-x^2)< g(x)< \min\{\Lambda  (\pi^2-x^2), M\}$ for all $x\in(-\pi,\pi)$.

\begin{rmk}
The range of $\mu$ stems from Proposition \ref{prop: a priori upper bound general m}, Proposition \ref{prop: upper barrier new}, and also Remark \ref{rmk: universality of large Lambda}.
It is not the largest range, but a convenient one that leads to a uniform-in-$\mu$ choice of $\Lam_*$ in the definitions of $\BW_0$ and $\BW$; see Remark \ref{rmk: universality of large Lambda}.
In fact, $\f{3}{4M}$ can be replaced by any fixed number that is smaller than $\f{c_*}{2M}$, where $c_*$ was defined in Proposition \ref{prop: upper barrier new}.
\end{rmk}

\begin{proof}
Since $\BW\subset \BM$, by Proposition \ref{prop: implicit mapping}, $\bfR$ maps $\mathbb{W}$ into $\BM$.
Denote $u(x) := \pi^2-x^2$.
Let us introduce an auxiliary mapping $\hat \bfR$ defined on $\mathbb{W}$ as a truncated version of $\bfR$:
\beq
\begin{split}
\hat\bfR(g)(x):= &\;
\begin{cases}
\bfR (g)(x), &\text{if $\lambda u(x)\le \bfR (g)(x)\le \min\{\Lambda u(x), M\}$},\\
\lambda u(x), &\text{if } \bfR (g)(x)<\lambda u(x),\\
\min \{\Lambda u(x), M\}, &\text{if } \bfR (g)(x) > \min \{\Lambda u(x), M\}.
\end{cases}
\end{split}
\label{eqn: truncated mapping hat R}
\eeq
    We claim that $\hat \bfR$ maps $\mathbb{W}$ into itself.
    It is not difficult to derive the evenness and monotonicity of $\hat{\bfR}(g)$ from that of $\bfR(g)$, $\lam u(x)$, and $\min \{\Lambda u(x), M\}$.
    It suffices to verify the lower semi-continuity of $\hat{\bfR}(g)$.
    By the definition of $\bfR$,
    \[
    0 = \lim_{x'\to x} F\big(x',\bfR(g)(x');g\big) - \partial_{x_2}\Psi(\pi,0;g)\cdot \frac{1-e^{-2\mu g(x')}}{2\mu g(x')}.
    \]
    Since $z\mapsto \f{1-e^{-z}}{z}$ is decreasing for $z>0$ and $g\in \BW$ is lower semi-continuous,
    \begin{align*}
    F(x,\bfR(g)(x);g) = &\; \partial_{x_2}\Psi(\pi,0;g)\cdot \frac{1-e^{-2\mu g(x)}}{2\mu g(x)}\\
    \geq &\; \limsup_{x'\to x}\partial_{x_2}\Psi(\pi,0;g)\cdot \frac{1-e^{-2\mu g(x')}}{2\mu g(x')}
    = \limsup_{x'\to x} F\big(x',\bfR(g)(x');g\big).
    \end{align*}
    By the continuity and monotonicity of $F$, we conclude that
    \[
    \bfR(g)(x)\leq \liminf_{x'\to x} \bfR(g)(x'),
    \]
    which implies that $\bfR(g)$ is lower semi-continuous.

    Take an arbitrary $x\in \BT$.
    If $\bfR(g)(x)\leq \min \{\Lambda u(x), M\}$,
    \begin{align*}
    \hat{\bfR}(g)(x) = &\; \max\{\bfR(g)(x),\lam u(x)\}
    \\
    \leq &\; \max\left\{\liminf_{x'\to x}\bfR(g)(x'), \, \liminf_{x'\to x}\lam u(x')\right\}
                \leq \liminf_{x'\to x} \max\{\bfR(g)(x'),\lam u(x')\}.
    \end{align*}
    Since it also trivially holds that
    \[
    \hat{\bfR}(g)(x) \leq \min \{\Lambda u(x), M\} = \liminf_{x'\to x} \min \{\Lambda u(x'), M\},
    \]
    we obtain that
    \[
    \hat{\bfR}(g)(x)
    \leq \liminf_{x'\to x} \min\left\{ \max\{\bfR(g)(x'),\lam u(x')\},\, \min \{\Lambda u(x'), M\}\right\}
    = \liminf_{x'\to x} \hat{\bfR}(g)(x').
    \]
    If otherwise $\bfR(g)(x) > \min \{\Lambda u(x), M\}$, by the lower semi-continuity of $\bfR(g)$, there exists $\d>0$ such that $\bfR(g)(x') > \min \{\Lambda u(x'), M\}$ for all $x'\in (x-\d,x+\d)$, which implies that
    $\hat{\bfR}(g)(x') = \min \{\Lambda u(x'), M\}$ for all $x'\in (x-\d,x+\d)$.
    Therefore,
    \[
    \hat{\bfR}(g)(x)
    = \min \{\Lambda u(x), M\} = \lim_{x'\to x} \min \{\Lambda u(x'), M\}
    = \liminf_{x'\to x} \hat{\bfR}(g)(x').
    \]
    This proves the lower semi-continuity of $\hat{\bfR}(g)$.

    Let $\mathbf{K}:= \CI\circ \hat{\bfR}\circ \CT$, which is a mapping from $\BW_0$ to itself.
    For any $g_1,g_2\in \BW_0$, by the definition of $\hat \bfR$, for all $x\in \BT$,
    \[
    \big|\hat \bfR(\CT(g_1))(x) - \hat \bfR(\CT(g_2))(x)\big|
    \leq |\bfR(\CT(g_1))(x) - \bfR(\CT(g_2))(x)|.
    \]
    By Proposition \ref{prop: continuity of R},
    \begin{align*}
     \big\|\hat \bfR(\CT(g_1)) - \hat \bfR(\CT(g_2))\big\|_{L^2}
     \leq &\; \big\|\bfR(\CT(g_1)) - \bfR(\CT(g_2))\big\|_{L^2} \\
     \leq &\; C_{M,\lambda, \Lambda} \|\CT(g_1)-\CT(g_2)\|_{L^2}^{2/5}
     = C_{M,\lambda, \Lambda} \|g_1-g_2\|_{L^2}^{2/5}.
    \end{align*}
    This gives
    \[
    \|\mathbf{K}(g_1) - \mathbf{K}(g_2)\|_{L^2(\BT)}  \leq C_{M,\lambda, \Lambda} \|g_1-g_2\|_{L^2(\BT)}^{2/5},
    \]
    which implies that $\mathbf{K}$ is continuous on $\mathbb{W}_0$ in the standard $L^2(\BT)$-topology.
    Since $\mathbb{W}_0$ is a convex, closed, and compact subset of $L^2(\BT)$, the Schauder fixed-point theorem implies that  $\mathbf{K}$  has a fixed point $\hat{g}_*\in\mathbb{W}_0$ in the sense that $\|\mathbf{K}(\hat{g}_*) - \hat{g}_*\|_{L^2}=0$.
Let $g_* := \CT(\hat{g}_*)\in \BW$.
Then for any $x\in \BT$,
\[
g_*(x) = \CT(\hat{g}_*)(x) = \CT\big(\mathbf{K}(\hat{g}_*)\big)(x) = \CT\circ \CI \circ \hat{\bfR}\circ \CT(\hat{g}_*)(x) = \hat{\bfR}(g_*)(x).
\]
By the definition of $\hat{\bfR}$, for all $x\in [-\pi,\pi]$,
\beq
\lam u(x)\leq g_*(x)\leq \min\{\Lam u(x),M\}.
\label{eqn: lower and upper bound for the fixed-point prelim}
\eeq

    Next, we show that $\bfR(g_*)(x) = \hat \bfR(g_*)(x)$ for all $x\in \BT$.
    It suffices to prove this for $x\in (-\pi,\pi)$, as we always have $\bfR(g_*)(\pm\pi) = \hat \bfR(g_*)(\pm\pi) = 0$ by definition.
        \begin{enumerate}[(i)]
    \item
    If $\bfR(g_*)(x_*) > \hat \bfR(g_*)(x_*)$ at some $x_*\in (-\pi,\pi)$, then by the definition of $\hat \bfR$,
    \[
    \bfR(g_*)(x_*) > \min\{\Lambda u(x_*),M\} = \hat \bfR(g_*)(x_*) = g_*(x_*).
    \]
    This together with \eqref{eqn: lower and upper bound for the fixed-point prelim}, Proposition \ref{prop: a priori upper bound general m}, and Proposition \ref{prop: upper barrier new} implies that     $\bfR(g_*)(x_*)<g_*(x_*)$, which is a contradiction.
    \item
    Similarly, if $\bfR(g_*)(x_*) < \hat \bfR(g_*)(x_*) $ at some $x_*\in (-\pi,\pi)$, then we have
    \[
    \bfR(g_*)(x_*) < \lambda u(x_*) = \hat \bfR(g_*)(x_*) = g_*(x_*).
    \]
    In this case, \eqref{eqn: lower and upper bound for the fixed-point prelim} and Proposition \ref{prop: lower barrier} give that $\bfR(g_*)(x_*)>g_*(x_*)$, which is again a contradiction.
    \end{enumerate}
    Therefore, we conclude that $g_*(x) = \bfR(g_*)(x) = \hat \bfR(g_*)(x)$ for all $x\in \BT$.
    This proves the desired existence.

    In fact, the above argument also shows that, if $g\in \BW$ is a fixed point of $\bfR$ in the sense that $\bfR(g)(x) = g(x)$ for all $\BT$, then     \[
    \lambda u(x)< g(x)< \min\{\Lambda u(x), M\},
    \]
    for all $x\in(-\pi,\pi)$.
    Finally, Proposition \ref{prop: implicit mapping} implies that $g(x)= \bfR(g)(x)$ is strictly decreasing on $[0,\pi]$.
\end{proof}
\end{prop}

\section{Regularity of the Fixed Point}
\label{sec: regularity}
In view of \eqref{eqn: def of R(g)}, every fixed point $g\in \mathbb{W}\subset\BM$ of $\bfR$ provides a solution to the problem \eqref{eqn: constraint for the boundary curve in Phi}, but it does not necessarily satisfy the original condition \eqref{eqn: constraint for Phi along the image of the patch boundary} (or equivalently, \eqref{eqn: constraint satisfied by phi along the patch boundary final}) for the uniformly rotating vortex patches.
In this section, we will show that, for $\mu$ in a possibly smaller range, every fixed point of $\bfR$ in $\mathbb{W}$ is continuous in $\BT$ and enjoys even higher regularity, so that it also makes \eqref{eqn: constraint for Phi along the image of the patch boundary} hold because of Remark \ref{rmk: equivalence of the conditions along boundary}.
In addition, we will also study the behavior of the fixed point $g$ near the end-points $\pm \pi$, which eventually leads to the conclusion that the resulting rotating vortex patch has 90-degree corners.

The main results of this section are Proposition \ref{prop: first regularity result} and Proposition \ref{prop: g slope at pi}.
At the end of this section, we will present the proof of Theorem \ref{thm: main existence theorem}.

\subsection{Non-degeneracy estimates}

For $\mu>0$, we define
\beq
F_\dag(x_1,x_2) = F_\dag(x_1,x_2;g):= F(x_1,x_2;g) -\pa_{x_2}\Psi(\pi,0;g)\cdot \f{1-e^{-2\mu x_2}}{2\mu x_2},
\label{eqn: def of F_dag}
\eeq
while for $\mu=0$, we let
\[
F_\dag(x_1,x_2) = F_\dag(x_1,x_2;g):= F(x_1,x_2;g) -\pa_{x_2}\Psi(\pi,0;g).
\]
Note that this is different from $H$ defined in \eqref{eqn: def of H as F minus RHS in the implicit mapping}.
If $g\in \BW$ is a fixed point of $\bfR$, then for each $x_1\in [-\pi,\pi]$, $x_2 = g(x_1)$ should be a root of $x_2\mapsto F_\dag(x_1,x_2;g)$.
In order to rule out potential discontinuity of $g$, in view of the implicit function theorem, it is tempting to show that at any root of the function $x_2\mapsto F_\dag(x_1,x_2;g)$, its $x_2$-derivative should be negative; in fact, it suffices to consider those points additionally satisfying $x_1 = g^{-1}(x_2)$.
This motivates Proposition \ref{prop: non-degeneracy along level set} and Proposition \ref{prop: non-degeneracy along level set x_1 greater than pi over 2} below, which provide quantitative estimates for the non-degeneracy of $F_\dag(x_1,x_2;g)$ in the $x_2$-direction
in the case $x_1\in [0,\f{\pi}{2}]$ and $x_1\in [\f{\pi}{2},\pi)$, respectively.

In the rest of this section, for simplicity, we shall focus on the case $\mu>0$.
Results for $\mu = 0$ will be stated in accompanying remarks, which can be justified similarly up to minor modification.

\begin{prop}
\label{prop: non-degeneracy along level set}
Assume that $g\in \BM$ satisfies $g(0)\leq M$.
Suppose $(x_1,x_2)$ satisfies $x_2\in (0,g(0))$, $x_1 = g^{-1}(x_2)$, and $x_1\in [0,\f{\pi}{2}]$.
Then it holds that
\beq
\pa_{x_2}\Psi(x_1,x_2;g) - \ka_* F(x_1,x_2;g)
\leq
-\d(1-\ka_*)\cdot
\pa_{x_2} \Psi(\pi,0;g)\cdot \f{1-e^{-2\mu x_2}}{2\mu x_2},
\label{eqn: pa_x2 Psi less than kappa_* F quantitative}
\eeq
where
\beq
\ka_*:=\f{2\mu x_2}{e^{2\mu x_2}-1},
\label{eqn: def of kappa_*}
\eeq
and
\beq
\d:=\min\left\{\f{1}{4(e^{2\mu M}-1)},\, \left(\f{x_1}{2\pi}\cdot \f{e^{-2\mu M}}{e^{2\mu x_2}-1-2\mu x_2}\right)^{1/2}\right\}.
\label{eqn: value of delta}
\eeq

If $(x_1,x_2)$ additionally satisfies
\beq
F(x_1,x_2;g)= \pa_{x_2}\Psi(\pi,0;g)\cdot \f{1-e^{-2\mu x_2}}{2\mu x_2},
\label{eqn: x_1 x_2 lies on the level set}
\eeq
then it further holds that
\beq
\pa_{x_2}\left[ F(x_1,x_2;g) -\pa_{x_2}\Psi(\pi,0;g)\cdot\f{1-e^{-2\mu x_2}}{2\mu x_2}\right]\leq
\d\cdot \pa_{x_2}\Psi(\pi,0;g)\cdot \f{d}{dx_2}\left[\f{1-e^{-2\mu x_2}}{2\mu x_2}\right].
\label{eqn: non-degeneracy along the level set}
\eeq

\begin{rmk}
\label{rmk: equivalent formulation of the non-degeneracy}
Given \eqref{eqn: pa_x2 Psi less than kappa_* F quantitative}, the proof of \eqref{eqn: non-degeneracy along the level set} under the additional condition \eqref{eqn: x_1 x_2 lies on the level set} is straightforward.
First note that
\[
\pa_{x_2} F(x_1,x_2;g)
= \f{1}{x_2}\big[\pa_{x_2}\Psi(x_1,x_2;g) - F(x_1,x_2;g)\big].
\]
With $\ka_*$ defined in \eqref{eqn: def of kappa_*},
\begin{align*}
\pa_{x_2}\Psi(\pi,0;g)\cdot \f{d}{dx_2}\left[\f{1-e^{-2\mu x_2}}{2\mu x_2}\right]
= &\; -\f{1}{x_2} \left(1-\f{2\mu x_2}{e^{2\mu x_2}-1}\right)\cdot \pa_{x_2} \Psi(\pi,0;g)\cdot \f{1-e^{-2\mu x_2}}{2\mu x_2}
\\
= &\; -\f{1}{x_2} (1-\ka_*)\cdot \pa_{x_2} \Psi(\pi,0;g)\cdot \f{1-e^{-2\mu x_2}}{2\mu x_2}.
\end{align*}
Under the condition \eqref{eqn: x_1 x_2 lies on the level set}, this further implies
\[
\pa_{x_2}\Psi(\pi,0;g)\cdot \f{d}{dx_2}\left[\f{1-e^{-2\mu x_2}}{2\mu x_2}\right]
= -\f{1}{x_2}(1-\ka_*)F(x_1,x_2;g).
\]
Combining these identities with \eqref{eqn: pa_x2 Psi less than kappa_* F quantitative} immediately yields \eqref{eqn: non-degeneracy along the level set}.
\end{rmk}

\begin{rmk}\label{rmk: mu=0 x_1 <= pi/2}
Under the same condition, in the case $\mu = 0$, \eqref{eqn: non-degeneracy along the level set} becomes
\[
\pa_{x_2}\big[ F(x_1,x_2;g) -\pa_{x_2}\Psi(\pi,0;g) \big]\leq
-\min\left\{\f{1}{8M},\, \f{1}{x_2}\left(\f{x_1}{4\pi}\right)^{1/2}\right\}\cdot
\pa_{x_2} \Psi(\pi,0;g).
\]
\end{rmk}
\end{prop}

We shall omit the $g$-dependence whenever it incurs no confusion.

In order to prove Proposition \ref{prop: non-degeneracy along level set}, we need some preparation.
For $z>0$ and $\th\in [-2\pi,2\pi]$, we denote \[
J(\th;z) := \int_0^\th \f{\sinh z}{\cosh z - \cos \al}\,d\al.
\]
It is not difficult to verify that
\beq
\begin{split}
J(\th;z)
=&\;
\begin{cases}
2\arctan (\coth\f{z}{2}\cdot \tan\f{\th}{2} ), & \mbox{if } \th\in (-\pi,\pi), \\
\pi+2\arctan (\tanh\f{z}{2}\cdot \tan\f{\th-\pi}{2} ), & \mbox{if } \th\in (0,2\pi).
\end{cases}
\end{split}
\label{eqn: anti-derivative of Poisson}
\eeq
For $z>0$ and $x_1,y_1\in [0,\pi]$, we denote
\beq
\begin{split}
\tilde{J} (x_1,y_1;z )
:= &\; \int_{-y_1}^{y_1} \f{\sinh z}{\cosh z- \cos (x_1-\al)} \, d\al\\
= &\; \int_{x_1-y_1}^{x_1+y_1} \f{\sinh z}{\cosh z- \cos \al} \, d\al = J(x_1+y_1;z)-J(x_1-y_1;z).
\end{split}
\label{eqn: def of tilde J}
\eeq

Let us prove a few properties of $\tilde{J}$.
\begin{lem}
\label{lem: properties of tilde J}
Let $z>0$.
\begin{enumerate}
\item For $x_1\in \BT$,
\[
y_1\mapsto \tilde{J}(x_1,y_1;z) \mbox{ is increasing on $[0,\pi]$}.
\]

\item For $y_1\in [0,\pi]$,
\[
x_1\mapsto \tilde{J}(x_1,y_1;z) \mbox{ is decreasing on $[0,\pi]$}.
\]

\item
If $x_1\in [\f{\pi}{2},\pi]$ and $y_1\in[0, x_1]$,
\[
z\mapsto \tilde{J}(x_1,y_1;z)\mbox{ is increasing on $(0,+\infty)$.}
\]

\item
For any $y_1\in [0,\pi]$, $\tilde{J}(\pi,y_1;z) \leq 2y_1$.

\item
If $x_1\in [0,\f{\pi}{2}]$, then for any $y_1\in [x_1,\pi]$,
$\tilde{J}(x_1,y_1;z) \geq 2y_1$.

\end{enumerate}

\begin{proof}
The first claim immediately follows from
\[
\pa_{y_1}\tilde{J}(x_1,y_1;z) = \f{\sinh z}{\cosh z - \cos(x_1-y_1)} + \f{\sinh z}{\cosh z - \cos(x_1+y_1)}\geq 0.
\]

To show the second claim, we calculate that
\[
\pa_{x_1}\tilde{J}(x_1,y_1;z) = \f{\sinh z}{\cosh z - \cos(x_1+y_1)} -\f{\sinh z}{\cosh z - \cos(x_1-y_1)}.
\]
This is non-positive since $\cos(x_1+y_1)-\cos(x_1-y_1) = -2\sin x_1 \sin y_1 \leq 0$.

When $x_1 \in [\f{\pi}{2},\pi]$ and $y_1\leq x_1$,
\begin{align*}
\pa_z \tilde{J}(x_1,y_1;z)
=&\; -\f{\sin (x_1+y_1)}{\cosh z -\cos(x_1+y_1)} + \f{\sin(x_1-y_1)}{\cosh z -\cos(x_1-y_1)}
\\
=&\; \f{2\sin y_1[\cos y_1-\cosh z \cos x_1]}{(\cosh z -\cos(x_1+y_1))(\cosh z -\cos(x_1-y_1))}.
\end{align*}
This is non-negative since $\cos y_1\geq \cos x_1$ and $\cos x_1 \leq 0$, which proves the third claim.

Since
\[
\pa_{y_1}\tilde{J}(\pi,y_1;z) = \f{2\sinh z}{\cosh z + \cos y_1}\mbox{ is increasing on $[0,\pi]$,}
\]
we find that
\[
y_1\mapsto \f{1}{y_1}\tilde{J}(\pi,y_1;z)\mbox{ is increasing on $[0,\pi]$,}
\]
which further implies
\[
\f{1}{y_1}\tilde{J}(\pi,y_1;z) \leq \f{1}{\pi}\tilde{J}(\pi,\pi;z) = 2.
\]
This proves the fourth claim.

To show the last claim, we first derive from \eqref{eqn: anti-derivative of Poisson} that
\begin{align*}
J(\th;z)
=&\;
\begin{cases}
2\arctan (\coth\f{z}{2}\cdot \tan\f{\th}{2} ), & \mbox{if } \th\in (-\pi,\pi), \\
\pi+2\arctan (\tanh\f{z}{2}\cdot \tan\f{\th-\pi}{2} ), & \mbox{if } \th\in (0,2\pi).
\end{cases}
\end{align*}
Hence,
\begin{align*}
&\; \tilde{J}(x_1,y_1;z) \\
= &\; J(x_1+y_1;z)-J(x_1-y_1;z)\\
= &\; \pi+2\arctan \left(\tanh\f{z}{2}\cdot \tan\f{x_1+y_1-\pi}{2} \right)
+2\arctan \left(\coth\f{z}{2}\cdot \tan\f{y_1-x_1}{2} \right).
\end{align*}
Recall that $x_1\leq \f{\pi}{2}$.
When $y_1\in[x_1, \pi-x_1]$, we find that
\[
\tan\f{x_1+y_1-\pi}{2}\leq 0,\quad
\tan\f{y_1-x_1}{2}\geq 0,
\]
so
\begin{align*}
\tilde{J}(x_1,y_1;z)
\geq &\; \pi+2\arctan \left(\tan\f{x_1+y_1-\pi}{2} \right)
+2\arctan \left(\tan\f{y_1-x_1}{2} \right)
= 2y_1.
\end{align*}
When $y_1\in (\pi-x_1,\pi)$, it must hold that $y_1 > \f{\pi}{2}$,
\beq
\tan\f{y_1-x_1}{2} \geq \tan\f{x_1+y_1 -\pi}{2} > 0,
\label{eqn: two terms in the numerator are positive}
\eeq
and
\[
\left(\tanh\f{z}{2}\cdot \tan\f{x_1+y_1-\pi}{2}\right) \left(\coth\f{z}{2}\cdot \tan\f{y_1-x_1}{2}\right)
< \tan\f{x_1+\pi-\pi}{2}\tan\f{\pi-x_1}{2} = 1.
\]
This allows us to further derive that
\begin{align*}
&\; \tilde{J}(x_1,y_1;z) \\
= &\; \pi+2\arctan \left(\f{\tanh\f{z}{2}\cdot \tan\f{x_1+y_1-\pi}{2} +\coth\f{z}{2}\cdot \tan\f{y_1-x_1}{2}}{1- \tan\f{x_1+y_1-\pi}{2} \tan\f{y_1-x_1}{2}}\right)
\\
= &\; \pi+2\arctan \left(\f{\tanh\f{z}{2}\cdot \tan\f{x_1+y_1-\pi}{2} +\coth\f{z}{2}\cdot \tan\f{y_1-x_1}{2}}{\tan\f{x_1+y_1-\pi}{2} + \tan\f{y_1-x_1}{2}}\cdot
\f{\tan\f{x_1+y_1-\pi}{2} + \tan\f{y_1-x_1}{2}}{1- \tan\f{x_1+y_1-\pi}{2} \tan\f{y_1-x_1}{2}}\right)
\\
= &\; \pi+2\arctan \left(\f{\tanh\f{z}{2}\cdot \tan\f{x_1+y_1 -\pi}{2} +\coth\f{z}{2}\cdot \tan\f{y_1-x_1}{2}}{\tan\f{x_1+y_1 -\pi}{2} + \tan\f{y_1-x_1}{2}}\cdot
\tan\left(y_1-\f{\pi}{2}\right)\right).
\end{align*}
Using \eqref{eqn: two terms in the numerator are positive} as well as the fact $\tanh\f{z}{2}<1<\coth\f{z}{2}$, we find that
\[
\f{\tanh\f{z}{2}\cdot \tan\f{x_1+y_1 -\pi}{2} +\coth\f{z}{2}\cdot \tan\f{y_1-x_1}{2}}{\tan\f{x_1+y_1 -\pi}{2} + \tan\f{y_1-x_1}{2}}
\geq \f{1}{2}\left(\tanh\f{z}{2} + \coth\f{z}{2}\right) \geq 1.
\]
Therefore,
\[
\tilde{J}(x_1,y_1;z)
\geq \pi+2\arctan \left(\tan\left(y_1-\f{\pi}{2}\right)\right)
= 2y_1.
\]

This completes the proof.
\end{proof}
\end{lem}

Motivated by \eqref{eqn: pa_x2 Psi less than kappa_* F quantitative}, we shall derive integral representations of $\pa_{x_2}\Psi$ and $F$.
\begin{lem}
\label{lem: integral representation of Psi_x2 and F}
Let $\tilde{J}(x_1,y_1;z)$ be defined in \eqref{eqn: def of tilde J}.
Let $g\in \BM$.
Then for $x_1\in [0,\pi]$ and $x_2>0$,
\begin{align*}
&\; \pa_{x_2}\Psi(x_1,x_2) - \f{1}{4\pi}\int_\Om e^{-2\mu y_2}\,dy\\
= &\; -\f{1}{4\pi}\int_0^{g(0)} e^{-2\mu y_2}\cdot \mathrm{sgn}(x_2-y_2) \tilde{J}\big(x_1,g^{-1}(y_2);|x_2-y_2|\big)\,dy_2.
\end{align*}
and
\begin{align*}
&\; F(x_1,x_2)-\f{1}{4\pi}\int_\Om e^{-2\mu y_2} \, dy
\\
=
&\; \f{1}{4\pi x_2}\int_{0}^{g(0)} \mathds{1}_{\{y_2\geq x_2/2\}}\cdot e^{-2\mu y_2}
\int_{|y_2-x_2|}^{y_2} \tilde{J}\big(x_1,g^{-1}(y_2);z\big)
\,dz\,dy_2
\\
&\; -\f{1}{4\pi x_2}\int_0^{g(0)} \mathds{1}_{\{y_2\leq x_2/2\}}\cdot e^{-2\mu y_2}
\int_{y_2}^{x_2-y_2}
\tilde{J}\big(\pi,g^{-1}(y_2);z\big)
\, dz \,dy_2
\\
&\; +\f{1}{4\pi x_2}\int_{0}^{g(0)} e^{-2\mu y_2}\int_{\max\{x_2-y_2,y_2\}}^\infty
\tilde{J}\big(x_1,g^{-1}(y_2);z\big)-\tilde{J}\big(\pi,g^{-1}(y_2);z\big)
\, dz \,dy_2.
\end{align*}

\begin{proof}
By \eqref{eqn: integral representation of Psi} in Lemma \ref{lem: integral formula for Psi},
\begin{align*}
&\; \pa_{x_2}\Psi(x_1,x_2)\\
= &\; -\f{1}{4\pi}\int_\Om e^{-2\mu y_2}
\left[ \f{\sinh(x_2-y_2)}{\cosh(x_2-y_2)- \cos(x_1-y_1)}
-1\right] dy
\\
= &\; -\f{1}{4\pi}\int_0^{g(0)} e^{-2\mu y_2}
\int_{-g^{-1}(y_2)}^{g^{-1}(y_2)} \f{\sinh(x_2-y_2)}{\cosh(x_2-y_2)- \cos(x_1-y_1)}\,dy_1\,dy_2
+ \f{1}{4\pi}\int_\Om e^{-2\mu y_2}\,dy
\\
= &\; -\f{1}{4\pi}\int_0^{g(0)} e^{-2\mu y_2}\cdot \mathrm{sgn}(x_2-y_2) \tilde{J}\big(x_1,g^{-1}(y_2);|x_2-y_2|\big)\,dy_2
+ \f{1}{4\pi}\int_\Om e^{-2\mu y_2}\,dy.
\end{align*}

For the integral representation of $F$, we derive from \eqref{eqn: integral representation of Psi} and \eqref{eqn: def of F} that
\begin{align*}
&\; F(x_1,x_2)-\f{1}{4\pi}\int_\Om e^{-2\mu y_2} \, dy
\\
= &\; -\f{1}{4\pi x_2}\int_\Om e^{-2\mu y_2}
\cdot \ln \left(\f{\cosh(x_2-y_2)- \cos(x_1-y_1)}{\cosh y_2 -\cos(x_1-y_1)}\right)dy
\\
&\; -\f{1}{4\pi x_2}\int_\Om e^{-2\mu y_2}
\cdot \ln \left(\f{\cosh y_2- \cos(x_1-y_1)}{\cosh y_2 -\cos(\pi-y_1)}\right)dy
\\
= &\; -\f{1}{4\pi x_2}\int_0^{g(0)} e^{-2\mu y_2}\int_{-g^{-1}(y_2)}^{g^{-1}(y_2)}
\int_{y_2}^{|x_2-y_2|} \f{\sinh z}{\cosh z- \cos(x_1-y_1)}\,dz\,dy_1\,dy_2
\\
&\; +\f{1}{4\pi x_2}\int_0^{g(0)} e^{-2\mu y_2}
\int_{-g^{-1}(y_2)}^{g^{-1}(y_2)} \int_{y_2}^\infty
\f{\sinh z}{\cosh z- \cos(x_1-y_1)} - \f{\sinh z}{\cosh z -\cos(\pi-y_1)}\, dz \,dy_1 \, dy_2
\\
= &\; -\f{1}{4\pi x_2}\int_0^{g(0)} e^{-2\mu y_2}
\int_{y_2}^{|x_2-y_2|} \tilde{J}\big(x_1,g^{-1}(y_2);z\big)
\,dz\,dy_2
\\
&\; +\f{1}{4\pi x_2}\int_0^{g(0)} e^{-2\mu y_2}\int_{y_2}^\infty
\tilde{J}\big(x_1,g^{-1}(y_2);z\big)-\tilde{J}\big(\pi,g^{-1}(y_2);z\big)
\, dz \,dy_2
\\
=
&\; \f{1}{4\pi x_2}\int_{0}^{g(0)} \mathds{1}_{\{y_2\geq x_2/2\}}\cdot e^{-2\mu y_2}
\int_{|y_2-x_2|}^{y_2} \tilde{J}\big(x_1,g^{-1}(y_2);z\big)
\,dz\,dy_2
\\
&\; -\f{1}{4\pi x_2}\int_0^{g(0)} \mathds{1}_{\{y_2\leq x_2/2\}}\cdot e^{-2\mu y_2}
\int_{y_2}^{x_2-y_2}
\tilde{J}\big(\pi,g^{-1}(y_2);z\big)
\, dz \,dy_2
\\
&\; +\f{1}{4\pi x_2}\int_{0}^{g(0)} e^{-2\mu y_2}\int_{\max\{x_2-y_2,y_2\}}^\infty
\tilde{J}\big(x_1,g^{-1}(y_2);z\big)-\tilde{J}\big(\pi,g^{-1}(y_2);z\big)
\, dz \,dy_2.
\end{align*}
This completes the proof.
\end{proof}
\end{lem}

The following lemma plays a crucial role in proving Proposition \ref{prop: non-degeneracy along level set}.
\begin{lem}
\label{lem: estimate for Psi_x2 - kappa F}
For $x_1\in [0,\pi]$ and $x_2>0$, with $\ka_*$ defined in \eqref{eqn: def of kappa_*}, it holds that
\begin{align*}
&\; \pa_{x_2}\Psi(x_1,x_2) -\ka_* F(x_1,x_2)
- (1-\ka_*)\cdot \f{1}{4\pi}\int_\Om e^{-2\mu y_2}\,dy\\
\leq
&\; -\f{1}{4\pi}\int_0^{g(0)} \mathds{1}_{\{y_2\leq x_2/2\}}\cdot e^{-2\mu y_2} \tilde{J}\big(x_1,g^{-1}(y_2);x_2-y_2\big)\,dy_2
\\
&\;
- \f{1}{4\pi} \left(1 + \f{\ka_*}{2\mu x_2}\right)
\int_0^{g(0)}
\left(1-e^{-2\mu (2x_2-2y_2)}\right)\\
&\;\qquad\qquad \qquad \qquad \quad \cdot
\mathds{1}_{\{\f12 x_2\leq y_2\leq x_2\}} \cdot e^{-2\mu y_2} \tilde{J}\big(x_1,g^{-1}(y_2);x_2-y_2\big)\,dy_2
\\
&\; + \f{\ka_*}{4\pi}\int_0^{g(0)} \mathds{1}_{\{y_2\leq x_2/2\}}\cdot \f{x_2-2y_2}{x_2} \cdot e^{-2\mu y_2}
\tilde{J}\big(\pi,g^{-1}(y_2);x_2-y_2\big)\, dy_2.
\end{align*}

\begin{proof}
Thanks to Lemma \ref{lem: integral representation of Psi_x2 and F},
\beq
\begin{split}
&\; \pa_{x_2}\Psi(x_1,x_2) - \ka_* F(x_1,x_2)
- (1-\ka_*)\cdot \f{1}{4\pi}\int_\Om e^{-2\mu y_2}\,dy\\
= &\; -\f{1}{4\pi}\int_0^{g(0)} \mathds{1}_{\{y_2\leq x_2/2\}}\cdot e^{-2\mu y_2} \tilde{J}\big(x_1,g^{-1}(y_2);x_2-y_2\big)\,dy_2
\\
&\; -\f{1}{4\pi}\int_0^{g(0)} \mathds{1}_{\{x_2/2\leq y_2\leq x_2\}}\cdot e^{-2\mu y_2} \tilde{J}\big(x_1,g^{-1}(y_2);x_2-y_2\big)\,dy_2
\\
&\; +\f{1}{4\pi}\int_0^{g(0)}  \mathds{1}_{\{y_2\geq x_2\}}\cdot e^{-2\mu y_2} \tilde{J}\big(x_1,g^{-1}(y_2);y_2-x_2\big)\,dy_2
\\
&\; -\f{\ka_*}{4\pi x_2}\int_{0}^{g(0)} \mathds{1}_{\{y_2\geq x_2/2\}}\cdot e^{-2\mu y_2}
\int_{|y_2-x_2|}^{y_2} \tilde{J}\big(x_1,g^{-1}(y_2);z\big)
\,dz\,dy_2
\\
&\; +\f{\ka_*}{4\pi x_2}\int_0^{g(0)} \mathds{1}_{\{y_2\leq x_2/2\}}\cdot e^{-2\mu y_2}
\int_{y_2}^{x_2-y_2}
\tilde{J}\big(\pi,g^{-1}(y_2);z\big)
\, dz \,dy_2
\\
&\; -\f{\ka_*}{4\pi x_2}\int_{0}^{g(0)} e^{-2\mu y_2}\int_{\max\{x_2-y_2,y_2\}}^\infty
\tilde{J}\big(x_1,g^{-1}(y_2);z\big)-\tilde{J}\big(\pi,g^{-1}(y_2);z\big)
\, dz \,dy_2
\\
=: &\; I_1+I_2+I_3+I_4+I_5+I_6.
\end{split}
\label{eqn: splitting Psi_x2-F}
\eeq

We first claim that
\beq
\begin{split}
I_4 \leq &\;
-\f{1}{4\pi} \int_{0}^{g(0)}
\f{\ka_*}{2\mu  x_2 }\left(e^{2\mu \min\{x_2,2y_2-2x_2\}} - 1\right) \\
&\;\qquad \qquad \quad \cdot \mathds{1}_{\{y_2\geq x_2\}}\cdot e^{-2\mu y_2} \tilde{J}\big(x_1,g^{-1}(y_2);y_2-x_2\big)\,dy_2.
\end{split}
\label{eqn: bound for I_4}
\eeq
To show this, we start by writing
\begin{align*}
&\; \int_0^{g(0)} \mathds{1}_{\{y_2\geq x_2/2\}}\cdot e^{-2\mu y_2}
\int_{|y_2-x_2|}^{y_2} \tilde{J}\big(x_1,g^{-1}(y_2);z\big)
\,dz\,dy_2
\\
=
&\; \int_0^{g(0)}
\int_0^\infty \mathds{1}_{\{y_2\geq x_2/2\}} \mathds{1}_{\{|y_2-x_2|\leq z \leq y_2\}}\cdot e^{-2\mu y_2}\tilde{J}\big(x_1,g^{-1}(y_2);z\big)
\,dz\,dy_2
\\
=
&\; \int_0^\infty
\int_0^{g(0)}\mathds{1}_{\{y_2\geq x_2/2\}}
\mathds{1}_{\{y_2 \geq z\}}
\mathds{1}_{\{x_2-z\leq y_2 \leq x_2+z\}} \cdot e^{-2\mu y_2} \tilde{J}\big(x_1,g^{-1}(y_2);z\big)
\,dy_2\,dz
\\
=
&\; \int_0^\infty
\int_{\max\{z,x_2-z\}}^{z+x_2}
\mathds{1}_{\{y_2\leq g(0)\}}\cdot e^{-2\mu y_2} \tilde{J}\big(x_1,g^{-1}(y_2);z\big)
\,dy_2\,dz.
\end{align*}
By Lemma \ref{lem: properties of tilde J},
\beq
y_2\mapsto \tilde{J}\big(x_1,g^{-1}(y_2);z\big)\mbox{ is decreasing},
\label{eqn: monotonicity of tilde J in y_2}
\eeq
so
\begin{align*}
&\; \int_0^{g(0)} \mathds{1}_{\{y_2\geq x_2/2\}}\cdot e^{-2\mu y_2}
\int_{|y_2-x_2|}^{y_2} \tilde{J}\big(x_1,g^{-1}(y_2);z\big)
\,dz\,dy_2
\\
\geq &\; \int_0^\infty
\left(\int_{\max\{z,x_2-z\}}^{z+x_2}\mathds{1}_{\{y_2\leq g(0)\}} \cdot e^{-2\mu y_2}
\,dy_2\right) \tilde{J}\big(x_1,g^{-1}(z+x_2);z\big)\,dz
\\
= &\; \int_0^\infty \mathds{1}_{\{z\leq g(0)-x_2\}}
\left(\int_{\max\{z,x_2-z\}}^{z+x_2}
e^{-2\mu y_2}
\,dy_2\right) \tilde{J}\big(x_1,g^{-1}(z+x_2);z\big)\,dz.
\end{align*}
In the last line, we used the fact that, when $z\geq g(0)-x_2$, it holds that
\[
\tilde{J}\big(x_1,g^{-1}(z+x_2);z\big)
= \tilde{J}(x_1,0;z) = 0.
\]
Now we evaluate the integral above and obtain that
\begin{align*}
&\; \int_0^{g(0)} \mathds{1}_{\{y_2\geq x_2/2\}}\cdot e^{-2\mu y_2}
\int_{|y_2-x_2|}^{y_2} \tilde{J}\big(x_1,g^{-1}(y_2);z\big)
\,dz\,dy_2
\\
\geq &\;
\int_0^\infty \mathds{1}_{\{z+x_2\leq g(0)\}}
\cdot \f{1}{2\mu}\left(e^{-2\mu \max\{z,x_2-z\}} - e^{-2\mu (z+x_2)}\right) \tilde{J}\big(x_1,g^{-1}(z+x_2);z\big)\,dz
\\
= &\;
\int_{x_2}^\infty \mathds{1}_{\{y_2\leq g(0)\}}
\cdot \f{1}{2\mu}\left(e^{-2\mu \max\{y_2-x_2,x_2-(y_2-x_2)\}} - e^{-2\mu y_2}\right) \tilde{J}\big(x_1,g^{-1}(y_2);y_2-x_2\big)\,dy_2
\\
= &\;
\int_{0}^\infty \mathds{1}_{\{y_2\leq g(0)\}}
\cdot \f{1}{2\mu}\left(e^{-2\mu \max\{-x_2,2x_2-2y_2\}} - 1\right) \\
&\;\qquad \cdot \mathds{1}_{\{y_2 \geq x_2\}} \cdot e^{-2\mu y_2} \tilde{J}\big(x_1,g^{-1}(y_2);y_2-x_2\big)\,dy_2.
\end{align*}
This proves \eqref{eqn: bound for I_4}.

Using \eqref{eqn: bound for I_4}, we deduce that
\beq
\begin{split}
&\;I_3+I_4\\
\leq
&\;
-\f{1}{4\pi} \int_{0}^{g(0)}
\left[\f{\ka_*}{2\mu  x_2 }\left(e^{2\mu \min\{x_2,2y_2-2x_2\}} - 1\right) - 1\right] \\
&\;\qquad \qquad \quad \cdot \mathds{1}_{\{y_2\geq x_2\}}\cdot e^{-2\mu y_2} \tilde{J}\big(x_1,g^{-1}(y_2);y_2-x_2\big)\,dy_2
\\
=
&\; -\f{1}{4\pi} \int_{0}^{g(0)}
\left[\f{\ka_*}{2\mu  x_2 }\left(e^{2\mu (2y_2-2x_2)} - 1\right) - 1\right] \\
&\;\qquad \qquad \quad \cdot \mathds{1}_{\{x_2\leq y_2\leq \f32 x_2\}}\cdot e^{-2\mu y_2} \tilde{J}\big(x_1,g^{-1}(y_2);y_2-x_2\big)\,dy_2
\\
&\;
-\f{1}{4\pi} \left(\ka_* \cdot \f{e^{2\mu x_2} - 1}{2\mu x_2 } - 1\right) \int_{0}^{g(0)}
\mathds{1}_{\{y_2\geq \f32 x_2\}}\cdot e^{-2\mu y_2} \tilde{J}\big(x_1,g^{-1}(y_2);y_2-x_2\big)\,dy_2.
\end{split}
\label{eqn: bound for I_3+I_4}
\eeq
Note that the last line vanishes due to the definition of $\ka_*$.
To further simplify the bound, we observe that, by a change of variable $y_2':=2x_2-y_2$,
\begin{align*}
&\;
-\f{1}{4\pi} \int_{0}^{g(0)}
\left[\f{\ka_*}{2\mu  x_2 }\left(e^{2\mu (2y_2-2x_2)} - 1\right) - 1\right] \\
&\;\qquad \qquad \quad \cdot \mathds{1}_{\{x_2\leq y_2\leq \f32 x_2\}}\cdot e^{-2\mu y_2} \tilde{J}\big(x_1,g^{-1}(y_2);y_2-x_2\big)\,dy_2
\\
=
&\;
-\f{1}{4\pi} \int_{0}^{g(0)}
\left[\f{\ka_*}{2\mu  x_2 }\left(1-e^{-2\mu (2y_2-2x_2)} \right) - e^{-2\mu (2y_2-2x_2)} \right]  \\
&\;\qquad \qquad \quad \cdot \mathds{1}_{\{x_2\leq y_2\leq \f32 x_2\}} \cdot e^{-2\mu (2x_2-y_2)} \tilde{J}\big(x_1,g^{-1}(y_2);y_2-x_2\big)\,dy_2
\\
=
&\;
\f{1}{4\pi} \int_{2x_2-g(0)}^{2x_2}
\left[e^{-2\mu (2x_2-2y_2')} - \f{\ka_*}{2\mu  x_2 }\left(1-e^{-2\mu (2x_2-2y_2')} \right)\right]  \\
&\;\qquad \qquad \quad \cdot \mathds{1}_{\{\f12 x_2\leq y_2'\leq x_2\}} \cdot e^{-2\mu y_2'} \tilde{J}\big(x_1,g^{-1}(2x_2-y_2');x_2-y_2'\big)\,dy_2'
\\
\leq
&\;
\f{1}{4\pi} \int_{2x_2-g(0)}^{2x_2}
\left[e^{-2\mu (2x_2-2y_2')} - \f{\ka_*}{2\mu  x_2 }\left(1-e^{-2\mu (2x_2-2y_2')} \right)\right]
\\
&\;\qquad \qquad \quad \cdot
\mathds{1}_{\{\f12 x_2\leq y_2'\leq x_2\}} \cdot e^{-2\mu y_2'} \tilde{J}\big(x_1,g^{-1}(y_2');x_2-y_2'\big)\,dy_2'
\\
=
&\;
\f{1}{4\pi} \int_{2x_2-g(0)}^{2x_2}
\left[e^{-2\mu (2x_2-2y_2)} - \f{\ka_*}{2\mu x_2}\left(1-e^{-2\mu (2x_2-2y_2)} \right)\right]
\\
&\;\qquad \qquad \quad \cdot
\mathds{1}_{\{\f12 x_2\leq y_2\leq x_2\}} \cdot e^{-2\mu y_2} \tilde{J}\big(x_1,g^{-1}(y_2);x_2-y_2\big)\,dy_2.
\end{align*}
In the inequality above, we used the monotonicity of $g^{-1}$, \eqref{eqn: monotonicity of tilde J in y_2}, and the definition of $\ka_*$.
We claim that
\beq
\begin{split}
&\; \int_{2x_2-g(0)}^{2x_2}
\left[e^{-2\mu (2x_2-2y_2)} - \f{\ka_*}{2\mu x_2}\left(1-e^{-2\mu (2x_2-2y_2)} \right)\right]
\\
&\;\qquad \qquad \cdot
\mathds{1}_{\{\f12 x_2\leq y_2\leq x_2\}} \cdot e^{-2\mu y_2} \tilde{J}\big(x_1,g^{-1}(y_2);x_2-y_2\big)\,dy_2
\\
\leq
&\; \int_0^{g(0)} \left[e^{-2\mu (2x_2-2y_2)} - \f{\ka_*}{2\mu x_2}\left(1-e^{-2\mu (2x_2-2y_2)} \right)\right]
\\
&\;\qquad \cdot
\mathds{1}_{\{\f12 x_2\leq y_2\leq x_2\}} \cdot e^{-2\mu y_2} \tilde{J}\big(x_1,g^{-1}(y_2);x_2-y_2\big)\,dy_2.
\end{split}
\label{eqn: bounding the integral on x_2 to 3 halves x_2}
\eeq
Indeed,
\begin{itemize}
\item
If $x_2\geq g(0)$, the left-hand side simply vanishes as $x_2\leq 2x_2 - g(0)$.

\item
If $x_2\leq \f{g(0)}{2}$, we have $[\f12x_2,x_2]\subset [0,g(0)]$ and $[\f12x_2,x_2]\subset [2x_2-g(0),2x_2]$, so the equality in \eqref{eqn: bounding the integral on x_2 to 3 halves x_2} holds.

\item
If $x_2\in [\f{g(0)}{2},g(0)]$, we have $2x_2-g(0)\geq 0$ while $x_2\leq g(0)\leq 2x_2$, so \eqref{eqn: bounding the integral on x_2 to 3 halves x_2} holds as well.

\end{itemize}
With \eqref{eqn: bounding the integral on x_2 to 3 halves x_2}, we conclude that
\begin{align*}
&\;
-\f{1}{4\pi} \int_{0}^{g(0)}
\left[\f{\ka_*}{2\mu  x_2 }\left(e^{2\mu (2y_2-2x_2)} - 1\right) - 1\right] \\
&\;\qquad \qquad \cdot \mathds{1}_{\{x_2\leq y_2\leq \f32 x_2\}}\cdot e^{-2\mu y_2} \tilde{J}\big(x_1,g^{-1}(y_2);y_2-x_2\big)\,dy_2
\\
\leq
&\;
\f{1}{4\pi} \int_0^{g(0)} \left[e^{-2\mu (2x_2-2y_2)} - \f{\ka_*}{2\mu x_2}\left(1-e^{-2\mu (2x_2-2y_2)} \right)\right]
\\
&\;\qquad \quad \cdot
\mathds{1}_{\{\f12 x_2\leq y_2\leq x_2\}} \cdot e^{-2\mu y_2} \tilde{J}\big(x_1,g^{-1}(y_2);x_2-y_2\big)\,dy_2.
\end{align*}
This combined with \eqref{eqn: bound for I_3+I_4} gives that
\begin{align*}
&\;I_2+I_3+I_4\\
\leq
&\;
- \f{1}{4\pi} \left(1 + \f{\ka_*}{2\mu x_2} \right)
\int_0^{g(0)}
\left(1-e^{-2\mu (2x_2-2y_2)}\right)\\
&\;\qquad\qquad \qquad \qquad \qquad \cdot
\mathds{1}_{\{\f12 x_2\leq y_2\leq x_2\}} \cdot e^{-2\mu y_2} \tilde{J}\big(x_1,g^{-1}(y_2);x_2-y_2\big)\,dy_2.
\end{align*}

By Lemma \ref{lem: properties of tilde J},
$z\mapsto \tilde{J}(\pi,g^{-1}(y_2);z)$ is increasing,
so we find that
\beq
\begin{split}
I_5
\leq
&\; \f{\ka_*}{4\pi x_2}\int_0^{g(0)} \mathds{1}_{\{y_2\leq x_2/2\}} \cdot e^{-2\mu y_2}
\int_{y_2}^{x_2-y_2}
\tilde{J}\big(\pi,g^{-1}(y_2);x_2-y_2\big)
\, dz \,dy_2
\\
= &\; \f{\ka_*}{4\pi}\int_0^{g(0)} \mathds{1}_{\{y_2\leq x_2/2\}}\cdot \f{x_2-2y_2}{x_2} \cdot e^{-2\mu y_2}
\tilde{J}\big(\pi,g^{-1}(y_2);x_2-y_2\big)\, dy_2.
\end{split}
\label{eqn: estimate for I_5}
\eeq
Finally, thanks to Lemma \ref{lem: properties of tilde J}, $x_1\mapsto \tilde{J}(x_1,g^{-1}(y_2);z)$ is decreasing on $[0,\pi]$, so $I_6\leq 0$.

Summarizing all the estimates, we conclude that
\begin{align*}
&\; \pa_{x_2}\Psi(x_1,x_2) - \ka_* F(x_1,x_2)
- (1-\ka_*)\cdot \f{1}{4\pi}\int_\Om e^{-2\mu y_2}\,dy\\
\leq
&\; -\f{1}{4\pi}\int_0^{g(0)} \mathds{1}_{\{y_2\leq x_2/2\}}\cdot e^{-2\mu y_2} \tilde{J}\big(x_1,g^{-1}(y_2);x_2-y_2\big)\,dy_2
\\
&\;
- \f{1}{4\pi} \left(1 + \f{\ka_*}{2\mu x_2}\right)
\int_0^{g(0)}
\left(1-e^{-2\mu (2x_2-2y_2)}\right)\\
&\;\qquad\qquad \qquad \qquad \qquad \cdot
\mathds{1}_{\{\f12 x_2\leq y_2\leq x_2\}} \cdot e^{-2\mu y_2} \tilde{J}\big(x_1,g^{-1}(y_2);x_2-y_2\big)\,dy_2
\\
&\; + \f{\ka_*}{4\pi}\int_0^{g(0)} \mathds{1}_{\{y_2\leq x_2/2\}}\cdot \f{x_2-2y_2}{x_2} \cdot e^{-2\mu y_2}
\tilde{J}\big(\pi,g^{-1}(y_2);x_2-y_2\big)\, dy_2.
\end{align*}
This completes the proof.
\end{proof}
\end{lem}

Now we are ready to prove Proposition \ref{prop: non-degeneracy along level set}.

\begin{proof}[Proof of Proposition \ref{prop: non-degeneracy along level set}]
As is noted in Remark \ref{rmk: equivalent formulation of the non-degeneracy}, it suffices to prove \eqref{eqn: pa_x2 Psi less than kappa_* F quantitative}.

Using Lemma \ref{lem: estimate for Psi_x2 - kappa F} and the definition of  $\ka_*$ in \eqref{eqn: def of kappa_*}, we find that
\begin{align*}
&\; \pa_{x_2}\Psi(x_1,x_2) - \ka_* F(x_1,x_2)
\\
\leq
&\; -\f{1}{4\pi}\int_{0}^{x_2/2} e^{-2\mu y_2} \tilde{J}\big(x_1,g^{-1}(y_2);x_2-y_2\big)\,dy_2
\\
&\;
- \f{1}{4\pi} \cdot
\int_{x_2/2}^{x_2}\f{1-e^{-2\mu (2x_2-2y_2)}}{1-e^{-2\mu x_2}}
\cdot e^{-2\mu y_2} \tilde{J}\big(x_1,g^{-1}(y_2);x_2-y_2\big)\,dy_2
\\
&\; + \f{\ka_*}{4\pi}\int_0^{x_2/2} \f{x_2-2y_2}{x_2} \cdot e^{-2\mu y_2}
\tilde{J}\big(\pi,g^{-1}(y_2);x_2-y_2\big)\, dy_2
\\
&\; +(1-\ka_*)\cdot \f{1}{4\pi}\int_\Om e^{-2\mu y_2}\,dy.
\end{align*}
On the other hand, by \eqref{eqn: end point constant} in Lemma \ref{lem: integral formula for Psi},
\begin{align*}
\pa_{x_2} \Psi(\pi,0)
\leq &\; \f{1}{4\pi}\int_0^{g(0)} e^{-2\mu y_2} \cdot 4\cdot \f{g^{-1}(y_2)}{2}\, dy_2
+ \f{1}{4\pi}\int_\Om e^{-2\mu y_2} \, dy
= 2\cdot \f{1}{4\pi}\int_\Om e^{-2\mu y_2} \, dy,
\end{align*}
which gives
\[
-\d(1-\ka_*)\cdot
\pa_{x_2} \Psi(\pi,0)\cdot \f{1-e^{-2\mu x_2}}{2\mu x_2}
\geq - 2\d(1-\ka_*) \cdot \f{1}{4\pi}\int_\Om e^{-2\mu y_2} \, dy.
\]
So in order to show \eqref{eqn: pa_x2 Psi less than kappa_* F quantitative}, we only have to prove
\begin{align*}
&\; \int_{0}^{x_2/2} e^{-2\mu y_2} \tilde{J}\big(x_1,g^{-1}(y_2);x_2-y_2\big)\,dy_2
\\
&\;
+\int_{x_2/2}^{x_2}
\f{1-e^{-2\mu (2x_2-2y_2)}}{1-e^{-2\mu x_2}} \cdot
e^{-2\mu y_2} \tilde{J}\big(x_1,g^{-1}(y_2);x_2-y_2\big)\,dy_2
\\
&\; -\ka_* \int_0^{x_2/2} \f{x_2-2y_2}{x_2} \cdot e^{-2\mu y_2}
\tilde{J}\big(\pi,g^{-1}(y_2);x_2-y_2\big)\, dy_2
\\
\geq &\; (1+2\d)(1-\ka_*)\int_\Om e^{-2\mu y_2}\,dy.
\end{align*}
Since $x_2\in (0,g(0))$ and $g^{-1}(x_2) = x_1$,
\begin{align*}
\int_\Om e^{-2\mu y_2}\,dy
= &\; \int_0^{g(0)} e^{-2\mu y_2}\cdot 2g^{-1}(y_2)\,dy_2
\\
\leq &\; \int_0^{x_2} e^{-2\mu y_2}\cdot 2g^{-1}(y_2)\,dy_2
+ \int_{x_2}^{g(0)} e^{-2\mu y_2}\cdot 2x_1\,dy_2
\\
= &\; \int_0^{x_2} e^{-2\mu y_2}\cdot 2g^{-1}(y_2)\,dy_2
+ 2x_1\cdot \f{1}{2\mu} \big(e^{-2\mu x_2}-e^{-2\mu g(0)}\big).
\end{align*}
Combining this with the preceding inequality, we find that it suffices to show
\beq
\begin{split}
W \geq &\; (1+2\d)\cdot \f{1-\ka_*}{2\mu}\cdot 2x_1 \big(e^{-2\mu x_2}-e^{-2\mu g(0)}\big),
\end{split}
\label{eqn: pa_x2 Psi less than alpha F simplified}
\eeq
where
\begin{align*}
W:=&\; \int_{0}^{x_2/2} e^{-2\mu y_2}
\left[\tilde{J}\big(x_1,g^{-1}(y_2);x_2-y_2\big)\right.\\
&\;\qquad \qquad \left.-
\ka_*\cdot \f{x_2-2y_2}{x_2} \cdot
\tilde{J}\big(\pi,g^{-1}(y_2);x_2-y_2\big)
- 2(1+2\d)(1-\ka_*) g^{-1}(y_2)
\right] dy_2
\\
&\;
+ \int_{x_2/2}^{x_2} e^{-2\mu y_2}\left[
\f{1-e^{-2\mu (2x_2-2y_2)}}{1-e^{-2\mu x_2}}\cdot
\tilde{J}\big(x_1,g^{-1}(y_2);x_2-y_2\big)\right.\\
&\; \qquad \qquad \quad
- 2(1+2\d)(1-\ka_*) g^{-1}(y_2)\Big] dy_2.
\end{align*}

For $x_1\in [0,\f{\pi}{2}]$, by Lemma \ref{lem: properties of tilde J},
\begin{align*}
W \geq &\; \int_{0}^{x_2/2} e^{-2\mu y_2}
\left[2g^{-1}(y_2) - \ka_*\cdot \f{x_2-2y_2}{x_2} \cdot
2g^{-1}(y_2) - 2(1+2\d)(1-\ka_*) g^{-1}(y_2)\right] dy_2
\\
&\;
+ \int_{x_2/2}^{x_2} e^{-2\mu y_2}\left[
\f{1-e^{-2\mu (2x_2-2y_2)}}{1-e^{-2\mu x_2}}\cdot
2g^{-1}(y_2)
- 2(1+2\d)(1-\ka_*) g^{-1}(y_2)\right] dy_2
\\
=
&\; \int_{0}^{x_2/2} e^{-2\mu y_2}\cdot 2g^{-1}(y_2)\left[ \ka_* \cdot \f{2y_2}{x_2} - 2\d(1-\ka_*) \right] dy_2
\\
&\;
+ \int_{x_2/2}^{x_2} e^{-2\mu y_2}\cdot 2g^{-1}(y_2) \left[
\f{1-e^{-2\mu (2x_2-2y_2)}}{1-e^{-2\mu x_2}} - (1+2\d)(1-\ka_*) \right] dy_2.
\end{align*}
Let
\begin{align*}
\va_1(y_2):= &\; \ka_*\cdot \f{2y_2}{x_2} - 2\d(1-\ka_*),\\
\va_2(y_2):= &\; \f{1-e^{-2\mu (2x_2-2y_2)}}{1-e^{-2\mu x_2}} -(1+2\d)(1-\ka_*).
\end{align*}
It is clear that $\va_1$ is strictly increasing on $[0,\f{x_2}{2}]$, and $\va_2$ is strictly decreasing on $[\f{x_2}{2},x_2]$.
Moreover, by \eqref{eqn: value of delta} and the fact that $x_2\in [0,M]$,
\[
\d \leq \f{1}{2(e^{2\mu M}-1)} \leq
\f{1}{2(\f{e^{2\mu M}-1}{2\mu M}-1)} \leq \f{1}{2(\f{e^{2\mu x_2}-1}{2\mu x_2}-1)} = \f{\ka_*}{2(1-\ka_*)},
\]
which implies
\[
\va_1(0) <0,\quad \va_1\left(\f{x_2}{2}\right) = \ka_*-2\d (1-\ka_*)\geq 0,
\]
and
\[
\va_2\left(\f{x_2}{2}\right)
=1-(1+2\d)(1-\ka_*) \geq 0,\quad
\va_2(x_2) = -(1+2\d)(1-\ka_*) <0.
\]
Hence, there exists a unique $\tilde{x}_2 \in [0,\f{x_2}{2}]$ and $x_2^*\in [\f{x_2}{2},x_2]$ such that $\va_1(\tilde{x}_2) = 0$ and $\va_2(x_2^*) = 0$.
In particular,
\beq
\tilde{x}_2 = \f{x_2 \d(1-\ka_*)}{\ka_*}.
\label{eqn: def of tilde x_2}
\eeq
This allows us to derive that
\begin{align*}
W
\geq
&\; \int_{0}^{\tilde{x}_2} e^{-2\mu y_2} \cdot 2\pi \left[ \ka_* \cdot \f{2y_2}{x_2} - 2\d(1-\ka_*) \right]  dy_2
\\
&\; + \int_{\tilde{x}_2}^{x_2/2} e^{-2\mu y_2}
\cdot 2g^{-1}\big(x_2^*\big) \left[ \ka_* \cdot \f{2y_2}{x_2} - 2\d(1-\ka_*) \right] dy_2
\\
&\;
+ \int_{x_2/2}^{x_2} e^{-2\mu y_2} \cdot 2g^{-1}\big(x_2^*\big) \left[
\f{1-e^{-2\mu (2x_2-2y_2)}}{1-e^{-2\mu x_2}} - (1+2\d)(1-\ka_*) \right] dy_2
\\
=
&\; 2g^{-1} \big(x_2^*\big) \int_0^{x_2/2} e^{-2\mu y_2} \left[ \ka_* \cdot \f{2y_2}{x_2} - 2\d(1-\ka_*) \right] dy_2
\\
&\;
+ 2g^{-1} \big(x_2^*\big) \int_{x_2/2}^{x_2} e^{-2\mu y_2} \left[
\f{1-e^{-2\mu (2x_2-2y_2)}}{1-e^{-2\mu x_2}} - (1+2\d)(1-\ka_*) \right]  dy_2
\\
&\; + \big[2\pi-2g^{-1}\big(x_2^*\big)\big] \int_{0}^{\tilde{x}_2} e^{-2\mu y_2} \left[ \ka_* \cdot \f{2y_2}{x_2} - 2\d(1-\ka_*) \right]  dy_2
\\
=
&\; 2g^{-1}\big(x_2^*\big)\cdot \f{2\ka_*}{(2\mu)^2 x_2}\big(1-(1+\mu x_2) e^{-\mu x_2}\big)
\\
&\;
+ 2g^{-1}\big(x_2^*\big) \left[
\f{1-e^{-\mu x_2}}{1-e^{-2\mu x_2}} - (1-\ka_*) \right] \cdot \f{1}{2\mu} \big(e^{-\mu x_2}-e^{-2\mu x_2}\big)
\\
&\;
- 2g^{-1}\big(x_2^*\big) \cdot 2\d (1-\ka_*) \cdot \f{1}{2\mu} \big(1-e^{-2\mu x_2}\big)
\\
&\; + \big[2\pi-2g^{-1}\big(x_2^*\big)\big] \left[\f{2\ka_*}{(2\mu)^2 x_2}\big(1-(1+2\mu \tilde{x}_2) e^{-2\mu \tilde{x}_2}\big)
- 2\d(1-\ka_*)
\cdot \f{1}{2\mu}\big(1-e^{-2\mu \tilde{x}_2}\big)\right].
\end{align*}
Using the definition \eqref{eqn: def of kappa_*} of $\ka_*$ as well as the fact $ \d(1-\ka_*) = \f{\tilde{x}_2\ka_*}{x_2}$ due to \eqref{eqn: def of tilde x_2}, we find that
\begin{align*}
W
\geq
&\; 2g^{-1}\big(x_2^*\big) \cdot \f{2}{e^{2\mu x_2}-1}\cdot \f{1}{2\mu} \big(1-(1+\mu x_2) e^{-\mu x_2}\big)
\\
&\;
+ 2g^{-1}\big(x_2^*\big) \cdot
\f{-e^{\mu x_2} + 1 + 2\mu x_2}{e^{2\mu x_2}-1}
\cdot  \f{1}{2\mu} \big(e^{-\mu x_2}-e^{-2\mu x_2}\big)
\\
&\;
- 2g^{-1}\big(x_2^*\big) \cdot \f{1-\ka_*}{2\mu} \cdot 2\d \big(1-e^{-2\mu x_2}\big)
\\
&\; + \big[2\pi-2g^{-1}\big(x_2^*\big)\big] \left[\f{2\ka_*}{(2\mu)^2 x_2}\big(1-(1+2\mu \tilde{x}_2) e^{-2\mu \tilde{x}_2}\big)
- \f{2\ka_*}{2\mu x_2}\cdot \tilde{x}_2
\big(1-e^{-2\mu \tilde{x}_2}\big)\right]
\\
= &\; 2g^{-1}\big(x_2^*\big)\cdot \f{1}{ 2\mu }\cdot \f{1-(1+2\mu x_2) e^{-2\mu x_2}}{e^{2\mu x_2}-1}
- 2g^{-1}\big(x_2^*\big) \cdot \f{1-\ka_*}{2\mu}\cdot 2\d \big(1-e^{-2\mu x_2}\big)
\\
&\; +\big[2\pi-2g^{-1}\big(x_2^*\big)\big] \cdot \f{2\ka_*}{(2\mu)^2 x_2} \left[1-  e^{-2\mu \tilde{x}_2} - 2\mu \tilde{x}_2 \right]
\\
=
&\; 2g^{-1}\big(x_2^*\big)\cdot \f{1-\ka_*}{ 2\mu }\left[ e^{-2\mu x_2}
- 2\d \big(1-e^{-2\mu x_2}\big)\right]
\\
&\; +\big[2\pi-2g^{-1}\big(x_2^*\big)\big] \cdot \f{1}{2\mu} \cdot \f{2 (1-2\mu \tilde{x}_2 - e^{-2\mu \tilde{x}_2} )}{e^{2\mu x_2}-1}.
\end{align*}
Observe that \eqref{eqn: value of delta} together with the fact $x_2\in [0,M]$ implies $e^{-2\mu x_2} - 2\d (1-e^{-2\mu x_2}) \geq 0$.
Also note that $2(1-t-e^{-t})\in (-t^2,0)$ for all $t>0$, and $g^{-1}(x_2^*)\geq g^{-1}(x_2)= x_1$.
So we obtain that
\begin{align*}
W
\geq &\; 2x_1 \cdot \f{1-\ka_*}{ 2\mu }\left[ e^{-2\mu x_2}
- 2\d \big(1-e^{-2\mu x_2}\big)\right]
- 2\pi \cdot \f{1}{2\mu} \cdot \f{(2\mu \tilde{x}_2)^2}{e^{2\mu x_2}-1}
\\
= &\; 2x_1 \cdot \f{1-\ka_*}{ 2\mu }\left[ e^{-2\mu x_2}
- 2\d \big(1-e^{-2\mu x_2}\big)\right]
- 2\pi \cdot \f{1}{2\mu} \cdot \f{(2\mu x_2)^2}{e^{2\mu x_2}-1}
\left(\f{\d(1-\ka_*)}{\ka_*}\right)^2
\\
= &\; 2x_1 \cdot \f{1-\ka_*}{ 2\mu }\left[ e^{-2\mu x_2}
- 2\d \big(1-e^{-2\mu x_2}\big)\right]
- 2\pi \cdot \f{1-\ka_*}{2\mu} \cdot
\d^2 \big(e^{2\mu x_2}-1-2\mu x_2\big).
\end{align*}
Here we used the definitions of $\ka_*$ and $\tilde{x}_2$ in \eqref{eqn: def of kappa_*} and \eqref{eqn: def of tilde x_2} respectively.

Now given $g(0)\leq M$, in order to achieve \eqref{eqn: pa_x2 Psi less than alpha F simplified}, it suffices to guarantee that
\begin{align*}
&\; 2x_1 \cdot \f{1-\ka_*}{ 2\mu }\left[ e^{-2\mu x_2}
- 2\d \big(1-e^{-2\mu x_2}\big)\right]
- 2\pi \cdot \f{1-\ka_*}{2\mu} \cdot
\d^2 \big(e^{2\mu x_2}-1-2\mu x_2\big)
\\
\geq &\;
(1+2\d)\cdot \f{1-\ka_*}{2\mu}\cdot 2x_1 \big(e^{-2\mu x_2}-e^{-2\mu M}\big),
\end{align*}
which is equivalent to
\[
\d \cdot 2x_1 \big(1-e^{-2\mu M}\big)
+ \d^2 \cdot \pi \big(e^{2\mu x_2}-1-2\mu x_2\big)
\leq x_1 e^{-2\mu M}.
\]
By virtue of \eqref{eqn: value of delta},
\begin{align*}
\d \cdot 2x_1 \big(1-e^{-2\mu M}\big)
\leq &\; \f12 x_1 e^{-2\mu M},
\\
\d^2\cdot \pi \big(e^{2\mu x_2}-1-2\mu x_2\big)
\leq &\; \f12 x_1 e^{-2\mu M},
\end{align*}
so the above inequality follows.

This completes the proof.
\end{proof}

For the case $x_1\in [\f{\pi}{2},\pi)$, we can prove a similar non-degeneracy result as Proposition \ref{prop: non-degeneracy along level set}.
However, instead of \eqref{eqn: x_1 x_2 lies on the level set}, we will assume some smallness condition on $\mu M$.

\begin{prop}
\label{prop: non-degeneracy along level set x_1 greater than pi over 2}
Assume that $g\in \BM$ satisfies $g(0)\leq M$.
Suppose $(x_1,x_2)$ satisfies $x_2\in (0,g(0))$, $x_1 = g^{-1}(x_2)$,
and $x_1\in [\f{\pi}{2},\pi)$.
If $\mu M < \f12 \ln 2$,
it holds that
\beq
\pa_{x_2} \left[ F(x_1,x_2) -\pa_{x_2}\Psi(\pi,0)\cdot \f{1-e^{-2\mu x_2}}{2\mu x_2}\right] \leq
\d\cdot \pa_{x_2}\Psi(\pi,0)\cdot \f{d}{dx_2}\left[\f{1-e^{-2\mu x_2}}{2\mu x_2}\right],
\label{eqn: non-degeneracy along the level set x_1 greater than pi over 2}
\eeq
where
\[
\d := \f{1}{2(1-e^{-2\mu M})}-1 >0.
\]

\begin{rmk}\label{rmk: mu=0 x_1 >= pi/2}
Under the same condition, in the case $\mu = 0$, \eqref{eqn: non-degeneracy along the level set x_1 greater than pi over 2} reduces to
\[
\pa_{x_2}\big[ F(x_1,x_2;g) -\pa_{x_2}\Psi(\pi,0;g) \big]\leq
-\f{1}{4M}\cdot
\pa_{x_2} \Psi(\pi,0;g).
\]
\end{rmk}
\end{prop}

To prove this, we will need a new estimate for $\pa_{x_2}\Psi-F$ (cf.\;Lemma \ref{lem: estimate for Psi_x2 - kappa F}).
\begin{lem}
\label{lem: estimate for Psi_x2 - kappa F new}
For $x_1\in[\f{\pi}{2},\pi]$ and $x_2 > 0$ satisfying that $x_1\geq g^{-1}(x_2)$,
\begin{align*}
&\; \pa_{x_2}\Psi(x_1,x_2) - F(x_1,x_2)\\
\leq &\; -\f{1}{4\pi}\int_0^{g(0)} \mathds{1}_{\{y_2\leq x_2/2\}}\cdot \f{2y_2}{x_2}\cdot e^{-2\mu y_2} \tilde{J}\big(x_1,g^{-1}(y_2);x_2-y_2\big)\,dy_2
\\
&\; -\f{1}{4\pi}\int_0^{g(0)} \mathds{1}_{\{x_2/2\leq y_2\leq x_2\}}\cdot e^{-2\mu y_2} \tilde{J}\big(x_1,g^{-1}(y_2);x_2-y_2\big)\,dy_2
\\
&\; -\f{1}{4\pi x_2}\int_{0}^{g(0)} \mathds{1}_{\{x_2/2\leq y_2\leq x_2\}}\cdot e^{-2\mu y_2}
\int_{x_2-y_2}^{y_2} \tilde{J}\big(x_1,g^{-1}(y_2);z\big)
\,dz\,dy_2.
\end{align*}

\begin{proof}
Applying Lemma \ref{lem: integral representation of Psi_x2 and F}, we split $\pa_{x_2}\Psi - F$ as follows (cf.\;\eqref{eqn: splitting Psi_x2-F})
\begin{align*}
&\; \pa_{x_2}\Psi(x_1,x_2) - F(x_1,x_2)\\
= &\; -\f{1}{4\pi}\int_0^{g(0)} \mathds{1}_{\{y_2\leq x_2/2\}}\cdot e^{-2\mu y_2} \tilde{J}\big(x_1,g^{-1}(y_2);x_2-y_2\big)\,dy_2
\\
&\; -\f{1}{4\pi}\int_0^{g(0)} \mathds{1}_{\{x_2/2\leq y_2\leq x_2\}}\cdot e^{-2\mu y_2} \tilde{J}\big(x_1,g^{-1}(y_2);x_2-y_2\big)\,dy_2
\\
&\; +\f{1}{4\pi}\int_0^{g(0)}  \mathds{1}_{\{y_2\geq x_2\}}\cdot e^{-2\mu y_2} \tilde{J}\big(x_1,g^{-1}(y_2);y_2-x_2\big)\,dy_2
\\
&\; -\f{1}{4\pi x_2}\int_{0}^{g(0)} \mathds{1}_{\{y_2\geq x_2/2\}}\cdot e^{-2\mu y_2}
\int_{|y_2-x_2|}^{y_2} \tilde{J}\big(x_1,g^{-1}(y_2);z\big)
\,dz\,dy_2
\\
&\; +\f{1}{4\pi x_2}\int_0^{g(0)} \mathds{1}_{\{y_2\leq x_2/2\}}\cdot e^{-2\mu y_2}
\int_{y_2}^{x_2-y_2}
\tilde{J}\big(\pi,g^{-1}(y_2);z\big)
\, dz \,dy_2
\\
&\; -\f{1}{4\pi x_2}\int_{0}^{g(0)} e^{-2\mu y_2}\int_{\max\{x_2-y_2,y_2\}}^\infty
\tilde{J}\big(x_1,g^{-1}(y_2);z\big)-\tilde{J}\big(\pi,g^{-1}(y_2);z\big)
\, dz \,dy_2
\\
=: &\; \tilde{I}_1+\tilde{I}_2+\tilde{I}_3+\tilde{I}_4+\tilde{I}_5+\tilde{I}_6.
\end{align*}
Recall that $x_1\in [\f{\pi}{2},\pi]$ and $x_1\geq g^{-1}(x_2)$.
For $y_2\geq x_2$, it holds that $x_1\geq g^{-1}(y_2)$, so by Lemma \ref{lem: properties of tilde J}, $z\mapsto \tilde{J}(x_1,g^{-1}(y_2);z)$ is increasing.
This implies that
\begin{align*}
&\; \tilde{I}_3+\tilde{I}_4\\
\leq &\; \f{1}{4\pi}\int_0^{g(0)} \mathds{1}_{\{y_2\geq x_2\}}\cdot e^{-2\mu y_2} \\
&\; \qquad \cdot \left[\tilde{J}\big(x_1,g^{-1}(y_2);y_2-x_2\big)- \f{1}{x_2}
\int_{y_2-x_2}^{y_2} \tilde{J}\big(x_1,g^{-1}(y_2);y_2-x_2\big)
\,dz\right] dy_2
\\
&\; -\f{1}{4\pi x_2}\int_{0}^{g(0)} \mathds{1}_{\{x_2/2\leq y_2\leq x_2\}}\cdot e^{-2\mu y_2}
\int_{x_2-y_2}^{y_2} \tilde{J}\big(x_1,g^{-1}(y_2);z\big)
\,dz\,dy_2
\\
= &\; -\f{1}{4\pi x_2}\int_{0}^{g(0)} \mathds{1}_{\{x_2/2\leq y_2\leq x_2\}}\cdot e^{-2\mu y_2}
\int_{x_2-y_2}^{y_2} \tilde{J}\big(x_1,g^{-1}(y_2);z\big)
\,dz\,dy_2,
\end{align*}
and (cf.\;\eqref{eqn: estimate for I_5})
\[
\tilde{I}_5
\leq
\f{1}{4\pi}\int_0^{g(0)} \mathds{1}_{\{y_2\leq x_2/2\}}\cdot \f{x_2-2y_2}{x_2} \cdot e^{-2\mu y_2}
\tilde{J}\big(\pi,g^{-1}(y_2);x_2-y_2\big)\, dy_2.
\]
Using this as well as the fact that
$x_1\mapsto \tilde{J}(x_1,y_1;z)$ is decreasing on $[0,\pi]$, which also follows from \ref{lem: properties of tilde J}, we obtain that
\[
\tilde{I}_1+\tilde{I}_5
\leq
-\f{1}{4\pi}\int_0^{g(0)} \mathds{1}_{\{y_2\leq x_2/2\}}\cdot \f{2y_2}{x_2}\cdot e^{-2\mu y_2} \tilde{J}\big(x_1,g^{-1}(y_2);x_2-y_2\big)\,dy_2.
\]
Finally, the monotonicity of $x_1\mapsto \tilde{J}(x_1,g^{-1}(y_2);z)$ again implies that $\tilde{I}_6\leq 0$.

Combining these estimates yields the desired bound.
\end{proof}
\end{lem}

\begin{proof}[Proof of Proposition \ref{prop: non-degeneracy along level set x_1 greater than pi over 2}]

Thanks to the calculation in Remark \ref{rmk: equivalent formulation of the non-degeneracy}, we see that \eqref{eqn: non-degeneracy along the level set x_1 greater than pi over 2} is equivalent to
\beq
\pa_{x_2}\Psi(x_1,x_2) - F(x_1,x_2)
\leq -(1+\d)\left(\f{1-e^{-2\mu x_2}}{2\mu x_2}-e^{-2\mu x_2}\right) \pa_{x_2}\Psi(\pi,0),
\label{eqn: pa_x2 Psi less than F quantitative}
\eeq

Since $x_1 = g^{-1}(x_2)\geq \f{\pi}{2}$, for any $y_2\leq x_2$, it holds that $g^{-1}(y_2)\geq \f{\pi}{2}$, and thus
\[
[0,\pi]\subset \big[x_1-g^{-1}(y_2),x_1+g^{-1}(y_2)\big].
\]
As a result, for any $z>0$,
\[
\tilde{J}\big(x_1,g^{-1}(y_2);z\big)\geq \pi.
\]
Hence, by Lemma \ref{lem: estimate for Psi_x2 - kappa F new}, for $x_1\in[\f{\pi}{2},\pi]$ and $x_2\in (0,g(0))$,
\begin{align*}
&\; \pa_{x_2}\Psi(x_1,x_2) - F(x_1,x_2)
\\
\leq &\; -\f{1}{4\pi}\int_0^{x_2/2} \f{2y_2}{x_2}\cdot e^{-2\mu y_2} \cdot \pi \,dy_2
\\
&\;
-\f{1}{4\pi}\int_{x_2/2}^{x_2} e^{-2\mu y_2} \cdot \pi\,dy_2 -\f{1}{4\pi }\int_{x_2/2}^{x_2} e^{-2\mu y_2} \cdot \f{1}{x_2}\int_{x_2-y_2}^{y_2} \pi \,dz\,dy_2
\\
= &\; -\f{1}{4}\int_0^{x_2} \f{2y_2}{x_2}\cdot e^{-2\mu y_2} \,dy_2
\\
= &\; -\f{1}{4\mu}\cdot \f{1}{2\mu x_2}\left[1-(1+2\mu x_2)e^{-2\mu x_2}\right].
\end{align*}

On the other hand, recall \eqref{eqn: end point constant} in Lemma \ref{lem: integral formula for Psi}:
\[
\pa_{x_2}\Psi(\pi,0)
= \f{1}{4\pi}\int_0^{g(0)} e^{-2\mu y_2}\left[\tilde{J}\big(\pi, g^{-1}(y_2);y_2\big)+ 2g^{-1}(y_2) \right] dy_2,
\]
so
\[
\pa_{x_2}\Psi(\pi,0)
\leq \f{1}{4\pi}\int_0^{g(0)} e^{-2\mu y_2} (2\pi + 2\pi )\, dy_2
\leq \f{1}{2\mu}\left(1-e^{-2\mu M}\right).
\]

Combining the preceding two estimates into \eqref{eqn: pa_x2 Psi less than F quantitative}, we see that it suffices to achieve that
\[
-\f{1}{4\mu}\cdot \f{1}{2\mu x_2}\left[1-(1+2\mu x_2)e^{-2\mu x_2}\right]
\leq -(1+\d)\left(\f{1-e^{-2\mu x_2}}{2\mu x_2}-e^{-2\mu x_2}\right) \f{1}{2\mu}\left(1-e^{-2\mu M}\right),
\]
which is equivalent to
\[
e^{-2\mu M}\geq  1-\f{1}{2(1+\d)}.
\]
Then the desired statement follows.
\end{proof}

Let us summarize the results of Proposition \ref{prop: non-degeneracy along level set} and Proposition \ref{prop: non-degeneracy along level set x_1 greater than pi over 2} for later use.

\begin{cor}
\label{cor: non-degeneracy on fixed point}
Assume that $g\in \BM$ satisfies $g(0)\leq M$.
Let $\mu_1 \in (0, \f{1}{2M}\ln 2)$.
Then there exists some universal constant $C = C_{M,\mu_1}>0$ that only depends on $M$ and $\mu_1$, such that
for any $\mu\in [0,\mu_1]$ and any $(x_1,x_2)\in [0,\pi]\times (0,g(0))$ satisfying $x_1=g^{-1}(x_2)$ and
\[
F(x_1,x_2)= \pa_{x_2}\Psi(\pi,0)\cdot \f{1-e^{-2\mu x_2}}{2\mu x_2},
\]
it holds that
\begin{equation}\label{eqt: non-degeneracy on fixed point}
\pa_{x_2} \left[ F(x_1,x_2) - \pa_{x_2}\Psi(\pi,0)\cdot \f{1-e^{-2\mu x_2}}{2\mu x_2}\right]\leq - C_{M,\mu_1} x_1^{1/2}
\cdot \pa_{x_2}\Psi(\pi,0).
\end{equation}

\begin{proof}
We first consider the case $\mu>0$.
Since $\mu \leq \mu_1 < \f{1}{2M}\ln 2$, we find that for $x_2\in (0,g(0))\subset (0,M]$,
\beq
\f{d}{dx_2}\left[\f{1-e^{-2\mu x_2}}{2\mu x_2}\right] = -2\mu \cdot \f{1-(1+ 2\mu x_2)e^{-2\mu x_2}}{(2\mu x_2)^2} \leq -C\mu,
\label{eqn: bound for derivative of (1-e^(-x))/x}
\eeq
where $C>0$ is a universal constant.

If $x_1\in[0,\pi/2]$, Proposition \ref{prop: non-degeneracy along level set} and \eqref{eqn: bound for derivative of (1-e^(-x))/x} imply that
\beq
\pa_{x_2}\left[ F(x_1,x_2) -\pa_{x_2}\Psi(\pi,0)\cdot\f{1-e^{-2\mu x_2}}{2\mu x_2}\right]\leq
-\d\cdot \pa_{x_2}\Psi(\pi,0)\cdot C\mu,
\label{eqn: non-degeneracy along the level set repeat}
\eeq
where $C>0$ is universal, and where
\[
\d= \min\left\{\f{1}{4(e^{2\mu M}-1)},\, \left(\f{x_1}{2\pi}\cdot \f{e^{-2\mu M}}{e^{2\mu x_2}-1-2\mu x_2}\right)^{1/2}\right\}.
\]
Since $x_2\in (0,M]$ and $\mu \leq \mu_1$, we have $e^{-2\mu M}\geq e^{-2\mu_1 M} >\f12$ and
\[
\f{e^{-2\mu M}}{e^{2\mu x_2}-1-2\mu x_2} \geq
\f{e^{-2\mu M}}{e^{2\mu M}-1-2\mu M} = \f{e^{-4\mu M}}{1-(1+2\mu M)e^{-2\mu M}}
\geq \f{C}{(1-e^{-2\mu M})^2},
\]
where $C>0$ is universal.
Hence,
\begin{align*}
\d \mu
\geq &\; \mu\min\left\{\f{C}{1-e^{-2\mu M}},\, \f{Cx_1^{1/2}}{1-e^{-2\mu M}} \right\}
\geq
\f{C\mu x_1^{1/2}}{1-e^{-2\mu M}} \geq \f{Cx_1^{1/2}}{M},
\end{align*}
where $C>0$ is a universal constant.

If $x_1\in[\pi/2,\pi]$, then Proposition \ref{prop: non-degeneracy along level set x_1 greater than pi over 2} and \eqref{eqn: bound for derivative of (1-e^(-x))/x} imply that \eqref{eqn: non-degeneracy along the level set repeat} holds
with
\[
\delta = \frac{1}{2(1-e^{-2\mu M})} - 1.
\]
Let
\[
\r(z):= z \left(\f{1}{2(1-e^{-z})}-1\right).
\]
Observe that $\lim_{z\to 0^+}\r(z) = \f12$, and $\r(z)$ is continuous and positive on $(0,2\mu_1 M]$, so it admits a positive lower bound on $(0,2\mu_1 M]$, which only depends on $\mu_1 M$.
Hence,
\[
\d \mu
= \f{1}{2M}\cdot 2\mu M\left(\frac{1}{2(1-e^{-2\mu M})} - 1 \right)
\geq \f{1}{M}C_{\mu_1 M} \geq C_{M,\mu_1} x_1^{1/2},
\]
where $C_{M,\mu_1}>0$ depends only on $M$ and $\mu_1$.

Combining these two cases, we obtain \eqref{eqt: non-degeneracy on fixed point} from \eqref{eqn: non-degeneracy along the level set repeat} in the case $\mu>0$.

The case $\mu=0$ can be justified similarly by using Remark \ref{rmk: mu=0 x_1 <= pi/2} and Remark \ref{rmk: mu=0 x_1 >= pi/2}.
We omit the details.
\end{proof}

\begin{rmk}
\label{rmk: constraint on mu from the regularity}
Since \[
\ln 2 = \int_1^2 \f{1}{z}\,dz >  1\cdot \left(\f{3}2\right)^{-1} = \f23,
\]
it is legitimate to take $\mu_1 = \f{1}{3M}$ in this proposition.
In this case, the constant $C_{M,\mu_1}$ only depends on $M$.
\end{rmk}
\end{cor}

\subsection{Interior regularity}

Recall that the existence of a fixed point in $\mathbb{W}$ of $\bfR$ has been proved for $\mu\in[0,\mu_0]\cap [0,\f{3}{4M}]$ in Proposition \ref{prop: existence of fixed point}.
In the next proposition, we will show under the additional assumption $\mu\in [0,\f1{3M}]$ that any fixed point of $\bfR$ in $\mathbb{W}$ is in fact continuous on $\BT$ and belongs to $C_{loc}^{1,\alpha}((-\pi,\pi))$ for all $\alpha\in(0,1)$.

\begin{prop}\label{prop: first regularity result}
Let $M$ and $\mu_0$ be given by Proposition \ref{prop: a priori upper bound general m}.
Let $\BW$ be defined as in \eqref{eqt: def of function set D}.
Assume that $\mu\in [0,\mu_0]\cap [0,\f{1}{3M}]$. Let $g\in \mathbb{W}$ be a fixed point of $\bfR$, i.e., $\bfR(g)(x)=g(x)$ for all $x\in[-\pi,\pi]$.
Then $g\in C(\BT)\cap C^{1,\alpha}_{loc}((-\pi,\pi))$ for any $\alpha\in(0,1)$. Moreover, $g'(x)<0$ for $x\in(0,\pi)$, $g'(0)=0$, and $g'\in L^\infty(\BT)$ with $\|g'\|_{L^\infty(\BT)} \leq C_{M,\lam}$, where $C_{M,\lam}>0$ depends only on $M$ and $\lam$.

\begin{rmk}
One of the upper bounds $\f{1}{3M}$ in the range of $\mu$ is chosen in the spirit of Remark \ref{rmk: constraint on mu from the regularity} as a convenient choice, aiming for a uniform-in-$\mu$ bound for $\|g'\|_{L^\infty}$.
It is not the optimal one, but it may be replaced by any fixed number that is smaller than $\f{1}{2M}\ln 2$.
\end{rmk}
\begin{proof}
We only handle the case $\mu>0$.

Let $g\in \BW$ be a fixed point of $\bfR$.
We define as in \eqref{eqn: def of F_dag} that
\[
F_\dag(x_1,x_2;g) := F(x_1,x_2;g) - \pa_{x_2}\Psi(\pi,0;g)\cdot \frac{1-e^{-2\mu x_2}}{2\mu x_2}.
\]
Applying the standard elliptic theory to \eqref{eqn: equation for Psi}, we know that, for $\alpha\in(0,1)$, $\Psi\in C^{1,\alpha}_{loc}(\BT\times (0,+\infty))$, so $F\in C^{1,\alpha}_{loc}(\BT\times (0,+\infty))$ and thus $F_\dag\in C^{1,\alpha}_{loc}(\BT\times (0,+\infty))$.
By the definition of $\bfR$, that $g = \bfR(g)$ implies $F_\dag(x,g(x))\equiv 0$ on $[-\pi,\pi]$.

We first show that $g\in C(\BT)$.
Note that any function in $\mathbb{W}$ must be continuous at $x=\pm\pi$ (due to the lower and upper bounds) and at $x=0$ (due to the lower semi-continuity and the monotonicity on $[0,\pi]$).
Suppose $g$ is discontinuous at some $z_1\in(0,\pi)$.  By the monotonicity of $g$, we can define its one-sided limits at $z_1$ as
\[
g^+ := \lim_{x\to z_1^+}g(x)\quad \text{and} \quad  g^- := \lim_{x\to z_1^-}g(x).
\]
The discontinuity of $g$ at $z_1$ means $g^+ < g^-$. By the continuity of $F_\dag$,
\[
F_\dag (z_1,g^\pm) = \lim_{x\to z_1^{\pm}} F_\dag(x,g(x)) = 0.
\]
We claim that there exists some $z_2\in[g^+,g^-]\subset(0,M)$ such that
\begin{equation}\label{eqt: C1 regularity step 1}
F_\dag(z_1,z_2) = 0\quad \text{and}\quad \pa_{x_2}F_\dag(z_1,z_2) \geq  0.
\end{equation}
In fact, if $\pa_{x_2} F_\dag(z_1,g^+)\geq 0$ or $\pa_{x_2} F_\dag(z_1,g^-)\geq 0$, we can simply take $z_2 = g^+$ or $z_2=g^-$, respectively.
Otherwise, we have $\pa_{x_2} F_\dag(z_1,g^\pm)<0$.
Then it is straightforward to show that there exists some $z_2\in(g^+,g^-)$ such that \eqref{eqt: C1 regularity step 1} holds.
The claim is thus proved.
Since $g$ is strictly decreasing on $[0,\pi]$ by Proposition \ref{prop: existence of fixed point}, that $z_2\in [g^+,g^-]$ implies $z_1=g^{-1}(z_2)$ (see the definition of $g^{-1}$ in \eqref{eqn: def of g inverse}).
Now by Corollary \ref{cor: non-degeneracy on fixed point} (and also Remark \ref{rmk: constraint on mu from the regularity} as well as the assumption $\mu\leq \f{1}{3M}$), the facts $z_1=g^{-1}(z_2)$ and $F_\dag(z_1,z_2) = 0$ together imply that
\[
\pa_{x_2} F_\dag(z_1,z_2)
\leq - C_M z_1^{1/2}\cdot \pa_{x_2}\Psi(\pi,0) < 0, \]
which contradicts \eqref{eqt: C1 regularity step 1}.
Therefore, $g$ has no discontinuity point in $(0,\pi)$, so $g\in C(\BT)$.

The argument above together with Corollary \ref{cor: non-degeneracy on fixed point} and Remark \ref{rmk: constraint on mu from the regularity} in fact gives \[
\pa_{x_2}F_\dag(x,g(x)) \leq - C_{M} x^{1/2}\cdot \pa_{x_2}\Psi(\pi,0) \]
for all $x\in(0,\pi)$.
Since $g(x)\geq  \lambda (\pi^2-x^2)$ for all $x\in [-\pi,\pi]$, by Lemma \ref{lem: integral formula for Psi},
\[
\partial_{x_2}\Psi(\pi,0)
\geq \f{1}{4\pi}\int_\Om e^{-2\mu y_2}\, dy
\geq \f{1}{4\pi}\int_{-\pi}^{\pi} \int_0^{\lam(\pi^2-y_1^2)} e^{-2\mu y_2} \, dy_2\,dy_1 \geq C \lam,
\]
where $C>0$ is universal.
Here we used the fact that $\lam \pi^2 \leq M$ (see the definition of $\BW$ in \eqref{eqt: def of function set D} and also Remark \ref{rmk: universality of small lambda}) and $\mu M\leq \f13$.
Hence,
\beq
\pa_{x_2}F_\dag(x,g(x)) \leq - C_{M} \cdot \lam x^{1/2} < 0.
\label{eqn: x_2 estimate for F_dag}
\eeq
As a result, we can apply the implicit function theorem to find that
\beq
g'(x) = - \frac{\pa_{x_1}F_\dag (x,g(x))}{\pa_{x_2}F_\dag(x,g(x))}
= - \frac{\pa_{x_1}F (x,g(x))}{\pa_{x_2}F_\dag(x,g(x))}, \quad x\in(0,\pi).
\label{eqn: formula for g' implicit function thm}
\eeq
Since $\pa_{x_1}F(x_1,x_2)<0$ in $(0,\pi)\times (0,+\infty)$ by Proposition \ref{prop: monotonicity of F}, we have the strict inequality $g'(x)<0$ for $x\in(0,\pi)$.
Also, since it has been shown that $F_\dag\in C^{1,\alpha}_{loc}(\BT\times (0,+\infty))$, we can conclude that $g\in C^{1,\alpha}_{loc}((-\pi,\pi))$.

Lastly, we derive an upper bound for $|g'(x)|$.
We proceed into two cases: $x\in(0,\pi/2]$ and $x\in(\pi/2,\pi)$.
For $x\in(0,\pi/2]$, we have $g(x)\leq M$ and $g(x)\geq \lambda (\pi^2-x^2)\geq C\lambda$, where $C>0$ is universal.
By virtue of Lemma \ref{lem: upper bound for F x_1 derivative}, for $x \in (0,\pi/2]$,
\[
-\pa_{x_1}F(x,g(x))
\leq
\f{C\sin x}{g(x)} \ln \left(2 + \f{g(x)}{\sin x}\right)
+ \f{C}{g(x)} \int_0^{M} \ln \left(1 + \f{ \sin x \sin g^{-1}(y_2)}{y_2^2  + \f14\sin^2 x}\right)dy_2,
\]
where $C>0$ is universal.
To bound the integral above, we note that $g(x_1)\geq \lam (\pi^2 -x_1^2)$ for all $x_1\in [-\pi,\pi]$, so for any $y_2\geq 0$,
\[
\pi - g^{-1}(y_2)\leq \f{y_2}{\lam(\pi + g^{-1}(y_2))} \leq \f{y_2}{\lam \pi}.
\]
Hence, $\sin g^{-1}(y_2)\leq C\lam^{-1} y_2$ for any $y_2\geq 0$, where $C>0$ is universal.
This allows us to further derive that
\begin{align*}
\int_0^{M} \ln \left(1 + \f{ \sin x \sin g^{-1}(y_2)}{y_2^2  + \f14\sin^2 x}\right)dy_2
\leq &\; \int_0^{M}  \f{C\lam^{-1} \sin x \cdot y_2}{y_2^2  + \sin^2 x}\, dy_2
\\
\leq &\; C\lam^{-1} \sin x \cdot \ln\left(1+\f{M^2}{\sin^2 x}\right).
\end{align*}
Hence, for $x\in (0,\pi/2]$,
\[
-\pa_{x_1}F(x,g(x))
\leq C_{M,\lam}\cdot \sin x \cdot  \ln \left(2 + \f{M}{\sin x}\right),
\]
where $C_{M,\lam}$ depends on $M$ and $\lam$.
Combining this with \eqref{eqn: x_2 estimate for F_dag} and \eqref{eqn: formula for g' implicit function thm} yields that, for $x\in (0,\pi/2]$,
\[
|g'(x)| \leq  C_{M,\lam}\cdot x^{1/2} \ln \left(2 + \f{1}{x}\right) \leq C_{M,\lam},
\]
where $C_{M,\lam}>0$ depends on $M$ and $\lam$.
Since $g\in C^{1,\al}_{loc}((-\pi,\pi))$, this estimate applies to $x = 0$ as well, which gives $g'(0) = 0$.

As for $x\in(\pi/2,\pi)$, we decompose $\pa_{x_1}F = \pa_{x_1}F_1 + \pa_{x_1}F_2$.
By Lemma \ref{lem: upper bound for -F_1 x_1},
\begin{align*}
|\pa_{x_1}F_1(x,g(x))|
&\leq \f{C\sin x}{g(x)}
+ \f{C\sin x}{g(x)} \ln \left(1+\f{g(x)}{4\sin x}\right) \\
&\leq \frac{C\sin(\pi-x)}{\lambda(\pi^2-x^2)}  + C\sup_{z>0} z \ln \left(1+\f{1}{z}\right)
\leq C_{\lambda}.
\end{align*}
On the other hand, by Lemma \ref{lem: bounding F_2 x_1 in terms of F_2 x_2},
\begin{align*}
\left|\frac{\pa_{x_1}F_2(x,g(x))}{\pa_{x_2}F_2(x,g(x))}\right|
\leq &\; \f{g(x)\sinh g(x) \sin x}{g(x)\cosh g(x)-\sinh g(x)}
\leq C_M \cdot \frac{\sin(\pi-x)}{g(x)} \\
\leq &\; C_M \cdot \frac{\sin(\pi-x)}{\lambda(\pi^2-x^2)}\leq C_{M,\lambda}.
\end{align*}
Combining this with Lemma \ref{lem: monotonicity of F_1} and Lemma \ref{lem: monotonicity of F_2} yields that
\begin{align*}
\left|\frac{\pa_{x_1}F_2(x,g(x))}{\pa_{x_2}F(x,g(x))}\right|
\leq \left|\frac{\pa_{x_1}F_2(x,g(x))}{\pa_{x_2}F_2(x,g(x))}\right|
\leq C_{M,\lambda}.
\end{align*}
Also note that Corollary \ref{cor: non-degeneracy on fixed point} and Remark \ref{rmk: constraint on mu from the regularity} imply that
\begin{align*}
\frac{|\pa_{x_2}F(x,g(x))|}{|\pa_{x_2}F_\dag(x,g(x))|}
\leq  &\;
1 + \f{|\pa_{x_2}\Psi(\pi,0)|}{|\pa_{x_2}F_\dag(x,g(x))|}\cdot
\left|\left.\f{d}{dx_2}\right|_{x_2 = g(x)}\left[\f{1-e^{-2\mu x_2}}{2\mu x_2}\right]\right|\\
\leq &\; 1 +
\f{1}{C_{M} x^{1/2}} \cdot \mu \leq C_{M}.
\end{align*}
Hence, for $x\in(\pi/2,\pi)$, we derive from \eqref{eqn: formula for g' implicit function thm} that
\begin{align*}
|g'(x)| &\leq \frac{|\pa_{x_1}F_1(x,g(x))|}{|\pa_{x_2}F_\dag(x,g(x))|} + \frac{|\pa_{x_1}F_2(x,g(x))|}{|\pa_{x_2}F(x,g(x))|}
\cdot \frac{|\pa_{x_2}F(x,g(x))|}{|\pa_{x_2}F_\dag(x,g(x))|}\\
&\leq \frac{C_{\lambda}}{C_{M}\cdot \lam x^{1/2}} + C_{M,\lambda}\cdot C_{M}
\leq C_{M,\lambda},
\end{align*}
where $C_{M,\lambda}>0$ only depends on $M$ and $\lambda$.
This concludes the desired $L^\infty$-bound for $g'$.
\end{proof}
\end{prop}

\begin{rmk}
\label{rmk: analyticity of g}
Once the $C_{loc}^{1,\al}((-\pi,\pi))$-regularity of $g$ is achieved, some classic bootstrap argument in the elliptic free boundary problem (see e.g.\;\cite{KinderlehrerNirenbergSpruck1978}) allows us to raise the regularity of $g$ to be analytic in $(-\pi,\pi)$.
However, we are not going to justify this result here, as it can be derived more easily from the analyticity of the regular part of the patch boundary in the original coordinate; see the proof of Theorem \ref{thm: main existence theorem} at the end of Section \ref{sec: properties near the end points}.
\end{rmk}

\subsection{Properties near the end-points}
\label{sec: properties near the end points}
Let $g$ be a fixed point of $\bfR$ in $\BW$, i.e., $g(x) = \bfR(g)(x)$ for all $x\in \BT$.
In this part, we shall study properties of $g$ near the end-points $\pm \pi$.
The goal is to show that the graph of $g$ on $[-\pi,\pi]$ meets the horizontal axis at the end-points $(\pm \pi,0)$ with 45-degree angles, and moreover $g\in C^1([-\pi,\pi])$.
To achieve this, we first establish the following local expansion of $\Psi$ near the corner point $(\pi,0)$.

\begin{lem}\label{lem: local expansion}
Let $\mu\in [0,1]$ and $g\in \BW$.
Define for $r\in (0,1)$ that
    \begin{equation}\label{eqt: def of A(r;g)}
    A(r;g) := -\frac{1}{\ln r}\int_{3r}^\pi\frac{1}{z}\cdot \frac{g(\pi-z)/z}{1+(g(\pi-z)/z)^2}\, dz.
    \end{equation}
Then there exists some constant $C_\Lam>0$ that only depends on $\Lam$,
such that, for all $r\in(0,1)$ and $\alpha\in[0,\pi]$,
\begin{equation}\label{eqt: Psi local expansion}
 \left|\Psi(\pi - r\cos\alpha ,r\sin\alpha;g) - r\sin\alpha \cdot\partial_{x_2}\Psi(\pi,0;g) + \frac{\cos 2\alpha }{2\pi}A(r;g)\cdot r^2\ln r\right| \leq C_\Lam r^2,
 \end{equation}
\begin{equation}\label{eqt: Psi_x1 local expansion}
 \left|\pa_{x_1}\Psi(\pi - r\cos\alpha ,r\sin\alpha;g)  - \frac{\cos \alpha}{\pi}A(r;g)\cdot r\ln r\right| \leq C_\Lam r,
\end{equation}
and
\begin{equation}\label{eqt: Psi_x2 local expansion}
 \left|\pa_{x_2}\Psi(\pi - r\cos\alpha ,r\sin\alpha;g) - \partial_{x_2}\Psi(\pi,0;g) - \frac{\sin\alpha}{\pi}A(r;g)\cdot r\ln r\right| \leq C_\Lam r.
\end{equation}
Here, $A(r;g)$ satisfies
\beq
\liminf_{r\to 0^+} A(r;g) \geq \min\left\{\frac{2\pi\lambda}{1+(2\pi\lambda)^2}\,,\,\frac{2\pi\Lambda}{1+(2\pi\Lambda)^2}\right\}.
\label{eqn: lower bound for A r g}
\eeq
In particular, if the limit $\kappa := \lim_{r\to0^+}g(\pi-r)/r$ exists, then
\begin{equation}\label{eqn: A(r;g) limit}
\lim_{r\to 0^+} A(r;g) = \frac{\kappa}{1+\kappa^2}.
\end{equation}
\end{lem}

\begin{proof}
For $x_1\in \BT$, denote $\tilde x_1 := \pi-x_1$. Note that $g$ is $2\pi$-periodic.
Applying Lemma \ref{lem: integral formula for Psi}, we write
    \begin{align*}
    -4\pi \Psi(x_1,x_2) &= \int_{-\pi}^{\pi}\int_0^{g(y_1)} e^{-2\mu y_2}
    \left[\ln \left(\f{\cosh(x_2-y_2)- \cos(x_1-y_1)}{\cosh y_2-\cos(\pi-y_1)}\right)-x_2\right] dy_2\, dy_1\\
    &= \int_{-\pi}^{\pi}\int_0^{g(\pi-\tilde y_1)} e^{-2\mu y_2}
    \left[\ln \left(\f{\cosh(x_2-y_2)- \cos(\tilde x_1-\tilde y_1)}{\cosh y_2-\cos \tilde y_1 }\right)-x_2\right] dy_2 \, d\tilde y_1\\
    &= \int_{-\pi}^{\pi}\int_0^{g(\pi-\tilde y_1)} e^{-2\mu y_2}
    \left[\ln \left(\f{\sinh^2 \frac{x_2-y_2}{2} +\sin^2 \frac{\tilde x_1-\tilde y_1}{2} }{\sinh^2\frac{y_2}{2}+\sin^2\frac{\tilde y_1}{2}}\right)-x_2\right] dy_2 \, d\tilde y_1\\
    &=: W(\tilde x_1,x_2).
    \end{align*}
    Note that $\pa_{\tilde x_1}W(\tilde x_1,x_2) = 4\pi\pa_{x_1}\Psi(\pi-\tilde x_1,x_2)$ and $\pa_{x_2}W(\tilde x_1,x_2) = -4\pi\pa_{x_2}\Psi(\pi-\tilde x_1,x_2)$. We also define
    \[
    \hat W(\tilde x_1,x_2) := \int_{-\pi}^{\pi}\int_0^{g(\pi-\tilde y_1)} e^{-2\mu y_2}\left[\ln \left(\f{(x_2-y_2)^2+(\tilde x_1-\tilde y_1)^2}{y_2^2+\tilde y_1^2}\right)-x_2\right] dy_2 \, d\tilde y_1
    \]
    as a counterpart of $W$.
    Let $K(\tilde x_1,x_2) := W(\tilde x_1,x_2) - \hat W(\tilde x_1,x_2) $. It is not hard to check that:
    \begin{itemize}
    \item
    $\Delta K(\tilde x_1,x_2) = 0$ in $(-\pi, \pi)\times \BR$;

    \item
    $\pa_{\tilde x_1}K(0,0) = 0$ thanks to the even symmetry of $\tilde y_1\mapsto g(\pi-\tilde y_1)$.
    \item
    Moreover, since $|g(x)|\leq \Lam(\pi^2-x^2)$ for $x\in [-\pi,\pi]$, $|K(\tilde{x}_1,x_2)|\leq C_\Lam$ on $[-2,2]\times [-2,2]$, where $C_\Lam>0$ depends on $\Lam$.

    \end{itemize}
    For $r:=|(\tilde{x}_1,x_2)| < 1$, the standard gradient estimates for harmonic functions give that
    \[
    \big|K(\tilde{x}_1,x_2) - x_2 \pa_{x_2}K(0,0)\big| \leq C_\Lam r^2 ,
    \]
    and
    \[
    \big|\pa_{\tilde x_1}K(\tilde{x}_1,x_2)\big| +\big|\pa_{x_2}K(\tilde{x}_1,x_2) - \pa_{x_2}K(0,0) \big| \leq C_\Lam r.
    \]
    Therefore, it suffices to study the behavior of $\hat W$ near $(\tilde x_1,x_2)=(0,0)$.

    We first prove \eqref{eqt: Psi local expansion}. Let
    \begin{align*}
    \hat I(\tilde x_1,x_2) &:= \hat W(\tilde x_1,x_2) - x_2\pa_{x_2}\hat W(0,0)\\
    &= \int_{-\pi}^{\pi}\int_0^{g(\pi-\tilde y_1)}e^{-2\mu y_2}\left[\ln\left(\frac{(x_2-y_2)^2+(\tilde x_1-\tilde y_1)^2}{y_2^2+\tilde y_1^2}\right)+\frac{2x_2 y_2}{y_2^2 +\tilde y_1^2}\right] dy_2\, d\tilde y_1.
    \end{align*}
        Adopting the polar coordinates $(\tilde x_1,x_2) = (r\cos\alpha,r\sin\alpha)$ and $(\tilde y_1,y_2)=(\rho\cos\beta,\rho\sin\beta)$, we find
    \begin{align*}
    &\; \hat I(r\cos\alpha,r\sin\alpha) \\
    = &\; \int_{\tilde \Omega} e^{-2\mu \rho\sin\beta}\left[\ln\left(\frac{r^2+\rho^2 - 2r\rho \cos(\alpha-\beta)}{\rho^2}\right)+\frac{2r\rho\sin\alpha\sin\beta}{\rho^2}\right] \rho\, d\rho\, d\beta\\
    = &\; \int_{\tilde \Omega\cap S_{3r}} e^{-2\mu \rho\sin\beta}\left[\ln\left(1+\frac{r^2 - 2r\rho \cos(\alpha-\beta)}{\rho^2}\right)+\frac{2r\sin\alpha\sin\beta}{\rho}\right] \rho \, d\rho\, d\beta \\
    &\; + \int_{\tilde \Omega \setminus S_{3r}} e^{-2\mu \rho\sin\beta}\left[\ln\left(1+\frac{r^2 - 2r\rho \cos(\alpha-\beta)}{\rho^2}\right)+\frac{2r\sin\alpha\sin\beta}{\rho}\right] \rho \,d\rho \,d\beta \\
    =: &\; \hat I_1(r\cos\alpha,r\sin\alpha) + \hat I_2(r\cos\alpha,r\sin\alpha),
    \end{align*}
    where
    \[
    \tilde \Omega := \big\{(\rho,\beta):\; \r\geq 0,\, \b\in [0,\pi],\, \rho\cos\beta\in[-\pi,\pi],\, \rho\sin\beta \in[0,g(\pi-\rho\cos\beta)]\big\},
    \]
    and
    \[
    S_t :=\big\{(\rho,\beta):\; \r\geq 0,\, \b\in [0,\pi],\,\rho\cos\beta\in[-t,t],\,\rho\sin\beta\in[0,2\pi\Lambda t]\big\}.
    \]
    Note that for $g\in \mathbb{W}$, $g(\pi-\tilde y_1)\leq \Lambda u(\pi-\tilde y_1) \leq 2\pi\Lambda |\tilde y_1|$ for all $\tilde y_1\in[-\pi,\pi]$, so we have
    \[
    \tilde \Omega \cap S_{3r}
    \subset \big\{(\rho,\beta): \rho\in [0,C_\Lam r],\, \beta \in[0,\pi]\big\},
        \]
    and
    \begin{align*}
    \tilde \Omega \setminus S_{3r}
    = &\; \big\{(\rho,\beta): \r\geq 0,\, \b\in [0,\pi],\, \rho\cos\beta\in[-\pi,-3r]\cup[3r,\pi],\, \rho\sin\beta \in[0,g(\pi-\rho\cos\beta)]\big\}\\
    \subset &\; \big\{(\rho,\beta): \rho\in[3r,C_\Lam],\, \beta \in[0,\pi]\big\}.
    \end{align*}
    Here $C_\Lam>0$ only depends on $\Lam$.

    It then follows that     \begin{align*}
    &\; \big|\hat I_1(r\cos\alpha,r\sin\alpha)\big| \\
    \leq &\; \int_0^{C_\Lam r} \int_0^\pi\left|\r \ln\left(1+\frac{r^2 - 2r\rho \cos(\alpha-\beta)}{\rho^2}\right)+ 2r\sin\alpha\sin\beta\right| d\beta\, d\rho \\
    \leq &\; \pi\int_0^{C_\Lam r} \r \ln\left(1+\frac{r^2 + 2r\rho}{\rho^2}\right) + \r \left|\ln\left(1+\frac{r^2 - 2r\rho}{\rho^2}\right)\right|+ 2r\, d\rho \\
    = &\; \pi r^2 \int_0^{C_\Lam} s\ln\left(1+\frac{1 + 2s}{s^2}\right) + s\left|\ln\left(1+\frac{1 - 2s}{s^2}\right)\right| +2\, ds\\
    \leq &\;  C_\Lam r^2.
    \end{align*}
    As for $\hat I_2$, we have $\r \geq |\r\cos \b| \geq 3r$, so
    \[
    \left|-\f{2r}{\r}\cos (\al-\b) + \f{r^2}{\r^2}\right|\leq \f{2r}{\r} + \f{r^2}{\r^2} \leq \f79.
    \]
    By the Taylor expansion, $|\ln (1+z) - z + \f12z^2|\leq C|z|^3$ for all $|z|\leq \f79$, where $C>0$ is universal.
    Hence, we rewrite $\hat{I}_2$ as
    \begin{align*}
    &\;\hat I_2(r\cos\alpha,r\sin\alpha) \\
        = &\;\int_{\tilde \Omega \setminus S_{3r}} e^{-2\mu\rho\sin\beta}\\
    &\;\quad\cdot \left[\frac{2r}{\rho}\sin\alpha\sin\beta - \frac{2r}{\rho}\cos(\alpha-\beta) + \frac{r^2}{\rho^2} - \frac{2r^2}{\rho^2}\cos^2(\alpha-\beta) + \frac{r^3}{\rho^3}\cdot L\left(\frac{r}{\rho},\alpha,\beta\right) \right] \rho \, d\rho \, d\beta,
                        \end{align*}
    where $|L(\f{r}{\rho},\alpha,\beta)|\leq C$ for some universal $C>0$  for all $\rho\geq 3r$ and all $\alpha,\beta\in[0,\pi]$.
    Since $\tilde{y}_1\mapsto g(\pi-\tilde{y}_1)$ is even, we can further derive that
    \begin{align*}
    &\;\hat I_2(r\cos\alpha,r\sin\alpha) \\
                =  &\; \int_{\tilde \Omega \setminus S_{3r}} e^{-2\mu\rho\sin\beta}\\
     &\; \quad \cdot \left[- \frac{2r}{\rho}\cos\alpha \cos\beta - \frac{r^2}{\rho^2} \big(\cos 2\alpha\cos 2\beta+\sin 2\alpha \sin 2\beta\big) + \frac{r^3}{\rho^3}\cdot L\left(\frac{r}{\rho},\alpha,\beta\right) \right] \rho \, d\rho \, d\beta\\
    =  &\; \int_{\tilde \Omega \setminus S_{3r}} e^{-2\mu\rho\sin\beta} \cdot \frac{r^3}{\rho^2}\cdot L\left(\frac{r}{\rho},\alpha,\beta\right) d\rho \, d\beta
    \\
    &\; + r^2 \cos 2\alpha \int_{\tilde \Omega \setminus S_{3r}} \big(1-e^{-2\mu\rho\sin\beta} \big) \cdot \f{1}{\rho}\cdot\cos 2\beta\, d\rho \,d\beta\\
    &\;  - r^2\cos 2\al \int_{\tilde \Omega \setminus S_{3r}} \cos 2\beta \cdot \frac{1}{\rho}\, d\rho \, d\beta \\
    =:  &\; \hat I_{2,1}(r\cos\alpha,r\sin\alpha) + \hat I_{2,2}(r\cos\alpha,r\sin\alpha) + \hat I_{2,3}(r\cos\alpha,r\sin\alpha),
    \end{align*}
        It is clear that
    \[
    \big|\hat I_{2,1}(r\cos\alpha,r\sin\alpha)\big| \leq  Cr^3 \int_{3r}^{C_\Lam}\int_0^\pi \frac{1}{\rho^2}\,  d\beta \, d\rho \leq C r^2,
    \]
    and
    \[
    \big|\hat I_{2,2}(r\cos\alpha,r\sin\alpha)\big|
    \leq Cr^2 \int_{3r}^{C_\Lam}\int_0^\pi \mu \,  d\beta \, d\rho \leq C_\Lam r^2.
    \]
    Here we used the fact that $\mu\in[0,1]$.
    For $\hat{I}_{2,3}$, we use the left-right symmetry of $\tilde \Omega$ and $S_{3r}$ to get
    \begin{equation}\label{eqt: local expansion step 1}
    \begin{split}
    \int_{\tilde \Omega \setminus S_{3r}} \frac{\cos 2\beta}{\rho}\, d\rho \,d\beta
    &=2\int_{3r}^\pi\int_0^{g(\pi-\tilde y_1)}\frac{\tilde y_1^2-y_2^2}{(\tilde y_1^2+y_2^2)^2}\, dy_2 \, d\tilde y_1\\
        &=2\int_{3r}^\pi\frac{1}{\tilde y_1}\cdot \frac{g(\pi-\tilde y_1)/\tilde y_1}{1+(g(\pi-\tilde y_1)/\tilde y_1)^2}\, d\tilde y_1.
            \end{split}
    \end{equation}
    Hence, by \eqref{eqt: def of A(r;g)},      \[
    \hat I_{2,3}(r\cos\alpha,r\sin\alpha)
    = - r^2 \cos 2\al \int_{\tilde \Omega \setminus S_{3r}} \frac{\cos 2\beta}{\rho}\, d\rho\, d\beta
    = 2\cos 2\al\cdot r^2\ln r \cdot A(r;g).
    \]
Putting all the above estimates together leads to
\[
\left|\hat{I}(r\cos \al, r\sin \al) - 2\cos 2\al\cdot r^2\ln r \cdot A(r;g)\right|\leq C_\Lam r^2.
\]
Therefore,
    \begin{align*}
    &\;\left|\Psi(\pi - r\cos\alpha ,r\sin\alpha;g) - r\sin\alpha \cdot\partial_{x_2}\Psi(\pi,0;g) + \frac{\cos 2\al}{2\pi}A(r;g) \cdot r^2\ln r \right|\\
    = &\; \f{1}{4\pi} \left|W(r\cos\alpha ,r\sin\alpha) - r\sin\alpha \cdot\partial_{x_2}W(0,0) -  2\cos 2\alpha \cdot r^2\ln r\cdot A(r;g) \right|\\
    = &\; \f{1}{4\pi}\left|K(r\cos\alpha ,r\sin\alpha) - r\sin \al\cdot \partial_{x_2}K(0,0)  + \hat I(r\cos\alpha ,r\sin\alpha) -  2\cos 2\al\cdot r^2\ln r \cdot A(r;g)\right|\\
    \leq &\; \f{1}{4\pi}\big|K(\tilde{x}_1 ,x_2) - x_2\partial_{x_2}K(0,0)\big|  + \f{1}{4\pi}\left|\hat I(r\cos\alpha ,r\sin\alpha) -  2\cos 2\al\cdot r^2\ln r \cdot A(r;g)\right|\\
    \leq &\; C_\Lam r^2,         \end{align*}
    as desired.

    To prove \eqref{eqt: Psi_x1 local expansion}, we shall study $\pa_{\tilde x_1}\hat W(\tilde x_1,x_2)$.
    Still adopting the polar coordinates
    $(\tilde x_1,x_2) = (r\cos\alpha,r\sin\alpha)$ and $(\tilde y_1,y_2)=(\rho\cos\beta,\rho\sin\beta)$, we analogously calculate that
    \begin{align*}
    \pa_{\tilde x_1}\hat W(\tilde x_1,x_2) =&\; \int_{-\pi}^{\pi}\int_0^{g(\pi-\tilde y_1)} e^{-2\mu y_2}\cdot  \frac{2(\tilde x_1-\tilde y_1)}{(x_2-y_2)^2 +(\tilde x_1-\tilde y_1)^2}\, dy_2 \, d\tilde y_1\\
    =&\; 2\int_{\tilde \Omega} e^{-2\mu \rho\sin\beta}\cdot \frac{r\cos\alpha - \rho\cos\beta}{r^2+\rho^2 - 2r\rho \cos(\alpha-\beta)} \cdot \rho \, d\rho \, d\beta\\
    =&\; 2\int_{\tilde \Omega\cap S_{3r}} e^{-2\mu \rho\sin\beta}\frac{r\cos\alpha - \rho\cos\beta}{r^2+\rho^2 - 2r\rho \cos(\alpha-\beta)} \cdot \rho\, d\rho\, d\beta\\
    &\;+ 2\int_{\tilde \Omega\setminus S_{3r}} e^{-2\mu \rho\sin\beta} \frac{r\cos\alpha - \rho\cos\beta}{r^2+\rho^2 - 2r\rho \cos(\alpha-\beta)} \cdot \rho \, d\rho\, d\beta\\
    =: &\;\hat J_1(r\cos\alpha,r\sin\alpha) +  \hat J_2(r\cos\alpha,r\sin\alpha) ,
    \end{align*}
    where $\tilde \Omega $ and $S_{3r}$ are defined as above. Under the change of variable $s=\rho/r$, it follows that
    \begin{align*}
    \big|\hat J_1(r\cos\alpha,r\sin\alpha)\big|
    \leq
    &\; 2\int_0^\pi \int_0^{C_\Lam r} \frac{1}{(r^2+\rho^2 - 2r\rho \cos(\alpha-\beta))^{1/2}} \cdot \rho\, d\rho\, d\beta\\
    \leq
    &\; 2r \int_0^\pi \int_0^{C_\Lam} \frac{s}{(1+s^2 - 2s\cos(\alpha-\beta))^{1/2}}\, ds\, d\beta\\
    \leq
    &\; C_\Lam r \int_0^\pi \int_0^{C_\Lam} \frac{1}{|\sin (\al-\b)| + |s - \cos(\alpha-\beta)|}\, ds\, d\beta\\
    \leq
    &\; C_\Lam r \int_0^\pi \ln \left(\frac{|\sin (\al-\b)| + C_\Lam + 1}{|\sin (\al-\b)|} \right) d\beta\\
                    \leq &\; C_\Lam r.
    \end{align*}
    As for $\hat J_2$, we use the Taylor expansion as well as the left-right symmetry of $\tilde \Omega$ and $S_{3r}$ to derive that
    \begin{align*}
    &\;\hat J_2(r\cos\alpha,r\sin\alpha) \\
    =&\;2\int_{\tilde \Omega\setminus S_{3r}} e^{-2\mu \rho\sin\beta}\cdot
    \frac{-\cos\beta + \frac{r}{\rho}\cos\alpha}{1 - \frac{2r}{\rho} \cos(\alpha-\beta) +(\frac{r}{\rho})^2} \, d\rho\, d\beta\\
    =&\;2\int_{\tilde \Omega\setminus S_{3r}} e^{-2\mu \rho\sin\beta} \left[- \cos\beta + \frac{r}{\rho}\left(\cos\alpha - 2\cos\beta\cos(\alpha-\beta)\right) + \frac{r^2}{\rho^2} \cdot L_1\left(\frac{r}{\rho},\alpha,\beta\right)\right]d\rho\, d\beta\\
    =&\;2\int_{\tilde \Omega\setminus S_{3r}} e^{-2\mu \rho\sin\beta} \left[-\frac{r}{\rho}\cos\alpha \cos(2\beta) + \frac{r^2}{\rho^2} \cdot L_1\left(\frac{r}{\rho},\alpha,\beta\right)\right]d\rho \,d\beta\\
    =&\;2r^2\int_{\tilde \Omega\setminus S_{3r}} e^{-2\mu \rho\sin\beta}\cdot \frac{1}{\rho^2}\cdot L_1\left(\frac{r}{\rho},\alpha,\beta\right) d\rho\, d\beta \\
    &\;+ 2r\cos\alpha \int_{\tilde \Omega\setminus S_{3r}} \big(1-e^{-2\mu \rho\sin\beta}\big)\cdot \frac{1}{\rho}\cdot \cos 2\beta\, d\rho\, d\beta \\
    &\; - 2r\cos\alpha \int_{\tilde \Omega\setminus S_{3r}} \frac{\cos 2\beta}{\rho} \, d\rho\, d\beta\\
    =: &\; \hat J_{2,1}(r\cos\alpha,r\sin\alpha) + \hat J_{2,2}(r\cos\alpha,r\sin\alpha) + \hat J_{2,3}(r\cos\alpha,r\sin\alpha),
    \end{align*}
    where $|L_1(\f{r}{\rho},\alpha,\beta)|\leq C$ for some universal $C>0$  for all $\rho\geq 3r$ and all $\alpha,\beta\in[0,\pi]$.
    It is clear that
    \begin{align*}
    \big|\hat J_{2,1}(r\cos\alpha,r\sin\alpha)\big| \leq &\; C r^2\int_0^\pi\int_{3r}^{C_\Lam} \frac{1}{\rho^2} \, d\rho\, d\beta \leq Cr,\\
    \big|\hat J_{2,2}(r\cos\alpha,r\sin\alpha)\big| \leq &\; Cr \int_0^\pi\int_{3r}^{C_\Lam}  \mu\, d\rho\, d\beta \leq C_\Lam r.
    \end{align*}
    As for $\hat J_{2,3}$, we follow \eqref{eqt: def of A(r;g)} and \eqref{eqt: local expansion step 1} to obtain
    \[
    \hat J_{2,3}(r\cos\alpha,r\sin\alpha) = - 2r\cos \alpha\int_{\tilde \Omega \setminus S_{3r}} \frac{\cos 2\beta}{\rho}\, d\rho \, d\beta = 4\cos\alpha\cdot r\ln r \cdot A(r;g).
    \]
    Therefore, we have
    \[
    \left|\pa_{\tilde x_1}\hat W(r\cos\alpha,r\sin\alpha) - \hat J_{2,3}(r\cos\alpha,r\sin\alpha)\right| \leq C_\Lam r,
    \]
    and thus
    \begin{align*}
    &\;\left|\pa_{x_1}\Psi(\pi-r\cos\alpha,r\sin\alpha) - \frac{\cos\alpha}{\pi} A(r;g)\cdot r\ln r\right| \\
    = &\; \f{1}{4\pi}\left|\pa_{\tilde x_1}W(r\cos\alpha,r\sin\alpha) - 4\cos\alpha \cdot r\ln r\cdot A(r;g)\right|\\
    = &\; \f{1}{4\pi}\left|\pa_{\tilde x_1}K(r\cos\alpha,r\sin\alpha) + \pa_{\tilde x_1}\hat W(r\cos\alpha,r\sin\alpha) - \hat J_{2,3}(r\cos\alpha,r\sin\alpha)\right|\\
        \leq &\; C_\Lam r,
    \end{align*}
    as desired.

    One can show \eqref{eqt: Psi_x2 local expansion} in a similar way by studying $\pa_{x_2}\hat W(\tilde x_1,x_2) - \pa_{x_2}\hat W(0,0)$.         We omit the tedious details for brevity as the argument is largely the same.
    It remains to verify the claimed properties of $A(r;g)$.
    Since $g\in \mathbb{W}$, we have $\lambda u(x)\leq g(x)\leq \Lambda u(x)$ for all $x\in [-\pi,\pi]$, where $u(x) = \pi^2-x^2$ satisfies
    $\lim_{r\to 0^+} u(\pi-r)/r = 2\pi$.
    Also observe that the function $z\mapsto z/(1+z^2)$ is increasing on $[0,1]$ and decreasing on $[1,+\infty)$.
    It then follows from \eqref{eqt: def of A(r;g)} and the l'Hospital's rule that
    \begin{align*}
    \liminf_{r\to 0^+} A(r;g)
    \geq &\;
    \lim_{r\to 0^+} -\frac{1}{\ln r} \int_{3r}^\pi \frac{1}{z}\cdot \min\left\{\frac{\lam u(\pi-z)/z}{1+(\lam u(\pi-z)/z)^2},\, \frac{\Lam u(\pi-z)/z}{1+(\Lam u(\pi-z)/z)^2}\right\} dz
    \\
    = &\;
    \lim_{r\to 0^+}
    \min\left\{\frac{\lam u(\pi-3r)/(3r)}{1+(\lam u(\pi-3r)/(3r))^2},\, \frac{\Lam u(\pi-3r)/(3r)}{1+(\Lam u(\pi-3r)/(3r))^2} \right\}
    \\
    = &\;\min\left\{\frac{2\pi\lambda}{1+(2\pi\lambda)^2}\,,\,\frac{2\pi\Lambda}{1+(2\pi\Lambda)^2}\right\}.
    \end{align*}
    Moreover, if the limit $\kappa := \lim_{r\to0^+}g(\pi-r)/r$ exists, then also by the l'Hospital's rule,     \[\lim_{r\to 0^+} A(r;g) = \frac{\kappa}{1+\kappa^2}.\]
    The proof is thus completed.
\end{proof}

Recall the decomposition $\Psi=\Psi_1+\Psi_2$, where $\Psi_1$ and $\Psi_2$ are given in \eqref{eqn: representation of Psi_1} and \eqref{eqn: representation of Psi_2}, respectively.
In fact, one can show that $\Psi_2$ admits similar local expansions as those for $\Psi$ proved in Lemma \ref{lem: local expansion}.
In particular, those terms containing $\ln r$ in the expansion of $\Psi$ (see Lemma \ref{lem: local expansion}) are solely contributed by $\Psi_2$.
This result will be used in Section \ref{sec: flow field} later.

\begin{lem}\label{lem: local expansion of Psi2}
Let $\mu\in [0,1]$ and $g\in\mathbb{W}$.
Let $A(r;g)$ be defined in \eqref{eqt: def of A(r;g)} in Lemma \ref{lem: local expansion}. Then there exists some constant $C_\Lam>0$ that only depends on $\Lam$, such that, for all $r\in(0,1)$ and $\alpha\in[0,\pi]$,
\[
\left|\Psi_2(\pi - r\cos\alpha ,r\sin\alpha;g) - r\sin\alpha \cdot\partial_{x_2}\Psi_2(\pi,0;g) + \frac{\cos 2\alpha}{2\pi}A(r;g)\cdot r^2\ln r\right| \leq C_\Lam r^2,
\]
\[
\left|\pa_{x_1}\Psi_2(\pi - r\cos\alpha ,r\sin\alpha;g)  - \frac{\cos \alpha}{\pi}A(r;g)\cdot r\ln r\right| \leq C_\Lam r,
\]
and
\[
\left|\pa_{x_2}\Psi_2(\pi - r\cos\alpha ,r\sin\alpha;g) - \partial_{x_2}\Psi_2(\pi,0;g) - \frac{\sin\alpha}{\pi}A(r;g)\cdot r\ln r\right| \leq C_\Lam r.
\]

\begin{proof}
The proof is completely parallel to that of Lemma \ref{lem: local expansion}.
One only needs to replace the integral representation of $\Psi$ (see \eqref{eqn: integral representation of Psi}) by that of $\Psi_2$ (see \eqref{eqn: representation of Psi_2 old}) and derive analogously. We omit the details.
\end{proof}
\end{lem}

\begin{prop}\label{prop: g slope at pi}
Let $M$ and $\mu_0$ be given by Proposition \ref{prop: a priori upper bound general m}.
Let $\BW$ be defined as in \eqref{eqt: def of function set D}.
Assume $\mu\in [0,\mu_0]\cap [0,\f{1}{3M}]$, and let $g\in\mathbb{W}$ be a fixed point of $\bfR $.
Then $g'_-(\pi) := \lim_{x\to \pi^-}g(x)/(x-\pi) =-1$, with
\beq
\limsup_{x\to \pi^-}\left|\left(1+\f{g(x)}{x-\pi}\right) \ln(\pi-x)\right| \leq C_{\Lam},
\label{eqn: asymptotic behavior of g near end points}
\eeq
where $C_{\Lam}>0$ depends only on $\Lam$.
Moreover, $g'(\pi^-):= \lim_{x\to \pi^-}g'(x) = -1$.

By symmetry, $g_+'(-\pi) := \lim_{x\to -\pi^+} g(x)/(x+\pi) = 1$, and $g'(-\pi^+):=\lim_{x\to -\pi^+}g'(x) = 1$.
As a result, $g\in C^1([-\pi,\pi])$.

\begin{proof}
Let $g\in\mathbb{W}$ be a fixed point of $\bfR $.
By definition,
    \begin{equation}\label{eqt: corner regularity step 1}
 	\Psi(x,g(x);g)=\partial_{x_2}\Psi(\pi,0;g)\cdot \frac{1-e^{-2\mu g(x)}}{2\mu}
 	\end{equation}
  for all $x\in[-\pi,\pi]$, so for $x\in (0,\pi)$,
    \beq
 	\big|\Psi(x,g(x);g) - g(x) \partial_{x_2}\Psi(\pi,0;g) \big|
  = \partial_{x_2}\Psi(\pi,0;g) \cdot g(x)\left|\frac{1-e^{-2\mu g(x)}}{2\mu g(x)} -1\right|   \leq M\cdot \mu g(x)^2.
  \label{eqn: bound from the fixed-point equation}
 \eeq
In the last inequality, we used Lemma \ref{lem: bound for Psi_x_2}.
For $x\in (0,\pi)$, we define
\[
r(x):= \sqrt{(\pi-x)^2+g(x)^2},\quad \alpha(x) := \arctan \f{g(x)}{\pi-x}.
\]
Note that for $g\in \BW$, $r(x)\in [\pi-x,C_\Lam(\pi-x)]$ for some $C_\Lam>0$ which only depends on $\Lam$, and $\al(x) \in [\arctan (\pi\lam), \arctan(2\pi\Lam)]\subset (0,\f{\pi}{2})$ is well-defined.
By Lemma \ref{lem: local expansion}, for $x\in (0,\pi)$ satisfying that $r(x)<1$, \[
\left|\Psi(x,g(x);g) - g(x)\partial_{x_2}\Psi(\pi,0;g)
+ \frac{\cos 2\alpha(x)}{2\pi}A(r(x);g)\cdot r(x)^2\ln r(x)\right| \leq C_\Lam r(x)^2.
\]
Combining this with \eqref{eqn: bound from the fixed-point equation} yields that, whenever $r(x)<1$,
\[
\big| \cos 2\alpha(x) \cdot A(r(x);g)\cdot r(x)^2\ln r(x) \big| \leq (C_{\Lam} + \mu M) r(x)^2,
\]
i.e.,
\beq
\left|\f{1-\tan^2 \al(x)}{1+\tan^2\al(x)}\right| =
\big| \cos 2\alpha(x) \big| \leq \f{C_{\Lam}}{|A(r(x);g)\ln r(x)|}.
\label{eqn: bound for cos 2 alpha(x)}
\eeq
Here we used the assumption that $\mu M \leq \f13$.
Using \eqref{eqn: lower bound for A r g} in Lemma \ref{lem: local expansion} and the fact $\lim_{x\to \pi^-} r(x)= 0$, we find that $\lim_{x\to \pi^-} \cos 2\alpha(x) = 0$, so
\[
\lim_{x\to \pi^-} \alpha(x) = \f{\pi}{4},
\]
and
\[
g'_-(\pi) := \lim_{x\to \pi^-} \f{g(x)}{x-\pi} = -\lim_{x\to \pi^-} \tan \al(x) = -1.
\]
This further implies that $\lim_{x\to \pi^-} \f{r(x)}{\pi-x} = \sqrt{2}$, and by \eqref{eqn: A(r;g) limit} in Lemma \ref{lem: local expansion}, $\lim_{x\to \pi^-} A(r(x);g) = \f12$.
Combining them with \eqref{eqn: bound for cos 2 alpha(x)} yields that
\[
\limsup_{x\to \pi^-}\big|(1-\tan\al(x)) \ln (\pi-x)\big| \leq C_{\Lam},
\]
i.e.,
\[
\limsup_{x\to \pi^-}\left|\left(1+\f{g(x)}{x-\pi}\right) \ln (\pi-x)\right| \leq C_{\Lam}.
\]

    By Proposition \ref{prop: first regularity result}, we differentiate \eqref{eqt: corner regularity step 1} in $(0,\pi)$ to obtain
    \[
 	\pa_{x_1}\Psi(x,g(x);g) + g'(x)\pa_{x_2}\Psi(x,g(x);g) = \partial_{x_2}\Psi(\pi,0;g)\cdot g'(x) \cdot e^{-2\mu g(x)},
 	\]
    which gives
    \beq
 	g'(x) = - \frac{\pa_{x_1}\Psi(x,g(x);g)}{\pa_{x_2}\Psi(x,g(x);g) - \pa_{x_2}\Psi(\pi,0;g)\cdot e^{-2\mu g(x)}}.
    \label{eqn: formula for g' in terms of Psi}
 	\eeq
    Indeed, this is valid because, by Corollary \ref{cor: non-degeneracy on fixed point} as well as the fact $\Psi = x_2 F$, for $x\in(0,\pi)$,
    \begin{align*}
    &\;\pa_{x_2}\Psi(x,g(x);g) - \partial_{x_2}\Psi(\pi,0;g)\cdot e^{-2\mu g(x)} \\
    = &\;g(x)\pa_{x_2}F(x,g(x);g) + F(x,g(x);g) - \partial_{x_2}\Psi(\pi,0;g)\cdot e^{-2\mu g(x)} \\
    =&\; g(x)\left( \pa_{x_2} F(x,g(x);g) + \partial_{x_2}\Psi(\pi,0;g)\cdot \frac{1-(1+2\mu g(x))e^{-2\mu g(x)}}{2\mu g(x)^2} \right)\\
    =&\; g(x) \left.\pa_{x_2}\left(F(x_1,x_2;g) -\partial_{x_2}\Psi(\pi,0;g)\cdot \frac{1-e^{-2\mu x_2}}{2\mu x_2} \right)\right|_{(x_1,x_2) = (x,g(x))}\\
    \leq &\; - C_{M,\mu_0}  x^{1/2} g(x)\cdot \pa_{x_2}\Psi(\pi,0;g)\\
    <&\; 0.
    \end{align*}
    It has been proved that     $\lim_{x\to\pi^-}\alpha(x) = \pi/4$. We apply \eqref{eqt: Psi_x1 local expansion}, \eqref{eqt: Psi_x2 local expansion}, and \eqref{eqn: A(r;g) limit} in Lemma \ref{lem: local expansion} to derive that
    \[
    \lim_{x\to \pi^-} \frac{\pa_{x_1}\Psi(x,g(x);g)}{r(x) |\ln r(x)| }= \lim_{x\to \pi^-} \frac{\pa_{x_1}\Psi(\pi-r(x)\cos\alpha(x),r(x)\sin\alpha(x);g)}{r(x) |\ln r(x)| } = - \frac{\cos \f{\pi}{4}}{2\pi},
    \]
    and
    \begin{align*}
    &\;\lim_{x\to \pi^-} \frac{\pa_{x_2}\Psi(x,g(x);g) - \pa_{x_2}\Psi(\pi,0;g)}{ r(x) |\ln r(x)| } \\
    = &\;\lim_{x\to \pi^-} \frac{\pa_{x_2}\Psi(\pi-r(x)\cos\alpha(x),r(x)\sin\alpha(x);g) - \pa_{x_2}\Psi(\pi,0;g)}{ r(x)|\ln r(x)| }= - \frac{\sin \f{\pi}{4}}{2\pi}.
    \end{align*}
    Then it follows from \eqref{eqn: formula for g' in terms of Psi} that
    \begin{align*}
    \lim_{x\to \pi^-} g'(x)
        = &\; - \frac{\displaystyle \lim_{x\to \pi^-}  \frac{\pa_{x_1}\Psi(x,g(x);g)}{r(x)|\ln r(x)|}}{\displaystyle \lim_{x\to \pi^-}\frac{\pa_{x_2}\Psi(x,g(x);g)-\partial_{x_2}\Psi(\pi,0;g)}{r(x)|\ln r(x)|} + \partial_{x_2}\Psi(\pi,0;g)\cdot \frac{1- e^{-2\mu g(x)}}{r(x)|\ln r(x)|}}\\
        = &\; -\f{-\frac{\cos \f{\pi}{4}}{2\pi}}{-\frac{\sin \f{\pi}{4}}{2\pi} + 0}
    = - 1,
    \end{align*}
    as claimed.

    By symmetry, one can readily show that $g_+'(-\pi) = g'(-\pi^+) = 1$.

    Recall that by Proposition \ref{prop: first regularity result}, $g\in C_{loc}^1((-\pi,\pi))$.
    Combining this with the facts proved above,     we can conclude that $g\in C^1([-\pi,\pi])$.
\end{proof}

\begin{rmk}
The justification of $g'_-(\pi) = -1$, $g_+'(-\pi) = 1$ and \eqref{eqn: asymptotic behavior of g near end points} does not rely on the continuity or the higher regularity of $g$.
Besides, it is also clear from \eqref{eqn: lower bound for A r g} (also see \eqref{eqn: A(r;g) limit}) and \eqref{eqn: bound for cos 2 alpha(x)} that the statements
$g'_-(\pi) = g'(\pi^-) = -1$ and $g_+'(-\pi) = g'(-\pi^+) = 1$ proved here are consequences of the local geometry of $\Omega$ near the corner points $(\pm \pi,0)$ instead of the global behavior of $g$ in $\BT$.
\end{rmk}
\end{prop}
	 	
We conclude this section with the proof of Theorem \ref{thm: main existence theorem}, which states the existence of $m$-fold symmetric V-states with 90-degree corners and also characterizes their boundaries and angular velocities.

\begin{proof}[Proof of Theorem \ref{thm: main existence theorem}]
Let $M$ and $\mu_0$ be a pair of universal constants that makes Proposition \ref{prop: a priori upper bound general m} hold.
Define $m_0:=\lceil\max\{3M,\mu_0^{-1}\}\rceil$.
For any $m \geq m_0$, we have $\mu:=\f{1}{m}\in (0,\mu_0] \cap [0,\f{1}{3M}]$.
By Proposition \ref{prop: existence of fixed point}, there exists a fixed point of the mapping $\bfR$ in $\BW$, where $\bfR$ was defined in \eqref{eqn: def of R(g)} and where $\BW$ was given in \eqref{eqt: def of function set D}.
Take an arbitrary fixed point and denote it by $g\in \BW$.
In view of \eqref{eqn: def of g}, we define $f$ in terms of $g$ by \eqref{eqn: transformation from g to f}, i.e.,
\beq
f(\th):= e^{-\f{1}{m}g(m\th-\pi)}\quad (\th\in \BT). \label{eqn: def of f in terms of g}
\eeq
It is then straightforward to verify the statements of the theorem.
\begin{itemize}
\item That $f\in \BM_0$ follows from the definition of $\BM_0$ in \eqref{eqn: function set M_0}, definition of $\BM$ in \eqref{eqn: function set tilde M_0}, and \eqref{eqn: def of f in terms of g}.
Thanks to Proposition \ref{prop: first regularity result} on the regularity of $g$, we find $f\in C(\BT)$ and $f\in C^{1,\al}_{loc}((0,\f{2\pi}{m}))$ for any $\al\in (0,1)$.

\item
The fixed-point relation $g\equiv \bfR(g)$ implies \eqref{eqn: constraint for the boundary curve in Phi}.
Since $g\in C(\BT)$, this further implies \eqref{eqn: constraint for Phi along the image of the patch boundary}, which is equivalent to \eqref{eqn: constraint satisfied by phi along the patch boundary final} due to transformation \eqref{eqn: def of Phi}.
Hence, the statement \eqref{statement: rotating patch} follows.

\item
The property \eqref{statement: lower and upper bound for f} follows from the \eqref{eqn: def of f in terms of g} and the fact $g\in \BW$, where $M$, $\Lam$, and $\lam$ are the universal constants given in the definition of $\BW$ (also see the remark that follows \eqref{eqt: def of function set D_0}).

\item The $C^1$-regularity of $f$ in $[0,\f{2\pi}{m}]$ follows from Proposition \ref{prop: g slope at pi} and \eqref{eqn: def of f in terms of g},
which is the first part of the property \eqref{statement: local analyticity of f and piecewise C^1}.
The analyticity of $f$ in $(0,\f{2\pi}{m})$ will be justified later in this proof.

\item The property \eqref{statement: upper bound for f'} follows from Proposition \ref{prop: first regularity result} and \eqref{eqn: def of f in terms of g}.
In particular, $\|f'\|_{L^\infty(\BT)}\leq C_{M,\lam}$, where $C_{M,\lam}>0$ only depends on $M$ and $\lam$ and is thus universal.

\item The property \eqref{statement: derivative at the end points} follows from Proposition \ref{prop: g slope at pi} and \eqref{eqn: def of f in terms of g}.
\end{itemize}

In addition, for any $f\in \BM_0\cap C(\BT)$ that satisfies the stated properties in the theorem, by \eqref{eqn: constraint satisfied by phi along the patch boundary final}, the angular velocity of the rotating vortex patch is $-\pa_{x_1}\phi(1,0)$.
If we define $g$ in terms of $f$ by \eqref{eqn: def of g}, then $g\in \BW$, with the parameters $M$, $\Lam$, and $\lam$ in the definition of $\BW$ here given by the property \eqref{statement: lower and upper bound for f}.
Let $\Phi$ and $\Psi$ be given as in Lemma \ref{lem: integral formula for Psi}.
Thanks to \eqref{eqn: def of Phi} and Lemma \ref{lem: integral formula for Psi},
\[
-\pa_{x_1}\phi(1,0) = -\f{1}{m}\pa_{x_2}\Phi(\pi,0) = \f{1}{2\mu m}-\f{1}{m}\pa_{x_2}\Psi(\pi,0).
\]
Note that $\mu m = 1$.
By \eqref{eqn: end point constant}, $\pa_{x_2}\Psi(\pi,0) \geq C\lam$ for some universal constant $C>0$.
On the other hand, $\pa_{x_2}\Psi(\pi,0)\leq M$ due to Lemma \ref{lem: bound for Psi_x_2}.
Therefore,
\[
-\pa_{x_1}\phi(1,0) -\f12 = -\f{1}{m}\pa_{x_2}\Psi(\pi,0) \in \left[-\f{M}{m},-\f{C\lam}{m}\right].
\]
This proves \eqref{eqn: estimate for angular velocity}.
Since $D_0\subset \overline{B_1}$, \[
a = -\pa_{x_1}\phi(1,0) = -\pa_{x_1}\big(\G*\mathds{1}_{D_0}\big)(1,0)
= \f{1}{2\pi}\int_{D_0} \f{1-y_1}{|(1,0)-(y_1,y_2)|^2}\,dy > 0.
\]
We remark that the claim $a>0$ is also implied by \cite{GomezSerranoParkShiYao2021}.

The analyticity of $f$ in $(0,\f{2\pi}{m})$ can be justified by using standard bootstrap argument in the elliptic free boundary problem (see e.g.\;\cite{KinderlehrerNirenbergSpruck1978}, and also \cite{WangZhangZhou2026} which is more specifically on the boundary regularity of rotating vortex patches).
For completeness, we sketch the proof here by following \cite[Theorem 5.2]{HuangTong2025}.
Define \[
\psi_\dagger(x) := \phi(x) - \phi(1,0)
+ \f12 \pa_{x_1}\phi(1,0)\big(1-|x|^2\big).
\]
By \eqref{eqn: equation for phi} and \eqref{eqn: constraint satisfied by phi along the patch boundary final}, it satisfies
\[
-\Delta \psi_\dag = \mathds{1}_{D_0} + 2\pa_{x_1}\phi(1,0),\quad
\psi_\dag \equiv 0\mbox{ along }\pa D_0,
\]
which is in the form of an unstable free boundary problem \cite{AnderssonShahgholianWeiss2012,MonneauWeiss2007}.
Here the free boundary is exactly the patch boundary $\pa D_0 = \{(f(\th)\cos\th,f(\th)\sin \th):\, \th\in \BT\}$.
We have shown that $f\in C^{1,\al}_{loc}((0,\f{2\pi}{m}))$ for any $\al\in (0,1)$.
Fix $\th_0\in (0,\f{2\pi}{m})$ and denote $x_0:=(f(\th_0)\cos\th_0,f(\th_0)\sin\th_0)$.
To show $f$ is analytic near $\th_0$, in view of \cite[Theorem 3.1']{KinderlehrerNirenbergSpruck1978}, it suffices to verify that for some $\d>0$, \begin{enumerate}[(a)]
\item \label{claim: non-degeneracy of psi_dag along the patch boundary} $\na \psi_\dag \neq 0$ along $B_\d(x_0)\cap\pa D_0$;
\item \label{claim: psi_dag is C^1} $\psi_\dag\in C^1(B_\d(x_0))$;
\item \label{claim: psi_dag is C^2 on both sides up to the boundary} In $B_\d(x_0)$, $\psi_\dag$ is $C^2$ on either side $\pa D_0$ up to $\pa D_0$.
\end{enumerate}

To achieve \eqref{claim: non-degeneracy of psi_dag along the patch boundary}, we will actually prove
\beq
\na \psi_\dagger(x)\neq 0\quad \mbox{for all }x\in \left\{\big(f(\th)\cos\th,f(\th)\sin \th\big):\, \th\in \left(0,\f{2\pi}{m}\right)\right\}.
\label{eqn: non-degeneracy of psi_dag along the patch boundary}
\eeq
We derive by \eqref{eqn: def of zeta}, \eqref{eqn: def of Phi}, \eqref{eqn: relation between derivatives of phi and Phi}, and Lemma \ref{lem: integral formula for Psi} that, for $x = (x_1,x_2)\in \BH_+$ with $x_2>0$,
\begin{align*}
\psi_\dagger\circ \zeta^{-1}(x)
= &\; -\f{1}{m^2}\Phi(x)
+ \f1{2m} \pa_{x_2}\Phi(\pi,0) \big(1-|\zeta^{-1}(x)|^2\big)\\
= &\; -\f{1}{m^2}\Psi(x)
+ \f1{2m} \pa_{x_2}\Psi(\pi,0) \big(1-e^{-2x_2/m}\big)\\
= &\; -\f{x_2}{m^2}\left[F(x_1,x_2) - \pa_{x_2}\Psi(\pi,0) \cdot \f{1-e^{-2x_2/m}}{2x_2/m}\right].
\end{align*}
Thanks to the range of $\mu = \f{1}{m}$, we can apply Corollary \ref{cor: non-degeneracy on fixed point} to find that, if $x_1\in (-\pi,\pi)\setminus \{0\}$,
\[
\left.\f{\pa}{\pa x_2}\right|_{x_2 = g(x_1)}\left[\psi_\dagger\circ \zeta^{-1}(x_1,x_2)\right] >0,
\]
which implies $\na \psi_\dagger(x)\neq 0$ for all $x\in \{(f(\th)\cos\th,f(\th)\sin \th):\, \th\in (0,\f{\pi}{m})\cup (\f{\pi}{m},\f{2\pi}{m})\}$.
For $x_* = (f(\f{\pi}{m})\cos\f{\pi}{m},f(\f{\pi}{m})\sin \f{\pi}{m})$, we consider $\psi_\dag$ in the disk $B_{r_*}$ with $r_*:=|x_*| = f(\pi/m)$.
Since $\psi_\dag >0$ in $D_0$ by the maximum principle, we apply the Hopf lemma in $B_{r_*}$ to find $\f{\pa\psi_\dag}{\pa n}(x_*)<0$, where $n=n(x_*)$ denotes the outer unit normal vector at $x_*$ with respect to $D_0$.
This concludes \eqref{eqn: non-degeneracy of psi_dag along the patch boundary}.

The fact \eqref{claim: psi_dag is C^1} is trivial since $\psi_\dag\in C^{1,\al}_{loc}(\BR^2)$ for any $\al\in (0,1)$ due to the regularity theory for elliptic equations.

To verify \eqref{claim: psi_dag is C^2 on both sides up to the boundary}, it suffices to show that, for some $\d>0$, the original stream function $\phi$ defined by \eqref{eqn: equation for phi} is uniformly $C^2$ in $B_\d(x_0)\cap D_0$ and $B_\d(x_0)\setminus \overline{D_0}$; note that the uniform continuity of $\na^2\phi$ on one side of $\pa D_0$ allows us to extend it continuously up to $\pa D_0$ on that side.
The argument is similar to that in the proof of \cite[Theorem 5.2]{HuangTong2025}, so we only sketch it here.

Take $\d_1\in (0,1)$ to be sufficiently small so that $(\cos \f{2k\pi}{m},\sin \f{2k\pi}{m})\not \in B_{4\d_1}(x_0)$ for all $k= 0,\cdots, m-1$, and
moreover, for any $x\in B_{\d_1}(x_0)$, both $B_{2\d_1}(x)\cap D_0$ and $B_{2\d_1}(x)\setminus \overline{D_0}$ are simply connected, with $B_{2\d_1}(x)\cap \pa D_0$ being a $C^{1,\al}$-curve; the latter can be achieved since $\pa D_0$ has $C^{1,\al}$-regularity near $x_0$.
By a direct calculation, for any $x\in B_{\d_1}(x_0)$,
\[
\na^2 \phi(x) = -\f12 \r_{D_0}(x)\cdot Id + \f{1}{2\pi}\pv\int_{D_0}\s(x-y)\,dy,
\]
where
\[
\r_{D_0}(x) :=
\begin{cases}
1,&\mbox{if }x\in D_0\cap B_{\d_1}(x_0),\\
\f12 ,&\mbox{if }x\in \pa D_0\cap B_{\d_1}(x_0),\\
0, &\mbox{if }x\in B_{\d_1}(x_0)\setminus \overline{D_0},
\end{cases}
\]
and with $z = (z_1,z_2)\in \BR^2$,
\[
\s(z) := \f{1}{|z|^4}
\begin{pmatrix}
z_1^2-z_2^2& 2z_1z_2\\
2z_1z_2 & z_2^2-z_1^2
\end{pmatrix}.
\]
Here we used the fact that $\pa D_0$ in $B_{\d_1}(x_0)$ has $C^{1,\al}$-regularity.
Hence, for any $x\in B_{\d_1}(x_0)$,
\begin{align*}
|\na^2 \phi(x)|
\leq &\; C + \f{1}{2\pi}\int_{D_0\setminus B_{2\d_1}(x)}|\s(x-y)|\,dy
+ \f{1}{2\pi}\left|\pv\int_{D_0\cap B_{2\d_1}(x)}\s(x-y)\,dy\right|.
\end{align*}
We shall use the facts $D_0\subset B_2(x)$ and $|\s(z)|\leq C|z|^{-2}$ to bound the second term on the right-hand side, while we will apply \cite[Lemma 5.3]{HuangTong2025} to handle the last term. This leads to
\[
|\na^2 \phi(x)| \leq C + C|\ln \d_1|,
\]
where $C$ depends on $\al$ and the local $C^{1,\al}$-norm of $\pa D_0$.
See the details in the proof of \cite[Theorem 5.2]{HuangTong2025}.
Hence, $\na^2 \phi$ is uniformly bounded in $B_{\d_1}(x_0)$. This together with the implicit function theorem further implies that $\pa D_0\cap B_{\d_1}(x_0)$ has $C^{1,1}$-regularity.

Next we show the uniform continuity of $\na^2\phi$ on either side of $\pa D_0$ in $B_\d(x_0)$ for some $\d>0$.
Take $\d_2:=\d_1^2 \ll \d_1\ll 1$ and consider arbitrary $x,x'\in B_{\d_2}(x_0)\cap D_0$.
Denote $l:= |x-x'|$, $z:=(x+x')/2$, and let $r := l^{1/2}\gg l$.
In fact, $r\leq (2\d_2)^{1/2} < 2\d_1$.
We then bound $|\na^2 \phi(x)-\na^2\phi(x')|$ as follows:
\begin{align*}
&\; \big|\na^2 \phi(x)-\na^2 \phi(x')\big| \\
\leq &\; \f{1}{2\pi}\int_{D_0\setminus B_r(z)}|\s(x-y)-\s(x'-y)|\,dy
\\
&\; + \f{1}{2\pi}\int_{D_0\cap (B_r(z) \Delta B_r(x))}|\s(x-y)|\,dy
+ \f{1}{2\pi}\int_{D_0\cap (B_r(z)\Delta B_r(x'))}|\s(x'-y)|\,dy
\\
&\; + \f{1}{2\pi}\left|\pv\int_{D_0\cap B_r(x)} \s(x-y)\,dy - \pv\int_{D_0\cap B_r(x')} \s(x'-y)\,dy\right|.
\end{align*}
Here $A_1\Delta A_2$ denotes the symmetric difference of the two sets $A_1$ and $A_2$.
Using the $C^{1,1}$-regularity of $\pa D_0$ in $B_{2\d_2}(x_0)$ (note that $2\d_2\ll \d_1$ by assumption) and applying \cite[Lemma 5.3]{HuangTong2025} again, we can eventually obtain that
\[
\big|\na^2 \phi(x)-\na^2 \phi(x')\big|\leq C|x-x'|^{1/2},
\]
where $C$ depends on the local $C^{1,1}$-norm of $\pa D_0$.
Again, the details can be found in the proof of \cite[Theorem 5.2]{HuangTong2025}.
The case $x,x'\in B_{\d_2}(x_0)\setminus \overline{D_0}$ can be handled similarly.
This proves the desired claim in \eqref{claim: psi_dag is C^2 on both sides up to the boundary}.

Now we can apply \cite[Theorem 3.1']{KinderlehrerNirenbergSpruck1978} to conclude that $\pa D_0$ is analytic in a neighborhood of $x_0$, and thus $f$ is analytic in a neighborhood of $\th_0$.
Since $\th_0\in (0,\f{2\pi}{m})$ is arbitrary, $f$ is analytic in $(0,\f{2\pi}{m})$.
This proves the second part of the property \eqref{statement: local analyticity of f and piecewise C^1} of Theorem \ref{thm: main existence theorem}.
Let us mention that, the analyticity of $f$ together with \eqref{eqn: def of g} immediately implies the analyticity of $g$, which was claimed in Remark \ref{rmk: analyticity of g}.

Lastly, by Corollary \ref{cor: M=4 and mu leq 1/12}, $(M,\mu_0) = (4,\f{1}{12})$ makes Proposition \ref{prop: a priori upper bound general m} hold.
So the results of this theorem hold for all $m\geq \max\{3M, \mu_0^{-1}\} = 12$ with $M=4$.
Therefore, $m_0\leq 12$.
\end{proof}

\begin{rmk}
\label{rmk: optimality of m=12 explained}

As is mentioned in Remark \ref{rmk: claim optimality m = 12}, let us remark on the claimed threshold $12$ for $m_0$.

In our argument, we have encountered three constraints related to $\mu$ (or equivalently, on $m$) and $M$ for different purposes.
\begin{enumerate}
\item The first one comes from Proposition \ref{prop: a priori upper bound general m}, which is for establishing an a priori upper bound $M$ for $g$ preserved by $\bfR$.
There, $M$ and $\mu$ need to satisfy an implicit nonlinear relation, essentially \eqref{eqn: conditions for a valid pair of M and mu}.
In Remark \ref{rmk: (M,mu) pair with larger mu}, we numerically studied for each $\mu$ of interest the range of $M$ so that it can be a valid upper bound; see Table \ref{tab: smallest M for various mu}.

\item The second one lies in Proposition \ref{prop: upper barrier new}, which is to prove existence of an upper barrier.
There we need $2\mu M < c_*$.
It has been shown in Remark \ref{rmk: universality of large Lambda} that $c_*>\f{3}{2}$.

\item The third one stems from Proposition \ref{prop: non-degeneracy along level set x_1 greater than pi over 2}, which is needed for proving continuity (and higher regularity) of the fixed point $g\in \BW$ of $\bfR$, so that not only \eqref{eqn: constraint for the boundary curve in Phi} but also the original condition \eqref{eqn: constraint for Phi along the image of the patch boundary} for the rotating vortex patch is satisfied.
In this constraint, we need $2\mu M < \ln 2$.
Since $\ln 2 < 1$, this is more restrictive than the previous one.
\end{enumerate}

We summarize these constraints in Table \ref{tab: M for various mu plus the threshold due to regularity} for $\mu\in \{\f{1}{12},\f1{11}, \f{1}{10},\f{1}{9}\}$.
The middle column shows the first constraint for $M$ to become an a priori upper bound preserved by $\bfR$, which is directly taken from Table \ref{tab: smallest M for various mu}, while the right column shows the third constraint mentioned above due to Proposition \ref{prop: non-degeneracy along level set x_1 greater than pi over 2}.
We do not include the case $\mu\in \{\f18,\cdots, \f12\}$ here because one cannot even obtain a valid upper bound for $g$ in those cases.
Clearly, among the cases listed here, only when $\mu = \f{1}{12}$, the constraints combined give a non-empty admissible range of $M$, which allows us to complete the proof.
It is noteworthy that the upper bound $M = 4$ we proved for $\mu\in [0,\f{1}{12}]$ in Corollary \ref{cor: M=4 and mu leq 1/12} lies in the valid range of $M$ in the case $\mu = \f1{12}$, namely from $3.9526$ to $4.1589$
(see the second line of Table \ref{tab: M for various mu plus the threshold due to regularity}).
\end{rmk}

\begin{table}
\centering
\begin{tabular}{ccc}
\toprule[1pt]
& Constraints for $M$ & Constraints on $M$
\\ $\mu$& to serve as an upper bound & for proving the regularity of $g$
\\ & (Remark \ref{rmk: (M,mu) pair with larger mu}) & (Proposition \ref{prop: non-degeneracy along level set x_1 greater than pi over 2})
\\[3.5pt]
\toprule[0.5pt]
$\f{1}{12}$ & $3.9526 \lessapprox M \lessapprox 16.8312$ & $M< \f{\ln 2}{2\mu} \approx 4.1589$
\\[2.5pt]
$\f{1}{11}$ & $4.1461 \lessapprox M \lessapprox 14.4242$ & $M< \f{\ln 2}{2\mu} \approx 3.8123$
\\[2.5pt]
$\f{1}{10}$ & $4.4467 \lessapprox M \lessapprox 11.9733$ & $M< \f{\ln 2}{2\mu} \approx 3.4657$
\\[2.5pt]
$\f{1}{9}$ & $5.0364  \lessapprox M \lessapprox 9.3048$ & $M< \f{\ln 2}{2\mu} \approx 3.1192$
\\[2.5pt]
\bottomrule[1pt]
\end{tabular}
\caption{Constraints on $M$ for various $\mu$'s of interest.
The middle column shows the ranges of $M$ so that it can serve as a valid upper bound (see Remark \ref{rmk: (M,mu) pair with larger mu}) in the fixed-point argument, while the right column provides the range of $M$ in order to prove the regularity of the fixed point $g\in \BW$ (see Proposition \ref{prop: non-degeneracy along level set x_1 greater than pi over 2}).
Only in the case $\mu = \f{1}{12}$, the two constraints combined give a non-empty admissible range of $M$.}
\label{tab: M for various mu plus the threshold due to regularity}
\end{table}

\section{The Flow Field in the Co-rotating Frame}
\label{sec: flow field}

In this section, we study the modified stream function as well as the stationary flow field generated by the rotating vortex patch in the co-rotating frame.
The goal is to prove Theorem \ref{thm: heteroclinic trajectory} and Theorem \ref{thm: flow field}.

Let us start from the conclusion of Theorem \ref{thm: main existence theorem}.
Take an arbitrary $m\geq m_0$ ($m\in \BZ_+$), and assume that $f\in \BM_0\cap C(\BT)$ satisfies all the conditions of Theorem \ref{thm: main existence theorem}. Define $D_0 = D_0(f)$ by \eqref{eqn: form of D_0} and define $\phi$ by \eqref{eqn: equation for phi}.
By Theorem \ref{thm: main existence theorem}, $D_0(f)$ gives a uniformly rotating vortex patch in the sense that \eqref{eqn: constraint satisfied by phi along the patch boundary final} is satisfied, with the angular velocity being $-\pa_{x_1}\phi(1,0)$.
Then the modified stream function is defined as in \eqref{eqn: modified stream function repeat}, i.e.,
\beqo
\phi_\dag(x)  =\phi(x) - \frac{1}{2}\pa_{x_1}\phi(1,0)|x|^2.
\eeqo
In view of \eqref{eqn: flow field in the corotating frame}, its level sets correspond to streamlines in the co-rotating frame. Motivated by this, for given $V\in \BR$, we define $\s_V$ to be the set of all points $x\in \BR^2$ such that \beq
\phi(x) -\phi(1,0) + \frac{1}{2}\partial_{x_1}\phi(1,0)(1-|x|^2) = V, \label{eqn: general level set}
\eeq
i.e., $\s_V$ is the level set of $\phi_\dag$ with the value $\phi(1,0) - \frac{1}{2}\partial_{x_1}\phi(1,0) + V$ (also see \eqref{eqn: general level set in main thm}).

It is convenient to adopt the transform introduced in Section \ref{sec: conformal mapping decomposing stream function} once again.
Let $\mu := \f{1}{m}\in (0,1]$, and define $g$ in terms of $f$ by \eqref{eqn: def of g}.
Then
\begin{itemize}
\item $g\in \BM\cap C^1(\BT)$ (see the definition in \eqref{eqn: function set tilde M_0}). Also, $g\in \BW$ (see the definition in \eqref{eqt: def of function set D}), where the parameters $M$, $\Lam$, and $\lam$ in the definition of $\BW$ are given by those in Theorem \ref{thm: main existence theorem}.

\item $g$ is a fixed point of $\bfR$.

\item In addition,
\beq
g'_-(\pi)  = g'(\pi^-) = -1,\mbox{ and }g_+'(-\pi) = g'(-\pi^+) = 1.
\label{eqn: end-point properties of g}
\eeq
\end{itemize}
Note that such $g$ is not necessarily obtained by the fixed-point argument as above.

Let $\Phi(x_1,x_2) = \Phi(x_1,x_2;g)$ be given as in \eqref{eqn: def of Phi}.
For $(x_1,x_2)\in \BH = \BT\times \BR$, denote
\beq
\begin{split}
\Phi_\dagger(x_1,x_2) := &\;\Phi(x_1,x_2) - \partial_{x_2}\Phi(\pi,0)\cdot \frac{1-e^{-2\mu x_2}}{2\mu} \\
= &\; -m^2 \left[\phi_\dag\circ \zeta^{-1}(x_1,x_2)-\left(\phi(1,0) - \frac{1}{2}\pa_{x_1}\phi(1,0)\right)\right].
\end{split}
\label{eqn: def of Phi_dag}
\eeq
Then \eqref{eqn: general level set} becomes
\beq\label{eqn:level set on plane}
\Phi_\dag(x_1,x_2) = -\f{1}{\mu^2} V \quad \text{if and only if $(x_1,x_2)\in \Sigma_V:= \zeta(\s_V)$}.
\eeq

In what follows, we shall first study the solution set $\Sigma_V$ in the lower half-plane.
Denote the lower half-plane by $\BH_- := \BT\times (-\infty,0]$.
For this purpose, we consider $\Phi(x_1,-x_2;g)$ with $x_2\geq 0$ and relate it with $\Psi_2$ in the following lemma.

\begin{lem}
\label{lem: Phi in the lower half plane}
Let $\mu>0$ and $g\in \BM$.
For $x_2\geq 0$,
\[
\Phi(x_1,-x_2;g) = \Psi_2(x_1,x_2;g) + Bx_2,
\]
where $\Psi_2$ was given in \eqref{eqn: representation of Psi_2 old} (also see \eqref{eqn: equation for Psi_2}), and where
\[
B = B_{\mu,g} := \frac{1}{4\pi \mu} \int_{-\pi}^{\pi} e^{-2\mu g(y_1)} \, dy_1.
\]
As a result, $\partial_{x_2}\Phi(\pi,0) = -\partial_{x_2}\Psi_2(\pi,0) - B$.
\begin{proof}
Note that \eqref{eqn: integral representation of Phi} in Lemma \ref{lem: integral formula for Psi} is still valid.
Starting from that and using \eqref{eqn: integral of a ln function on torus}, we derive that
\begin{align*}
&\;\Phi(x_1,-x_2;g) \\
= &\; \frac{1}{4\pi}\int_{-\pi}^{\pi}\int_{g(y_1)}^{\infty}e^{-2\mu y_2} \left[\ln\left(\frac{\cosh(x_2+y_2)-\cos(x_1-y_1)}{\cosh(0-y_2) - \cos(\pi-y_1)}\right) + x_2\right] dy_2\, dy_1 \\
= &\;\frac{1}{4\pi}\int_{-\pi}^{\pi}\int_0^{\infty} e^{-2\mu y_2} \left[\ln\left(\frac{\cosh(x_2+y_2)-\cos(x_1-y_1)}{\cosh(0-y_2) - \cos(\pi-y_1)}\right) - x_2\right] dy_2 \, dy_1 \\
&\;- \frac{1}{4\pi}\int_{-\pi}^{\pi}\int_0^{g(y_1)} e^{-2\mu y_2} \left[\ln\left(\frac{\cosh(x_2+y_2)-\cos(x_1-y_1)}{\cosh(0-y_2) - \cos(\pi-y_1)}\right) - x_2\right] dy_2 \, dy_1 \\
&\; + \frac{x_2}{2\pi}\int_{-\pi}^{\pi}\int_{g(y_1)}^{\infty}e^{-2\mu y_2} \, dy_2\, dy_1
\\
=&\; - \frac{1}{4\pi}\int_{-\pi}^{\pi}\int_0^{g(y_1)}e^{-2\mu y_2} \left[\ln\left(\frac{\cosh(x_2+y_2)-\cos(x_1-y_1)}{\cosh(0-y_2) - \cos(\pi-y_1)}\right) - x_2\right] dy_2 \, dy_1 \\
&\; + \frac{x_2}{4\pi \mu} \int_{-\pi}^{\pi} e^{-2\mu g(y_1)}\, dy_1\\
= &\;\Psi_2(x_1,x_2;g) + x_2\cdot \frac{1}{4\pi \mu} \int_{-\pi}^{\pi} e^{-2\mu g(y_1)} \, dy_1.
\end{align*}
In the last equality, we used \eqref{eqn: representation of Psi_2 old}.
\end{proof}
\end{lem}

For $(x_1,x_2)\in \BH_+$, define
\beq
\Psi_\dagger(x_1,x_2) := \Phi_\dag(x_1,-x_2) = \Psi_2(x_1,x_2) + Bx_2 - \big(\partial_{x_2}\Psi_2(\pi,0) +  B\big)\cdot \f{e^{2\mu x_2}-1}{2\mu}.
\label{eqn: def of Psi_dag}
\eeq
With Lemma \ref{lem: Phi in the lower half plane}, \eqref{eqn:level set on plane} restricted to $\BH_-$ can be rewritten as
\beq
\Psi_\dag(x_1,x_2) = -\f{1}{\mu^2} V, \quad  (x_1,-x_2)\in \Sigma_V \cap \BH_-.
\label{eqn: level set on lower plane further simplify}
\eeq

\begin{rmk}
\label{rmk: the limit case mu = 0 in H_-}
It is also legitimate to consider the formal limiting case $\mu \to 0^+$, where $\Psi_\dag$ reduces to $\Psi_\dag(x_1,x_2) := \Psi_2(x_1,x_2) - x_2\partial_{x_2}\Psi_2(\pi,0)$.
\end{rmk}

In order to study the monotonicity of $\Psi_\dag$, we prove the following lemma.
\begin{lem}\label{lem: derivative comparison in lower plane}
Let $g\in \BM$ and let $\Psi_2$ be given in \eqref{eqn: representation of Psi_2 old}.
Then
$\partial_{x_2}\Psi_2(x_1,x_2) \leq \partial_{x_2}\Psi_2(\pi,x_2)< \partial_{x_2}\Psi_2(\pi,0)$ for all $x_1\in[-\pi,\pi]$ and $x_2>0$.
Besides, $\partial_{x_2}\Psi_2(x_1,0) < \partial_{x_2}\Psi_2(\pi,0)$ for all $x_1\in (-\pi,\pi)$.

\begin{proof}
To show $\partial_{x_2}\Psi_2(x_1,x_2)\leq \partial_{x_2}\Psi_2(\pi,x_2)$,
we derive from \eqref{eqn: representation of Psi_2 old} that
\begin{align*}
    \partial_{x_2}\Psi_2(x_1,x_2) &= - \frac{1}{4\pi}\int_{-\pi}^{\pi}\int_0^{g(y_1)}e^{-2\mu y_2} \left(\frac{\sinh(x_2+y_2)}{\cosh(x_2+y_2)- \cos(x_1-y_1)} - 1\right) dy_2 \, dy_1\\
    &= - \frac{1}{4\pi}\int_0^{g(0)}e^{-2\mu y_2} \int_{-g^{-1}(y_2)}^{g^{-1}(y_2)}\left(\frac{\sinh(x_2+y_2)}{\cosh(x_2+y_2)- \cos(x_1-y_1)} - 1\right) dy_1 \, dy_2
    \\
    &= - \frac{1}{4\pi}\int_0^{g(0)}e^{-2\mu y_2} \left(\tilde{J}\big(x_1,g^{-1}(y_2);x_2+y_2\big) - 2g^{-1}(y_2)\right)dy_2.
\end{align*}
where $\tilde{J}$ was defined in \eqref{eqn: def of tilde J}.
Since $g^{-1}(y_2)\in [0,\pi]$, by Lemma \ref{lem: properties of tilde J},
$x_1\mapsto \tilde{J}(x_1,g^{-1}(y_2);x_2+y_2)$ is decreasing on $[0,\pi]$, so $\partial_{x_2}\Psi_2(x_1,x_2)\leq \partial_{x_2}\Psi_2(\pi,x_2)$.
One can similarly justify $\pa_{x_2}\Psi_2(x_1,0) < \pa_{x_2}\Psi_2(\pi,0)$ for all $x_1\in (-\pi,\pi)$ since $g^{-1}(y_2)$ is not identically equal to $0$ or $\pi$ for $y_2\in [0,g(0)]$.

Next, we show that $\partial_{x_2}\Psi_2(\pi,x_2)< \partial_{x_2}\Psi_2(\pi,0)$ for all $x_2>0$. Let us now use another representation of $\Psi_2$ (see \eqref{eqn: representation of Psi_2} and \eqref{eqn: Poisson kernel}):
\[\Psi_2(x_1,x_2) = \int_{-\pi}^{\pi} P(x_1-y_1,x_2)h(y_1) \, dy_1 = \frac{1}{2\pi}\int_{-\pi}^{\pi} \frac{\sinh x_2}{\cosh x_2-\cos(x_1-y_1)}\cdot h(y_1) \, dy_1.\]
We then have
\[\partial_{x_2}\Psi_2(\pi,x_2) = \frac{1}{2\pi}\int_{-\pi}^{\pi} \frac{1+\cos y_1 \cosh x_2}{(\cosh x_2+\cos y_1)^2}\cdot h(y_1) \, dy_1,\]
and in particular, \[\partial_{x_2}\Psi_2(\pi,0) = \frac{1}{2\pi}\int_{-\pi}^{\pi} \frac{1}{1+\cos y_1}\cdot h(y_1)\, dy_1.\]
Note that the latter integral is well-defined, since by Lemma \ref{lem: monotonicity of h}, $h(\pm\pi)=h'(\pm\pi)=0$ and $h\in C^{1,\alpha}(\BT)$ for any $\alpha\in(0,1)$.
It follows that
\begin{align*}
&\; \partial_{x_2}\Psi_2(\pi,x_2)- \partial_{x_2}\Psi_2(\pi,0) \\
= &\; \frac{1}{2\pi}\int_{-\pi}^{\pi} \left(\frac{1+\cos y_1\cosh x_2}{(\cosh x_2 +\cos y_1)^2}- \frac{1}{1+\cos y_1}\right)h(y_1) \, dy_1\\
= &\; -\frac{1}{2\pi}\int_{-\pi}^{\pi} \frac{(\cosh x_2-1)(\cosh x_2+ 1 + \cos y_1 - \cos^2 y_1)}{(\cosh x_2+\cos y_1)^2(1+\cos y_1)}\cdot h(y_1) \, dy_1\\
< &\; 0.
\end{align*}
The lemma is thus proved.
\end{proof}
\end{lem}

The monotonicity of $\Psi_\dag$ then follows immediately.
\begin{lem}\label{lem: monotonicity in lower plane}
$\partial_{x_1}\Psi_\dagger(x_1,x_2)<0$ for $(x_1,x_2)\in(0,\pi)\times [0,+\infty)$, and $\partial_{x_1}\Psi_\dagger(x_1,x_2)=0$ for $(x_1,x_2)\in\{0,\pi\}\times [0,+\infty)$.

$\partial_{x_2}\Psi_\dagger(x_1,x_2)<0$ for $(x_1,x_2)\in\BT\times (0,+\infty)$, and $\partial_{x_2}\Psi_\dagger(x_1,0)<0$ for $x_1\in (-\pi,\pi)$.

As a result, $\Psi_\dag$ has no critical points in $\BT\times [0,+\infty)\setminus \{(\pm \pi,0)\}$.
\begin{proof}
Since $\partial_{x_1}\Psi_\dagger(x_1,x_2) = \partial_{x_1}\Psi_2(x_1,x_2)$, the claims on $\partial_{x_1}\Psi_\dagger$ follow from Lemma \ref{lem: monotonicity of F_2}.

We use Lemma \ref{lem: derivative comparison in lower plane} to directly calculate that, for $x_2>0$,
\begin{align*}
\partial_{x_2}\Psi_\dagger(x_1,x_2) &= \partial_{x_2}\Psi_2(x_1,x_2) + B - \left(\partial_{x_2}\Psi_2(\pi,0) +  B\right)\cdot e^{2\mu x_2}\\
&\leq
\partial_{x_2}\Psi_2(x_1,x_2) - \partial_{x_2}\Psi_2(\pi,0)
< 0,
\end{align*}
as desired.
The case $x_2 = 0$ and $x_1\in (-\pi,\pi)$ can be justified similarly using Lemma \ref{lem: derivative comparison in lower plane}.
\end{proof}
\end{lem}

In view of \eqref{eqn: constraint for the boundary curve in Phi} and \eqref{eqn:level set on plane},  the graph of $g$ (which was denoted by $\g$) is a subset of $\Sigma_V$ with $V = 0$.
It forms a heteroclinic streamline that connects the two end-points $(\pm \pi,0)$ in the upper half-plane $\BH_+$, and thanks to Proposition \ref{prop: g slope at pi}, it meets the horizontal axis at $(\pm \pi,0)$ with 45-degree angles.
In the next proposition, we will show that in the lower half-plane $\BH_-$, $\Sigma_0$ can be described by another heteroclinic streamline, which also connects $(\pm \pi,0)$ and meets the horizontal with 45-degree angles.

\begin{prop}\label{prop: 0-level set in lower plane}
Let $\mu>0$ and $g\in \mathbb{W}$ be defined as in the beginning of this section.
Let $\Phi$ and $\Phi_\dag$ be given as in \eqref{eqn: def of Phi} and \eqref{eqn: def of Phi_dag}, respectively.
For $V= 0$, there exists a unique function $\hat g\in \BM$ (see the definition of $\BM$ in \eqref{eqn: function set tilde M_0}) such that
$\Sigma_0 \cap \BH_- = \{(x_1,-\hat{g}(x_1)):\, x_1\in \BT\}$ (see the definition of $\Sigma_0$ in \eqref{eqn:level set on plane}).
Moreover, $\hat g$ satisfies the following:
\begin{enumerate}[(i)]
\item $\hat g'(0)=0$, and $\hat g'(x_1)<0$ for $x_1\in(0,\pi)$.
\item $\hat g'_-(\pi) = \hat g'(\pi^-) = -1$, $\hat g'_+(-\pi) = \hat g'(-\pi^+) = 1$, and
\beq
\limsup_{x\to \pi^-}\left|\left(1+\f{\hat{g}(x)}{x-\pi}\right) \ln(\pi-x)\right| \leq C_{M,\Lam},
\label{eqn: asymptotic behavior of hat g near end points}
\eeq
where $C_{M,\Lam}>0$ depends on $M$ and $\Lam$.
\item $\hat g\in C^1([-\pi,\pi])$ and $\hat{g}$ is analytic in $(-\pi,\pi)$.
\end{enumerate}

\begin{proof}
First observe that $\Psi_\dag(x_1,0) = \Psi_2(x_1,0) = h(x_1)$.
By Lemma \ref{lem: monotonicity of h}, $h(x_1)>0$ for $x_1\in(-\pi,\pi)$ and $h(\pm \pi)=0$.
On the other hand, it is straightforward to verify that $\lim_{x_2\to+\infty}\Psi_\dag(x_1,x_2) = -\infty$.
Moreover, Lemma \ref{lem: monotonicity in lower plane} states that $\partial_{x_2}\Psi_\dagger(x_1,x_2)<0$ for all $x_1\in [-\pi,\pi]$ and $x_2>0$. As a consequence, for each $x_1\in[-\pi,\pi]$, there is a unique $\hat g(x_1)\in[0,+\infty)$ such that $\Psi_\dagger(x_1,\hat g(x_1)) = 0$, i.e., \eqref{eqn: level set on lower plane further simplify} holds with $V=0$. This proves the existence of the desired function $\hat g$.
It is clear that $\hat{g}(\pm \pi) = 0$, and when $x_1\neq \pm \pi$, $\hat{g}(x_1)>0$.
Moreover, for any given $\e>0$, $\Psi_\dag(\pi,\e)<0$.
By the continuity of $\Psi_\dag$, there exists $\d>0$, such that $\Psi_\dag(x_1,\e)<0$ for all $x_1\in [\pi-\d,\pi]$.
Hence, $\hat{g}(x_1)\in [0,\e]$ whenever $x_1\in [\pi-\d,\pi]$, which implies that $\lim_{x_1\to \pi^-}\hat{g}(x_1) = 0$.
By evenness, we also have $\lim_{x_1\to -\pi^+}\hat{g}(x_1) = 0$.

By the implicit function theory and Lemma \ref{lem: monotonicity in lower plane}, \[\hat g'(x_1) = - \frac{\partial_{x_1}\Psi_\dagger(x_1,\hat g(x_1))}{\partial_{x_2}\Psi_\dagger(x_1,\hat g(x_1))} \leq 0,\quad x_1\in [0,\pi].\]
This means $\hat g\in \BM$.
In particular, since $\partial_{x_1}\Psi_\dagger(0,x_2) = 0$ for all $x_2>0$ and $\partial_{x_1}\Psi_\dagger(x_1,x_2) <0$ for all $(x_1,x_2)\in(0,\pi)\times (0,+\infty)$, we have $\hat g'(0)=0$ and $\hat g'(x_1)<0$ for $x_1\in(0,\pi)$.

Since $\Psi_2(x)$ is harmonic in $\BH_+$ (see \eqref{eqn: equation for Psi_2}), it is analytic in $x=(x_1,x_2)$ for $x_2>0$, and so is $\Psi_\dagger(x)$. Hence, by the implicit function theory, we can conclude that $\hat g$ is analytic in $(-\pi,\pi)$.

One can use \eqref{eqn: end-point properties of g}, Lemma \ref{lem: local expansion of Psi2}, and the level set condition \eqref{eqn: level set on lower plane further simplify} to show $\hat g'_-(\pi)  = \hat g'(\pi^-) = -1$, $\hat g'_+(-\pi)  = \hat g'(-\pi^+) = 1$ and \eqref{eqn: asymptotic behavior of hat g near end points}, following exactly the same idea of proving Proposition \ref{prop: g slope at pi} by Lemma \ref{lem: local expansion}. We omit the details for the sake of brevity.
This further implies $\hat g\in C^1([-\pi,\pi])$.
\end{proof}

\begin{rmk}
One may also consider the case $\mu=0$ (see Remark \ref{rmk: the limit case mu = 0 in H_-}), whose proof would be completely parallel.
We omit the details. \end{rmk}
\end{prop}

We are ready to prove Theorem \ref{thm: heteroclinic trajectory}.

\begin{proof}[Proof of Theorem \ref{thm: heteroclinic trajectory}]
As is stated above, we define $D_0 = D_0(f)$ by \eqref{eqn: form of D_0} and define $\phi$ by \eqref{eqn: equation for phi}.
Let $g$ be defined in terms of $f$ by \eqref{eqn: def of g}; $g$ satisfies the properties listed at the beginning of this section.
Let the modified stream function $\phi_\dag$ be defined as in \eqref{eqn: modified stream function repeat}.
Let $\hat{g} = \hat{g}(\th)$ be uniquely determined in Proposition \ref{prop: 0-level set in lower plane}, and let (cf.~\eqref{eqn: transformation from g to f})
\[
\hat{f}(\th) :=  e^{\f{1}{m}\hat{g}(m\th-\pi)}\quad (\th \in \BT).
\]
Then the existence, uniqueness, and the claimed properties of $\hat{f}$ follow immediately from Proposition \ref{prop: 0-level set in lower plane} as well as \eqref{eqn: def of zeta} and \eqref{eqn: def of Phi}.
We omit the details, but only note that
\beq
\phi_\dag(x) - \phi_\dag(1,0)
= 0 \quad \mbox{along }\pa D_0\cup \pa \hat{D}_0.
\label{eqn: the streamlines are the level sets of phi dag}
\eeq
Note that by definition, $\phi_\dag(1,0) = \phi(1,0) - \frac{1}{2}\pa_{x_1}\phi(1,0)$.
One can also use \eqref{eqn: def of Phi_dag}, \eqref{eqn: def of Psi_dag}, and Lemma~\ref{lem: monotonicity in lower plane} to show that $x\cdot \na \phi_\dag >0$ for all $|x|>1$, which further implies that
\beq
\phi_\dag(x) - \phi_\dag(1,0)
> 0 \quad \mbox{for all }x\in \BR^2\setminus \overline{\hat{D}_0}.
\label{eqn: phi dag is positive outside the streamlines}
\eeq

It is straightforward to derive from \eqref{eqn: equation for phi}, \eqref{eqn: estimate for angular velocity}, and the fact that $a= -\pa_{x_1}\phi(1,0)>0$ that
\beq
-\D\phi_\dag = \mathds{1}_{D_0}+2\pa_{x_1}\phi(1,0)
\begin{cases}
>0,& \mbox{if }x\in D_0,\\
<0,& \mbox{if }x\in \big(\overline{D_0}\big)^c.
\end{cases}
\label{eqn: equation for phi_dag}
\eeq
Then by the maximum principle and \eqref{eqn: the streamlines are the level sets of phi dag}, we find that
\beq
\phi_\dag(x) - \phi_\dag(1,0)
\begin{cases}
> 0, & \mbox{if }x\in D_0,\\
< 0, & \mbox{if }x\in \hat{D}_0\setminus\overline{D_0}.
\end{cases}
\label{eqn: sign of phi dag in and out the two streamlines}
\eeq
Combining this with \eqref{eqn: the streamlines are the level sets of phi dag} and \eqref{eqn: phi dag is positive outside the streamlines}, we proved the desired claim on the sign of $\phi_\dag(x) - \phi_\dag(1,0)$.
\end{proof}

Next we turn to investigate the flow field $v$ in the co-rotating frame.
For that purpose, we also need to study \eqref{eqn:level set on plane} in the upper half-plane $\BH_+$.
Thanks to \eqref{eqn: decomposition of Phi} in Lemma \ref{lem: integral formula for Psi}, $\Phi_\dag$ defined in \eqref{eqn: def of Phi_dag} can be equivalently written as
\beq
\Phi_\dag (x_1,x_2;g)
= \Psi(x_1,x_2;g) - \partial_{x_2}\Psi(\pi,0;g) \cdot \frac{1-e^{-2\mu x_2}}{2\mu}.
\label{eqn: equivalent representation of Phi dag}
\eeq
We can show the monotonicity of $\Phi_\dag$ as follows.

\begin{lem}
\label{lem: monotonicity of Psi dag tilde}
Let $g\in \BM$.
Let $M$ and $\Lam$ be the constants given in Theorem \ref{thm: main existence theorem}.
Let $\Phi_\dag$ be defined as in \eqref{eqn: def of Phi_dag}.
Then the following holds.
\begin{enumerate}
\item
$\partial_{x_1}\Phi_\dag(x_1,x_2)<0$ for $(x_1,x_2)\in(0,\pi)\times [0,+\infty)$,
and $\partial_{x_1}\Phi_\dag(x_1,x_2)=0$ for $(x_1,x_2)\in\{0,\pi\}\times [0,+\infty)$.

\item $\partial_{x_2}\Phi_\dag(0,x_2)<0$ for any $x_2 \geq g(0)$.

\item
Suppose $m\geq 4$.
If $\pa_{x_2}\Phi_\dag(\pi,x_2^*;g)\leq 0$ for some $x_2^*\geq g(0)$, then $\pa_{x_2}\Phi_\dag(x_1,x_2;g)< 0$ in $\BT\times (x_2^*,+\infty)$.
If the strict inequality holds in the condition, i.e.,  $\pa_{x_2}\Phi_\dag(\pi,x_2^*;g)< 0$ for some $x_2^*\geq g(0)$, then
$\pa_{x_2}\Phi_\dag(x_1,x_2;g)< 0$ in $\BT\times [x_2^*,+\infty)$.

\end{enumerate}

\begin{proof}
Observe that $\partial_{x_1}\Phi_\dag(x_1,x_2) = \partial_{x_1}\Psi(x_1,x_2)$.
When $x_2>0$, the first claim follows from Proposition \ref{prop: monotonicity of F}.
If $x_2 = 0$, by definition, $\Phi_\dag(x_1,0) = h(x_1)$, so the claim follows from Lemma \ref{lem: monotonicity of h}.

To justify the second statement, we calculate with \eqref{eqn: equivalent representation of Phi dag} that
\[
\pa_{x_2} \Phi_\dag (x_1,x_2;g) = \pa_{x_2}\Psi(x_1,x_2;g) - \partial_{x_2}\Psi(\pi,0;g)e^{-2\mu x_2}.
\]
It suffices to show $\pa_{x_2}\Psi(0,x_2;g)\leq 0$ for all $x_2 \geq g(0)$.
We derive from \eqref{eqn: integral representation of Psi} that
\begin{align*}
\pa_{x_2}\Psi(0,x_2)
= &\; -\f{1}{4\pi}\int_0^{g(0)}\int_{-g^{-1}(y_2)}^{g^{-1}(y_2)} e^{-2\mu y_2}
\left[\f{\sinh(x_2-y_2)}{\cosh(x_2-y_2)- \cos y_1}
-1\right] dy_1\,dy_2
\\
= &\; -\f{1}{4\pi}\int_0^{g(0)} e^{-2\mu y_2}
\left[4\arctan\left(\coth\f{x_2-y_2}{2}\tan\f{g^{-1}(y_2)}{2}\right) - 2g^{-1}(y_2)\right] dy_2 < 0.
\end{align*}
Here we used the fact $x_2 \geq g(0)$ and the calculation in \eqref{eqn: anti-derivative of Poisson}.

To show the last claim, we start from deriving an upper bound for $\pa_{x_2}\Psi(\pi,x_2)$.
Since $\pa_{x_2}\Psi(x_1,x_2)$ is harmonic in $\BT\times (g(0),+\infty)$ and even in $x_1$, we find that for all $x_2 > x_2^*\geq g(0)$,
\begin{align*}
\pa_{x_2}\Psi(\pi,x_2)
= &\; \int_\BT P(\pi-y_1,x_2-x_2^*)\cdot \pa_{x_2}\Psi(y_1,x_2^*)\,dy_1\\
= &\; \int_\BT P(\pi-y_1,x_2-x_2^*)\big[\pa_{x_2}\Psi(y_1,x_2^*)-\pa_{x_2}\Psi(\pi,x_2^*)\big]\,dy_1 \\
&\;
+ \int_\BT P(\pi-y_1,x_2-x_2^*)\cdot \pa_{x_2}\Psi(\pi,x_2^*)\,dy_1.
\end{align*}
Here $P$ is the Poisson kernel defined in \eqref{eqn: Poisson kernel}.
We will show in Lemma \ref{lem: monotonicity of Psi_x_2 in x_1} below that, as long as $x_2 \geq g(0)$, $x_1\mapsto \pa_{x_2}\Psi(x_1,x_2)$ is increasing on $[0,\pi]$, so by the evenness, $\pa_{x_2}\Psi(y_1,x_2^*)-\pa_{x_2}\Psi(\pi,x_2^*)\leq 0$ for all $y_1\in \BT$.
On the other hand, $P(\pi-y_1,x_2-x_2^*) \geq P(\pi,x_2-x_2^*) >0$.
Hence,
\[
\pa_{x_2}\Psi(\pi,x_2)
\leq \int_\BT P(\pi,x_2-x_2^*)\big[\pa_{x_2}\Psi(y_1,x_2^*)-\pa_{x_2}\Psi(\pi,x_2^*)\big]\,dy_1
+ \pa_{x_2}\Psi(\pi,x_2^*).
\]
Observe that (see \eqref{eqn: integral representation of Psi})
\[
\int_\BT \pa_{x_2}\Psi(x_1,x_2^*)\,dx_1
= -\f{1}{4\pi}\int_\Omega e^{-2\mu y_2}\int_\BT \left[\f{\sinh(x_2^*-y_2)}{\cosh(x_2^*-y_2)-\cos(x_1-y_1)}-1\right] dx_1\, dy
= 0.
\]
We thus obtain that
\[
\pa_{x_2}\Psi(\pi,x_2)
\leq \big[1- 2\pi P(\pi,x_2-x_2^*)\big]\pa_{x_2}\Psi(\pi,x_2^*)
= \f{2}{e^{x_2-x_2^*}+1}\cdot \pa_{x_2}\Psi(\pi,x_2^*).
\]

By \eqref{eqn: equivalent representation of Phi dag},
\begin{align*}
\pa_{x_2}\Phi_\dag (\pi,x_2)
=&\; \pa_{x_2}\Psi(\pi,x_2) - \partial_{x_2}\Psi(\pi,0) \cdot e^{-2\mu x_2}
\\
\leq &\; \f{2}{e^{x_2-x_2^*}+1}\cdot \pa_{x_2}\Psi(\pi,x_2^*) - \partial_{x_2}\Psi(\pi,0) \cdot e^{-2\mu x_2}
\\
= &\; \f{2}{e^{x_2-x_2^*}+1}\left[\pa_{x_2}\Psi(\pi,x_2^*)
-\partial_{x_2}\Psi(\pi,0) \cdot e^{-2\mu x_2^*}\right]\\
&\;+ \partial_{x_2}\Psi(\pi,0) \cdot e^{-2\mu x_2}\left[\f{2e^{2\mu (x_2-x_2^*)}}{e^{x_2-x_2^*}+1}- 1\right].
\end{align*}
Notice that $\pa_{x_2}\Psi(\pi,0)>0$,
\[
\pa_{x_2}\Psi(\pi,x_2^*)
-\partial_{x_2}\Psi(\pi,0) \cdot e^{-2\mu x_2^*} = \pa_{x_2}\Phi_\dag (\pi,x_2^*) \leq 0,
\]
and, with $m\geq 4$ (i.e., $\mu \leq \f{1}{4}$) and $x_2-x_2^*>0$,
\[
\f{2e^{2\mu (x_2-x_2^*)}}{e^{x_2-x_2^*}+1} - 1
\leq \f{2e^{\f12 (x_2-x_2^*)}}{e^{x_2-x_2^*}+1} - 1 < 0.
\]
Hence, we can conclude that $\pa_{x_2}\Phi_\dag (\pi,x_2) < 0$ for all $x_2 > x_2^*$.
The case of the strict inequality $\pa_{x_2}\Phi_\dag (\pi,x_2^*) < 0$ can be handled similarly, which we omit here.
Finally, applying Lemma \ref{lem: monotonicity of Psi_x_2 in x_1} again yields that $\pa_{x_2}\Phi_\dag (x_1,x_2) < 0$ for all $x_1\in \BT$ and $x_2 > x_2^*$.
\end{proof}
\end{lem}

We used the following result in the proof above.
\begin{lem}\label{lem: monotonicity of Psi_x_2 in x_1}
If $x_2\geq g(0)$, $x_1\mapsto \pa_{x_2}\Psi(x_1,x_2)$ is increasing on $[0,\pi]$.
\begin{proof}

Recall that we have shown in \eqref{eqn: formula for Psi_x_1} that
\[
\pa_{x_1}\Psi(x_1,x_2)
= \f{1}{4\pi}\int_0^{g(0)} e^{-2\mu y_2} \cdot \ln \left(\f{\cosh (x_2-y_2)-\cos(x_1-g^{-1}(y_2))}{\cosh(x_2-y_2)-\cos(x_1 + g^{-1}(y_2))}\right)dy_2.
\]
So we can calculate that
\begin{align*}
&\; \partial_{x_1 x_2}^2\Psi (x_1, x_2) \\
= &\;
\frac{1}{4\pi}\int_{0}^{g(0)} e^{-2\mu y_2} \left(\frac{\sinh(x_2-y_2)}{\cosh(x_2-y_2)-\cos(x_1-g^{-1}(y_2))}-\frac{\sinh(x_2-y_2)}{\cosh(x_2-y_2)-\cos(x_1+g^{-1}(y_2))} \right)dy_2\\
= &\; \frac{1}{4\pi}\int_{0}^{g(0)} e^{-2\mu y_2}\cdot\frac{\sinh(x_2-y_2)(\cos(x_1-g^{-1}(y_2))-\cos(x_1+g^{-1}(y_2)))}{(\cosh(x_2-y_2)-\cos(x_1-g^{-1}(y_2)))(\cosh(x_2-y_2)-\cos(x_1+g^{-1}(y_2)))}\, dy_2\\
= &\; \frac{1}{2\pi}\int_{0}^{g(0)} e^{-2\mu y_2}\cdot\frac{\sinh(x_2-y_2)\sin x_1 \sin g^{-1}(y_2)}{(\cosh(x_2-y_2)-\cos(x_1-g^{-1}(y_2)))(\cosh(x_2-y_2)-\cos(x_1+g^{-1}(y_2)))}\, dy_2.
\end{align*}
Since $x_2\geq g(0)$ and $y_2\in [0,g(0)]$, the integrand is non-negative, so $\pa_{x_1x_2}^2 \Psi(x)\geq 0$, which implies that $x_1\mapsto \pa_{x_2}\Psi(x_1,x_2)$ is increasing on $[0,\pi]$.
\end{proof}
\end{lem}

In the proof of Theorem \ref{thm: flow field}, we will also need the following  geometric lemma for reflecting part of the patch $D_0$ in a suitable way.

\begin{lem}
\label{lem: reflected point still lies in D_0}
Let $m\in \BZ_+$, $f$, and $D_0$ be given as in the condition of Theorem \ref{thm: flow field} (also see Theorem \ref{thm: heteroclinic trajectory}).
Then the following holds.
\begin{enumerate}
\item There exists a universal $r_*\in(0,1)$ that is independent of $m$, such that for any $r\in [r_*,1)$, if $(x_1,x_2)\in \overline{D_0}$ satisfies $x_1> r$, then $(2r-x_1,x_2)\in D_0$.

\item
Let $r_m := f(\f{\pi}{m})\cos\f{\pi}{m}\in (0,1)$.
If $(4\cos^2 \f{\pi}{m}-1)f(\f{\pi}{m}) > 1$, then for any
\beq
(x_1,x_2)\in \left\{(r\cos\th,r\sin\th)\in \BR^2:\, \th \in \left[-\f{\pi}{m},\f{\pi}{m}\right],\, r\leq f(\th),\, r\cos \th > r_m \right\},
\label{eqn: points in one petal}
\eeq
it holds that $(2r_m-x_1,x_2) \in D_0$.
\end{enumerate}

\begin{proof}
To show the first statement, we note that, by the assumption, $f$ satisfies all the conditions of Theorem \ref{thm: main existence theorem}; in particular, $f$ is Lipschitz on $\BT$, with $\|f'\|_{L^\infty(\BT)}\leq C_0$ for some universal $C_0>0$.
Assume $r\in[\f34,1)$ and take $(x_1,x_2)\in \overline{D_0}$ satisfying $x_1> r$.
Then it must hold that $x_1 \in (r,1]$, $x_2\in [0,\sqrt{1-r^2}]$, and $2r-x_1 \geq 2r-1\geq \f12$.
Since $D_0$ is symmetric with respect to the horizontal axis, we may assume $x_2\geq 0$ without loss of generality.
If $x_2 = 0$, it is trivial to derive $(2r-x_1,0)\in D_0$ from $(x_1,0)\in \overline{D_0}$.
If $x_2 > 0$, we define
\[
\al:=\arctan \f{x_2}{x_1},\quad \b:=\arctan \f{x_2}{2r-x_1}.
\]
Then
\beq
0 < \al < \b \leq \arctan\f{\sqrt{1-r^2}}{2r-1}.
\label{eqn: range of alpha and beta}
\eeq
In order to guarantee $(2r-x_1,x_2)\in D_0$, we only have to confirm that
\beq
|(2r-x_1,x_2)| < f(\b).
\label{eqn: condition for the reflected point in D_0}
\eeq
Since $x_2 = |(x_1,x_2)|\sin \al = |(2r-x_1,x_2)|\sin \b$,
\[
|(2r-x_1,x_2)| = |(x_1,x_2)|\cdot \f{\sin \al }{\sin \b}
\leq f(\al)\left[1-\f{\sin \b - \sin \al }{\sin \b}\right].
\]
In the last inequality, we used the fact that $(x_1,x_2)\in \overline{D_0}$.
Using the Lipschitz bound for $f$ and the fact $\b\geq \al$, we find that $f(\al) \leq f(\b) + C_0(\b-\al)$.
Putting this into the preceding estimate and applying the mean-value theorem, we obtain that
\begin{align*}
|(2r-x_1,x_2)|
\leq &\; \big(f(\b) + C_0(\b-\al)\big)-f(\al)\cdot \f{\cos \b}{\sin \b}\cdot (\b-\al)
\\
= &\; f(\b) - \f{\b-\al}{\tan \b}\cdot \big(f(\al)- C_0\tan \b\big).
\end{align*}
Since $f(\al) \geq |(x_1,x_2)| > r$ while $\b$ satisfies \eqref{eqn: range of alpha and beta}, we find that there exists a universal $r_*\in[\f34,1)$, such that as long as $r\in [r_*,1)$, we have $f(\al) >  C_0\tan \b$.
This implies \eqref{eqn: condition for the reflected point in D_0} and thus $(2r-x_1,x_2)\in D_0$.
This proves the first statement.

To prove the second statement, it suffices to show that for any $(x_1,x_2)$ satisfying \eqref{eqn: points in one petal}, $|(2r_m-x_1,x_2)|<f(\f{\pi}{m})$.
We derive that, for $x = (x_1,x_2)$ satisfying \eqref{eqn: points in one petal},
\begin{align*}
|(2r_m-x_1,x_2)|^2
= &\; 4r_m^2 -4x_1r_m + |x|^2\\
\leq &\; 4r_m^2 -4x_1 r_m + \min\left\{x_1^2 \cos^{-2} \left(\f{\pi}{m}\right),\, 1\right\}.
\end{align*}
Note that,
as a function of $x_1$, the last line is decreasing when $x_1\in [\cos \f{\pi}{m},1]$, while it is a quadratic function with the leading coefficient being positive when $x_1\in (r_m,\cos \f{\pi}{m}]$.
So in order to bound it from above, it suffices to compare its values when $x_1 \to r_m^+$ and when $x_1 = \cos \f{\pi}{m}$.
By the assumption,
\begin{align*}
&\; \left[4r_m^2 -4r_m\cdot r_m + r_m^2 \cos^{-2} \left(\f{\pi}{m}\right)\right]-
\left[4r_m^2 -4\cos\left(\f{\pi}{m}\right)r_m + 1\right]\\
= &\; \left(1-f\left(\f{\pi}{m}\right)\right)\left[4f\left(\f{\pi}{m}\right)\cos^2\left(\f{\pi}{m}\right) - f\left(\f{\pi}{m}\right) - 1\right] >0,
\end{align*}
so \[
|(2r_m-x_1,x_2)|^2
< 4r_m^2 -4r_m\cdot r_m + r_m^2 \cos^{-2} \left(\f{\pi}{m}\right) = f\left(\f{\pi}{m}\right)^2,
\]
which proves the desired claim.
\end{proof}
\end{lem}

Now let us prove Theorem \ref{thm: flow field} on characterizations of the flow field $v$ in the co-rotating frame.

\begin{proof}[Proof of Theorem \ref{thm: flow field}]
Recall that we denoted $\mathcal{S}:= \{(\cos \f{2k\pi}{m},\,\sin \f{2k\pi}{m}):\,k = 0,1,\cdots, m-1\}$.
That $v(x) = (0,0)$ for $x = (0,0)$ follows from the symmetry of $\phi_\dag$.
For $x\in \CS$, that $v(x) = (0,0)$ follows from the facts that $\phi_\dag$ is constant along $\pa D_0$ due to Theorem \ref{thm: heteroclinic trajectory}, and that $\pa D_0$ is piecewise $C^1$ and admits a 90-degree angle at $x$.

For convenience, for $x\in \BR^2\setminus \{0\}$, we denote
\[
\pa_r \phi_\dag(x) := \f{x}{|x|}\cdot \na \phi_\dag(x),\quad
\pa_\th \phi_\dag(x) := \f{x^\perp}{|x|}\cdot \na \phi_\dag(x).
\]
Note that $\pa_\th$ is not the differentiation with respect to the angle component in the polar coordinate.
Apparently, $v_r = \pa_\th \phi_\dag$ and $v_\th = -\pa_r \phi_\dag$.
By Lemma \ref{lem: monotonicity in lower plane}, \eqref{eqn: def of Phi_dag}, and \eqref{eqn: def of Psi_dag}, $\pa_r \phi_\dag (x)>0$ for all $x\in \BR^2$ such that $|x|\geq 1$ and $x\not \in \mathcal{S}$.
By Lemma \ref{lem: monotonicity in lower plane}, Lemma \ref{lem: monotonicity of Psi dag tilde}, \eqref{eqn: def of zeta} and \eqref{eqn: def of Phi}, we can deduce that
\beq
\pa_\th \phi_\dag (x)
\begin{cases}
<0,&\mbox{if }x\in \left\{\left(r\cos \th,\,r\sin\th\right):\, r>0,\, \th \in \left(0,\f{\pi}{m}\right)\right\},\\
>0,&\mbox{if }x\in \left\{\left(r\cos \th,\,r\sin\th\right):\, r>0,\, \th \in \left(\f{\pi}{m},\f{2\pi}{m}\right)\right\},\\
=0,&\mbox{if }x\in \left\{\left(r\cos \th,\,r\sin\th\right):\, r>0,\, \th = 0,\f{\pi}{m}\right\},
\end{cases}
\label{eqn: sign of pa_theta phi_dag}
\eeq
and moreover, $\pa_r \phi_\dag (x) < 0$ if
\[
x\in \left\{\left(r\cos \f{\pi}{m},\,r\sin\f{\pi}{m}\right):\, r\in \left(0,f\left(\f{\pi}{m}\right)\right]\right\}.
\]
Let $r_*\in (0,1)$ be the universal constant defined in Lemma \ref{lem: reflected point still lies in D_0}, and take an arbitrary $r\in [r_*,1)$.
Define an open set
\[
Q_r := \{(2r-x_1,x_2):\, (x_1,x_2)\in D_0,\, x_1 > r\},
\]
and let
\[
\tilde{\phi}_\dag(x_1,x_2):= \phi_\dag(2r-x_1,x_2),\quad (x_1,x_2)\in \BR^2.
\]
By Lemma \ref{lem: reflected point still lies in D_0}, $Q_r\subset D_0$, so
\[
-\D\big(\phi_\dag -\tilde{\phi}_\dag\big) = 0 \mbox{ on }Q_r,\quad
\big(\phi_\dag -\tilde{\phi}_\dag\big)\big|_{\pa Q_r\cap \{x_1 = r\}} = 0.
\]
Observe that $(x_1,x_2)\in \pa Q_r\cap \{x_1<r\}$ if and only if $(2r-x_1,x_2)\in \pa D_0 \cap \{x_1>r\}$.
So whenever $r<1$, $\pa Q_r\cap \{x_1<r\}$ is non-empty and it is contained in a Lipschitz curve, and
by Lemma \ref{lem: reflected point still lies in D_0}, $\pa Q_r\cap \{x_1<r\} \subset D_0$.
By \eqref{eqn: the streamlines are the level sets of phi dag} and \eqref{eqn: sign of phi dag in and out the two streamlines}, \[
\big(\phi_\dag -\tilde{\phi}_\dag\big)\big|_{\pa Q_r\cap \{x_1 < r\}} > 0.
\]
Now applying the maximum principle, we find that $\phi_\dag -\tilde{\phi}_\dag \geq 0$ on $Q_r$, and it is not identically zero.
By the Hopf lemma, for any $(r,x_2)\in D_0$,
\[
2\pa_{x_1}\phi_\dag(r,x_2) = \pa_{x_1}(\phi_\dag -\tilde{\phi}_\dag)(r,x_2) < 0.
\]
Since $r\in [r_*,1)$ is arbitrary, we conclude that
$\pa_{x_1} \phi_\dag < 0$ for all $x \in \{(x_1,x_2)\in D_0:\, x_1 \in [r_*,1)\}$.
Observe that, for $(x_1,x_2) = (r\cos\th,r\sin \th)$ with $x_1 =r\cos \th> 0$, it holds that
\[
\pa_r \phi_\dag(x) = \f{1}{\cos\th} \cdot \pa_{x_1}\phi_\dag(x) + \tan \th \cdot \pa_\th \phi_\dag(x).
\]
Combining this with the preceding conclusion and \eqref{eqn: sign of pa_theta phi_dag} yields that $\pa_r\phi_\dag(x)<0$ for all $x$ in
\beq
\left\{(x_1,x_2) = (r\cos\th,r\sin\th)\in D_0:\, x_1 \in [r_*,1),\, \th\in \left[-\f{\pi}{m},\f{\pi}{m}\right]\right\}.
\label{eqn: the region near corners in which pa_r phi_dag is negative}
\eeq
In particular, $\pa_r \phi_\dag (x)< 0$ for all $x \in \{(x_1,0):\, x_1 \in [r_*,1)\}$.
Finally, using the $m$-fold symmetry and the facts that $v_r = \pa_\th \phi_\dag$ and $v_\th = -\pa_r \phi_\dag$, we obtain the desired claims on the sign properties of $v_r$ and $v_\theta$.

The set $\CC$ of the critical points of $\phi_\dag$ is well-defined since $\phi_\dag$ has $C^{1,\al}_{loc}$-regularity for all $\al\in (0,1)$ due to the standard elliptic regularity theory.
Apparently, $\CC$ is a closed set having $m$-fold symmetry with respect to the origin, and $\CS\cup\{(0,0)\}\subset \CC$.
Denote
\begin{align*}
\mathcal{Y} := &\; \bigcup_{k = 0}^{m-1} \left\{\left(r\cos\f{2k\pi}{m},\, r\sin \f{2k\pi}{m}\right):\, r\in (0,r_*)\right\}\\
&\; \cup \bigcup_{k = 0}^{m-1} \left\{\left(r\cos\f{(2k+1)\pi}{m},\, r\sin \f{(2k+1)\pi}{m}\right):\,r\in \left(f\left(\f{\pi}{m}\right),1\right)\right\},
\end{align*}
where $r_*\in (0,1)$ is defined as above.
It follows from the analysis in the previous paragraph that $\mathcal{C}\subset \mathcal{Y}\cup \CS\cup\{(0,0)\}$.
It remains to show that $\CC$ is a finite set.
Since $\CS\cup\{(0,0)\}$ is finite and $\pa_\th \phi_\dag = 0$ in $\mathcal{Y}$, it suffices to show that $\pa_r \phi_\dag$ has at most finitely many zeros in $\mathcal{Y}$.

By \eqref{eqn: equation for phi_dag}, $\phi_\dag$ is analytic in $\BR^2\setminus \pa D_0$.
\begin{itemize}
\item Observe that $(s,0)\in D_0$ when $s\in (-f(\f{\pi}{m}),1)$, so the function $\va_1(s):=\pa_{x_1}\phi_\dag(s,0)$ is analytic in the interval $(-f(\f{\pi}{m}),1)$.
It is not identically zero since for $s\in [r_*,1)$, $\va_1(s) = \pa_r\phi_\dag(s,0) < 0$ as is shown in the previous paragraph.
As a result, $\va_1$ must have isolated zeros in the interval $(-f(\f{\pi}{m}),1)$, which implies that $\pa_r\phi_\dag$ has at most finitely many zeros in the set $\{(r,0):\, r\in (0,r_*)\}$.

\item Similarly, consider the function
\[
\va_2(s):= \pa_r \phi_\dag\left(s\cos\f{\pi}{m},\, s\sin \f{\pi}{m}\right),\quad s\in \left[f\left(\f{\pi}{m}\right),+\infty\right).
\]
It is analytic in the interval $(f(\f{\pi}{m}),+\infty)$.
Recall that $\Phi_\dag$ was defined in \eqref{eqn: def of Phi_dag}, and by Lemma \ref{lem: monotonicity of Psi dag tilde}, $\pa_{x_2}\Phi_\dag(0,-m\ln f(\frac{\pi}{m}))=\pa_{x_2}\Phi_\dag(0,g(0))<0$.
By \eqref{eqn: def of Phi_dag} and the conformal mapping \eqref{eqn: def of zeta}, this translates to $\va_2(f(\frac{\pi}{m})) = \pa_r\phi_\dag(f(\frac{\pi}{m})\cos\frac{\pi}{m},f(\frac{\pi}{m})\sin\frac{\pi}{m})<0$.
As a consequence, $\va_2(s)<0$ for $s$ in a neighborhood of $f(\f{\pi}{m})$ thanks to the continuity of $\na \phi_\dag$.
This further means that $\va_2$ is not identically zero.
Combining the facts above, we find that $\va_2$ can only have isolated zeros in $[f(\f{\pi}{m}),+\infty)$, which implies that $\pa_r\phi_\dag$ can have at most finitely many zeros in the set $\{(r\cos\f{\pi}{m},r\sin\f{\pi}{m}):\, r\in (f(\f{\pi}{m}),1)\}$.
\end{itemize}
By virtue of the $m$-fold symmetry, we can conclude that $\pa_r \phi_\dag$ has at most finitely many zeros in $\mathcal{Y}$, and thus $\CC$ is a finite set.

When \eqref{eqn: condition for folding one petal} holds, i.e.,  $2\cos\f{2\pi}{m}+1> e^{M/m}$, we can use the estimate $f(\f{\pi}{m})\geq e^{-M/m}$ in Theorem \ref{thm: main existence theorem} to find $(4\cos^2 \f{\pi}{m}-1)f(\f{\pi}{m}) > 1$.
Then the second statement of Lemma \ref{lem: reflected point still lies in D_0} applies, so with $r_m := f(\f{\pi}{m})\cos \f{\pi}{m}$, we can reflect
\[
\left\{(r\cos\th,r\sin\th)\in \BR^2:\, \th \in \left[-\f{\pi}{m},\f{\pi}{m}\right],\, r\leq f(\th),\, r\cos \th > r_m \right\}
\]
with respect to the axis $\{x_1= r_m\}$ and argue as above to analogously obtain that $\pa_{x_1}\phi_\dag(r_m,0)<0$.
If we define $\Phi_\dag$ as in \eqref{eqn: def of Phi_dag}, then this implies $\pa_{x_2}\Phi_\dag(\pi,-m\ln r_m) <0$ (see \eqref{eqn: def of zeta}).
Note that $-m\ln r_m \geq -m\ln f(\f{\pi}{m}) = g(0)$, and the condition \eqref{eqn: condition for folding one petal} automatically requires $m> 4$.
Hence, by Lemma \ref{lem: monotonicity of Psi dag tilde}, $\pa_{x_2}\Phi_\dag(x_1,x_2)< 0$ in $\BT\times [-m\ln r_m,+\infty)$, which then translates to $\pa_r \phi_\dag(x) < 0$ for all $x\in \overline{B_{r_m}}\setminus \{(0,0)\}$.

Finally, recall that we have shown $\pa_r \phi_\dag (x)< 0$ in the set \eqref{eqn: the region near corners in which pa_r phi_dag is negative}; in particular, this holds for all $x\in D_0$ that satisfies $|x|\geq \f{r_*}{\cos (\pi/m)}$.
Since $f(\f{\pi}{m}) \geq e^{-M/m}$ with $M$ being universal, there exists a universal $m_1 \geq 4$, such that $f(\f{\pi}{m}) \geq \f{r_*}{\cos(\pi/m)}$ as long as $m\geq m_1$.
In this case, $\pa_{x_1}\phi_\dag(f(\f{\pi}{m}),0)<0$.
Then the same argument as in the preceding paragraph leads to the conclusion that $\pa_r \phi_\dag(x) < 0$ for all $x\in \overline{B_{f(\pi/m)}}\setminus \{(0,0)\}$.
Therefore, $\pa_r \phi_\dag(x) < 0$ for all $x\in D_0\setminus\{(0,0)\}$, which allows us to improve the characterizations of $v_\theta$ and $\CC$.

This completes the proof.
\end{proof}

\appendix

\section{Proof of Corollary \ref{cor: M=4 and mu leq 1/12}}
\label{sec: proof of M = 4 mu = 1/12}

Recall that Corollary \ref{cor: M=4 and mu leq 1/12} states that
Proposition \ref{prop: a priori upper bound general m} holds with $M=4$ and $\mu_0 = \f1{12}$.
We present its proof as follows, which is purely analytic.

\begin{proof}[Proof of Corollary \ref{cor: M=4 and mu leq 1/12}]
Thanks to the arguments in the proof of Proposition \ref{prop: a priori upper bound general m}, it suffices to verify \eqref{eqn: sign conditions related to A(s)} with $M=4$, i.e., for any $\mu\in [0,\f1{12}]$,
\beq
\z_0(4; 8\mu)\leq 0,\quad \r_1(4;\mu)\leq 0,
\label{eqn: sign conditions related to A(s) simplified}
\eeq
where $\zeta_0$ and $\r_1$ were defined in \eqref{eqn: def of zeta_0} and \eqref{eqn: def of rho_1(M,mu)}, respectively.

To verify the first inequality in \eqref{eqn: sign conditions related to A(s) simplified}, we note that
\beq
\begin{split}
\ln 2 = &\; \ln\left(1+\f13\right) - \ln\left(1-\f13\right)
\\
= &\; \sum_{k=1}^\infty \f{(-1)^{k-1}}{k}\cdot \f{1}{3^k} - \f{(-1)^{k-1}}{k}\cdot \f{(-1)^k}{3^k}
= \sum_{k=1}^\infty \f{2}{2k-1}\cdot \f{1}{3^{2k-1}}\\
\leq &\; \f{2}{3}+\f{2}{3\cdot 3^3}\sum_{j = 0}^\infty\f{1}{9^{j}}
=\f{25}{36}.
\end{split}
\label{eqn: bound for ln 2}
\eeq
Then by \eqref{eqn: A(0)} and the monotonicity of $\z_0(M;c)$, for any $\mu \in [0,\f14]$,
\begin{align*}
\z_0(4;8\mu) \leq &\; \z_0(4;2)
=
- 2 \ln \left(\f{1-e^{-4}}{2}\right)
- 4 \cdot \f{1-e^{-2}}{2}
\\
\leq &\;
2\int_{1-e^{-4}}^1 \f{1}{s}\,ds + 2\ln 2 - 2 + 2e^{-2}
\\
\leq &\;
\f{2e^{-4}}{1-e^{-4}} + 2\ln 2 - 2 + 2e^{-2}
< 0.
\end{align*}
The last inequality follows from \eqref{eqn: bound for ln 2} and the fact that $e^2>(\sum_{k=0}^2 \f{1}{k!})^2 >6$.

Next we justify the second inequality in \eqref{eqn: sign conditions related to A(s) simplified}.
By the definition of $\r_1$ in \eqref{eqn: def of rho_1(M,mu)}, \beq
\begin{split}
e^{8\mu}\r_1(4;\mu)
= &\;
\f{\pi^2}{8}
+ \f12
\left(\f{1-e^{-4(1+2\mu)}}{1+2\mu} + \f{1-e^{-4(1-2\mu)}}{1-2\mu} -2 \right)
\\
&\;+ \mu \sum_{k = 2}^\infty \f{1}{k^2} \left(\f{(-1)^{k}}{k+2\mu} +\f{1}{k-2\mu}\right)\\
&\; - \f{1}{\mu}  \left(1-\f{1-e^{-8\mu}}{8\mu}\right)
+2\ln 2 \cdot \f{e^{8\mu}-1}{8\mu}.
\end{split}
\label{eqn: exp times rho_1 with M=4}
\eeq
The goal is to confirm that it is negative for all $\mu\in [0,\f1{12}]$.

We start from studying the third term on the right-hand side.
For $\mu\in [0,\f12]$,
\begin{align*}
&\; \sum_{k = 2}^\infty  \f{1}{k^2} \left(\f{(-1)^{k}}{k+2\mu} +\f{1}{k-2\mu}\right)\\
= &\; \sum_{j = 1}^\infty  \f{1}{(2j)^2} \left(\f{1}{2j+2\mu} +\f{1}{2j-2\mu}\right)
+ \sum_{j = 1}^\infty  \f{1}{(2j+1)^2} \left(\f{-1}{2j+1+2\mu} +\f{1}{2j+1-2\mu}\right)
\\
= &\; \sum_{j = 1}^\infty  \f{4j}{(2j)^2[(2j)^2 - (2\mu)^2]}
+ \sum_{j = 1}^\infty  \f{4\mu}{(2j+1)^2[(2j+1)^2 - (2\mu)^2]}
\\
\leq &\; \f{1}{4-4\mu^2} - \f{1}{3} + \sum_{j = 1}^\infty  \f{2}{2j[(2j)^2 - 1]}
+ 4\mu \sum_{j = 1}^\infty  \f{1}{(2j+1)^2[(2j+1)^2 - 1]}\\
= &\; \f{1}{4-4\mu^2} - \f{1}{3} + \sum_{j = 1}^\infty \left(\f{1}{2j-1}+\f{1}{2j+1}-\f{2}{2j}\right)
+ 4\mu \sum_{j = 1}^\infty \left(\f{1}{(2j+1)^2 - 1} - \f{1}{(2j+1)^2}\right)
\\
< &\; \f{1}{4-4\mu^2} - \f{1}{3} + 2\ln 2 - 1 + 4\mu \cdot \left(\f{1}{8}-\f19 + \f{1}{24}\right).
\end{align*}
In the last line, we used the identity $\sum_{n = 1}^\infty\f{(-1)^{n-1}}{n} = \ln 2$ as well as the fact that $(2j+1)^2 < (2j+3)^2-1$.
Hence, for $\mu\in [0,\f1{12}]$,
\[
\sum_{k = 2}^\infty  \f{1}{k^2} \left(\f{(-1)^{k}}{k+2\mu} +\f{1}{k-2\mu}\right)
< \f{36}{143} - \f13 + 2\cdot\f{25}{36}-1 + \f{4}{12}\cdot \f{1}{18} = \f{1258}{3861}.
\]
Here we also used \eqref{eqn: bound for ln 2}.

Next we claim that
\[
\mu \mapsto \f{1-e^{-4(1+2\mu)}}{1+2\mu} + \f{1-e^{-4(1-2\mu)}}{1-2\mu} \mbox{ is increasing for $\mu \in[0,\f12)$}.
\]
Indeed, for $z>0$,
\[
\f{d^2}{dz^2}\left[\f{1-e^{-z}}{z}\right]
= \f{2e^{-z}}{z^3} \left(e^z - 1-z-\f12 z^2\right)\geq 0,
\]
so
\begin{align*}
&\; \f{d}{d\mu}\left[\f{1-e^{-4(1+2\mu)}}{1+2\mu} + \f{1-e^{-4(1-2\mu)}}{1-2\mu} \right]\\
= &\; 32 \left[\left.\f{d}{dz}\right|_{z = 4(1+2\mu)}\left(\f{1-e^{-z}}{z}\right) -\left.\f{d}{dz}\right|_{z = 4(1-2\mu)}\left(\f{1-e^{-z}}{z}\right)\right] \geq 0.
\end{align*}
Moreover,
\[
\mu \mapsto - \f{1}{\mu} \left(1-\f{1-e^{-8\mu}}{8\mu}\right) + \f{e^{8\mu}-1}{8\mu}
\mbox{ is increasing for $\mu > 0$}.
\]
This is because
\begin{align*}
\f{1}{z}\left(1-\f{1-e^{-z}}{z}\right)
= \f{1}{z}\int_0^1 \big(1-e^{-zt}\big)\,dt = \int_0^1 \int_0^t e^{-zs}\,ds\,dt,
\end{align*}
which is decreasing for $z\in (0,+\infty)$, and
\[
\f{e^z-1}{z} =\int_0^1 e^{zt}\,dt,
\]
which is increasing for $z\in (0,+\infty)$.

Combining the estimates above with \eqref{eqn: exp times rho_1 with M=4} yields that, for $\mu \in [0,\f1{12}]$,
\beq
\begin{split}
e^{8\mu}\r_1(4;\mu)
< &\;
\f{\pi^2}{8}
+ \f12
\left(\f{1-e^{-14/3}}{7/6} + \f{1-e^{-10/3}}{5/6} -2 \right)
\\
&\;+ \f{1}{12}\cdot \f{1258}{3861}
- 12 \left(1-\f{1-e^{-2/3}}{2/3}\right)
+2\ln 2 \cdot \f{e^{2/3}-1}{2/3}.
\end{split}
\label{eqn: exp times rho_1 with M=4 bound 2}
\eeq
To show the right-hand side is negative, we need some additional numerical facts.
\begin{itemize}
\item
Observe that
\begin{align*}
\f{\pi}{6}
= &\; \arctan \f{\sqrt{3}}{3} = \sum_{k = 0}^\infty \f{(-1)^k}{2k+1}\cdot 3^{-\f{2k+1}{2}}\\
< &\; 3^{-\f12}
- \f{1}{3} \cdot 3^{-\f{3}{2}}
+ \f{1}{5} \cdot 3^{-\f{5}{2}}
- \f{1}{7} \cdot 3^{-\f{7}{2}}
+ \f{1}{9} \cdot 3^{-\f{9}{2}}.
\end{align*}
The last inequality follows from the fact that the power series here is an alternating series, with $k\mapsto \f{1}{2k+1}\cdot 3^{-(2k+1)/2}$ being decreasing in $k$.
Hence,
\[
\pi^2 < 36\cdot \f13\left(1-\f19 + \f1{45} - \f{1}{189} + \f{1}{729}\right)^2 = \f{2143134436}{217005075}.
\]
\item
By Taylor expansion,
\[
e^{2/3} = \sum_{k = 0}^\infty \f{1}{k!} \left(\f{2}{3}\right)^k
<\sum_{k = 0}^2 \f{1}{k!} \left(\f{2}{3}\right)^k + \sum_{k = 3}^\infty \f{1}{2!} \left(\f{2}{3}\right)^2\cdot \left(\f{1}{3}\cdot \f{2}{3}\right)^{k-2} =  \f{41}{21}.
\]
and similarly, a sharper estimate is that
\[
e^{2/3} = \sum_{k = 0}^\infty \f{1}{k!} \left(\f{2}{3}\right)^k
<\sum_{k = 0}^4 \f{1}{k!} \left(\f{2}{3}\right)^k + \sum_{k = 5}^\infty \f{1}{4!} \left(\f{2}{3}\right)^4\cdot \left(\f{1}{5}\cdot \f{2}{3}\right)^{k-4} =  \f{2051}{1053}.
\]
\end{itemize}
Plugging these inequalities as well as \eqref{eqn: bound for ln 2} into \eqref{eqn: exp times rho_1 with M=4 bound 2}, we finally obtain that
\begin{align*}
e^{8\mu}\r_1(4;\mu)
< &\;
\f{1}{8} \cdot \f{2143134436}{217005075}
+ \f12
\left(\f{1-(21/41)^7}{7/6} + \f{1-(21/41)^5}{5/6} -2 \right)
\\
&\;+ \f{1}{12}\cdot \f{1258}{3861}
- 12 \left(1-\f{1-\f{1053}{2051}}{2/3}\right)
+2\cdot \f{25}{36} \cdot \f{\f{2051}{1053}-1}{2/3}
\\
= &\;  -\f{6105170933824130733679}{3541526869556179847192850} < 0,
\end{align*}
which proves the second inequality in \eqref{eqn: sign conditions related to A(s) simplified}.

This completes the proof.
\end{proof}

\section{Some Calculus Lemmas}
\label{sec: calculus lemmas}
In the proof of Lemma \ref{lem: bounding F_2 x_1 in terms of F_2 x_2}, we used the following two auxiliary calculus lemmas.
\begin{lem}
\label{lem: simplification of Q_i}
For $i = 1,2$, we denote
\[
Q_i(x_1,x_2,y_1) := \pa_{x_i}\left(\f{P}{x_2}\right)(x_1-y_1,x_2)
+ \pa_{x_i}\left(\f{P}{x_2}\right)(x_1+y_1,x_2),
\]
where $P$ is the Poisson kernel given in \eqref{eqn: Poisson kernel}.
Then
\begin{align*}
Q_1(x_1,x_2,y_1)
= &\; -\f{x_2\sinh x_2 \sin x_1}
{\pi x_2^2(\cosh x_2- \cos (x_1-y_1))^2(\cosh x_2- \cos (x_1+y_1))^2}
\\
&\;\cdot \big[
\cos y_1 (\cosh^2 x_2 + 1 + \cos^2 x_1 - \cos^2 y_1)
-2\cosh x_2 \cos x_1\big],
\end{align*}
and
\begin{align*}
Q_2(x_1,x_2,y_1)
= &\; -\f{x_2\cosh x_2-\sinh x_2}{\pi x_2^2(\cosh x_2- \cos (x_1-y_1))^2(\cosh x_2- \cos (x_1+y_1))^2}
\\
&\;\cdot \big[(k(x_2)+\cos x_1\cos y_1) (\cosh x_2 - \cos x_1 \cos y_1)^2
\\
&\; \quad
+ (k(x_2) + 2\cosh x_2 - \cos x_1 \cos y_1) \sin^2 x_1 \sin^2 y_1\big],
\end{align*}
where
\beq
k(x_2) := \f{\sinh x_2 \cosh x_2-x_2}{x_2\cosh x_2-\sinh x_2}.
\label{eqn: def of k(x_2)}
\eeq
\begin{proof}
The proof is only a direct calculation.
In fact, by \eqref{eqn: Poisson kernel},
\[
\pa_{x_1}\left[\f{1}{x_2}P(x_1,x_2)\right]
= -\f{x_2\sinh x_2 }{2\pi x_2^2}\cdot \f{\sin x_1}{(\cosh x_2- \cos x_1)^2}.
\]
Also recall that, in the proof of Lemma \ref{lem: monotonicity of P/x_2}, we showed that
\begin{align*}
\pa_{x_2}\left[\f{1}{x_2}P(x_1,x_2)\right]
= &\; \f{x_2\cosh x_2-\sinh x_2}{2\pi x_2^2}\cdot \f{-k(x_2) - \cos x_1}{(\cosh x_2- \cos x_1)^2},
\end{align*}
where $k(x_2)$ was defined as in \eqref{eqn: def of k(x_2)}.
Then
\begin{align*}
&\; Q_1(x_1,x_2,y_1)\\
= &\; -\f{x_2\sinh x_2}{2\pi x_2^2}
\left[ \f{\sin (x_1-y_1)}{(\cosh x_2- \cos (x_1-y_1))^2} + \f{\sin (x_1+y_1)}{(\cosh x_2- \cos (x_1+y_1))^2}\right]
\\
= &\; -\f{x_2\sinh x_2}
{2\pi x_2^2(\cosh x_2- \cos (x_1-y_1))^2(\cosh x_2- \cos (x_1+y_1))^2}
\\
&\;\cdot \big[ \cosh^2 x_2 (\sin (x_1-y_1)+ \sin (x_1+y_1))\\
&\;\quad -2\cosh x_2 (\sin(x_1-y_1)\cos(x_1+y_1)+\sin(x_1+y_1)\cos(x_1-y_1))\\
&\;\quad +\sin(x_1-y_1)\cos^2(x_1+y_1) +\sin(x_1+y_1)\cos^2(x_1-y_1)\big]
\\
= &\; -\f{x_2\sinh x_2}
{2\pi x_2^2(\cosh x_2- \cos (x_1-y_1))^2(\cosh x_2- \cos (x_1+y_1))^2}
\\
&\;\cdot \big[2\cosh^2 x_2 \sin x_1 \cos y_1
-2\cosh x_2 \sin 2x_1\\
&\;\quad +\sin(x_1-y_1)(1-\sin^2(x_1+y_1)) +\sin(x_1+y_1)(1-\sin^2(x_1-y_1))\big]
\\
= &\; -\f{x_2\sinh x_2}
{2\pi x_2^2(\cosh x_2- \cos (x_1-y_1))^2(\cosh x_2- \cos (x_1+y_1))^2}
\\
&\;\cdot \big[ 2\cosh^2 x_2 \sin x_1 \cos y_1
-4\cosh x_2 \sin x_1\cos x_1\\
&\;\quad +2\sin x_1\cos y_1 (1-\sin(x_1+y_1)\sin(x_1-y_1))\big]
\\
= &\; -\f{2x_2\sinh x_2 \sin x_1}
{2\pi x_2^2(\cosh x_2- \cos (x_1-y_1))^2(\cosh x_2- \cos (x_1+y_1))^2}
\\
&\;\cdot \big[
\cos y_1 (\cosh^2 x_2 + 1 + \cos^2 x_1 - \cos^2 y_1)
-2\cosh x_2 \cos x_1\big].
\end{align*}
Similarly,
\begin{align*}
&\;Q_2(x_1,x_2,y_1)\\
= &\; \f{x_2\cosh x_2-\sinh x_2}{2\pi x_2^2}
\left[ \f{-k(x_2) - \cos (x_1-y_1)}{(\cosh x_2- \cos (x_1-y_1))^2} + \f{-k(x_2) - \cos (x_1+y_1)}{(\cosh x_2- \cos (x_1+y_1))^2}\right]
\\
= &\; \f{x_2\cosh x_2-\sinh x_2}
{2\pi x_2^2(\cosh x_2- \cos (x_1-y_1))^2(\cosh x_2- \cos (x_1+y_1))^2}
\\
&\;\cdot \big[(-2k(x_2)-\cos(x_1-y_1)-\cos(x_1+y_1))\cosh^2 x_2\\
&\;\quad + 2 \cosh x_2 (\cos(x_1+y_1)+\cos(x_1-y_1))k(x_2)
\\
&\;\quad + 4\cosh x_2 \cos (x_1-y_1) \cos(x_1+y_1)
\\
&\;\quad + (-k(x_2) - \cos (x_1-y_1))\cos^2(x_1+y_1)
+ (-k(x_2) - \cos (x_1+y_1))\cos^2(x_1-y_1)\big]
\\
= &\; \f{x_2\cosh x_2-\sinh x_2}{2\pi x_2^2(\cosh x_2- \cos (x_1-y_1))^2(\cosh x_2- \cos (x_1+y_1))^2}
\\
&\;\cdot \big[-2k(x_2)\cosh^2 x_2
+ 4k(x_2) \cosh x_2 \cos x_1 \cos y_1
-2\cos x_1\cos y_1\cosh^2 x_2
\\
&\;\quad - k(x_2)(\cos^2(x_1+y_1) + \cos^2(x_1-y_1))
\\
&\;\quad
+ 4\cosh x_2 (\cos^2 x_1\cos^2 y_1 - \sin^2 x_1 \sin^2 y_1)\\
&\;\quad - \cos (x_1-y_1) \cos(x_1+y_1)[\cos(x_1+y_1)+\cos(x_1-y_1)]
\big]
\\
= &\; \f{x_2\cosh x_2-\sinh x_2}{2\pi x_2^2(\cosh x_2- \cos (x_1-y_1))^2(\cosh x_2- \cos (x_1+y_1))^2}
\\
&\;\cdot \big[-2k(x_2)\cosh^2 x_2
+ 4k(x_2) \cosh x_2 \cos x_1 \cos y_1
\\
&\;\quad - 2k(x_2)(\cos^2 x_1\cos^2y_1 + \sin^2x_1\sin^2y_1)
-2\cos x_1\cos y_1\cosh^2 x_2
\\
&\;\quad
+ (\cos^2 x_1\cos^2 y_1 - \sin^2 x_1 \sin^2 y_1)
(4\cosh x_2 - 2\cos x_1\cos y_1) \big]
\\
= &\; -\f{2(x_2\cosh x_2-\sinh x_2)}{2\pi x_2^2(\cosh x_2- \cos (x_1-y_1))^2(\cosh x_2- \cos (x_1+y_1))^2}
\\
&\;\cdot \big[(k(x_2)+\cos x_1\cos y_1) (\cosh x_2 - \cos x_1 \cos y_1)^2
\\
&\; \quad
+ (k(x_2) + 2\cosh x_2 - \cos x_1 \cos y_1) \sin^2 x_1 \sin^2 y_1\big].
\end{align*}
This completes the proof.
\end{proof}
\end{lem}

\begin{lem}
\label{lem: bound for k(x_2)}
Let $k = k(x_2)$ be defined in \eqref{eqn: def of k(x_2)}.
Then $k(x_2)\in (2,2\cosh x_2)$ for any $x_2 > 0$.
\begin{proof}
Since for any $x_2 >0$,
\begin{align*}
&\; \f{d}{dx_2}\big[\sinh x_2 \cosh x_2-x_2 - 2(x_2\cosh x_2-\sinh x_2)\big]
\\
= &\; \cosh^2 x_2 + \sinh^2 x_2 - 1 - 2\cosh x_2
- 2x_2\sinh x_2 +2\cosh x_2
\\
= &\; 2\sinh^2 x_2 - 2x_2\sinh x_2 > 0,
\end{align*}
we find that
\[
\sinh x_2 \cosh x_2-x_2 > 2(x_2\cosh x_2-\sinh x_2)
\]
for any $x_2 > 0$.
This implies the desired lower bound.

We similarly note that
\begin{align*}
&\; \f{d}{dx_2}\big[2\cosh x_2(x_2\cosh x_2-\sinh x_2) - (\sinh x_2\cosh x_2 - x_2)\big]\\
= &\; 2\cosh^2 x_2 + 4x_2 \cosh x_2 \sinh x_2 - 3\big(\cosh^2 x_2 + \sinh^2 x_2\big) + 1
\\
= &\; 4x_2 \cosh x_2 \sinh x_2 - 4\sinh^2 x_2
> 0,
\end{align*}
which, by the same argument, gives the upper bound.
\end{proof}
\end{lem}

\section{Numerical Visualizations}\label{sec: numerical}
In this section, we introduce an empirically efficient approach for numerically computing a fixed point $g=\bfR(g)$ for $\mu \geq 0$, and then we present some numerical results on the corresponding V-states with 90-degree corners.

Recall that a fixed point $g$ solves \eqref{eqn: constraint for the boundary curve in Psi}, i.e.,
\[
\Psi(x,g(x);g) = \partial_{x_2}\Psi(\pi,0;g)\cdot \frac{1-e^{-2\mu g(x)}}{2\mu},
\quad x\in\BT.
\]
In the case $\mu=0$, the term $(1-e^{-2\mu g(x)})/(2\mu)$ should be understood as $g(x)$. Inspired by this, we consider an induced dynamic equation
\begin{equation}\label{eqt:numerical dynamic}
\pa_t g(x,t) = \Psi\big(x,g(x,t);g(\cdot,t)\big) - \partial_{x_2}\Psi\big(\pi,0;g(\cdot,t)\big)\cdot \frac{1-e^{-2\mu g(x,t)}}{2\mu},
\end{equation}
with the initial condition $g(x,0) \in\BM$. Apparently, a steady state of \eqref{eqt:numerical dynamic} in $\BM$ is a fixed point of $\bfR$.
Since $\Psi$ is even in $x_1$ and $\pa_{x_1}\Psi(x_1,x_2)\leq 0$ for $(x_1,x_2)\in[0,\pi]\times [0,+\infty)$ (see Proposition \ref{prop: monotonicity of F} and Lemma \ref{lem: monotonicity of h}), the above evolution preserves for all time $t\geq 0$ the properties that $x\mapsto g(x,t)$ is even on $[-\pi,\pi]$ and non-increasing on $[0,\pi]$, and that $g(\pm \pi,t)=0$. Then one can solve \eqref{eqt:numerical dynamic} numerically for $x\in[-\pi,\pi]$ and $t\geq 0$ with suitable initial condition $g(x,0)\in\BM$ using some standard numerical methods.
It is observed that, with suitable space-time discretization and numerical schemes, the numerical solution to \eqref{eqt:numerical dynamic} converges rapidly to an approximate steady state as $t$ increases, in the sense that the residual
\[
\mathrm{Re}(g(\cdot,t)) := \left\|\Psi\big(x,g(x,t);g(\cdot,t)\big) - \partial_{x_2}\Psi\big(\pi,0;g(\cdot,t)\big)\cdot \frac{1-e^{-2\mu g(x,t)}}{2\mu}\right\|_{L^\infty_x(\BT)}
\]
numerically converges to $0$.
Hence, given $\mu$, we can find an approximate fixed point of $\bfR$ by solving \eqref{eqt:numerical dynamic} numerically with a prescribed tolerance on the residual $\mathrm{Re}(g(\cdot,t))$.
We omit the details of the implementation, but let us make a few remarks on the numerical results.
\begin{itemize}
\item Although we have only proved analytically the existence of a fixed point $g\in \BM\cap C(\BT)$ for $\mu\in[0,\mu_0]$ with a possibly small $\mu_0$, our numerical method manages to find such fixed points for all $\mu\in[0,\f12]$.

\item
We observe that, for each $\mu\in[0,\f12]$, the numerical solutions to \eqref{eqt:numerical dynamic} starting from different initial data in $\BM$ all converge to the same numerical steady state (up to a very small discretization error).
This suggests the uniqueness of the fixed point $g=\bfR(g)$ in $\BM$ for each $\mu$, although the uniqueness remains open.
We shall denote the fixed point of $\bfR$ for each $\mu$ by $g_\mu = g_\mu(x)$.

\item It is also numerically observed that the graph of $g_\mu(x)$ on $[-\pi,\pi]$ meets the horizontal axis at $(\pm \pi,0)$ with 45-degree angles.
This can be heuristically explained as follows.
Let $g(x,t)$ be a solution to \eqref{eqt:numerical dynamic} with some suitable initial condition $g(x,0)\in \BM$.
Denote $r = r(x,t) := \sqrt{(\pi-x)^2+g(x,t)^2}$. Combining \eqref{eqt:numerical dynamic} and \eqref{eqt: Psi local expansion} in Lemma \ref{lem: local expansion}, for $r(x,t)\ll 1$ (which gives $|\pi-x|\ll 1$), we find that
\[
\left|\pa_t g(x,t) - \frac{(\pi-x)^2-g(x,t)^2}{2\pi} \cdot A(r(x,t);g(\cdot,t)) \cdot |\ln r(x,t)|\right| \leq C r(x,t)^2,
\]
where $A(r;g)$ was defined in \eqref{eqt: def of A(r;g)}. With slight abuse of the notation, let $\kappa = \kappa(x,t) := g(x,t)/(\pi-x)$.
Suppose that for all $t$ of interest and for all $x\in(0,\pi)$ sufficiently close to $\pi$, $\kappa(x,t)/(1+\kappa(x,t)^2)>\delta$ for some constant $\delta>0$.
This implies $A(r(x,t);g(\cdot,t))>C_\delta$ for some $C_\delta>0$ (see \eqref{eqt: def of A(r;g)} and also \eqref{eqn: lower bound for A r g}) whenever $|\pi-x|\ll 1$.
It follows from the preceding inequality and a direct calculation that
\[
\pa_t(1-\kappa)^2 \leq -\tilde C_\delta(1-\kappa)^2\cdot r|\ln r| + C r|1-\kappa^2|,
\]
where $\tilde C_\delta,\, C>0$.
The first term on the right-hand side dominates when $\tilde C_\delta|1-\kappa||\ln r| \geq 2 C (1+\kappa)$, or equivalently,
\[
\left|\f{1-\kappa}{1+\kappa}\right| \geq \f{2C}{\tilde C_\delta |\ln r|}.
\]
This drives $\kappa$ to approach $1$ whenever $r\ll 1$, which explains our numerical observation.
\end{itemize}
Once such $g_\mu \in \BM\cap C(\BT)$ is obtained, in the case $\mu = \f{1}{m}$ with $m\in \BZ_+$, we can use \eqref{eqn: transformation from g to f} to define $f$ and then use \eqref{eqn: form of D_0} to obtain the corresponding V-state $D_0 = D_0(f)$.
Thanks to the above-mentioned property of $g_\mu$ near the end-points $\pm \pi$, the resulting V-state $D_0$ has 90-degree corners.

Below, let us illustrate the fixed points $g_\mu$ for different values of $\mu=\f{1}m$ along with the corresponding V-states with 90-degree corners. All the results are produced by the numerical method introduced above with a stopping criterion that the residual $\mathrm{Re}(g(\cdot,t))$ should drop below $10^{-6}$ under sufficiently fine discretization.

In Figure \ref{fig:multiple}, we plot the graphs of the fixed points $g_\mu$ on $[-\pi,\pi]$ for several different values of $\mu=\f{1}m$ (including $\mu=0$, which formally corresponds to $m=\infty$). The numerical result suggests that $g_\mu$ depends continuously and monotonically on the parameter $\mu$, and that $g_\mu$ converges to $g_0$ as $\mu\to 0^+$. However, rigorous justification of such observations has not been achieved.

\begin{figure} \centering
\includegraphics[width=0.9\textwidth]{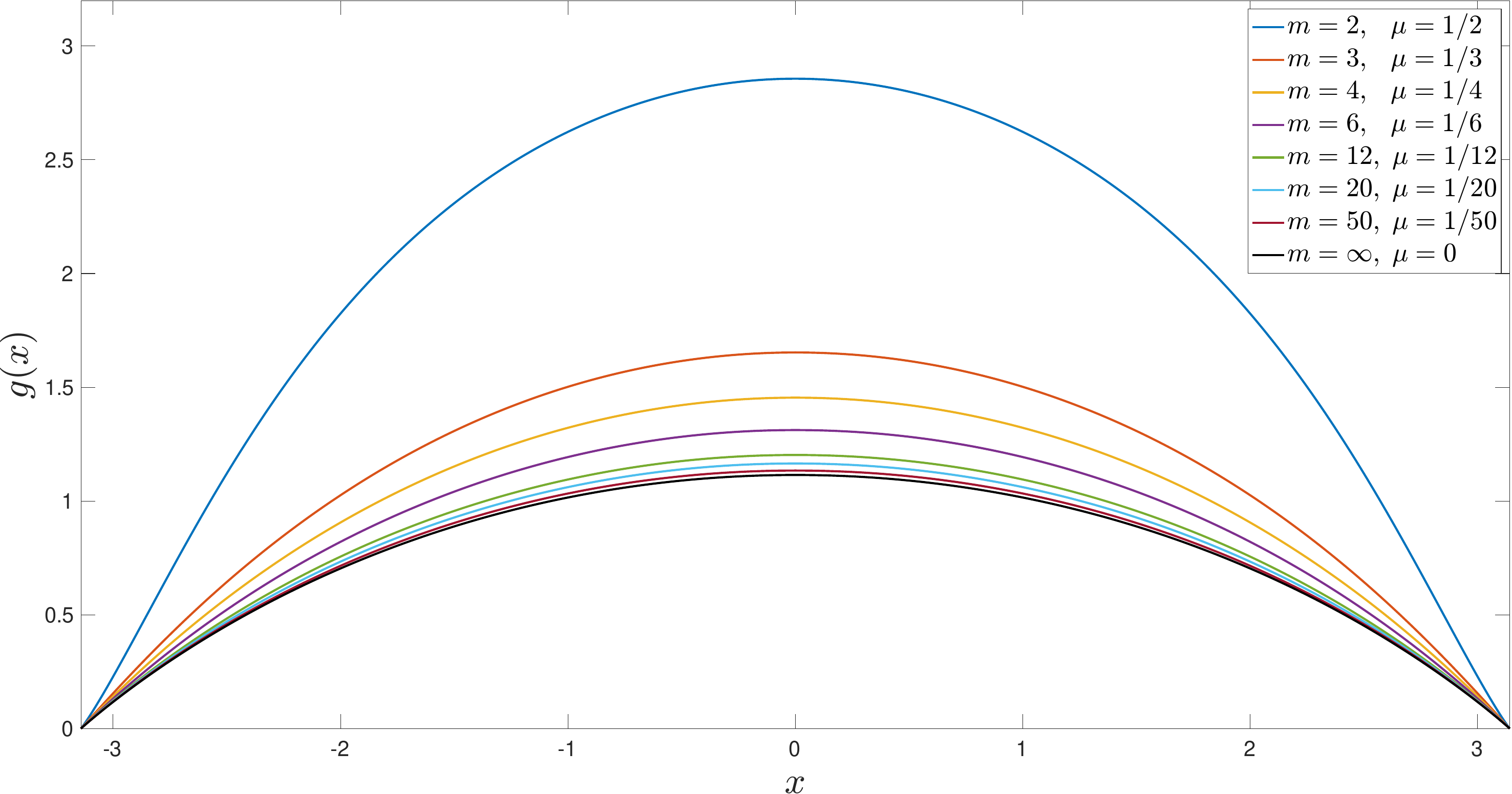}
\caption[Multiple fixed points]{\small Graphs of the fixed points $g_\mu(x)$ on $[-\pi,\pi]$ for different values of $\mu=\f{1}m$.}
\label{fig:multiple}
\end{figure}

Figure \ref{fig:12-fold} shows (in the case $m=12$) a 12-fold V-state with 90-degree corners corresponding to the fixed point $g_{1/12}$, together with multiple level sets (i.e., streamlines) of the corresponding modified stream function $\phi_\dagger$ (see the definition in \eqref{eqn: modified stream function repeat}). One can clearly see that the corners are of 90 degrees by zooming in sufficiently around one of the corner points. Figure \ref{fig:2-fold} exhibits a similar set of plots in the case $m=2$ in order to illustrate that our numerical method also applies to the case with a relatively large $\mu$ ($\mu = \f12$, which is beyond our analytic result in Theorem \ref{thm: main existence theorem}). Again, one needs to zoom in around a corner point to see the 90-degree angle.
It is worth mentioning that this vortex patch was reported before in e.g.\;\cite{CerretelliWilliamson2003,LuzzattoFegizWilliamson2010}.

Figure \ref{fig:m-fold} displays the $m$-fold V-states with 90-degree corners corresponding to the fixed points $g_{1/m}$ for $m\in \{3,4,5,6,20,50\}$. In each sub-figure, the dashed black curve represents the unit circle; the solid blue curve is the boundary of the V-state and also the part of $\s_0$ (which is a special level set of the modified stream function $\phi_\dag$ introduced in \eqref{eqn: general level set in main thm}) lying inside the unit circle; and the solid red curve is the part of $\s_0$ outside the unit circle.
The numerical results for $m = 3,4,5,6$ agree very well with those in \cite{HassainiaMasmoudiWheeler2020} (also see the earlier results in \cite{WuOvermanZabusky1984}).


\begin{figure}[!p] \centering
    \begin{subfigure}[b]{0.42\textwidth}
        \includegraphics[width=0.98\textwidth]{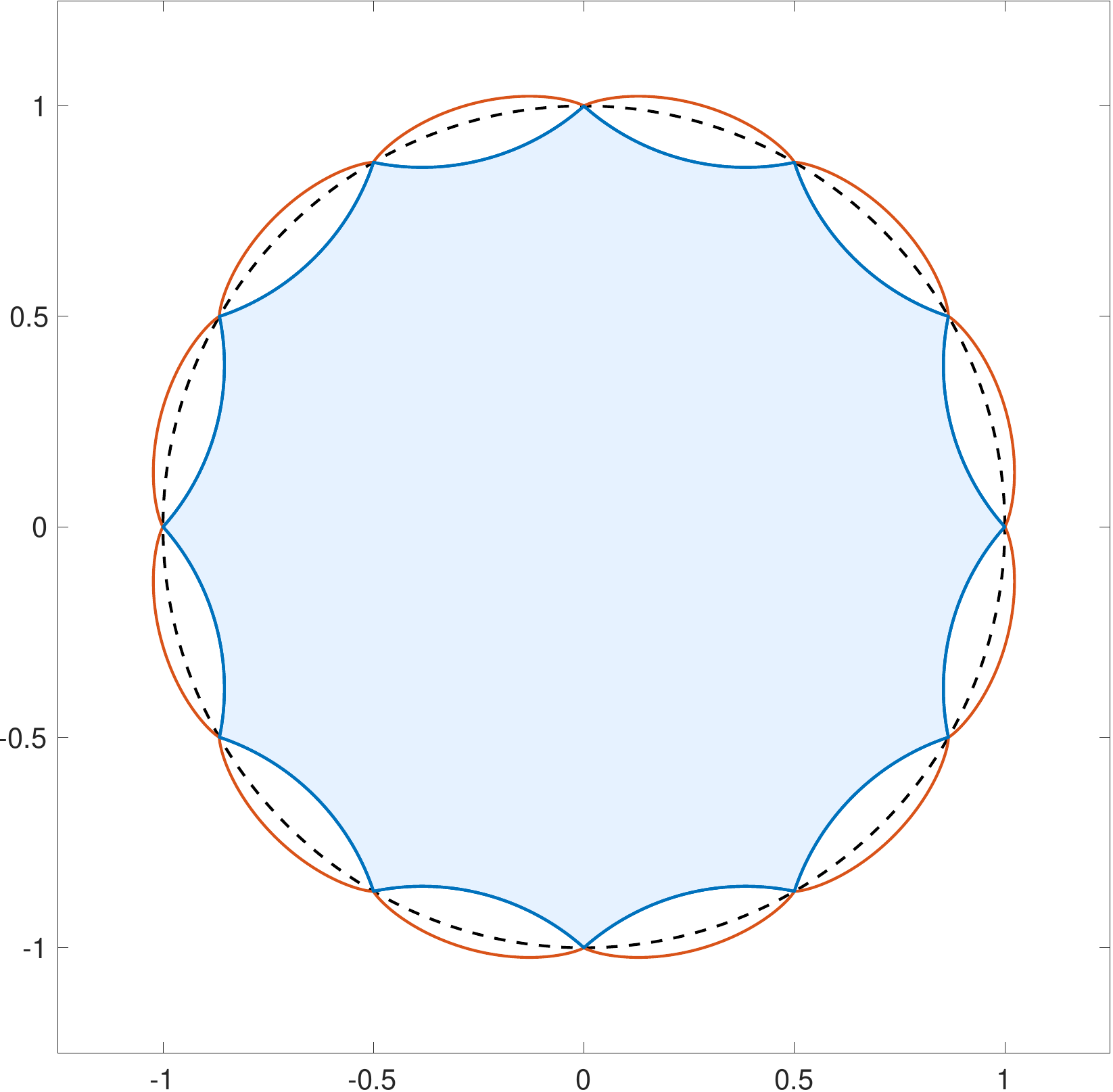}
        \caption{\small A $12$-fold V-state with $90$-degree corners}
    \end{subfigure}
    \begin{subfigure}[b]{0.42\textwidth}
        \includegraphics[width=0.98\textwidth]{contour_m=12}
        \caption{\small Level sets of $\phi_{\dagger}$}
    \end{subfigure}\\
    \vspace{3mm}
    \begin{subfigure}[b]{0.42\textwidth}
        \includegraphics[width=1\textwidth]{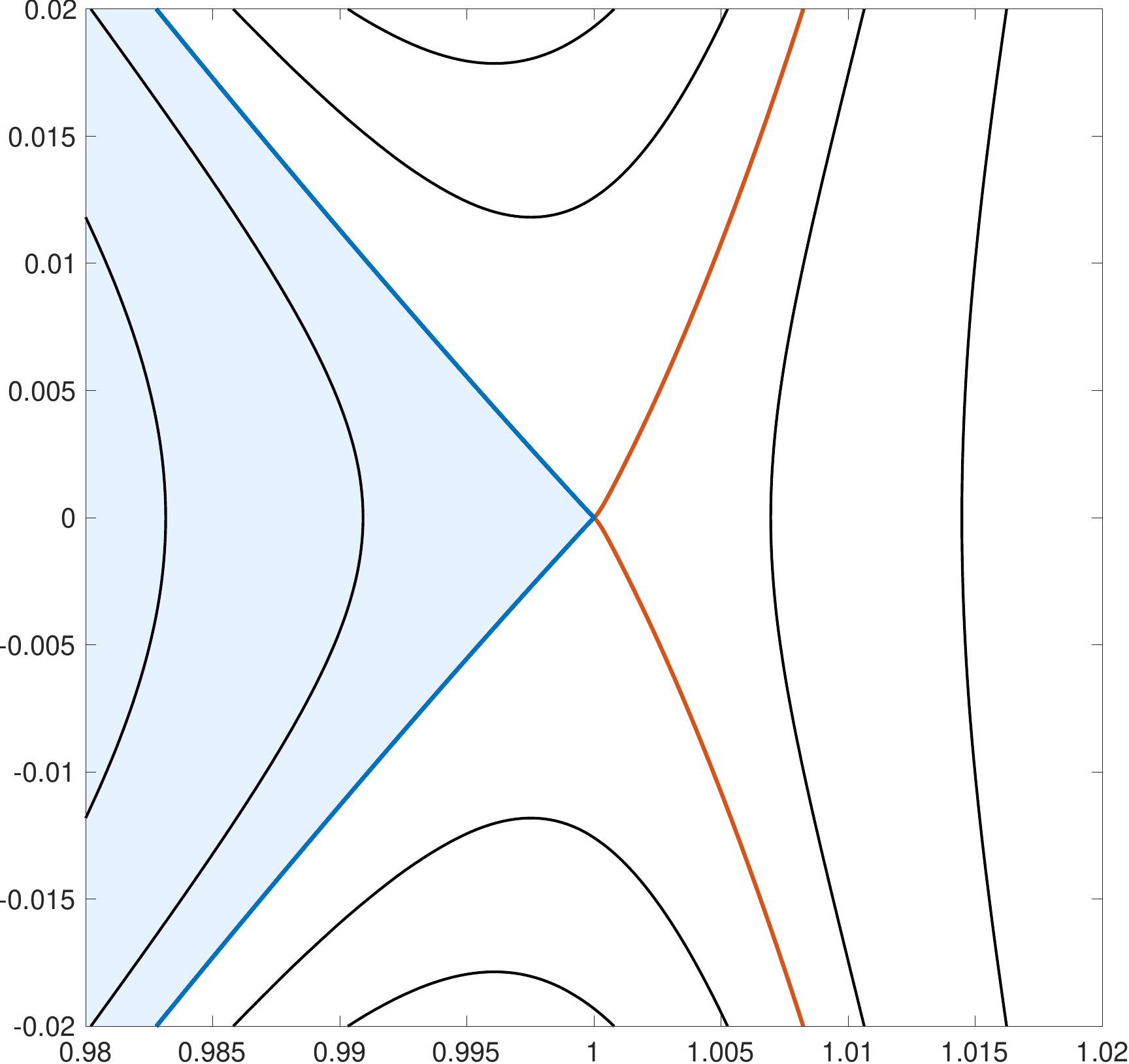}
        \caption{\small A zoom-in around $(1,0)$}
    \end{subfigure}
    \begin{subfigure}[b]{0.41\textwidth}
        \includegraphics[width=1\textwidth]{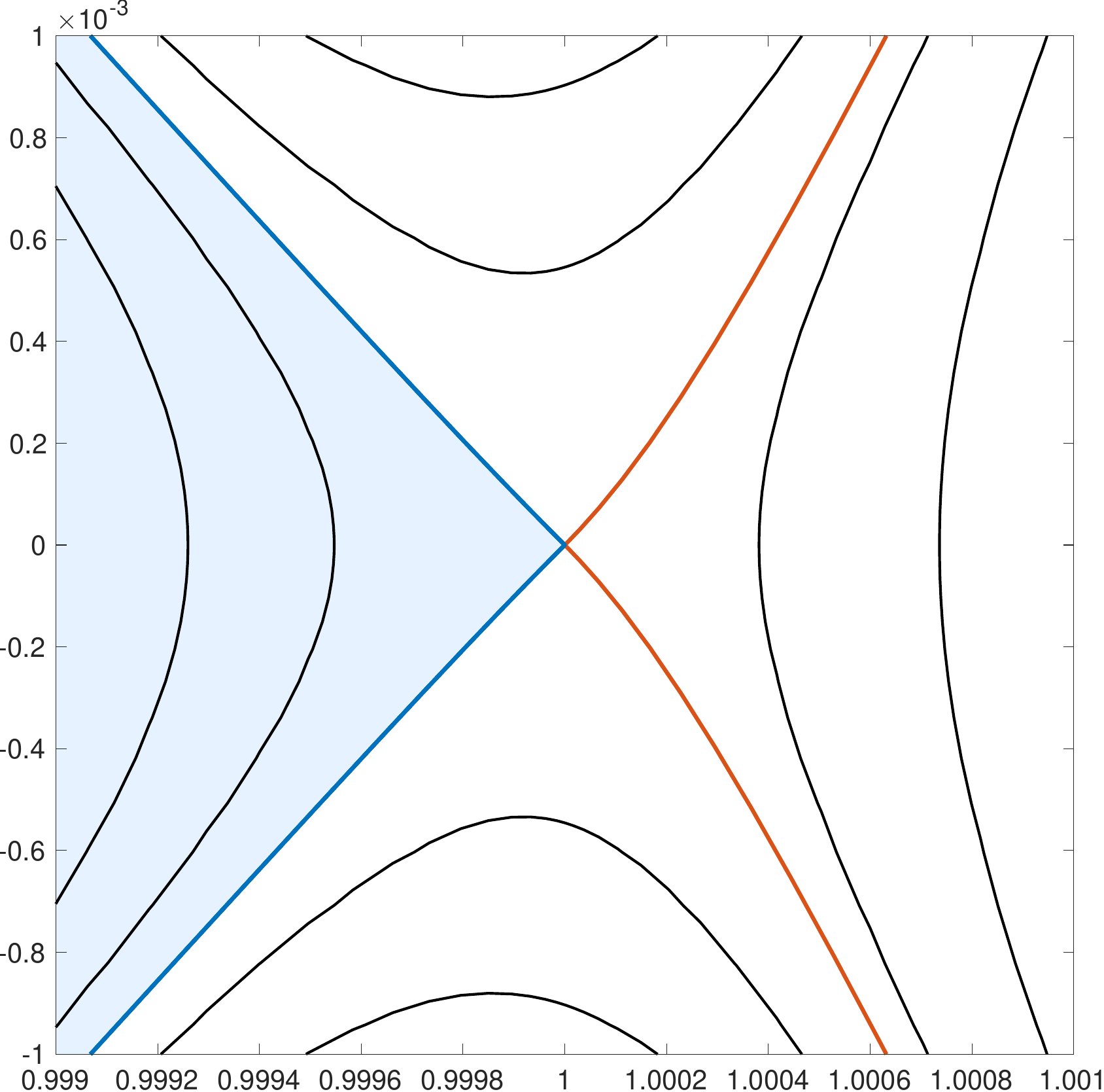}
        \caption{\small A further zoom-in around $(1,0)$}
    \end{subfigure}
    \caption[12-fold]{\small (A) A 12-fold uniformly rotating V-state with 90-degree corners. The dashed black curve represents the unit circle; the light blue region indicates the interior of the patch $D_0$ where $\om\equiv 1$; the solid blue curve is the boundary of the V-state and also the part of $\s_0$ (defined in \eqref{eqn: general level set in main thm}) inside the unit circle; and the solid red curve is the part of $\s_0$ outside the unit circle. (B) Level sets of $\phi_\dagger$. The solid black curves stand for selected level sets of $\phi_\dagger$ other than $\s_0$. (C) A zoom-in view around the point $(1,0)$, one of the corner points of the patch. (D) A further zoom-in view at a smaller scale around the point $(1,0)$. One can see that the blue curve and the red curve meet at the corner point and form four 90-degree angles (see Remark \ref{rmk: 4 90-degree angles at corner}).}
    \label{fig:12-fold}
\end{figure}

\begin{figure}[!p] \centering
    \begin{subfigure}[b]{0.42\textwidth}
        \includegraphics[width=0.98\textwidth]{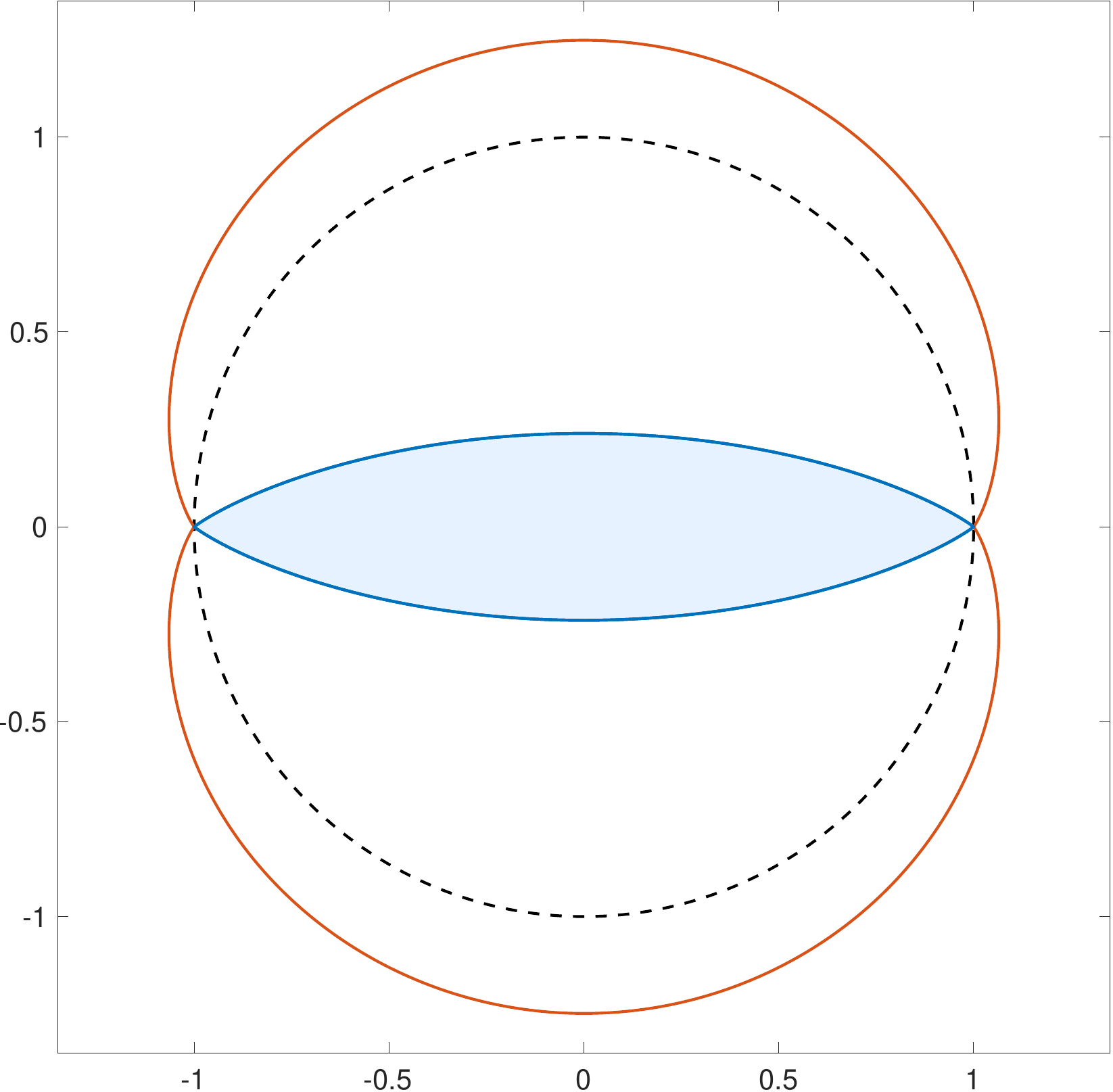}
        \caption{\small A 2-fold V-state with 90-degree corners}
    \end{subfigure}
    \begin{subfigure}[b]{0.42\textwidth}
        \includegraphics[width=0.98\textwidth]{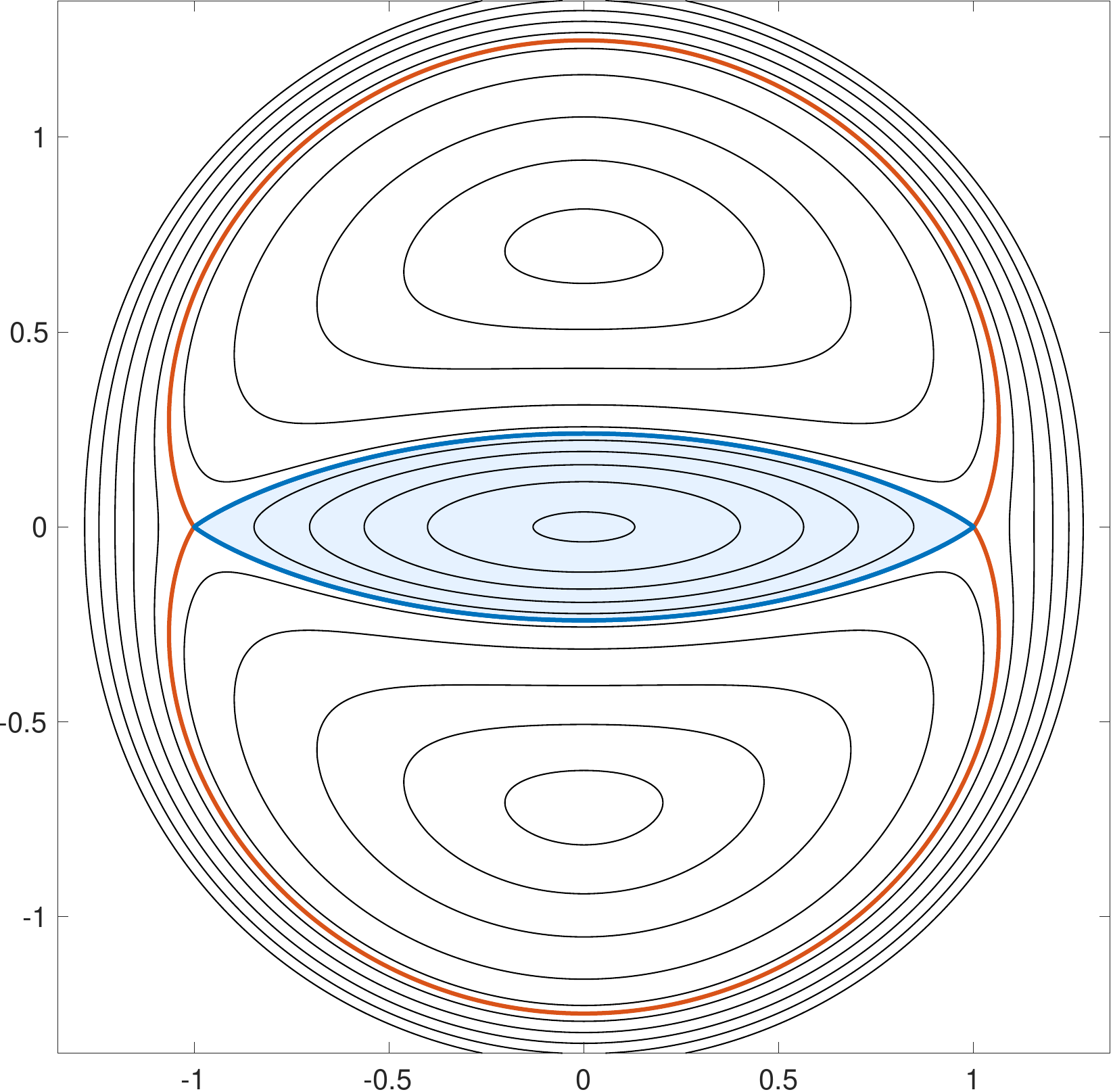}
        \caption{\small Level sets of $\phi_{\dagger}$}
    \end{subfigure}\\
    \vspace{3mm}
    \begin{subfigure}[b]{0.42\textwidth}
        \includegraphics[width=1\textwidth]{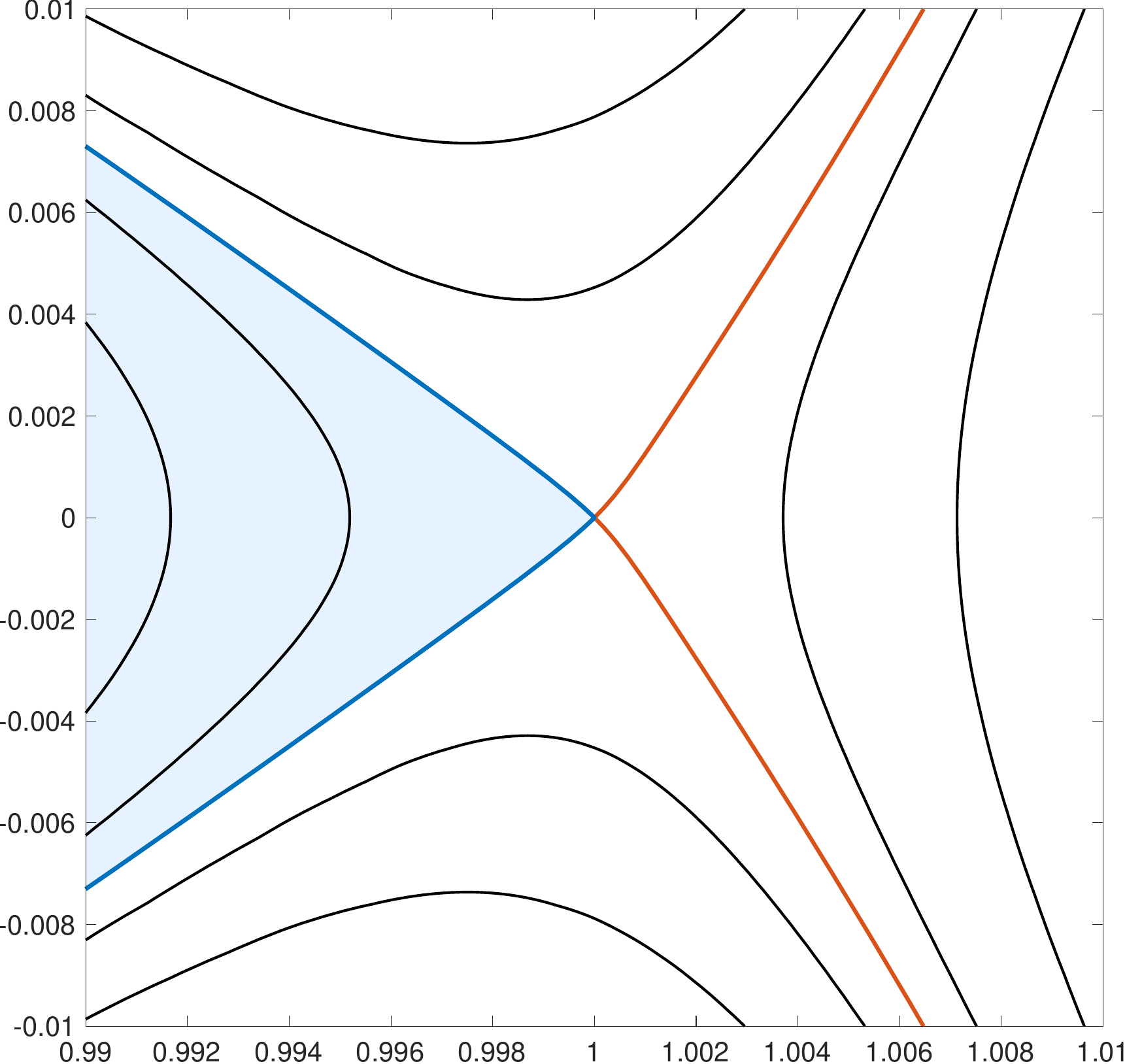}
        \caption{\small A zoom-in around $(1,0)$}
    \end{subfigure}
    \begin{subfigure}[b]{0.41\textwidth}
        \includegraphics[width=1\textwidth]{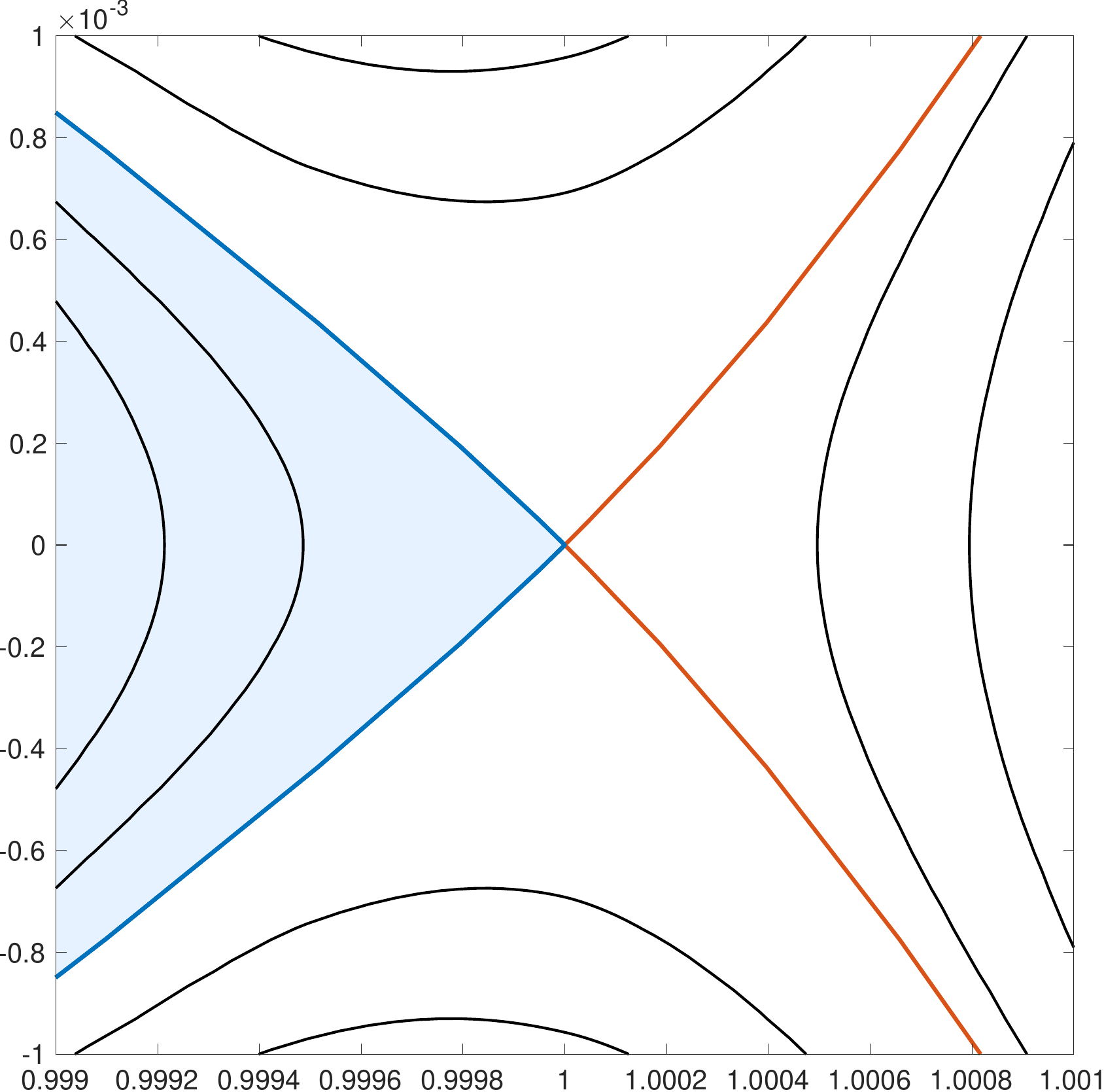}
        \caption{\small A further zoom-in around $(1,0)$}
    \end{subfigure}
    \caption[2-fold]{\small (A) A 2-fold uniformly rotating V-state with 90-degree corners. The dashed black curve represents the unit circle; the light blue region indicates the interior of the patch $D_0$ where $\om\equiv 1$; the solid blue curve is the boundary of the V-state and also the part of $\s_0$ (defined in \eqref{eqn: general level set in main thm}) inside the unit circle; and the solid red curve is the part of $\s_0$ outside the unit circle. (B) Level sets of $\phi_\dagger$. The solid black curves stand for selected level sets of $\phi_\dagger$ other than $\s_0$. (C) A zoom-in view around the point $(1,0)$, one of the corner points of the patch. (D) A further zoom-in view at a smaller scale around the point $(1,0)$.}
    \label{fig:2-fold}
\end{figure}

\begin{figure}[!p] \centering
    \begin{subfigure}[b]{0.35\textwidth}
        \includegraphics[width=1\textwidth]{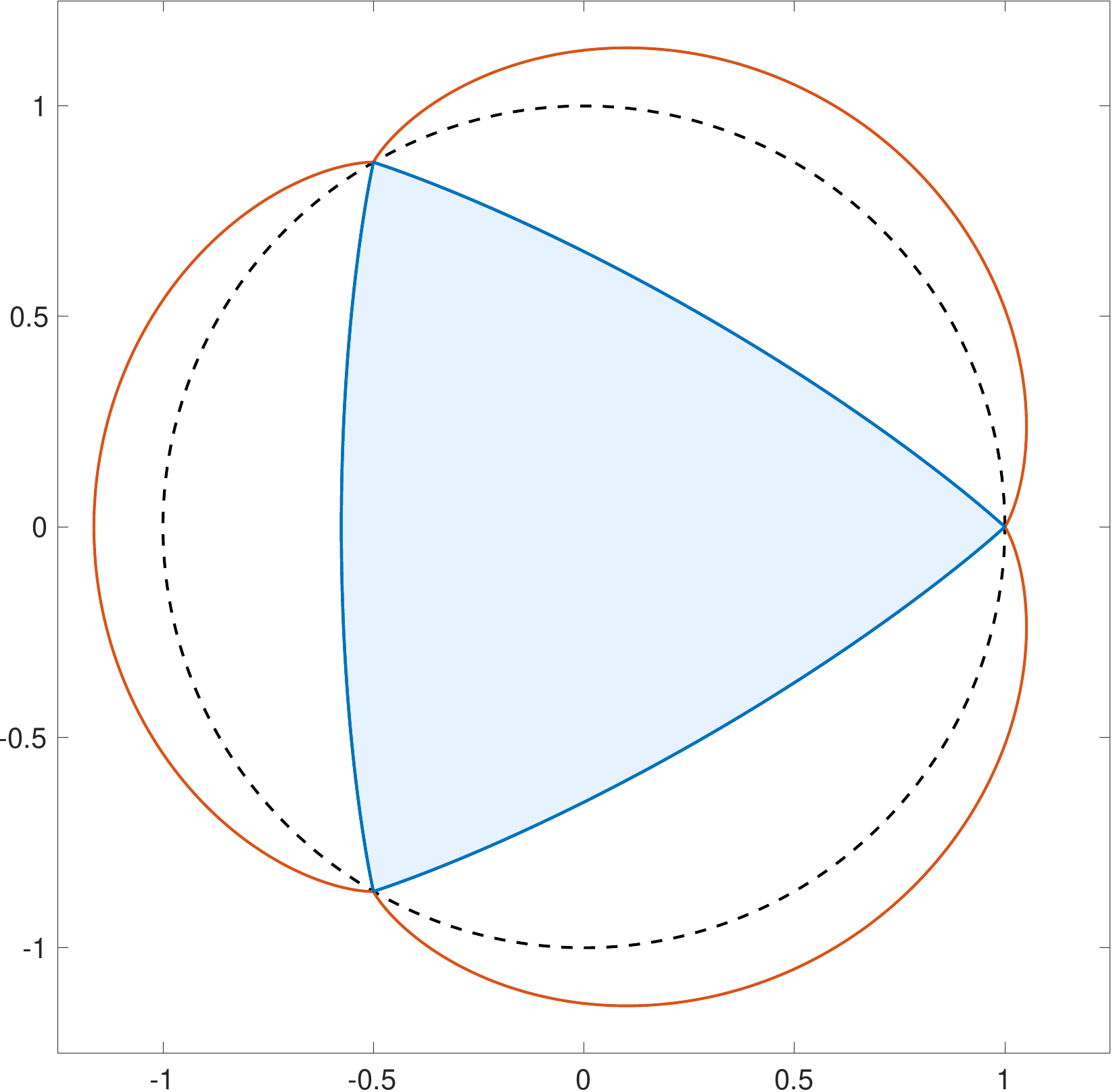}\vspace{-1.5mm}
        \caption{\small $m=3$}
    \end{subfigure}\qquad\quad
    \begin{subfigure}[b]{0.35\textwidth}
        \includegraphics[width=1\textwidth]{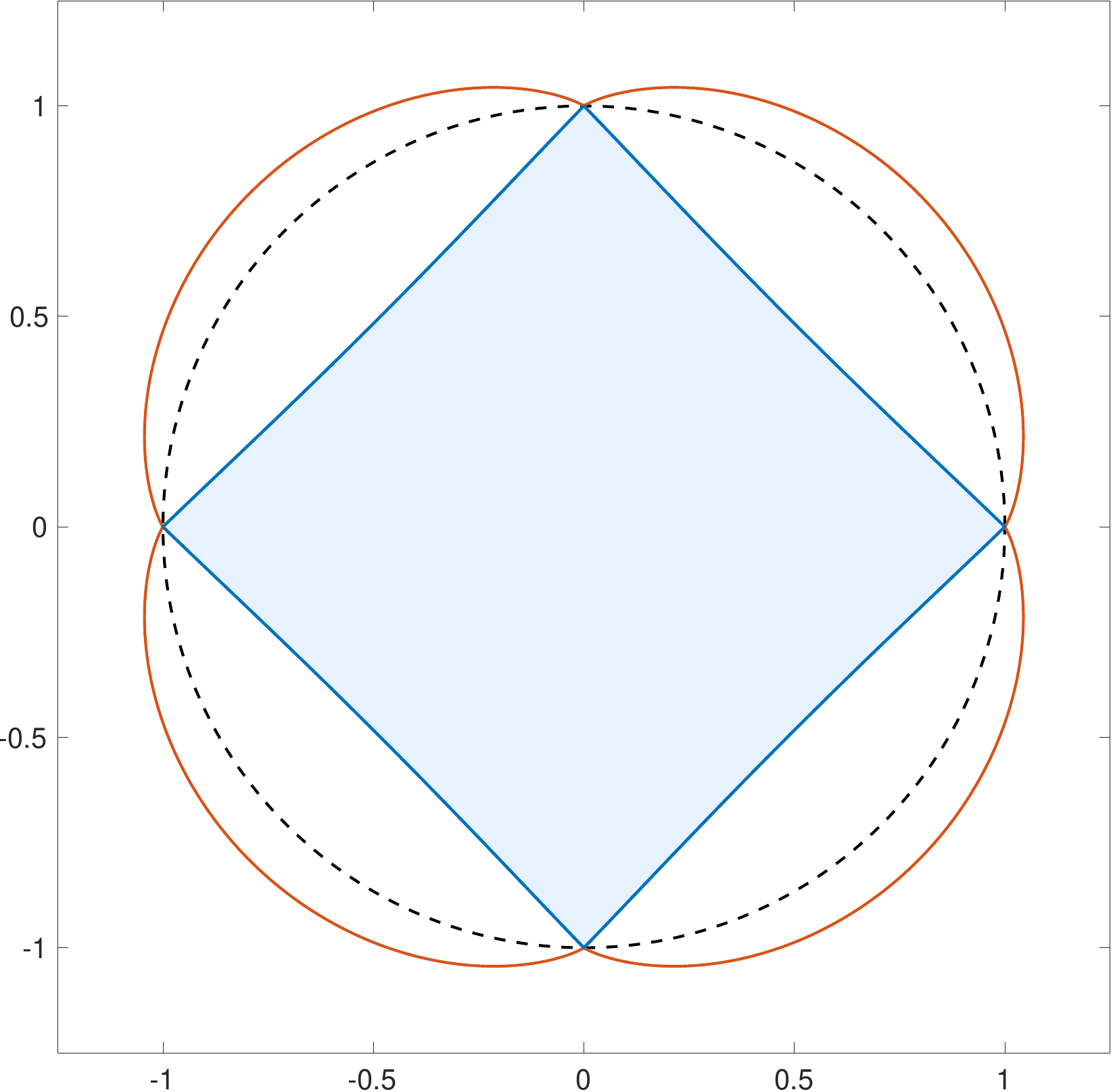}\vspace{-1.5mm}
        \caption{\small $m=4$}
    \end{subfigure}\\
    \vspace{2mm}
    \begin{subfigure}[b]{0.35\textwidth}
        \includegraphics[width=1\textwidth]{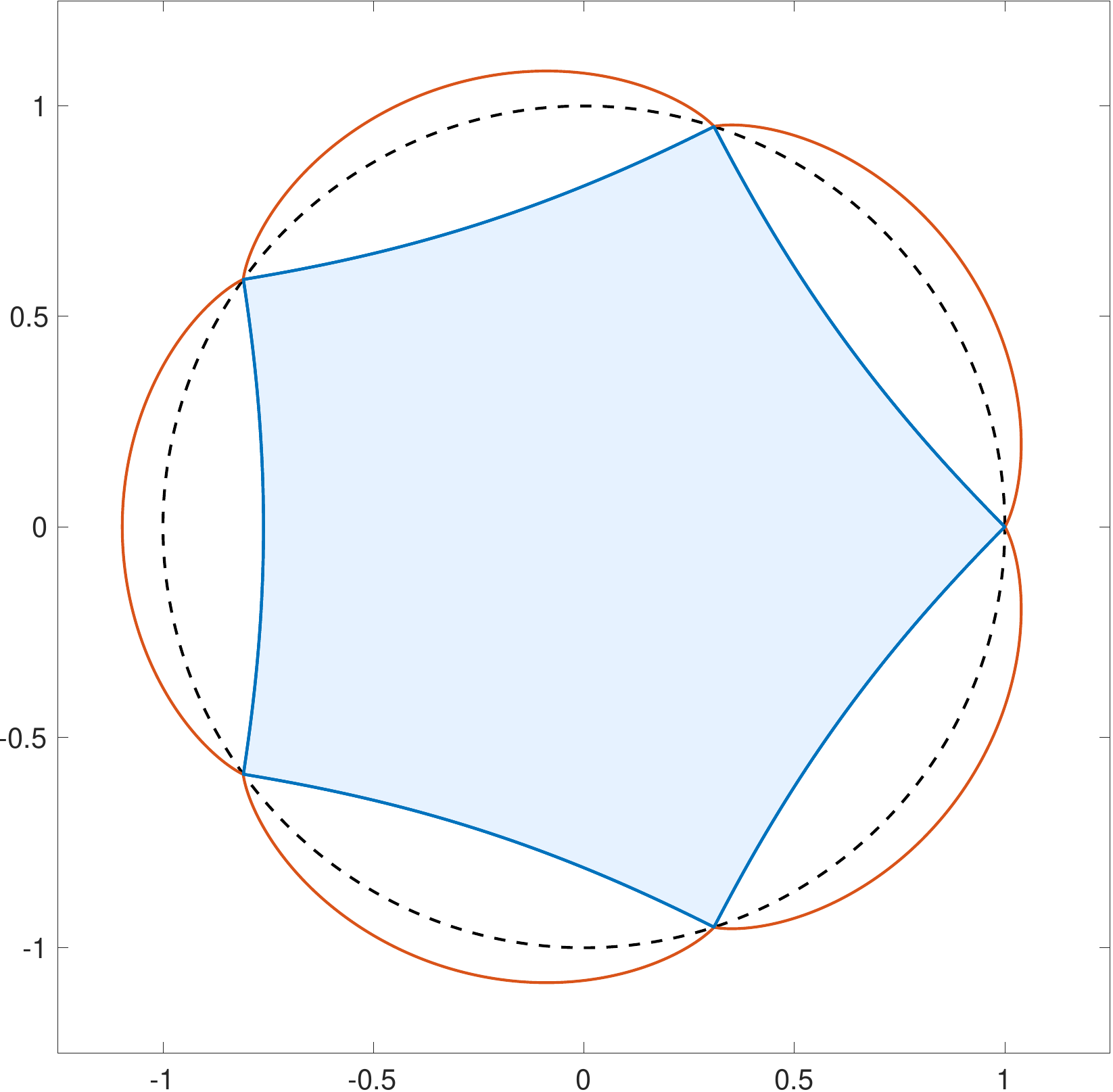}\vspace{-1.5mm}
        \caption{\small $m=5$}
    \end{subfigure}\qquad\quad
    \begin{subfigure}[b]{0.35\textwidth}
        \includegraphics[width=1\textwidth]{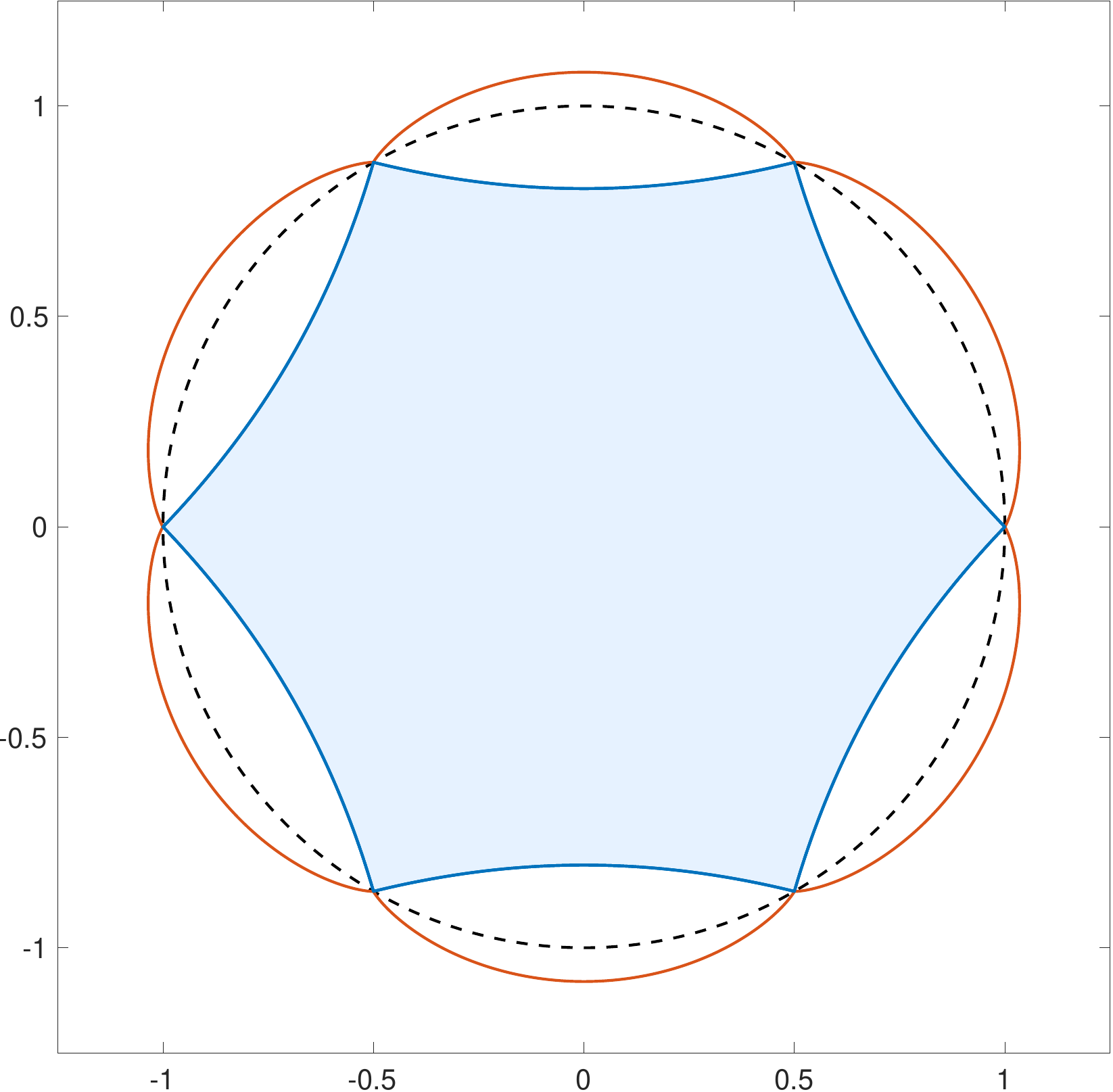}\vspace{-1.5mm}
        \caption{\small $m=6$}
    \end{subfigure}\\
    \vspace{2mm}
    \begin{subfigure}[b]{0.35\textwidth}
        \includegraphics[width=1\textwidth]{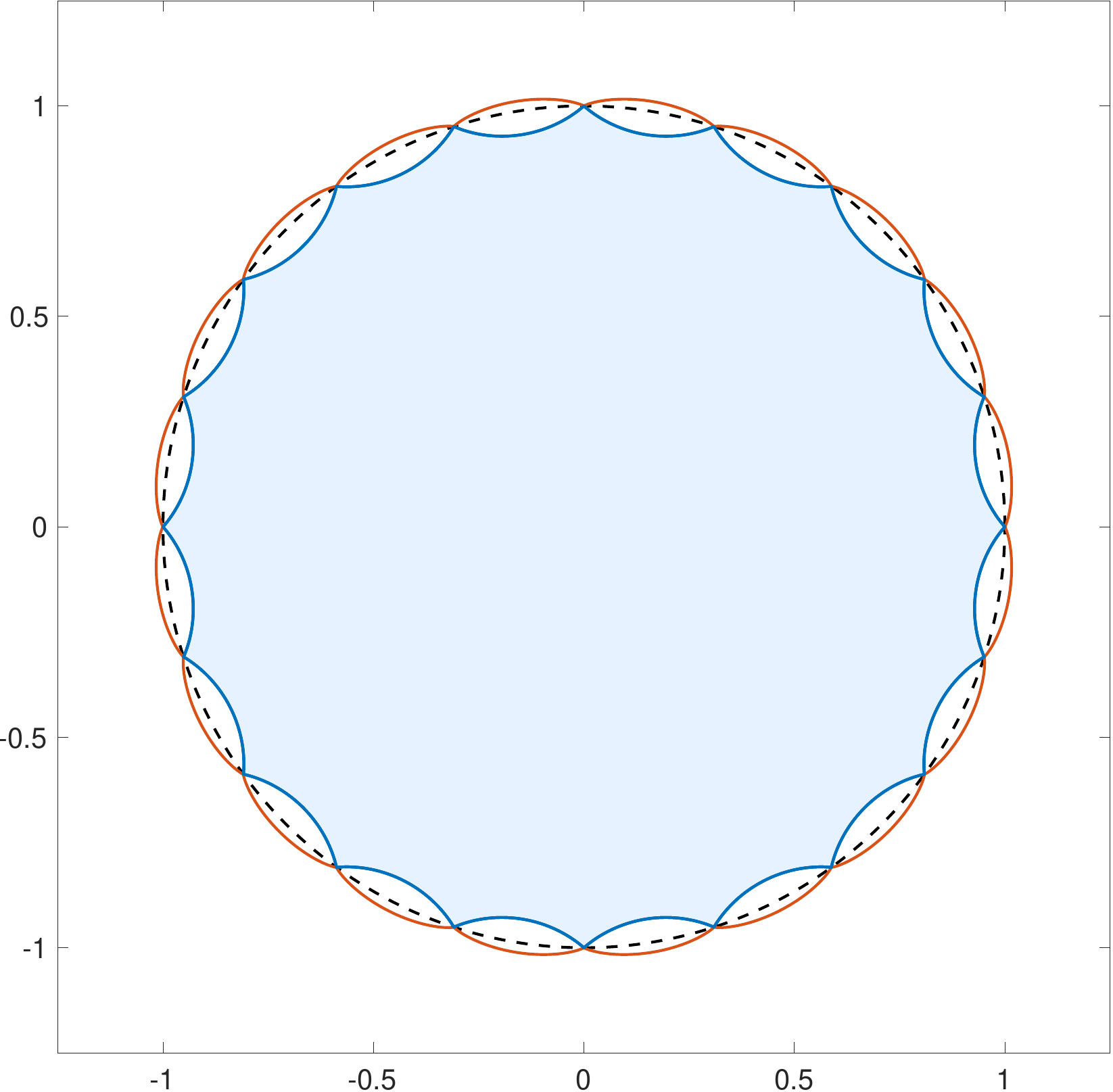}\vspace{-1.5mm}
        \caption{\small $m=20$}
    \end{subfigure}\qquad\quad
    \begin{subfigure}[b]{0.35\textwidth}
        \includegraphics[width=1\textwidth]{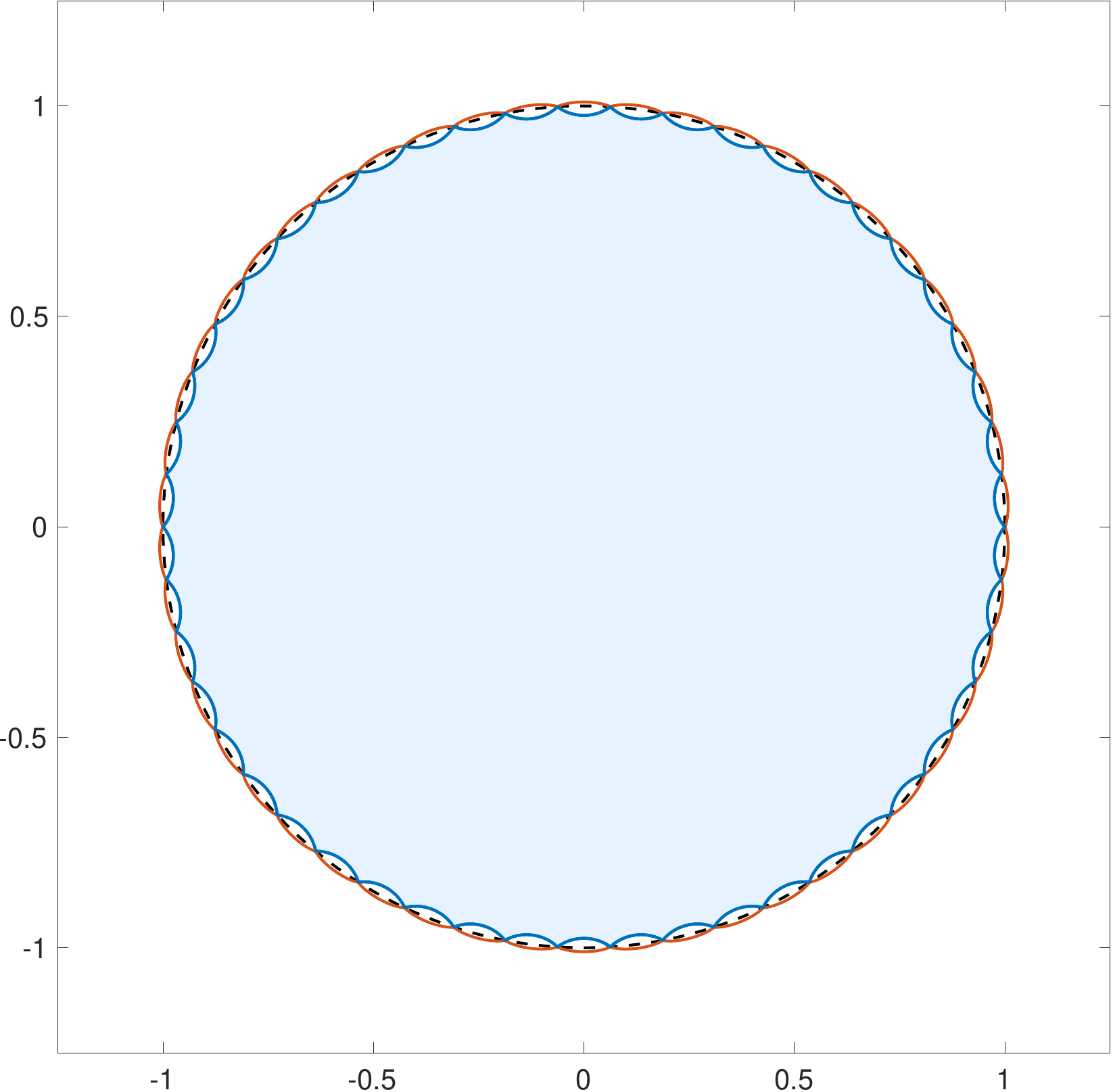}\vspace{-1.5mm}
        \caption{\small $m=50$}
    \end{subfigure}
    \vspace{-3mm}
    \caption[M-fold]{\small $m$-fold symmetric V-states with 90-degree corners for different $m$'s.
    In each sub-figure, the dashed black curve represents the unit circle; the light blue region indicates the interior of the patch $D_0$ where $\om\equiv 1$; the solid blue curve is the boundary of the V-state and also the part of $\s_0$ (defined in \eqref{eqn: general level set in main thm}) inside the unit circle; and the solid red curve is the part of $\s_0$ outside the unit circle.
    }
    \label{fig:m-fold}
\end{figure}

\end{document}